\documentclass[11pt]{report}
\usepackage{amssymb,amsmath,amsfonts,amsthm,mathrsfs}
\usepackage{graphicx,graphics,psfrag,epsfig,cancel}
\usepackage[colorlinks=true, pdfstartview=FitV, linkcolor=blue, citecolor=red, urlcolor=blue]{hyperref}
\usepackage{dsfont}

\usepackage{caption,subfig,color}
\usepackage{tikz,tkz-tab}
\usetikzlibrary{calc}

\usetikzlibrary{arrows,shapes,positioning}
\usetikzlibrary{decorations.markings}
\tikzstyle arrowstyle=[scale=1.25]
\tikzstyle directed=[postaction={decorate,decoration={markings,
    mark=at position .6 with {\arrow[arrowstyle]{stealth}}}}]
\tikzstyle reverse directed=[postaction={decorate,decoration={markings,
    mark=at position .5 with {\arrowreversed[arrowstyle]{stealth};}}}]

\usepackage{geometry}
\newtheorem{theorem}{Theorem}[chapter]
\newtheorem{assumption}{Assumption}
\newtheorem{remark}{Remark}[chapter]
\newtheorem{lemma}{Lemma}[chapter]
\newtheorem{proposition}{Proposition}[chapter]
\newtheorem{corollary}{Corollary}[chapter]

\newcommand{\N}{{\mathbb N}}
\newcommand{\Z}{{\mathbb Z}}
\newcommand{\R}{{\mathbb R}}
\newcommand{\C}{{\mathbb C}}

\newcommand{\A}{{\mathbb A}}
\newcommand{\B}{{\mathbb B}}
\newcommand{\E}{{\mathbb E}}

\newcommand{\M}{{\mathbb M}}
\newcommand{\D}{{\mathbb D}}
\newcommand{\Dbar}{\overline{\mathbb D}}
\newcommand{\U}{{\mathscr U}}
\newcommand{\Ubar}{\overline{\mathscr U}}
\newcommand{\cercle}{{\mathbb S}^1}

\newcommand{\cF}{\mathscr{F}}
\newcommand{\G}{{\mathcal G}}
\newcommand{\cG}{{\mathscr G}}
\newcommand{\cL}{{\mathscr L}}

\newcommand{\md}{\mathrm{d}}
\newcommand{\bqs}{\begin{equation*}}
\newcommand{\eqs}{\end{equation*}}
\newcommand{\bqq}{\begin{equation}}
\newcommand{\eqq}{\end{equation}}
\renewcommand{\Re}{\mathrm{Re}}
\renewcommand{\Im}{\mathrm{Im}}
\newcommand{\mbi}{\mathbf{i}}
\newcommand{\bfh}{\mathbf{h}}
\newcommand{\rme}{\mathrm{e}}

\begin{document}

\title{Nonlinear orbital stability of stationary shock profiles\\
for the Lax-Wendroff scheme}

\author{Jean-Fran\c{c}ois {\sc Coulombel} \& Gr\'egory {\sc Faye}\thanks{Institut de Math\'ematiques de Toulouse - UMR 5219, Universit\'e de 
Toulouse ; CNRS, Universit\'e Paul Sabatier, 118 route de Narbonne, 31062 Toulouse Cedex 9 , France. Research of J.-F. C. was supported 
by ANR project HEAD under grant agreement ANR-24-CE40-3260. G.F. acknowledges support from Labex CIMI under grant agreement 
ANR-11-LABX-0040, from ANR project Indyana under grant agreement ANR-21-CE40-0008 and from ANR project HEAD under grant 
agreement ANR-24-CE40-3260. Emails: {\tt jean-francois.coulombel@math.univ-toulouse.fr}, {\tt gregory.faye@math.univ-toulouse.fr}}}
\date{\today}
\maketitle

\begin{abstract}
In this article we study the spectral, linear and nonlinear stability of stationary shock profile solutions to the Lax-Wendroff scheme for hyperbolic 
conservation laws. We first clarify the spectral stability of such solutions depending on the convexity of the flux for the underlying conservation 
law. The main contribution of this article is a detailed study of the so-called Green's function for the linearized numerical scheme. As evidenced 
on numerical simulations, the Green's function exhibits a highly oscillating behavior ahead of the leading wave before this wave reaches the 
shock location. One of our main results gives a quantitative description of this behavior. Because of the existence of a one-parameter family 
of stationary shock profiles, the linearized numerical scheme admits the eigenvalue $1$ that is embedded in its continuous spectrum, which 
gives rise to several contributions in the Green's function. Our detailed analysis of the Green's function describes these contributions by means 
of a so-called activation function. For large times, the activation function describes how the mass of the initial condition accumulates along the 
eigenvector associated with the eigenvalue $1$ of the linearized numerical scheme. We can then obtain sharp decay estimates for the linearized 
numerical scheme in polynomially weighted spaces, which in turn yield a nonlinear orbital stability result for spectrally stable stationary shock profiles. 
This nonlinear result is obtained despite the lack of uniform $\ell^1$ estimates for the Green's function of the linearized numerical scheme, the lack 
of such estimates being linked with the dispersive nature of the numerical scheme. This dispersive feature is in sharp contrast with previous results 
on the orbital stability of traveling waves or discrete shock profiles for parabolic perturbations of conservation laws.
\end{abstract}

\bigskip

\bigskip

\noindent For the purpose of open access, this work is distributed under a Creative Commons Attribution | 4.0 International licence: 
\href{https://creativecommons.org/licenses/by/4.0/}{https://creativecommons.org/licenses/by/4.0/}.

\begin{flushleft}
\includegraphics{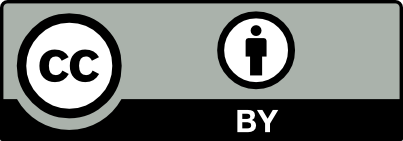}
\end{flushleft}

\newpage

\tableofcontents

%%%%%%%%%%%%%%%%%%%%%%%%%%%%%%%%%%%%%%%%%%%%%%%%%%%%%%%%%%%%%%%%%%%%%%%%
\chapter*{Notation}

Throughout this article, we let $\N^*$ denote the set of positive integers $\{ 1,2,3,\dots \}$ and $\N$ denote the set of integers 
(including $0$). We let $\R$ denote the set of real numbers, $\R^+:=[0,+\infty)$ the set of nonnegative real numbers and 
$\R^{+*}:=(0,+\infty)$ the set of positive numbers. We also let $\C$ denote the set of complex numbers. We shall use the 
notation:
\begin{align*}
&\U:= \{\zeta\in\C\,|\,|\zeta|>1\}\,,\quad\D:=\{\zeta\in\C\,|\,|\zeta|<1\}\,,\quad\cercle:=\{\zeta\in\C\,|\,|\zeta|=1\}\,,\\
&\Ubar:=\U\cup\cercle\,,\quad\Dbar:=\D\cup\cercle\,.
\end{align*}
If $w$ is a complex number, the notation $B_r(w)$ stands for the (round) open ball in $\C$ centered at $w$ and with radius $r>0$, that is $B_r(w) 
:= \{ z \in \C \, | \, |z-w|<r \}$. We shall also use ``square balls'' in the complex plane. Namely, for any $r>0$, the notation $\mathbf{B}_r(w)$ stands 
for the open set:
$$
\mathbf{B}_r(w) \, := \Big\{ z \in \C \, \big| \, \max \big( |\text{\rm Re} (z-w)|,|\text{\rm Im} (z-w)| \big)<r \Big\} \, .
$$
The notation $\overline{B_r(w)}$, resp. $\overline{\mathbf{B}_r(w)}$, refers to the closure of $B_r(w)$, resp. $\mathbf{B}_r(w)$, in $\C$. 
For any positive number $r$, we let $r \, \cercle$ denote the circle in $\C$ centered at the origin and with radius $r$. We also let 
${\mathscr M}_{n,k} (\C)$ denote the set of $n \times k$ matrices with complex entries. If $n=k$, we simply write ${\mathscr M}_n (\C)$.
\bigskip

For $q \in [1,+\infty)$, we let $\ell^q(\Z;\C)$ denote the space of complex valued sequences $\mathbf{v}=(v_j)_{j\in \Z}$ indexed by $\Z$ and 
such that the quantity:
$$
\sum_{j \in \Z} \, |v_j|^q
$$
is finite. The $1/q$-th power of this quantity defines a norm with which $\ell^q(\Z;\C)$ becomes a Banach space. This norm is denoted 
$\| \cdot \|_{\ell^q}$. For $q=+\infty$, we let $\ell^\infty(\Z;\C)$ denote the space of bounded complex valued sequences indexed by $\Z$. 
This space is equipped with the norm:
$$
\sup_{j \in \Z} \, |v_j| \, ,
$$
which we shall refer to as $\| \cdot \|_{\ell^\infty}$. When equipped with this norm, the space $\ell^\infty(\Z;\C)$ is a Banach algebra (the product 
between two sequences being here understood in the pointwise sense $(v \, w)_j :=v_j \, w_j$ for each $j \in \Z$). Sequences will be denoted 
with bold letters, while their evaluation at a given integer will be denoted with standard letters.
\bigskip

We let $C$, resp. $c$, denote some (large, resp. small) positive constants that may vary throughout the text (sometimes within the same line). 
A typical example is the inequality:
$$
\forall \, x \in \R^+ \, ,\quad x \, \exp(-c \, x) \, \le C \, \exp(-c \, x) \, ,
$$
where, of course, the constant $c$ on the right-hand side of the inequality is not the same as in the left-hand side. The dependance of the 
constants on the various involved parameters is made precise throughout the article.

In order to avoid overloading some expressions, we sometimes write:
$$
\sum_{m=0}^x
$$
for a sum that runs over the integer $m$ from $m=0$ to $m=x$, even when $x$ is a positive real number that is not an integer. In that case, 
it is understood that the sum runs up to the largest number that is less than $x$. This will allow us to avoir using the integer part of several 
quantities.

%%%%%%%%%%%%%%%%%%%%%%%%%%%%%%%%%%%%%%%%%%%%%%%%%%%%%%%%%%%%%%%%%%%%%%%%%%
\chapter{Introduction}
\label{chapter1}

This article is devoted to a detailed stability analysis of stationary shock profiles for the Lax-Wendroff scheme. The Lax-Wendroff scheme is a finite 
difference approximation of hyperbolic conservation laws that is formally second order accurate (at least for smooth solutions), see \cite{Kroner,LeVeque}. 
As any second order accurate numerical scheme, the Lax-Wendroff scheme gives rise to spurious oscillations once the solution to the conservation 
law develops discontinuities, that is, once shock waves have appeared. This is evidenced for instance in the numerical simulation reported in Figure 
\ref{fig:simulation-num-1} below. This undesirable numerical feature is usually corrected by introducing flux limiters or other more involved numerical 
treatments (essentially non-oscillatory schemes, weighted non-oscillatory schemes and so on). Our goal here is to show that despite the formation of 
oscillatory wave trains in the numerical computation of shock waves, the Lax-Wendroff scheme gives rise to \emph{stable} stationary shock profiles 
for \emph{convex} or \emph{concave} fluxes, and for even other situations. The whole point is to clarify the meaning of ``stable'' in the previous sentence.

Following a long line of research, we aim at studying this stability problem by first showing a \emph{spectral} stability result, then turning this spectral 
stability result into a \emph{linear} stability result by proving sharp decay estimates for the linearized numerical scheme. Obtaining these sharp linear 
decay estimates is the cornerstone of our work. The final step of the analysis is to use the linear decay estimates in order to obtain an \emph{orbital} 
stability result for the nonlinear dynamics. The approach is definitely not new. For the Lax-Wendroff scheme, it was followed by Smyrlis \cite{Smyrlis} 
who dealt with the exact same problem as the one we are looking at here. However, the functional framework that we use is much larger than the one 
in \cite{Smyrlis} where the introduction of \emph{exponential} weights stabilizes the linearized operator. Actually, the introduction of exponential weights 
gives rise to a spectral gap for some carefully chosen parameters, thus bypassing a detailed stability analysis that would take the flux properties into 
account. We consider here a larger and, probably, more natural functional framework in which the spectrum of the linearized operator depends on the 
flux of the conservation law.

The overall strategy is the same as the one followed in the contributions \cite{godillon,Coeuret1,Coeuret2}. We start from a given discrete shock profile, 
which is quite easy for the Lax-Wendroff scheme since this numerical scheme exhibits \emph{exact}, piecewise constant, stationary discrete shock 
profiles that are the mere projections on the numerical grid of the continuous step shock. We linearize the Lax-Wendroff scheme around this stationary 
solution and try to locate the spectrum of the linearized evolution operator. Enlarging the functional framework has some major impacts, the first of 
which being that we have to take into account the (neutral) eigenvalue $1$ for the linearized numerical scheme. This eigenvalue arises by ``translation 
invariance'', as explained in \cite{Smyrlis} (see also \cite{Serre-notes} for a general overview on discrete shock profiles). Translation invariance means 
here that there exists a one-parameter smooth family of stationary discrete shock profiles\footnote{For continuous problems, such as parabolic 
perturbations of conservation laws, the one-parameter family of profiles is obtained by simply translating in space a given profile, but the construction 
of shock profiles is unfortunately more complicated for discrete problems.}. The rest of the spectrum is located by analyzing the so-called Evans function 
associated with the shock profile. The Evans function plays the role of a characteristic polynomial for the linearized operator. The general construction of 
the Evans function is detailed in \cite{Serre-notes}, see also \cite{godillon,Coeuret1}. In our case, the construction is far easier because our reference 
discrete shock is piecewise constant so the Evans function reduces to what could be referred to as a Lopatinskii determinant (by analogy with the stability 
analysis of shock waves \cite{benzoni-serre}). Our problem is similar to the stability analysis of the discrete shock profiles for the Godunov scheme 
\cite{BGS} in which the shock profiles are also piecewise constant. For scalar equations, as we consider here, the expression of the Lopatinskii 
determinant is simple enough so that we can analyze the location of its zeroes in the case where the flux of the conservation law is either convex 
or concave. Namely, we shall show that for convex or concave fluxes, the Lopatinskii determinant has only $1$ as a (simple) zero in the region 
$\Ubar$ of the complex plane. This corresponds, in the terminology of \cite{Serre-notes} to a \emph{spectrally stable} situation. Our analysis 
encompasses a particular (symmetric) case that was considered in the seminal work \cite{HHL}.

Once we have located the spectrum of the operator, we can construct the so-called spatial Green's function, that is the fundamental solution of the 
resolvent equation. For spectrally stable configurations, we show that the spatial Green's function has a meromorphic extension near $1$ and it can 
also be holomorphically extended near any other point of the unit circle. This is a key step towards the decomposition and the proof of sharp decay 
bounds for the Green's function of the linearized evolution operator. In the parabolic case, the Green's function satisfies uniform (in time) $\ell^1$ 
estimates (in space) and it also satisfies decaying (in time) $\ell^\infty$ estimates (in space), just like the heat kernel. For the Lax-Wendroff scheme, 
there is unfortunately no hope to obtain such favorable uniform bounds. Indeed, it is already known that when applied to the transport equation, the 
Green's function of the Lax-Wendroff scheme does not enjoy uniform $\ell^1$ estimates. This failure of $\ell^1$ stability, that is linked to the 
\emph{dispersive} behavior of the Lax-Wendroff scheme, has been identified in a general context by Thom\'ee \cite{Thomee} for convolution 
operators. The analysis of $\ell^1$ instability  was later refined in \cite{hedstrom1,hedstrom2} in order to make the instability growth rates percise. 
Rather than following \cite{hedstrom1,hedstrom2} for the Lax-Wendroff scheme, we shall build here on the recent work \cite{jfcAMBP} by one of 
the authors. The analysis in \cite{jfcAMBP} gives a precise description of the Green's function for this numerical scheme in the context of the 
Cauchy problem for the transport equation\footnote{The quantitative estimates in \cite{jfcAMBP} are more accurate in many regimes than the ones 
in \cite{hedstrom1,hedstrom2} and this is crucial for several arguments that we use below.}. We extend here the analysis of \cite{jfcAMBP} to the 
context of the shock profile stability analysis in which Fourier analysis is no longer available due to spatial variations. We therefore substitute the 
so-called spatial dynamics approach in place of Fourier analysis and use the inverse Laplace transform to obtain a representation formula for the 
Green's function of the linearized operator. The main difficulty that we shall face arises from the singularity of the spatial Green's function at the 
eigenvalue $1$ (that is imbedded in the continuous spectrum). This pole gives rise to a leading contribution in the Green's function which, following 
\cite{Coeuret2}, we refer to as an \emph{activation function}. This function can be thought of as a primitive of the Green's function for the Cauchy 
problem. Detailed bounds on this activation function are given in Appendix \ref{appendixA} at the end of this article. In particular, a lengthy -though 
crucial- argument here is to obtain a uniform bound for the activation function. This is rather trivial in the parabolic setting thanks to uniform $\ell^1$ 
estimates for the Green's function but such uniform $\ell^1$ estimates are known to fail in the dispersive setting. Our uniform bound for the activation 
function is reminiscent of the main result in \cite{ELR} where such a uniform bound is also the cornerstone of the argument.

Once we have an accurate description of the Green's function with sharp bounds for each term in its decomposition, the nonlinear orbital stability 
result follows more or less by using standard tools as in \cite{BHR}. Since our Green's function does not satisfy uniform $\ell^1$ estimates and 
exhibits a dispersive instability\footnote{The $\ell^\infty$ decay of the Green's function for the Lax-Wendroff scheme is also slower than for the 
parabolic situation that is considered in \cite{BHR}.}, we can not complete an orbital stability for small $\ell^1$ perturbations. We need to work 
in \emph{polynomially weighted} $\ell^p$ spaces and carefully play with the weights in the definition of the norms in order to recover some 
integrability in time, which is a crucial feature of bootstrap arguments. This is less favorable than the situation for semi-discrete shocks that 
is dealt with in \cite{BHR} where the time translation invariance also allows to use a larger functional framework. The fully discrete situation 
does not give as many tools to deal with the nonlinear argument.
\bigskip

The plan of this article is the following. In Chapter \ref{chapter2} we introduce the numerical scheme, we give the expression of the reference 
shock profile that we consider and state our main results. For the sake of clarity, we have split the main results in four theorems, the first two 
being devoted to \emph{spectral} stability, the third one being devoted to \emph{linear} stability, and the fourth and main one being devoted to 
\emph{nonlinear orbital} stability. We also report on some numerical simulations that illustrate these results. Chapter \ref{chapter3} is devoted 
to the spectral analysis. The linear decay estimates are proved in Chapter \ref{chapter4} with the help of several key estimates that are given 
in Appendix \ref{appendixA}. Eventually, the nonlinear analysis is detailed in Chapter \ref{chapter5}.

%%%%%%%%%%%%%%%%%%%%%%%%%%%%%%%%%%%%%%%%%%%%%%%%%%%%%%%%%%%%%%%%%%%%%%%%%%
\chapter{The Lax-Wendroff scheme. Notation and main results}
\label{chapter2}

%%%%%%%%%%%%%%%%%%%%%%%%%%%%%%%%%%
\section{The Lax-Wendroff scheme and its stationary shock profiles}
\label{section2.1}

We consider in this article a scalar conservation law of the form:
\begin{equation}
\label{law}
\partial_t u + \partial_x f(u) = 0, \quad x\in\R,
\end{equation}
with a smooth flux $f \in \mathscr{C}^\infty(\R;\R)$. Since the solutions to the Cauchy problem for \eqref{law} generically develop singularities in finite 
time \cite{Dafermos,Serre1}, we directly consider piecewise smooth solutions to \eqref{law}, which we assume to be stationary for simplicity:
\begin{equation}
\label{shock}
u(t,x) = \left\{
\begin{array}{ll}
u_\ell, & x<0,\\
u_r, & x>0.
\end{array}
\right.
\end{equation}
For \eqref{shock} to be a weak solution to \eqref{law}, the so-called Rankine-Hugoniot relation must hold:
\begin{equation}
\label{RH}
f(u_\ell) = f(u_r) \, ,
\end{equation}
and we also enforce the so-called entropy criterion \cite{Serre1}:
\begin{equation}
\label{entropy}
f'(u_r) < 0 < f'(u_\ell) \, .
\end{equation}
The Lax entropy inequalities \eqref{entropy} imply that the characteristics stemming from either side of the shock wave (that is located here at $x=0$) 
enter the shock in positive time, as depicted in Figure \ref{fig:characteristics}.

\begin{figure}[h!]
\begin{center}
\begin{tikzpicture}[scale=1.25,>=latex]
\draw[black,->] (-4,0) -- (4,0);
\draw (4.1,0) node[below] {$x$};
\draw[thick,black,->] (0,-0.5)--(0,4);
\draw (0,4) node[right] {$t$};
\draw[thick,blue,->] (-2,0) -- (0,2);
\draw (-1,1) node[left] {$\big\{ x_0+f'(u_\ell)t,x_0<0 \big\} \,$};
\draw[thick,blue,->] (2,0) -- (0,3);
\draw (1.4,1) node[right] {$\big\{ x_0+f'(u_r)t,x_0>0 \big\}$};
\draw (0.15,0) node[below]{$0$};
\draw (3,2.5) node {$u_r$};
\draw (-3,2.5) node {$u_\ell$};
\end{tikzpicture}
\caption{The characteristics on either side of the shock.}
\label{fig:characteristics}
\end{center}
\end{figure}
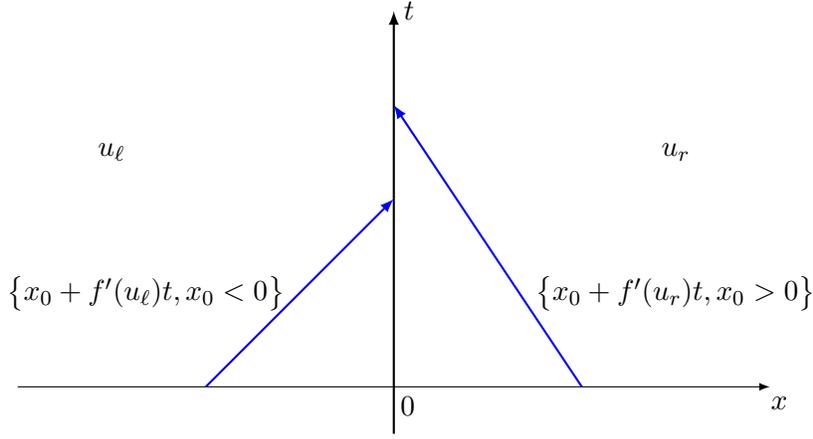

Our goal here is to understand the influence of a high order discretization procedure on the stability of the shock \eqref{shock} and more specifically 
whether dispersion in a numerical scheme may rule out linear or nonlinear stability. Let us recall that a rather complete existence and stability theory 
for monotone schemes has been developed in \cite{Jennings} but monotone schemes are at most first order accurate \cite{Kroner,LeVeque}. We 
therefore pursue the analysis of \cite{Smyrlis} and try to develop a rather complete stability analysis of discrete shock profiles for the Lax-Wendroff 
scheme. This is a model situation for a high order finite difference scheme and we hope that some of our arguments below may prove useful for 
other discretization procedures of conservation laws.

We therefore introduce a space step $\Delta x>0$ and a time step $\Delta t>0$ that are always chosen such that the ratio $\lambda:=\Delta t/\Delta x$ 
is kept constant. For any couple of integers $n \in \N$ and $j \in \Z$, the solution to \eqref{law} is approximated, on the time-space cell $[n \, \Delta t, 
(n+1) \, \Delta t) \times [j \, \Delta x,(j+1) \, \Delta x)$ by a constant value $u_j^n$, that is iteratively defined with respect to $n$ according to the formula:
\begin{equation}
\label{schemeLW}
\forall \, (n,j) \in \N \times \Z \, ,\quad 
u_j^{n+1} = u_j^n - \lambda \left( \cF_\lambda(u_j^n,u_{j+1}^n)-\cF_\lambda(u_{j-1}^n,u_{j}^n) \right),
\end{equation}
with a \emph{numerical flux} $\cF_\lambda$ that is defined by:
\begin{equation}
\label{fluxLW}
\forall \, (u,v) \in \R^2 \, ,\quad 
\cF_\lambda(u,v) := \dfrac{1}{2}\left(f(u)+f(v)\right) - \dfrac{\lambda}{2}f'\left(\frac{u+v}{2}\right)\left(f(v)-f(u)\right).
\end{equation}
The numerical scheme \eqref{schemeLW}-\eqref{fluxLW}, that is referred to as the Lax-Wendroff scheme and dates back to \cite{LaxWendroff}, 
is a formally second order accurate approximation procedure for \eqref{law}, see \cite{Kroner,LeVeque}. The initial condition $(u_j^0)_{j \in \Z}$ 
for \eqref{schemeLW} will be chosen as a small perturbation of a discretized version of the shock wave \eqref{shock}.

A very specific feature of the numerical scheme \eqref{schemeLW}-\eqref{fluxLW} is the fact that it captures \emph{exactly} the stationary piecewise 
constant solutions to \eqref{law}. Namely, if we consider the weak solution \eqref{shock} to \eqref{law}, then the following sequence:
\begin{equation}
\label{DSP}
\overline{u}_j^n:=\left\{
\begin{array}{ll}
u_\ell, & j \le 0,\\
u_r, & j \ge 1,
\end{array}
\right.
\quad n \in \N \, ,
\end{equation}
defines a stationary solution to \eqref{schemeLW}-\eqref{fluxLW} since it does not depend on $n$ and satisfies (we omit the index $n$ in $\overline{u}_j$ 
from now on):
$$
\forall \, j \in \Z \, ,\quad \cF_\lambda(\overline{u}_j,\overline{u}_{j+1}) = \cF_\lambda(\overline{u}_{j-1},\overline{u}_j) = f(u_\ell) = f(u_r) \, ,
$$
the final equality coming from the Rankine-Hugoniot relation \eqref{RH}. The aim of this article is to study the stability of the stationary solution 
$\overline{\mathbf{u}}=(\overline{u}_j)_{j \in \Z}$ defined in \eqref{DSP} with respect to the dynamics of \eqref{schemeLW}-\eqref{fluxLW}. Some 
numerical experiments are reported below in Section \ref{section2.3}. Let us recall that the piecewise constant discrete shock \eqref{DSP} is not 
the only discrete shock for the Lax-Wendroff scheme \eqref{schemeLW} associated 
with \eqref{shock}. As explained in \cite{Smyrlis}, a stationary discrete shock profile for \eqref{shock} is a real valued sequence $\mathbf{u}=(u_j)_{j \in \Z}$ 
such that:
$$
\forall \, j \in \Z \, ,\quad \cF_\lambda(u_j,u_{j+1}) = \cF_\lambda(u_{j-1},u_j) \, ,
$$
and
$$
\lim_{j \rightarrow -\infty} \, u_j \, = \, u_\ell \, ,\quad \lim_{j \rightarrow +\infty} \, u_j \, = \, u_r \, .
$$
The sequence in \eqref{DSP} is a particular case of a stationary discrete shock profile but there are many other ones. We refer to \cite{Serre-notes} 
and references therein for a thorough description of the existence theory (both in the scalar and system cases), and to \cite{Smyrlis} for the specific 
case of the Lax-Wendroff scheme \eqref{schemeLW}-\eqref{fluxLW} in the scalar case. We recall the following result that is one of the achievements 
in \cite{Smyrlis}:

\begin{theorem}[Smyrlis \cite{Smyrlis}]
\label{thm:Smyrlis}
Let the shock \eqref{shock} satisfy the Rankine-Hugoniot condition \eqref{RH} and the entropy condition \eqref{entropy}. Let $\lambda$ satisfy 
the so-called CFL condition:
\begin{equation}
\label{CFLSmyrlis}
\max \, (\lambda \, f'(u_\ell),\lambda \, |f'(u_r)|) \, < \, 1 \, .
\end{equation}
Then there exist $\underline{\theta}>0$ and a one-parameter family of stationary discrete shock profiles $\mathbf{v}^\theta = (v_j^\theta)_{j \in \Z}$, 
$\theta \in (-\underline{\theta},\underline{\theta})$, that satisfies the following properties:
\begin{itemize}
 \item[(i)] $\mathbf{v}^0=\overline{\mathbf{u}}$ is the piecewise constant discrete shock \eqref{DSP},
 \item[(ii)] for every $j \in \Z$, the map $\theta \mapsto v_j^\theta$ is $\mathscr{C}^\infty$ on the interval $(-\underline{\theta},\underline{\theta})$,
 \item[(iii)] there exist $\delta>0$ and $C>0$ such that for any $\theta \in (-\underline{\theta},\underline{\theta})$, the discrete shock profile 
 $\mathbf{v}^\theta$ converges towards its end states exponentially fast at rate $\delta$, namely:
$$
\forall \, j \in \N \, ,\quad \big| v_j^\theta -u_r \big| \, + \, \big| v_{-j}^\theta -u_\ell \big| \, \le \, C \, {\rm e}^{- \delta \, j} \, ,
$$
and furthermore 
$$
\forall \, j \in \Z \, ,\quad \big| v_j^\theta -v_j^0 \big| \, \le \, C \, |\theta| \, {\rm e}^{- \delta \, |j|} \, ,
$$
 \item[(iv)] for every $\theta$ in the interval $(-\underline{\theta},\underline{\theta})$, the ``excess mass'' of $\mathbf{v}^\theta$ equals $\theta$, 
 namely\footnote{The series is indeed convergent because stationary shock profiles converge exponentially fast towards their end states.}:
$$
\sum_{j \in \Z} \, v_j^\theta \, -v_j^0 \, = \, \theta \, .
$$
\end{itemize}
\end{theorem}
\bigskip

Since the numerical scheme \eqref{schemeLW}-\eqref{fluxLW} is conservative, it preserves mass. More precisely, given an initial sequence 
$\mathbf{h}$ in $\ell^1(\Z;\R)$, we want to consider the dynamics of \eqref{schemeLW}-\eqref{fluxLW} for the initial condition $\mathbf{v}^0 
+\mathbf{h}$ (see the numerical experiments in Section \ref{section2.3}). Let $(u_j^n)_{j \in \Z,n\in \N}$ denote the corresponding solution 
to \eqref{schemeLW}-\eqref{fluxLW}. By mass conservation, we have:
$$
\forall \, n \in \N \, ,\quad \sum_{j \in \Z} \, u_j^n-v_j^0 \, = \, \sum_{j \in \Z} \, u_j^0-v_j^0 \, = \, \sum_{j \in \Z} \, h_j \, .
$$
In particular, if we can prove that $(u_j^n)_{j \in \Z,n\in \N}$ converges, as $n$ tends to $+\infty$, towards a stationary discrete shock profile, 
then this limit can only be $\mathbf{v}^\theta$ where $\theta$ denotes the mass of the initial perturbation $\mathbf{h}$. In other words, we 
rather intend to show that the whole curve $\{ \mathbf{v}^\theta, \theta \in (-\underline{\theta},\underline{\theta}) \}$ of stationary discrete 
shock profiles is \emph{orbitally stable}.
\bigskip

To study the stability of \eqref{DSP}, or rather the orbital stability of the curve $\{ \mathbf{v}^\theta, \theta \in (-\underline{\theta},\underline{\theta}) \}$, 
we follow a common approach that is based on first studying the spectral stability of \eqref{DSP} and then on proving linear and nonlinear decay 
estimates. We therefore introduce the linearized numerical scheme that is obtained by linearizing \eqref{schemeLW}-\eqref{fluxLW} around the 
constant solution $\mathbf{v}^0$ that is given by \eqref{DSP}. For future use, we introduce the notation:
\begin{equation}
\label{defalphalrm}
\alpha_{\ell,r} := \lambda \, f'(u_{\ell,r}) \, ,\quad \alpha_m := \lambda \, f'\left(\frac{u_\ell+u_r}{2}\right) \, .
\end{equation}
The linearization of \eqref{schemeLW}-\eqref{fluxLW} around the discrete shock profile \eqref{DSP} leads to the iteration:
\begin{equation}
\label{linearizedLW}
\mathbf{w}^{n+1}=\mathscr{L} \, \mathbf{w}^{n}, \quad \mathbf{w}^{n}=(w_j^n)_{j\in\Z}, 
\end{equation}
and the bounded linear operator $\mathscr{L}:\ell^q(\Z;\C)\rightarrow \ell^q(\Z;\C)$ is defined by:
\begin{equation}
\label{linear}
(\mathscr{L} \, \mathbf{w})_j :=
\left\{
\begin{split}
&w_j - \dfrac{\alpha_r}{2}\left( w_{j+1}-w_{j-1}\right) + \dfrac{\alpha_r^2}{2}\left( w_{j+1}-2 \, w_j+w_{j-1}\right) \, , & j \ge 2, \\
&w_j - \dfrac{\alpha_\ell}{2}\left( w_{j+1}-w_{j-1}\right) + \dfrac{\alpha_\ell^2}{2}\left( w_{j+1}-2 \, w_j+w_{j-1}\right) \, , & j \le -1, \\
&w_0 - \dfrac{1}{2} \left( \alpha_r \, w_1 - \alpha_\ell \, w_{-1} \right) 
+ \dfrac{1}{2} \left( \alpha_r \, \alpha_m \, w_{1} - \alpha_\ell \, (\alpha_\ell+\alpha_m) \, w_0 +\alpha_\ell^2 \, w_{-1}\right) \, , & j=0, \\
&w_1 - \dfrac{1}{2} \left( \alpha_r \, w_2 - \alpha_\ell \, w_{0} \right) 
+ \dfrac{1}{2} \left( \alpha_r^2 \, w_2 - \alpha_r \, (\alpha_r+\alpha_m) \, w_1 + \alpha_\ell \, \alpha_m \, w_0 \right) \, , & j=1.
\end{split}
\right.
\end{equation}
Let us observe that the values $\alpha_\ell=-\alpha_r$ and $\alpha_m=0$ correspond to the particular ``symmetric'' case studied in \cite{HHL}. This 
case corresponds for instance to an even function $f$ with respect to the mid-point $(u_\ell+u_r)/2$. Actually, the spectral stability result of \cite{HHL} 
is, to some extent, the starting point for our analysis in Chapter \ref{chapter3} below.

%%%%%%%%%%%
\section{Main results}
\label{section2.2}

Our first main result is a spectral stability result for the stationary discrete shock \eqref{DSP} when the flux $f$ is either convex or concave.

\begin{theorem}[Spectral stability for convex or concave fluxes]
\label{thm1}
Let the flux $f$ in \eqref{law} be either convex or concave, and let the weak solution \eqref{shock} satisfy the Rankine-Hugoniot relation \eqref{RH} 
and the entropy inequalities \eqref{entropy}. Then under the condition\footnote{This condition is a mere restriction on the ratio $\lambda=\Delta t/\Delta x$. 
We recall that $\alpha_\ell$ and $\alpha_r$ are defined in \eqref{defalphalrm}. The condition \eqref{CFL} is the exact same one as in the existence theorem 
for stationary discrete shock profiles, see \eqref{CFLSmyrlis}.}:
\begin{equation}
\label{CFL}
\max \, (\alpha_\ell,|\alpha_r|) < 1 \, ,
\end{equation}
the operator $\mathscr{L}$ has no spectrum in $\Ubar \setminus \{ 1 \}$. Moreover, $1$ is an eigenvalue of $\mathscr{L}$ in $\ell^q(\Z;\C)$ for any 
$q \in [1,+\infty]$.
\end{theorem}

Actually, it will turn out in Chapter \ref{chapter3} that in the case of a convex or concave flux $f$, the operator $\mathscr{L}$ will have no spectrum in 
the exterior of the curve\footnote{This curve is actually an ellipse that is centered at $1-\alpha^2$.}:
\begin{equation}
\label{courbespectre}
\left\{1-2\,\alpha^2\sin^2\dfrac{\xi}{2}+\mathbf{i}\,\alpha\sin\xi\,|\,\xi\in\R\right\}\,,\quad\alpha:=\max(\alpha_\ell,|\alpha_r|)\,.
\end{equation}
The exterior of this curve will always be denoted $\mathscr{O}$ in Chapter \ref{chapter3}. It can be parametrized as:
$$
\mathscr{O} \, = \, \Big\{ z=x+\mbi \, y \in \C \, | \, (x-1+\alpha^2)^2 \, + \, \alpha^2 \, y^2 \, > \, \alpha^4 \Big\} \, ,
$$
see for instance Figure \ref{fig:regionO} in Chapter \ref{chapter3} below where $\mathscr{O}$ corresponds to the complement of the grey shaded area.
\bigskip

We wish to encompass slightly more general situations than the sole case of convex (or concave) fluxes. Indeed, our goal is to show that spectral 
stability of a stationary discrete shock is a \emph{sufficient} condition for linear and orbital nonlinear stability. In practice, this requires having at our 
disposal a convenient tool that locates the spectrum of $\mathscr{L}$ just as the characteristic polynomial does so for a matrix. Constructing such 
a tool, which we shall refer to as a \emph{Lopatinskii determinant} (in analogy with the shock wave theory for systems of conservation laws, see 
e.g. \cite{benzoni-serre}), is the purpose of part of Chapter \ref{chapter3} below. Examples of numerical calculations of such Lopatinskii determinants 
can be found in \cite{BLBS} in the context of numerical boundary conditions for transport equations but the analysis is entirely similar here.

Going beyond the case of a convex or concave flux requires making a stability assumption that is given as Assumption \ref{hyp-stabspectrale} 
in Chapter \ref{chapter3} below. This stability assumption is meant to exclude the possibility of having some spectrum of $\mathscr{L}$ in 
$\Ubar \setminus \{ 1 \}$. Our hope is that Assumption \ref{hyp-stabspectrale} can be ``easily'' verified (or proved not to hold) in specific 
situations. The generalization of Theorem \ref{thm1} is the following result.

\begin{theorem}
\label{thm1bis}
Let the weak solution \eqref{shock} satisfy the Rankine-Hugoniot relation \eqref{RH} and the entropy inequalities \eqref{entropy}. Then under the 
condition \eqref{CFL} on the parameter $\lambda$, we can define a function $\underline{\Delta}$ that is holomorphic on a neighborhood of $\Ubar$ 
(see \eqref{defDelta} in Chapter \ref{chapter3}). Then $1$ is an eigenvalue of $\mathscr{L}$ in $\ell^q(\Z;\C)$ for any $q \in [1,+\infty]$. Moreover, 
under Assumption \ref{hyp-stabspectrale} below on the location of the zeroes of $\underline{\Delta}$, the operator $\mathscr{L}$ has no spectrum 
in $\Ubar \setminus \{ 1 \}$.
\end{theorem}

Once we know that spectral stability holds, we expect that the favorable localization of the spectrum of $\mathscr{L}$ will imply some decay estimates 
for the semigroup of operators $\{ \mathscr{L}^n \, | \, n \in \N \}$. We gather here these estimates but, to some extent, our main result is rather Theorem 
\ref{thmGreen} that will be stated in Chapter \ref{chapter4} below. Theorem \ref{thmGreen} gives accurate bounds for the Green's function of the operator 
$\mathscr{L}$ so that, with classical convolution estimates, we can infer estimates on the norm of $\mathscr{L}^n$ in some polynomially weighted spaces. 
These spaces are defined as follows. We recall that $\ell^q(\Z;\C)$ denotes the space of complex valued sequences that are indexed by $\Z$ and such 
that their $q$-th power is integrable (bounded sequences if $q=+\infty$). The norm is denoted $\| \cdot \|_{\ell^q}$. Given a real number $\gamma \ge 0$, 
we then introduce the polynomially weighted $\ell^q$ space:
\begin{equation}
\label{defellqgamma}
\ell^q_\gamma(\Z;\C) \, := \, \left\{ \mathbf{h} \in \ell^q(\Z;\C) \, ~|~ \, \big( (1+|j|^\gamma) \, h_j \big)_{j \in \Z} \in \ell^q(\Z;\C) \right\} \, .
\end{equation}
For $\mathbf{h} \in \ell^q_\gamma$, we denote:
$$
\| \bfh \|_{\ell^q_\gamma} \, := \, \left\| \big( (1+|j|^\gamma) \, h_j \big)_{j \in \Z} \right\|_{\ell^q} \, ,
$$
the norm of $\mathbf{h}$, so that $\ell^q_\gamma(\Z;\C)$ is endowed with a Banach space structure. Our main decay estimates for the operator 
$\mathscr{L}$ read as follows.

\begin{theorem}
\label{thmLineaire}
Let the weak solution \eqref{shock} satisfy the Rankine-Hugoniot relation \eqref{RH} and the entropy inequalities \eqref{entropy}. Let the parameter 
$\lambda$ satisfy the condition \eqref{CFL} and let Assumption \ref{hyp-stabspectrale} on the zeroes of $\underline{\Delta}$ be satisfied. Then for 
any real numbers $\gamma_2 \ge \gamma_1 \ge 0$, there exists a constant $C$ such that, for any $\mathbf{h} \in \ell^1_{\gamma_2}(\Z;\C)$ that 
satisfies:
$$
\sum_{j \in \Z} \, h_j \, = \, 0 \, ,
$$
there holds:
\begin{align*}
\forall \, n \ge 1 \, ,\quad \| \mathscr{L}^n \, \mathbf{h} \|_{\ell^1_{\gamma_1}} \, & \le \, 
\dfrac{C}{n^{\gamma_2-\gamma_1-1/8}} \, \| \mathbf{h} \|_{\ell^1_{\gamma_2}} \, ,\\
\| \mathscr{L}^n \, \mathbf{h} \|_{\ell^\infty_{\gamma_1}} \, & \le \, 
\dfrac{C}{n^{\gamma_2-\gamma_1+\min(1/3,\gamma_1)}} \, \| \mathbf{h} \|_{\ell^1_{\gamma_2}} \, .
\end{align*}
\end{theorem}

The first estimate in Theorem \ref{thmLineaire} is consistent with all known results on the behavior of the Lax-Wendroff scheme for the transport 
equation. Indeed, if one chooses $\gamma_2=\gamma_1=0$, the first estimate in Theorem \ref{thmLineaire} corresponds to the well-known 
(weak) instability of the Lax-Wendroff scheme in either the $\ell^1$ or $\ell^\infty$ norm. The growth $n^{1/8}$ is known to be sharp, see e.g. 
\cite{hedstrom1,hedstrom2,jfcAMBP}. Since we wish to obtain a nonlinear orbital stability result by means of a bootstrap argument, some decay 
should be gained one way or another, and the introduction of the polynomial weight is a way to compensate for the weak instability of the numerical 
scheme (and to gain enough time integrability as well). If we compare with the more standard Cauchy problem on $\Z$ for the transport equation, 
the main new difficulty that we are facing here is the eigenvalue $1$ for the operator $\mathscr{L}$, this eigenvalue being imbedded into the 
continuous spectrum. Nevertheless, our decay estimates in Theorem \ref{thmLineaire} (and some complementary estimates that are detailed 
in Chapter \ref{chapter4}) are strong enough to yield the following \emph{orbital nonlinear stability} result for the family of stationary discrete 
shock profiles exhibited in Theorem \ref{thm:Smyrlis}.

\begin{theorem}
\label{thmNLS}
Let the weak solution \eqref{shock} satisfy the Rankine-Hugoniot relation \eqref{RH} and the entropy inequalities \eqref{entropy}. Let the parameter 
$\lambda$ satisfy the condition \eqref{CFL} and let Assumption \ref{hyp-stabspectrale} on the zeroes of $\underline{\Delta}$ be satisfied. Let now 
$\beta,\sigma \in \R^+$ satisfy $\beta+\sigma \geq \frac{5}{12}$ and $0\leq \sigma < \beta+\frac{1}{8}$. We define the constant:
\bqq
\label{defgamma}
\gamma \, := \, \sigma+\beta+\frac{1}{8} \, .
\eqq
Then there exist some positive real numbers $C_0,\epsilon>0$ such that for any sequence $\bfh \in \ell^1_\gamma(\Z;\R)$ satisfying
\bqs
\| \bfh \|_{\ell^1_\gamma}<\epsilon\, ,
\eqs
then one has
\bqs
\theta \, := \, \sum_{j\in\Z} h_j \in (-\underline{\theta},\underline{\theta}) \, ,
\eqs
and the solution $(\mathbf{u}^n)_{n\in\N}$ of the Lax-Wendroff scheme \eqref{schemeLW}-\eqref{fluxLW} with the initial condition 
$\mathbf{u}^0:=\overline{\mathbf{u}}+\bfh$ is well-defined. Furthermore, if we introduce the sequence $(\mathbf{p}^n)_{n\in\N}$ defined as 
\bqs
\forall \, n \in \N \, ,\quad \mathbf{p}^n \, := \, \mathbf{u}^n-\mathbf{v}^\theta \, ,
\eqs
then for all $n\in\N$ one has $\mathbf{p}^n \in \ell^1_\beta (\Z;\R)$ together with the estimates:
\bqs
\forall \, n \in \N \, ,\quad 
\left\| \mathbf{p}^n \right\|_{\ell^1_\beta} \, \leq \, \frac{C_0}{(1+n)^\sigma} \, \| \bfh \|_{\ell^1_\gamma} \,,  \text{ and } \, 
\left\| \mathbf{p}^n \right\|_{\ell^\infty_\beta} \, \leq \, \frac{C_0}{(1+n)^{\sigma+\frac{11}{24}}} \, \| \bfh \|_{\ell^1_\gamma} \, ,
\eqs
so that $\mathbf{u}^n$ tends to $\mathbf{v}^\theta$ in $\ell^\infty_\beta(\Z;\R)$ (and also in $\ell^1_\beta(\Z;\R)$ if $\sigma$ is positive).
\end{theorem}

%%%%%%%%%%%%%%%%
\section{Numerical experiments}
\label{section2.3}

We first present some numerical computations of stationary shock profiles. We consider for simplicity the Burgers equation, which corresponds to 
the convex flux function $f(u)=u^2/2$. By Theorem \ref{thm1} above, every stationary discrete shock of the form \eqref{DSP} with $u_\ell>0$ and 
$u_r=-u_\ell$ is spectrally and therefore nonlinearly orbitally stable (in the sense of Theorem \eqref{thmNLS}) if the parameter $\lambda$ is chosen 
to satisfy the CFL stability condition. Figure \ref{fig:chocs-discrets} displays three possible stationary discrete shock profiles for the Lax-Wendroff 
scheme with the same end states $u_\ell=1/2$, $u_r=-1/2$. The Rankine-Hugoniot conditions \eqref{RH} and the Lax entropy condition 
\eqref{entropy} are satisfied. The CFL parameter $\lambda$ is chosen to be 1/2 so that the CFL condition \eqref{CFL} is also satisfied.
%Donner valeur parametre CFL dans les calculs

\begin{figure}\centering
\includegraphics[scale=0.3]{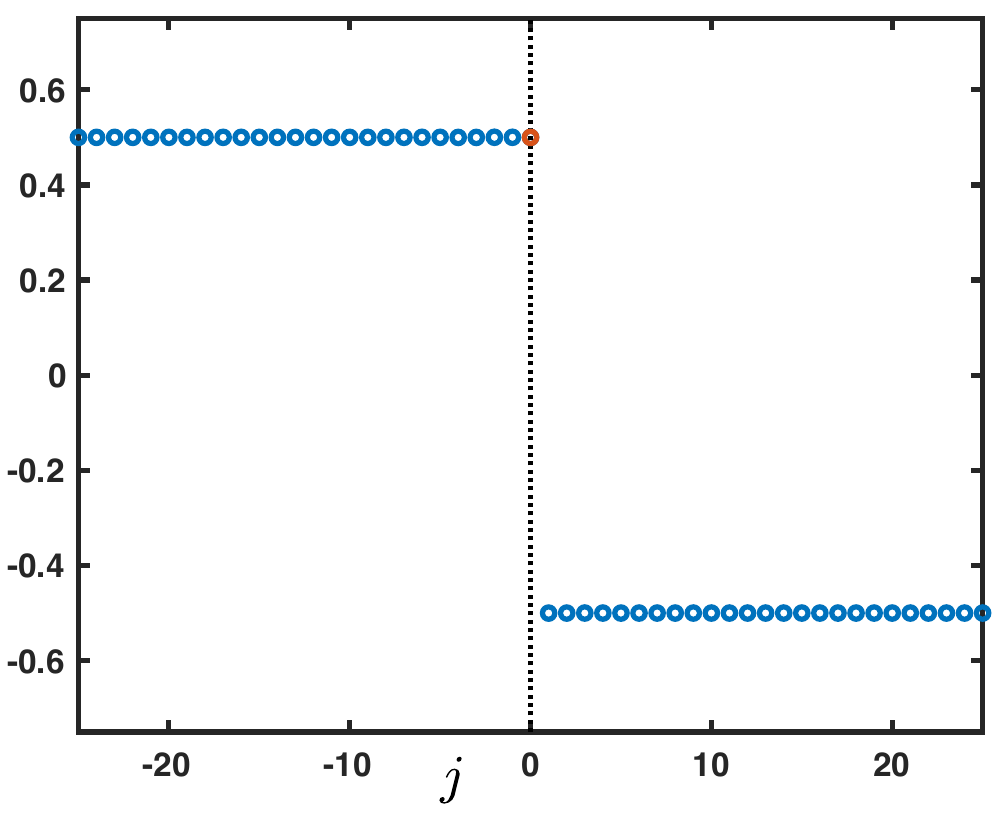}
\includegraphics[scale=0.3]{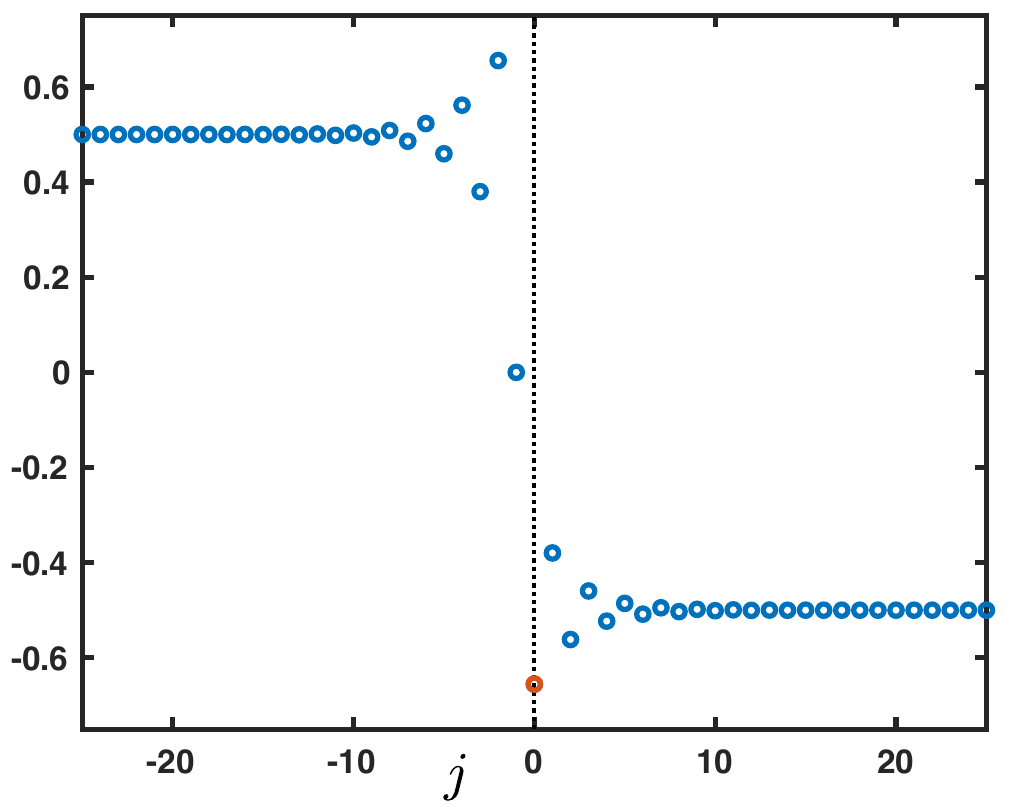}
\includegraphics[scale=0.3]{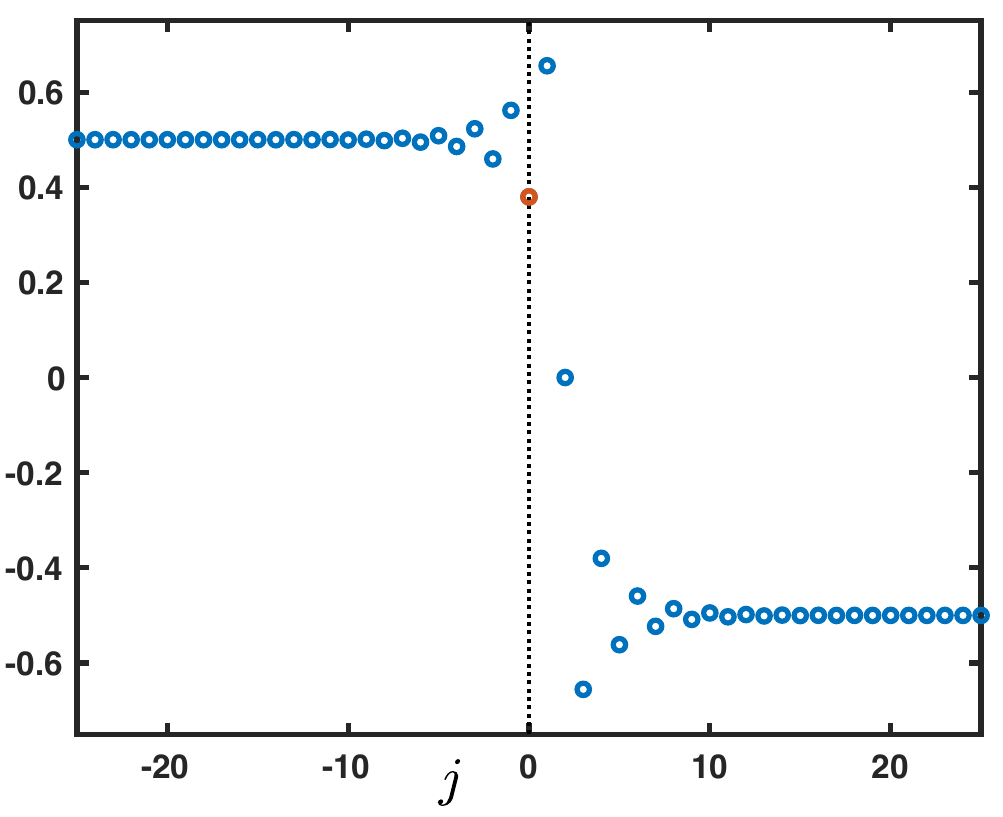}
\caption{Discrete shock profiles for the Lax-Wendroff scheme applied to the Burgers equation. Left: the reference shock \eqref{shock}. Middle: a discrete 
shock profile with same end states but negative mass difference ($\theta<0$ in Theorem \ref{thm:Smyrlis}). Right: a discrete shock profile with same end 
states but positive mass difference ($\theta>0$ in Theorem \ref{thm:Smyrlis}).}
\label{fig:chocs-discrets}
\end{figure}

As follows from Theorem \ref{thm:Smyrlis}, discrete shock profiles can be parametrized, at least for small enough mass perturbations, by their 
mass difference with respect to the reference discrete shock profile \eqref{DSP}. The first graph on the left of Figure \ref{fig:chocs-discrets} 
corresponds to the reference shock \eqref{shock} with end states $u_\ell=1/2$, $u_r=-1/2$. The value of that discrete shock at $j=0$ is highlighted 
in red. The middle and right graphs in Figure \ref{fig:chocs-discrets} correspond to stationary discrete shock profiles as given by Theorem 
\ref{thm:Smyrlis}, one being with negative mass difference (middle graph), and the other one being with positive mass difference (on the right 
of Figure \ref{fig:chocs-discrets}).

It turns out that the family $\{ \mathbf{v}^\theta \}$ of stationary discrete shock profiles given in Theorem \ref{thm:Smyrlis} can be parametrized 
\emph{globally} for the Burgers equation. This was already mentioned in \cite{Smyrlis} and we report here on the numerical computation of the 
whole family. As a first observation, we remark that the translation of the shock profile \eqref{DSP}, namely:
\begin{equation}
\label{DSP-translate}
\begin{cases}
u_\ell \, , & j \le 1 \, ,\\
u_r \, , & j \ge 2 \, ,
\end{cases}
\end{equation}
is also a stationary discrete shock profile for \eqref{schemeLW} with mass difference $u_\ell-u_r=1$ with respect to \eqref{DSP} (the only difference 
with \eqref{DSP} is in the cell labeled with the index $j=1$). In other words, if we can parametrize the family $\mathbf{v}^\theta$ of Theorem 
\ref{thm:Smyrlis} for $\theta \in [0,1]$ with $\mathbf{v}^0$ being equal to the stationary discrete shock profile \eqref{DSP} and $\mathbf{v}^1$ being 
equal to the translated discrete shock profile \eqref{DSP-translate}, then by repeated translations -either to the left or to the right- one can parametrize 
a \emph{global} family $\{ \mathbf{v}^\theta \, | \, \theta \in \R \}$ where $\theta$ still refers to the mass perturbation. This is illustrated in Figure 
\ref{fig:chocs-famille-globale} where we plot the evaluation at $j=0$ and $j=1$ of the family of stationary discrete shock profiles with $\theta \in [-2,2]$. 
For $\theta=1$, both $v_0^1$ and $v_1^1$ equal $1/2$ in agreement with \eqref{DSP-translate}. For $\theta=-1$, the translation of the reference 
discrete shock is to the left. Both curves that are plotted in Figure \ref{fig:chocs-famille-globale} are translates one of the other.

\begin{figure}\centering
\includegraphics[scale=0.5]{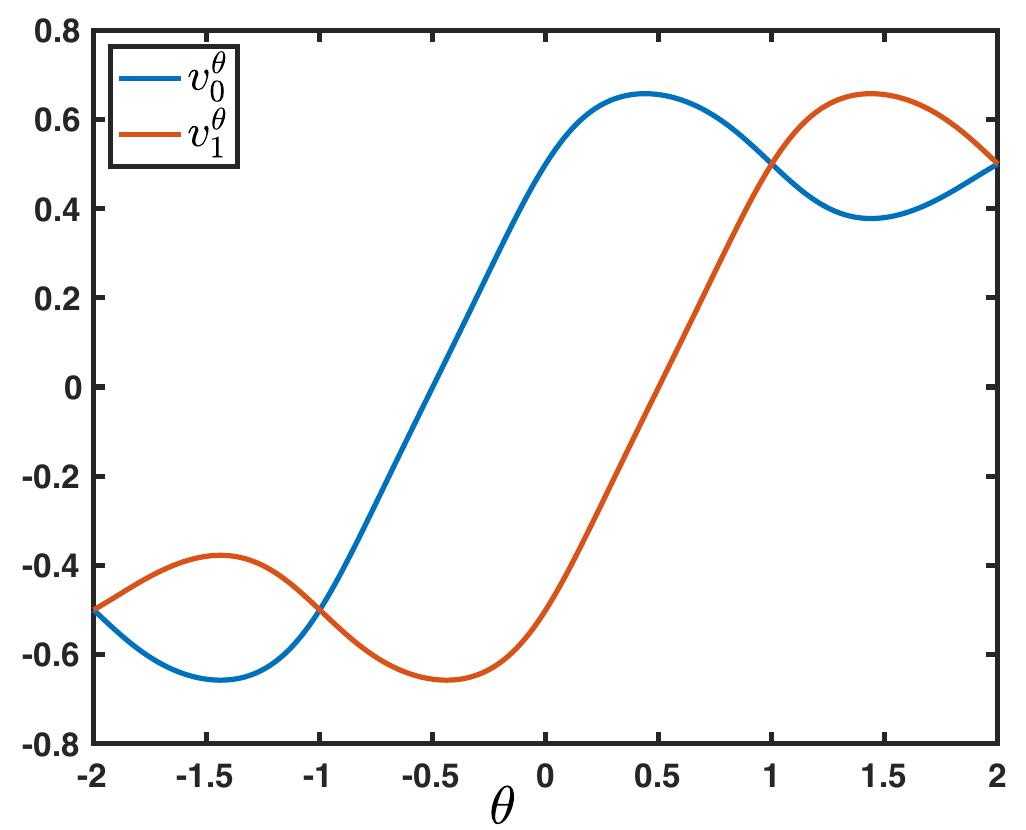}
\caption{The evaluation at $j=0$ and $j=1$ of the family of stationary discrete shock profiles $\mathbf{v}^\theta$, $\theta \in \R$.}
\label{fig:chocs-famille-globale}
\end{figure}

\bigskip

We now report on some computations that illustrate the dynamics of the numerical scheme \eqref{schemeLW}. In Figure \ref{fig:simulation-num-1}, 
we give several plots corresponding to the evolution of a zero mass perturbation of the reference discrete shock \eqref{DSP}. By appealing to Theorem 
\ref{thmNLS}, we expect that the numerical solution converges asymptotically towards \eqref{DSP} and this is what we observe. The convergence is 
illustrated in Figure \ref{fig:simulation-num-1} where two waves, one emanating from the right and one from the left of the shock, hit the shock one 
after the other, giving first rise to a translation of the initial shock to the right (due to positive mass excess) and then going back to the reference 
discrete shock \eqref{DSP}. Another view of this computation is given in Figure \ref{fig:simulation-1-tx} where the plot is in the $(j,n)$ plane. The 
translation of the shock to the right after the first wave hit the shock is more visible.
\bigskip

\begin{figure}\centering
\includegraphics[scale=0.3]{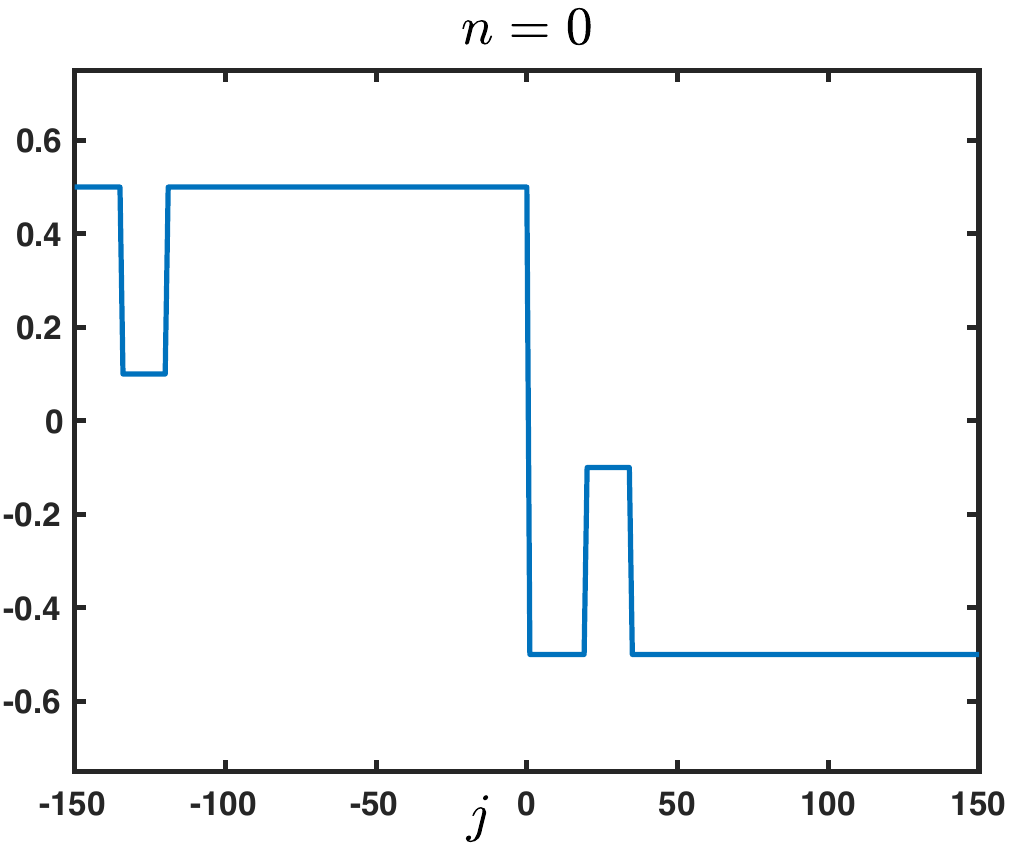}
\includegraphics[scale=0.3]{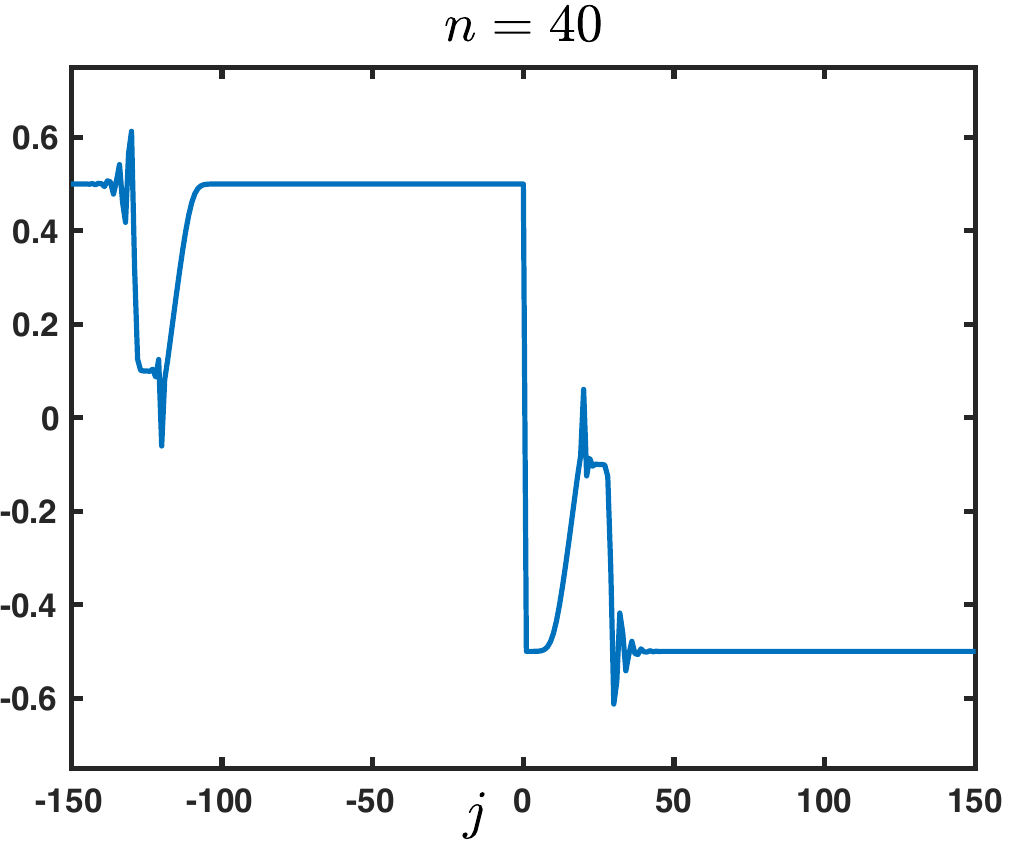}
\includegraphics[scale=0.3]{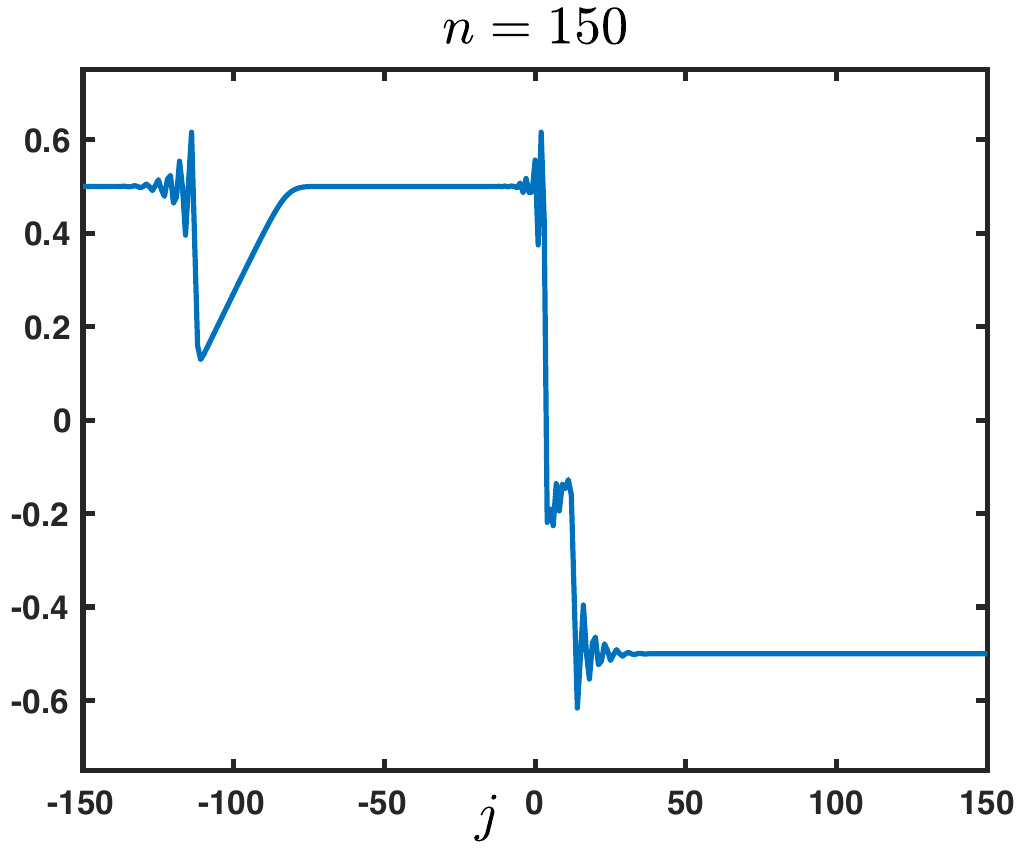} \\
\includegraphics[scale=0.3]{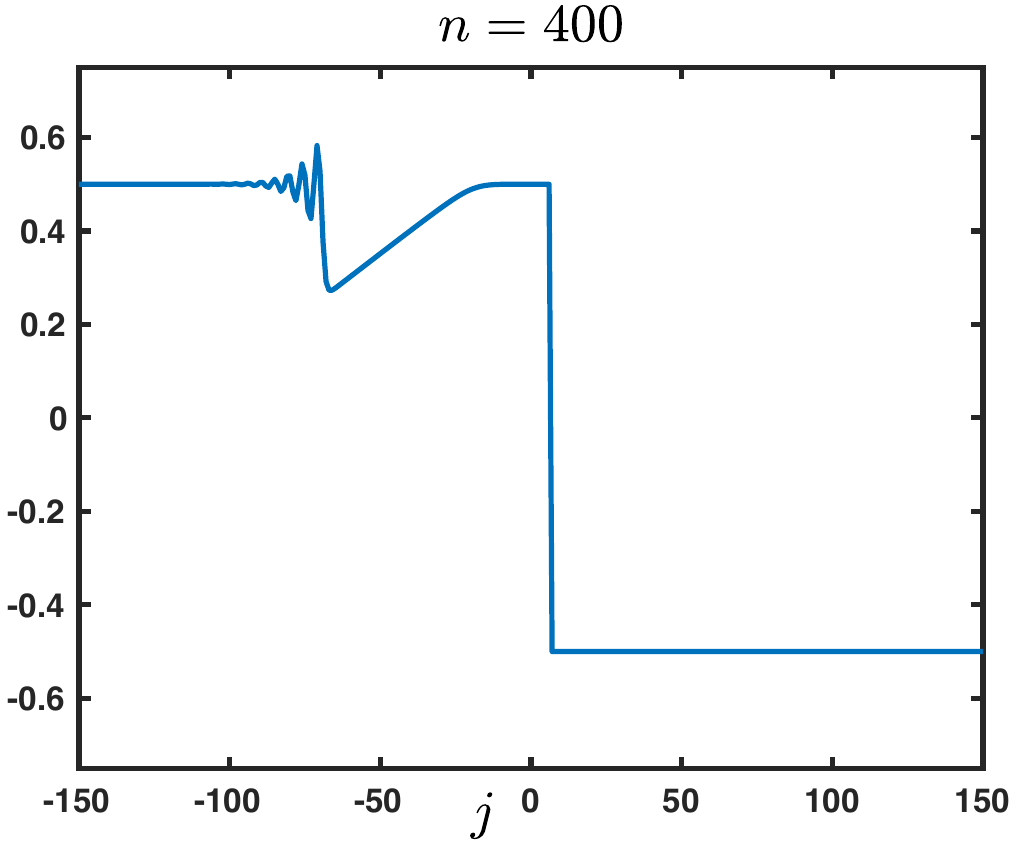}
\includegraphics[scale=0.3]{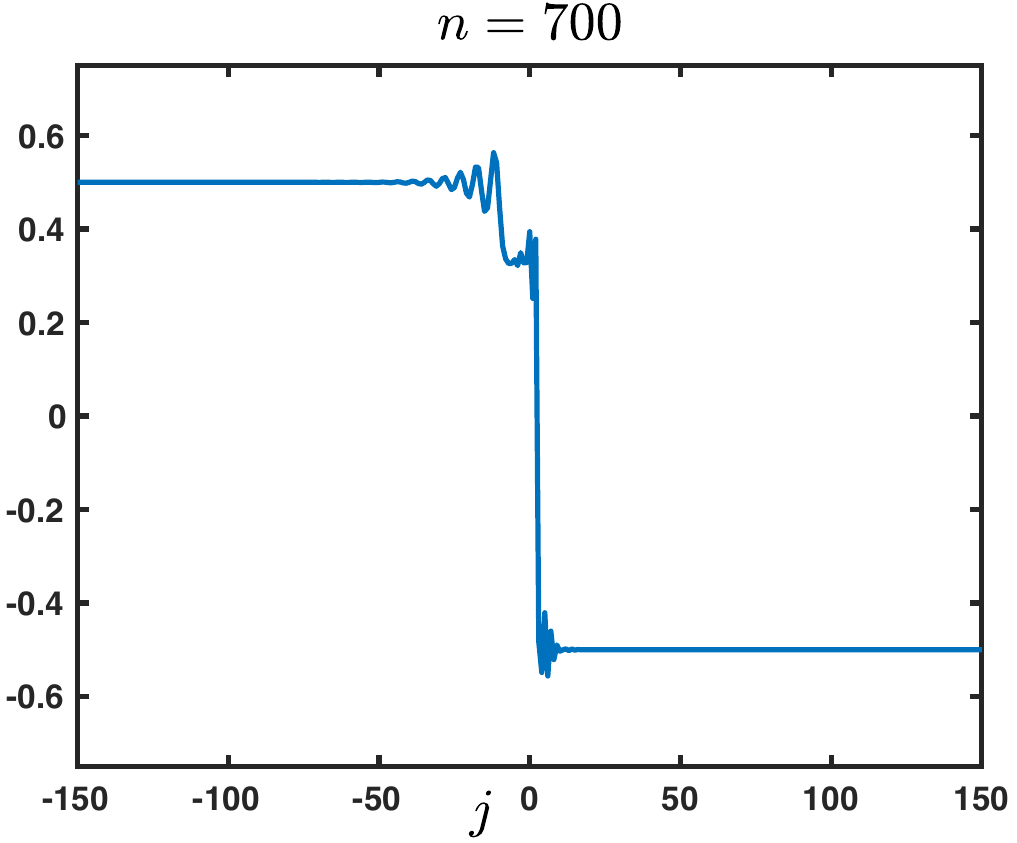}
\includegraphics[scale=0.3]{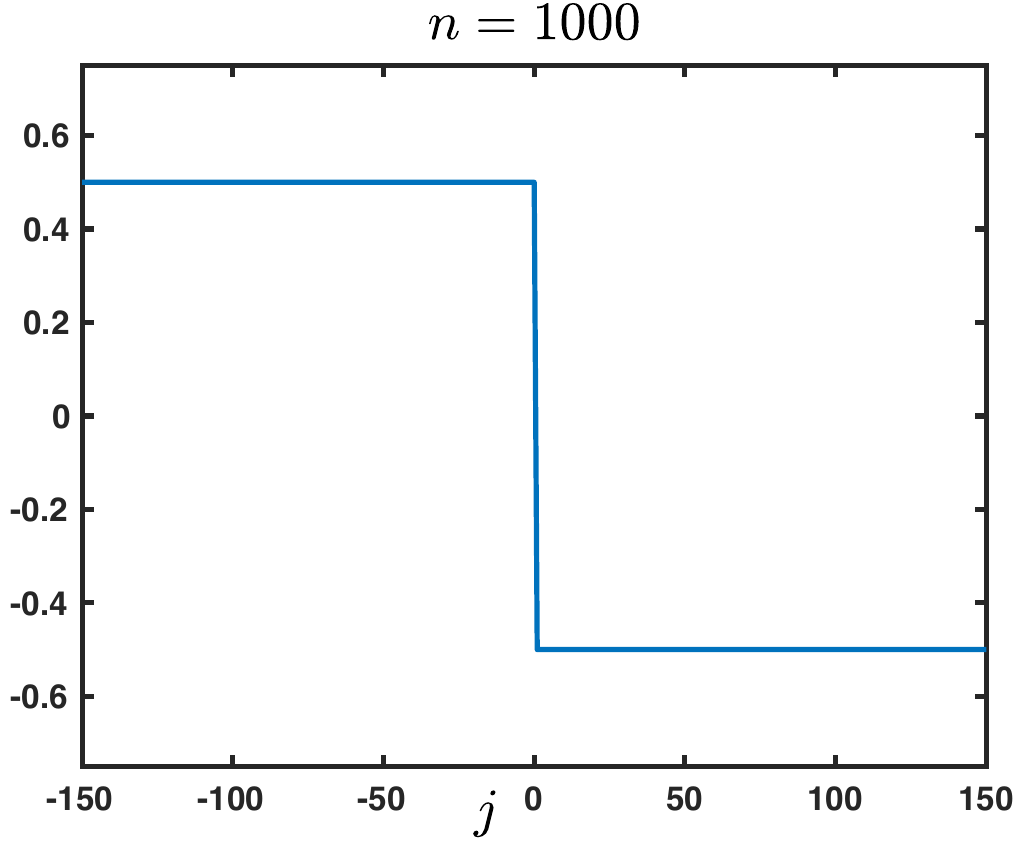}
\caption{Evolution of a perturbation with zero mass of the reference discrete shock \eqref{shock}. First line (from left to right): the initial condition, 
the solution at $n=40$, the solution at $n=150$. Second line (from left to right): the solution at $n=400$, the solution at $n=700$, the solution at 
$n=+\infty$ (convergence towards the reference shock \eqref{shock}).}
\label{fig:simulation-num-1}
\end{figure}

\begin{figure}\centering
\includegraphics[scale=0.5]{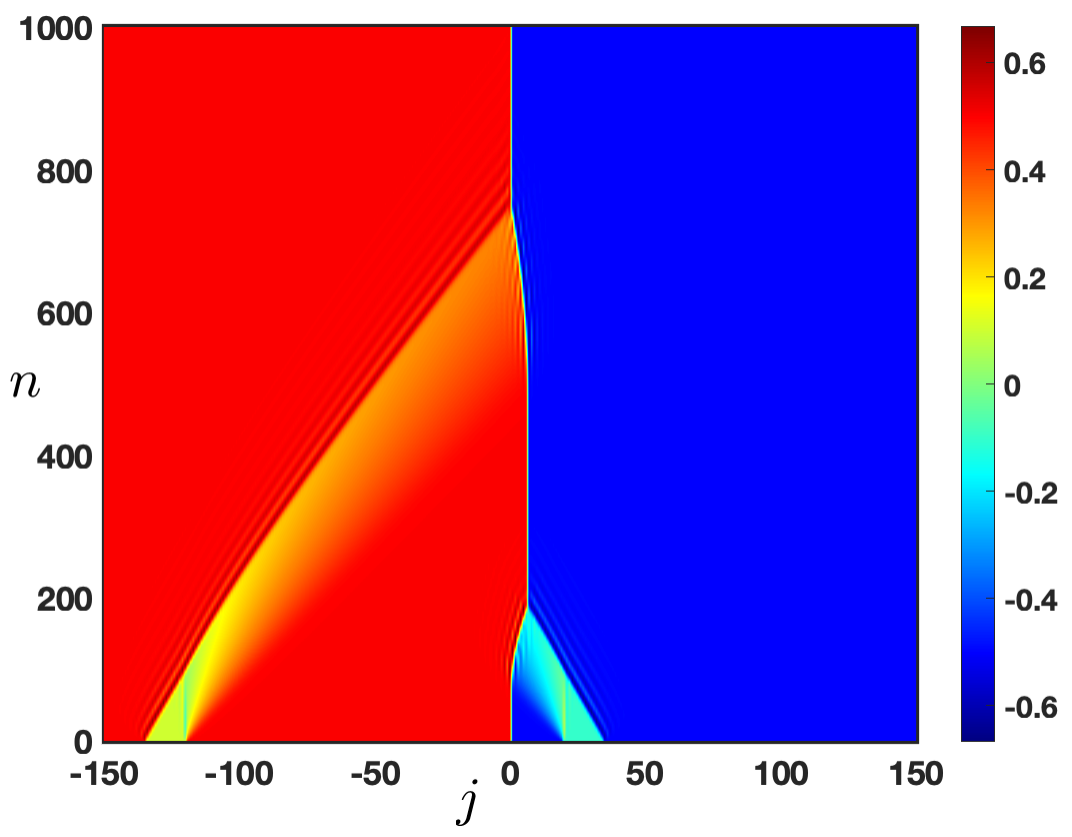}
\caption{Evolution of a perturbation with zero mass of the reference discrete shock \eqref{shock}.}
\label{fig:simulation-1-tx}
\end{figure}

As a conclusion, we illustrate the behavior of the Lax-Wendroff scheme \eqref{schemeLW} for a positive mass initial perurbation, see Figures 
\ref{fig:simulation-num-2} and \ref{fig:simulation-2-tx}. The asymptotic state is a stationary discrete shock that corresponds to a non-integer 
value of $\theta>0$, thus displaying the typical oscillations associated with the Lax-Wendroff scheme. Despite these oscillations, these solutions 
are spectrally and nonlinearly stable.

\begin{figure}\centering
\includegraphics[scale=0.3]{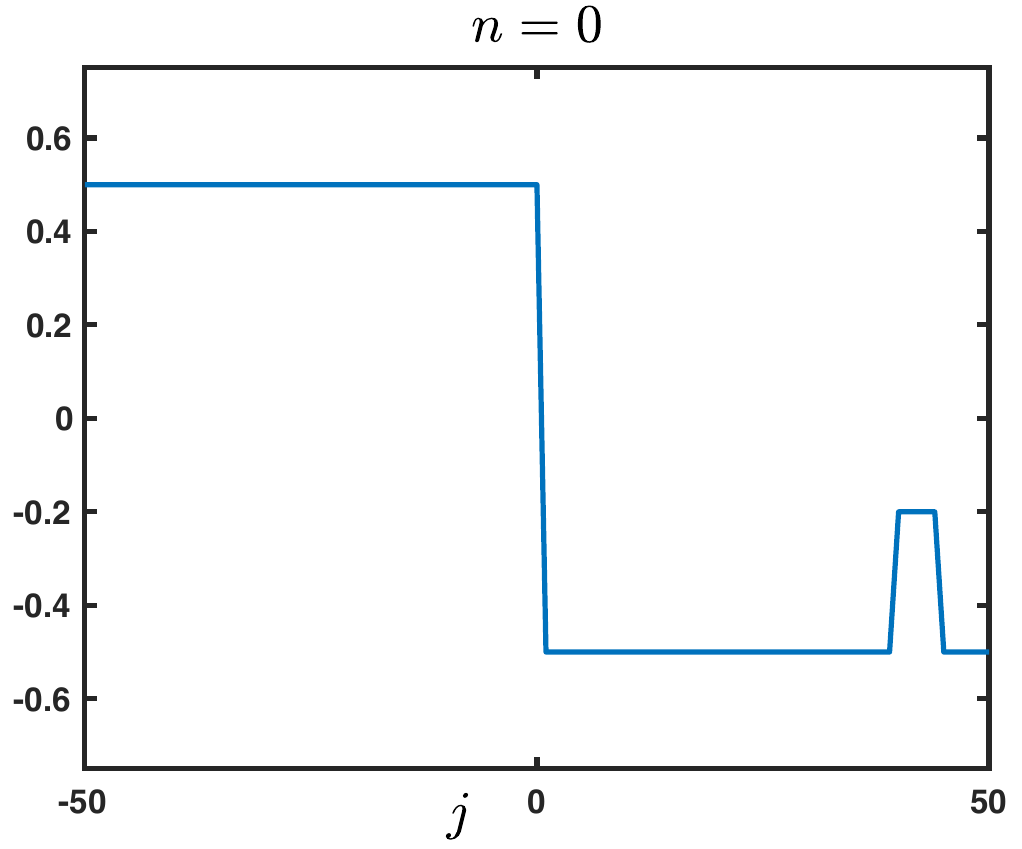}
\includegraphics[scale=0.3]{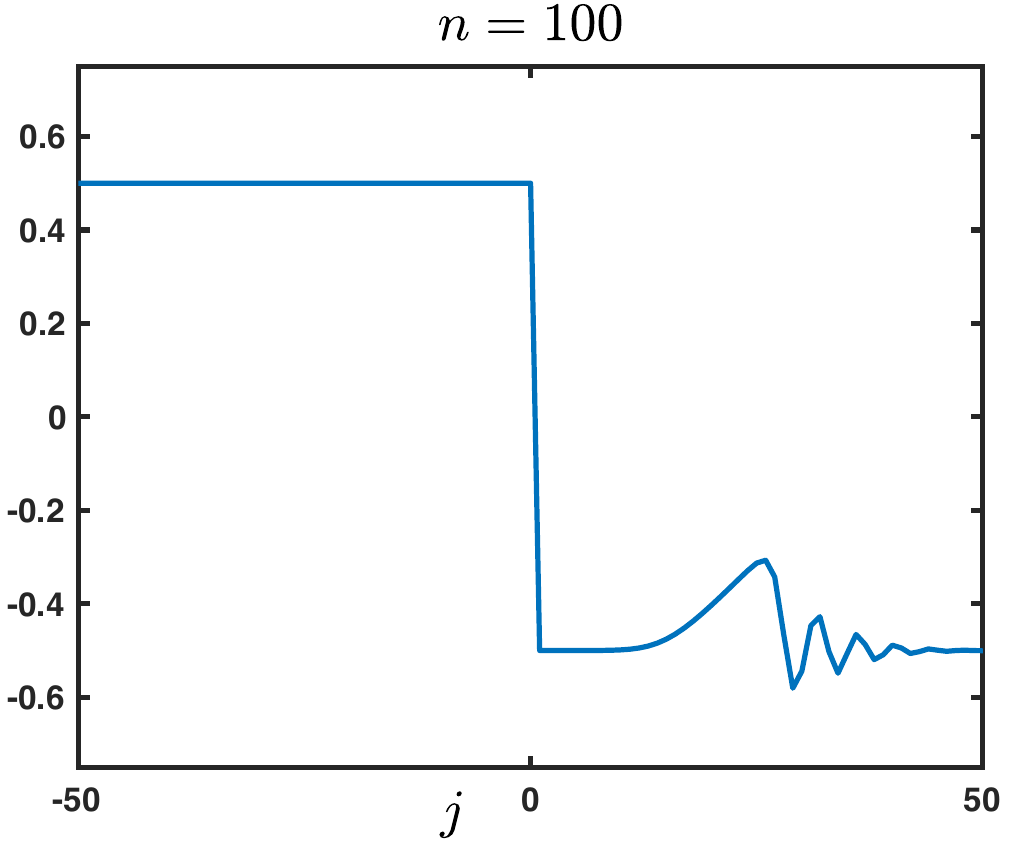}
\includegraphics[scale=0.3]{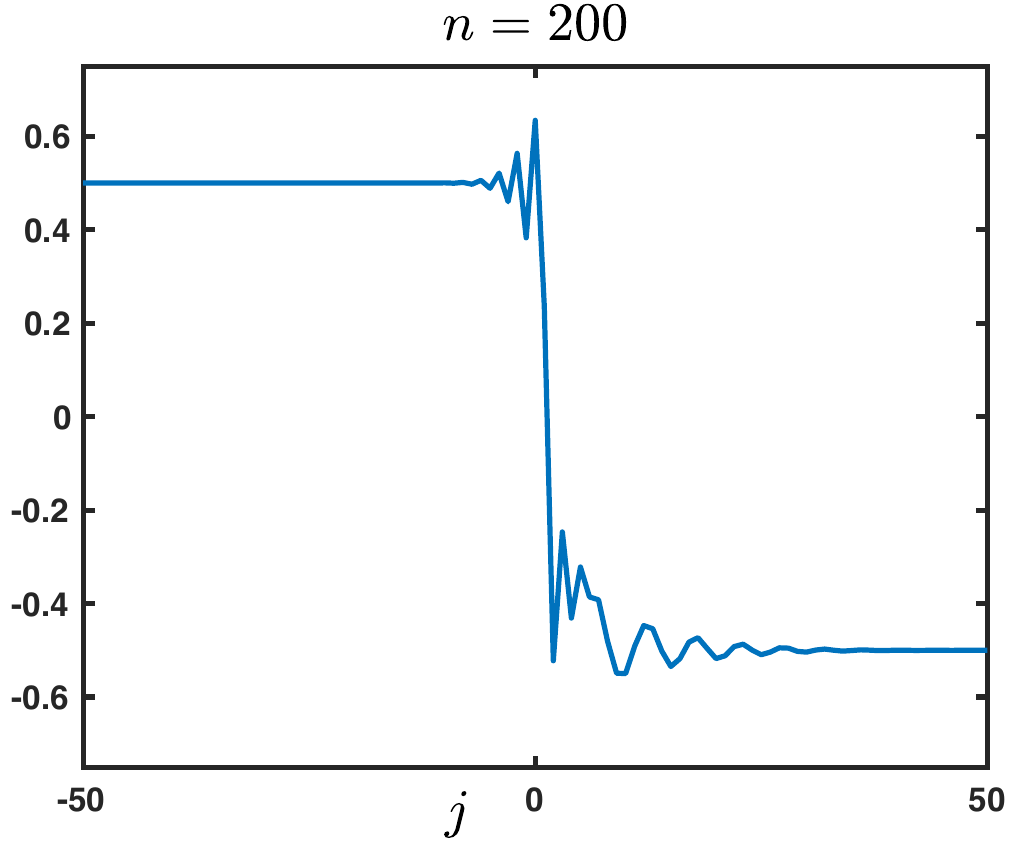} \\
\includegraphics[scale=0.3]{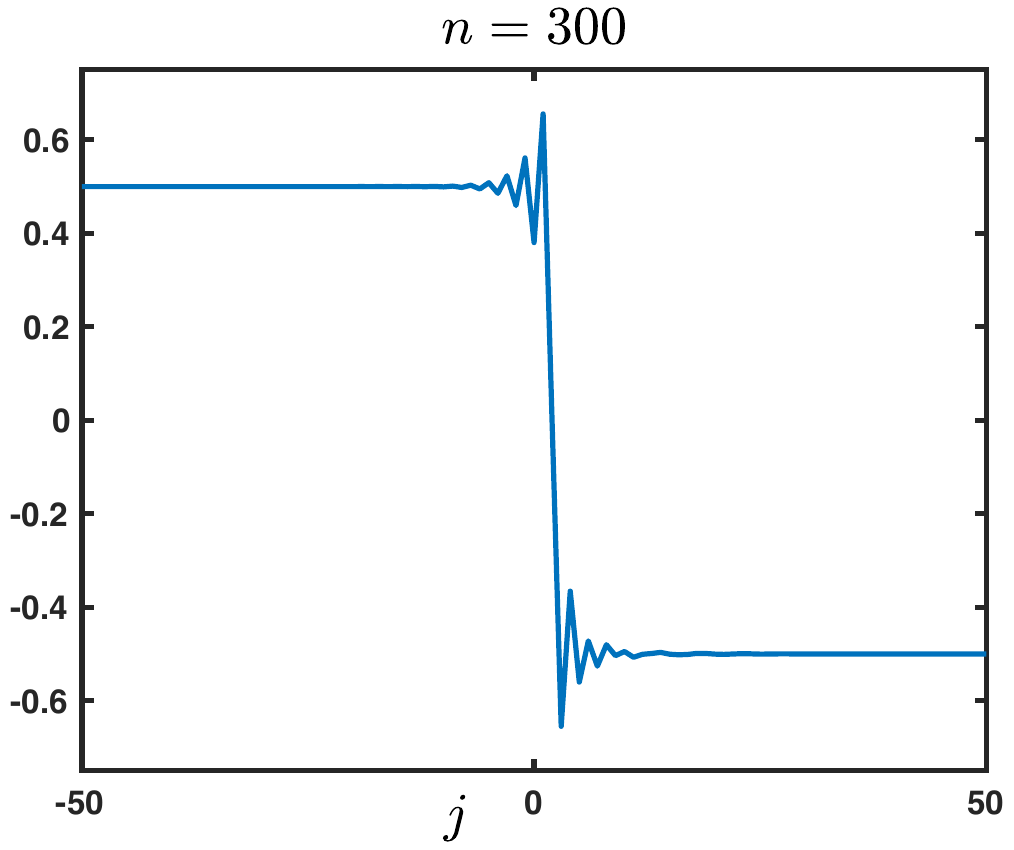}
\includegraphics[scale=0.3]{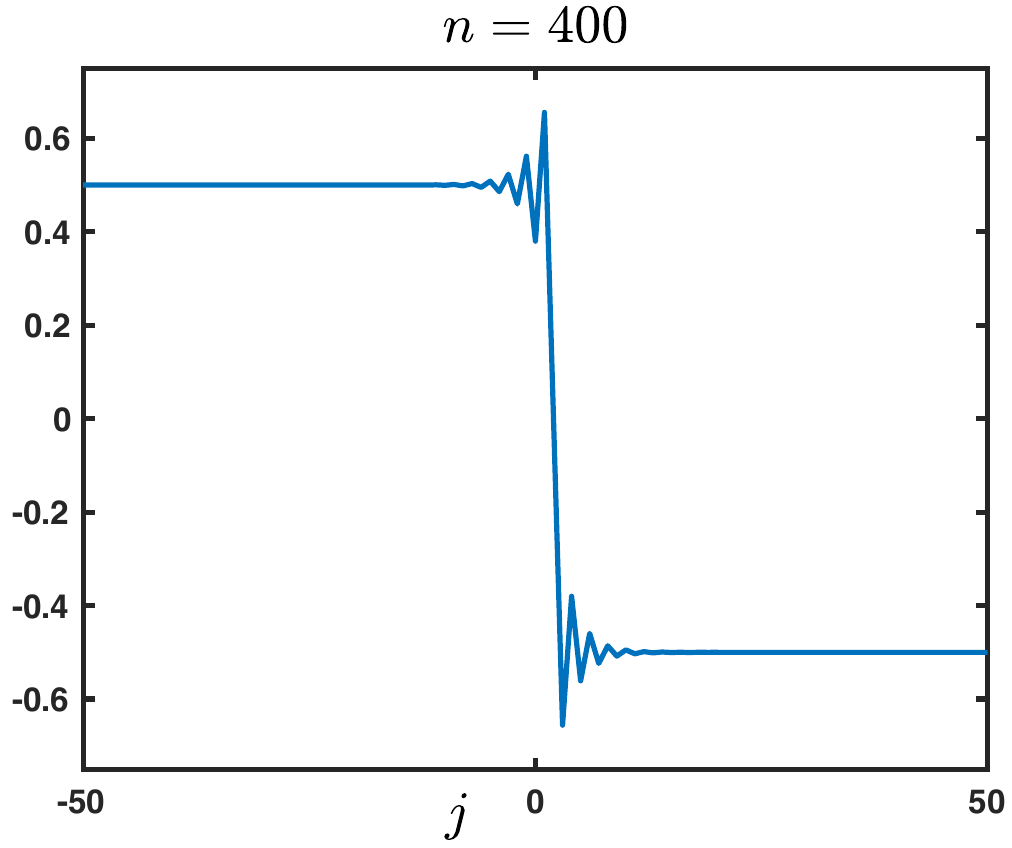}
\caption{Evolution of a perturbation with positive mass of the reference discrete shock \eqref{shock}. First line (from left to right): the initial condition, 
the solution at $n=100$, the solution at $n=200$. Second line (from left to right): the solution at $n=300$, the solution at $n=400$ (convergence 
towards a discrete shock profile).}
\label{fig:simulation-num-2}
\end{figure}

\begin{figure}\centering
\includegraphics[scale=0.5]{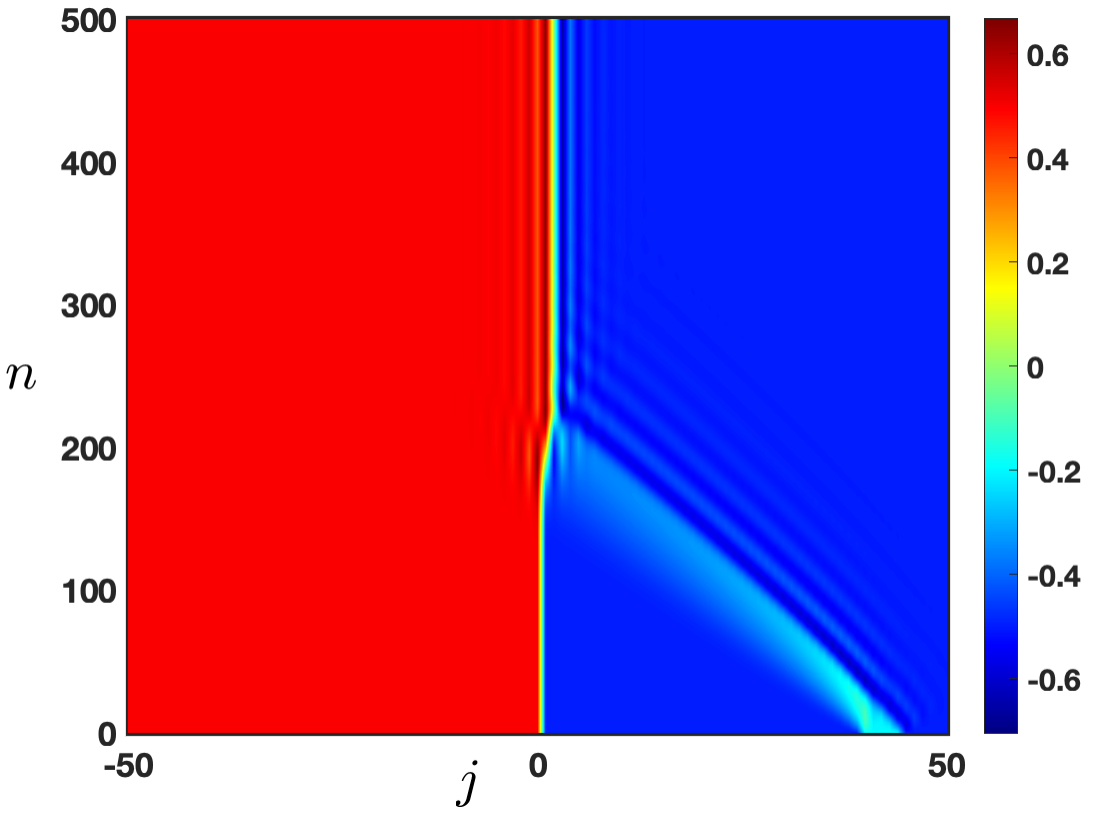}
\caption{Evolution of a perturbation with positive mass of the reference discrete shock \eqref{shock}.}
\label{fig:simulation-2-tx}
\end{figure}

%%%%%%%%%%%%%%%%%%%%%%%%%%%%%%%%%%%%%%%%%%%%%%%%%%%%%%%%%%%%%%%%%%%%%%%%%%%
\chapter{Spectral stability}
\label{chapter3}

Localizing the spectrum of the linearized operator $\mathscr{L}$ proceeds in several steps. In the first Section below, we analyze the eigenvalue 
problem\footnote{Since the only discrete shock profile that will appear in this Chapter is the piecewise constant one in \eqref{DSP}, we feel free 
to use the notation $\mathbf{v}$ for generic sequences. No possible confusion can be made with the family of discrete shock profiles $\mathbf{v}^\theta$ 
that will not be mentioned in this Chapter.} $\mathscr{L} \, \mathbf{v}=z\, \mathbf{v}$ for $z$ in the exterior $\mathscr{O}$ of the curve \eqref{courbespectre}. 
We show that the existence of a nonzero eigensequence $\mathbf{v} \in \ell^q(\Z;\C)$ is equivalent to a scalar equation $\underline{\Delta}(z)=0$, 
where the so-called Lopatinskii determinant $\underline{\Delta}$ is a holomorphic function on $\mathscr{O}$, and, actually, even on a larger region 
of the complex plane. This result is independent of the considered space $\ell^q(\Z;\C)$, $1 \le q \le +\infty$, and the Lopatinskii determinant 
$\underline{\Delta}$ does not depend on $q$. Since the discrete shock profile $\overline{\mathbf{u}}$ in \eqref{DSP} is piecewise constant and the numerical 
scheme \eqref{schemeLW} only involves a three point stencil, the Lopatinskii determinant $\underline{\Delta}$ is explicitly computable\footnote{The 
situation is opposite to the case of the Lax-Friedrichs scheme, for instance, where the discrete shock profiles depend on the spatial variable and the 
localization of the spectrum of the linearized operator involves an Evans function that is not accessible analytically, see \cite{godillon,Coeuret1} or 
\cite{Serre-notes}. Shock profiles such as \eqref{shock} for the Lax-Wendroff scheme rather look like the stationary shock profiles for the Godunov 
scheme that are studied in \cite{BGS}.}. We analyze the zeroes of the function $\underline{\Delta}$ in the case of a convex (or concave) flux $f$ in 
\eqref{law} which seems to be new, up to our knowledge. The symmetric case $\alpha_r = -\alpha_\ell$, $\alpha_m=0$ is dealt with in \cite{HHL}. 
We shall also show in Section \ref{section3-4} that for a non-convex flux, the Lopatinskii determinant $\underline{\Delta}$ can have zeroes in the 
unstable region $\U=\{z\in\C\,|\,|z|>1\}$ or that it can have a double root at $1$. In Section \ref{section3-2}, we use our knowledge on 
$\underline{\Delta}$ to compute the so-called spatial Green's function, that is, the solution to the resolvent problem:
$$
(z \, \mathrm{Id} -\mathscr{L}) \, \mathcal{G}^{j_0}(z) \, = \, \boldsymbol{\delta}_{j_0} \, ,
$$
where $\boldsymbol{\delta}_{j_0}$ denotes the discrete Dirac mass located at the index $j_0 \in \Z$:
$$
\forall \, j \in \Z \, ,\quad (\boldsymbol{\delta}_{j_0})_j \, := \begin{cases}
1 \, ,& \text{\rm if $j=j_0$,} \\
0 \, ,& \text{\rm otherwise.}
\end{cases}
$$
The construction of the spatial Green's function and the estimates we find on it will directly show that the operator $\mathscr{L}$ has no spectrum 
outside the curve \eqref{courbespectre} as long as $f$ is convex or concave, as claimed in Theorem \ref{thm1}. Theorem \ref{thm1bis}, that considers 
more general flux functions $f$, will follow from similar arguments. The spatial Green's function is our starting point for the analysis of the large time 
behavior of the linearized numerical scheme \eqref{linearizedLW}.

%%%%%%%%%%%%%%%%%%%
\section{The Lopatinskii determinant}
\label{section3-1}

We consider a complex number $z$ in the exterior $\mathscr{O}$ of the curve \eqref{courbespectre} and we shall be first looking for solutions 
$\mathbf{v} =(v_j)_{j \in \Z} \in \ell^q(\Z;\C)$ to the eigenvalue problem $\mathscr{L} \, \mathbf{v}=z\, \mathbf{v}$. Our goal is to reduce the 
existence of a non-trivial solution $\mathbf{v}$ to an equation of the form $\underline{\Delta}(z)=0$ where the holomorphic function $\underline{\Delta}$ 
plays the role of a characteristic polynomial for the operator $\mathscr{L}$. Specifying the relation $(\mathscr{L} \, \mathbf{v})_j=z\, v_j$ to those 
indices $j \le -1$ or $j \ge 2$ (see \eqref{linear} for the definition of the operator $\mathscr{L}$), we are led to solving the dispersion relations:
\begin{subequations}
\label{modekappa}
\begin{align}
\dfrac{\alpha_\ell \, (\alpha_\ell-1)}{2} \, \kappa^2 +(1-\alpha_\ell^2-z) \, \kappa +\dfrac{\alpha_\ell \, (\alpha_\ell+1)}{2} &=0 \, ,\label{modekappal} \\
\dfrac{\alpha_r \, (\alpha_r-1)}{2} \, \kappa^2 +(1-\alpha_r^2-z) \, \kappa +\dfrac{\alpha_r \, (\alpha_r+1)}{2} &=0 \, .\label{modekappar}
\end{align}
\end{subequations}
The behavior of the roots of the above dispersion relations \eqref{modekappa} is encoded in the following result.

\begin{lemma}
\label{lem1}
Let the conditions \eqref{entropy} and \eqref{CFL} be satisfied, and let $z \in \C$ belong to the exterior $\mathscr{O}$ of the curve 
\eqref{courbespectre}. Then the dispersion relation \eqref{modekappal} has one solution $\kappa_\ell(z) \in \U$ and one solution 
$\kappa_\ell^u(z) \in \D \setminus \{ 0 \}$. Both functions depend holomorphically on $z$ over $\mathscr{O}$, and they can be 
holomorphically extended to the set:
$$
\C \setminus \left\{ 1-\alpha_\ell^2 \, +\mathbf{i} \, t \, \alpha_\ell \, \sqrt{1-\alpha_\ell^2}\,\Big|\,t\in[-1,1]\right\}\,,
$$
on which they satisfy $\kappa_\ell(z) \neq \kappa_\ell^u(z)$ and $\kappa_\ell(z) \, \kappa_\ell^u(z) \neq 0$.

Furthermore, assuming still that $z \in \C$ belongs to the exterior $\mathscr{O}$ of the curve \eqref{courbespectre}, the dispersion relation 
\eqref{modekappar} has one solution $\kappa_r(z) \in \D \setminus \{ 0 \}$ and one solution $\kappa_r^u(z) \in \U$. Both functions depend 
holomorphically on $z$ over $\mathscr{O}$, and they can be holomorphically extended to the set:
$$
\C \setminus \left\{ 1-\alpha_r^2 \, +\mathbf{i} \, t \, \alpha_r \, \sqrt{1-\alpha_r^2} \, \Big|\,t\in[-1,1]\right\} \, ,
$$
on which they satisfy $\kappa_r(z) \neq \kappa_r^u(z)$ and $\kappa_r(z) \, \kappa_r^u(z) \neq 0$.
\end{lemma}

\noindent Let us quickly observe that, for $\beta \in [-1,1]$,  the compact domain $D_\beta$ that is delimited by the ellipse:
$$
\left\{1-2\,\beta^2\sin^2\dfrac{\xi}{2}+\mathbf{i}\,\beta\sin\xi\,\Big|\,\xi\in\R\right\}\,=\,\Big\{1-\beta^2+\beta^2\,\cos\xi+\mathbf{i}\,\beta\sin\xi\,\big|\,\xi\in\R\Big\}\,,
$$
satisfies $D_{\beta_1}\subset D_{\beta_2}$ as long as $|\beta_1|\le|\beta_2|$. This is the reason why the exterior $\mathscr{O}$ of the curve 
\eqref{courbespectre} does not contain any element of the two curves:
$$
\left\{ 1-2 \, \alpha_\ell^2 \sin^2 \dfrac{\xi}{2} +\mathbf{i} \, \alpha_\ell \sin \xi \, \Big| \, \xi \in \R \right\} \quad \text{\rm and} \quad 
\left\{ 1-2 \, \alpha_r^2 \sin^2 \dfrac{\xi}{2} +\mathbf{i} \, \alpha_r \sin \xi \, \Big| \, \xi \in \R \right\} \, .
$$
We also observe that the segment:
$$
\left\{ 1-\alpha_\ell^2 \, +\mathbf{i} \, t \, \alpha_\ell \, \sqrt{1-\alpha_\ell^2} \, \Big| \, t \in [-1,1] \right\} \, ,
$$
is located in the closed ball of $\C$ centered at $0$ and with radius $\sqrt{1-\alpha_\ell^2}$. Moreover, it is located within the curve \eqref{courbespectre}. 
Same for the analogous segment associated with the ``right'' state $u_r$ rather than with $u_\ell$ (with obvious modifications). In what follows, we shall 
mainly be interested in the fact that some quantities can be holomorphically extended through the unit circle $\cercle$. The exterior $\mathscr{O}$ of the 
curve \eqref{courbespectre} contains some elements of $\D$ far from the point $1$ where it is tangent to $\cercle$, and this is the reason why, sometimes, 
we need to consider the set:
$$
\mathscr{O} \cup \left\{ \zeta \in \C \, | \, |\zeta| > \max \Big( \sqrt{1-\alpha_\ell^2},\sqrt{1-\alpha_r^2} \Big) \right\} \, .
$$

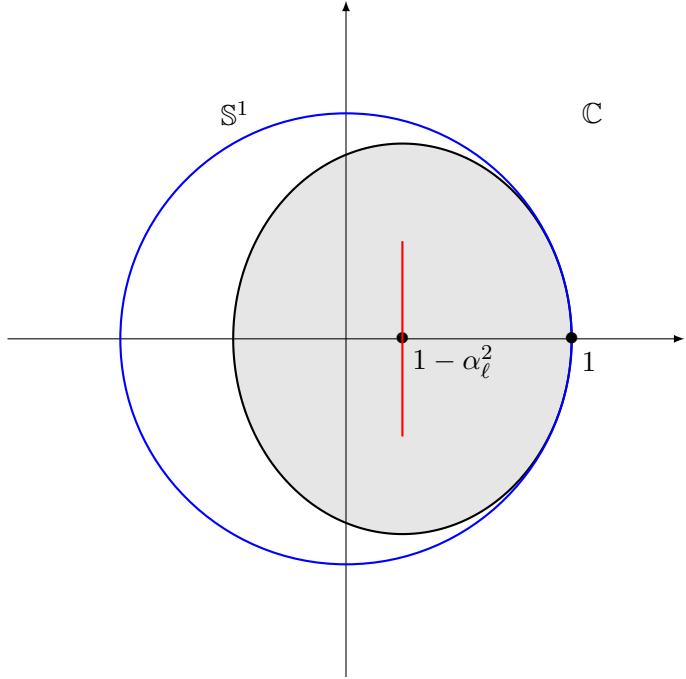
\begin{figure}[h!]
\begin{center}
\begin{tikzpicture}[scale=3,>=latex]
\draw [thick,fill=gray!20,samples=200,domain=-pi:pi] plot ({1-1.5*sin(\x/2 r)*sin(\x/2 r)},{0.5*sqrt(3)*sin(\x r)});
\draw[black,->] (-1.5,0) -- (1.5,0);
\draw[black,->] (0,-1.5)--(0,1.5);
\draw[thick,blue] (0,0) circle (1);
\draw (-0.6,1) node[right] {$\cercle$};
\draw (1,1) node[right] {$\C$};
\draw (1,0) node {$\bullet$};
\draw (0.25,0) node {$\bullet$};
\draw (1,-0.1) node[right] {$1$};
\draw (0.25,-0.1) node[right] {$1-\alpha_\ell^2$};
\draw [red,thick,domain=-1:1] plot (0.25,{sqrt(3)*\x/4});
\end{tikzpicture}
\caption{Locating the spectrum of the operator $\mathscr{L}$. In blue: the unit circle. In black: the curve \eqref{courbespectre}. The region 
$\mathscr{O}$ is the complement of the grey shaded area. In red: the segment $[1-\alpha_\ell^2-\mbi \, \alpha_\ell \, \sqrt{1-\alpha_\ell^2}, 
1-\alpha_\ell^2+\mbi \, \alpha_\ell \, \sqrt{1-\alpha_\ell^2}]$ outside of which one can holomorphically extend $\kappa_\ell$ and $\kappa_\ell^u$. 
The chosen parameter is $\alpha_\ell=\sqrt{3}/2$ with $\alpha_\ell\ge |\alpha_r|$ so that $\alpha=\alpha_\ell$.}
\label{fig:regionO}
\end{center}
\end{figure}

\begin{proof}[Proof of Lemma \ref{lem1}]
We shall only give the proof of Lemma \ref{lem1} for the case of Equation \eqref{modekappal}, the case of Equation \eqref{modekappar} being entirely similar. 
We begin with some preliminary observations. First of all, the curve \eqref{courbespectre} is a closed curve that is enclosed within the closed unit disk $\Dbar$, 
and that encompasses a strictly convex region. It is actually an ellipse that is centered at $1-\alpha^2$ with axis of half-length $\alpha^2$ and $\alpha$. This 
ellipse is tangent to the unit circle $\cercle$ from within at $1$, see Figure \ref{fig:regionO} for an illustration. As we have pointed out above, thanks to our choice 
for $\alpha$, the curve \eqref{courbespectre} encompasses both curves:
$$
\Big\{ 1-2 \, \alpha_{\ell,r}^2 \sin^2 \dfrac{\xi}{2} -\mathbf{i} \, \alpha_{\ell,r} \sin \xi \, \Big| \, \xi \in \R \Big\} \, .
$$
Let us also note that the exterior $\mathscr{O}$ of the curve \eqref{courbespectre} is a connected set and that, when $z$ belongs to $\mathscr{O}$, Equation 
\eqref{modekappal} has no root $\kappa \in \cercle$, for otherwise we would have:
$$
z =1-2 \, \alpha_\ell^2 \sin^2 \dfrac{\theta}{2} -\mathbf{i} \, \alpha_\ell \sin \theta
$$
for some $\theta \in \R$, and this fact would imply that $z$ belongs either to the interior of the curve \eqref{courbespectre} or to its boundary (depending 
whether $\alpha=\alpha_\ell$ or $\alpha=\alpha_r$), which is precluded here by our assumption $z \in \mathscr{O}$. Consequently, the number of roots 
of Equation \eqref{modekappal} in $\U$, resp. in $\D$, does not depend on $z \in \mathscr{O}$, and these two numbers (whose sum is $2$) are determined 
by letting $z$ tend to infinity. In that case, one root of Equation \eqref{modekappal} tends to zero and the other one tends to infinity, which gives the first half 
of Lemma \ref{lem1}.

Once the functions $\kappa_\ell$ and $\kappa_\ell^u$ have been defined on the exterior $\mathscr{O}$ of the curve \eqref{courbespectre}, it only remains 
to determine how they can be holomorphically extended. The product $\kappa_\ell (z) \, \kappa_\ell^u(z)$ equals $-(1+\alpha_\ell)/(1-\alpha_\ell)$. It is 
then rather easy to observe that the two roots of \eqref{modekappal} have same modulus if and only if $z$ belongs to the segment:
$$
\Big\{ 1-\alpha_\ell^2 \, +\mathbf{i} \, t \, \alpha_\ell \, \sqrt{1-\alpha_\ell^2} \, \Big| \, t \in [-1,1] \Big\} \, ,
$$
and that segment is located within the interior of the curve \eqref{courbespectre}. If $\alpha=\alpha_\ell$, the segment is actually part of one axis of the 
ellipse. Away from that segment, we can therefore always extend $\kappa_\ell (z)$ as the root of largest modulus to \eqref{modekappal}, and since that 
root is necessarily simple\footnote{The endpoints of the segment correspond precisely to the two values of $z$ at which \eqref{modekappal} has a double 
root, the so-called ``glancing'' points. These two glancing points lie here in the unit disk $\D$ since the Lax-Wendroff scheme is dissipative.}, it depends 
locally holomorphically on $z$. The proof of Lemma \ref{lem1} is thus complete.
\end{proof}

\begin{remark}
\label{remark1}
We observe from \eqref{modekappa} that in the case $\alpha_r=-\alpha_\ell$, there holds $\kappa_r(z)=1/\kappa_\ell(z)$ for any $z$ in the exterior 
$\mathscr{O}$ of the curve \eqref{courbespectre}.
\end{remark}

Specifying the relation $(\mathscr{L} \, \mathbf{v})_j=z\, v_j$ to the indices $j \ge 2$, we obtain:
$$
\forall \, j \ge 2 \, ,\quad 
\dfrac{\alpha_r \, (\alpha_r-1)}{2} \, v_{j+1} +(1-\alpha_r^2-z) \, v_j +\dfrac{\alpha_r \, (\alpha_r+1)}{2} \, v_{j-1}=0 \, ,
$$
and because the sequence $\mathbf{v}=(v_j)_{j \in \Z}$ belongs to $\ell^q(\Z;\C)$, with $1 \le q \le +\infty$, we obtain, thanks to Lemma \ref{lem1}, 
the expression:
$$
\forall \, j \ge 1 \, ,\quad v_j =v_1 \, \kappa_r(z)^{j-1} \, .
$$
Since $\kappa_r(z)$ belongs to the unit disk $\D$, the sequence $(\kappa_r(z)^{j-1})_{j \ge 1}$ has exponential decay and therefore belongs to any 
of the spaces $\ell^p(\Z;\C)$, $1 \le p \le \infty$. Specifying now to the indices $j \le -1$, we obtain in a similar way:
$$
\forall \, j \le 0 \, ,\quad v_j =v_0 \, \kappa_\ell(z)^j \, .
$$
There now remains to determine whether we can find a nonzero pair $(v_0,v_1) \in \C^2$ such that, with the sequence $\mathbf{v}=(v_j)_{j \in \Z} \in 
\ell^q(\Z;\C)$ defined by:
$$
v_j :=\begin{cases}
v_1 \, \kappa_r(z)^{j-1} \, ,& \text{\rm if $j \ge 1$,} \\
v_0 \, \kappa_\ell(z)^j \, ,& \text{\rm if $j \le 0$,} 
\end{cases}
$$
then there holds $(\mathscr{L} \, \mathbf{v})_0=z\, v_0$ and $(\mathscr{L} \, \mathbf{v})_1=z\, v_1$, which would provide us with a nonzero solution to the 
eigenvalue problem $\mathscr{L} \, \mathbf{v}=z\, \mathbf{v}$. Substituting the values of $v_2$ and $v_{-1}$ and collecting the terms, we find that the 
eigenvalue problem $\mathscr{L} \, \mathbf{v}=z\, \mathbf{v}$ amounts to solving the $2 \times 2$ linear system:
\begin{equation}
\label{eigenvalue}
\begin{bmatrix}
\dfrac{\alpha_\ell}{2} \big( \alpha_\ell-\alpha_m +(1-\alpha_\ell) \, \kappa_\ell (z) \big) & -\dfrac{\alpha_r}{2} \, (1-\alpha_m) \\
\dfrac{\alpha_\ell}{2} \, (1+\alpha_m) & \dfrac{\alpha_r}{2} \left( \alpha_r-\alpha_m -\dfrac{1+\alpha_r}{\kappa_r (z)} \right)
\end{bmatrix} \, \begin{bmatrix} v_0 \\ v_1 \end{bmatrix} = 0 \, .
\end{equation}
Here we have used the relations \eqref{modekappal} and \eqref{modekappar} in order to simplify some coefficients of the linear system \eqref{eigenvalue} 
and also used the fact that $\kappa_r(z)$ is nonzero. Up to the harmless nonzero factor $\alpha_\ell \, \alpha_r/4$, we are led to study the so-called 
Lopatinskii determinant $\underline{\Delta}$ associated with the linear system \eqref{eigenvalue}. Its expression is explicitly given here by:
\begin{equation}
\label{defDelta}
\underline{\Delta}(z) := 1-\alpha_m^2 
+\Big( \alpha_\ell-\alpha_m +(1-\alpha_\ell) \, \kappa_\ell (z) \Big) \left( \alpha_r-\alpha_m -\dfrac{1+\alpha_r}{\kappa_r (z)} \right) \, .
\end{equation}
If $\underline{\Delta}(z)=0$, then we can find a nonzero solution to \eqref{eigenvalue}, which means that we can find a nonzero solution to the eigenvalue 
problem $\mathscr{L} \, \mathbf{v}=z\, \mathbf{v}$. The converse is also true. We have therefore reduced, for $z \in \mathscr{O}$, the eigenvalue problem 
for the operator $\mathscr{L}$ to determining the zeroes of the function $\underline{\Delta}$.

\begin{remark}
\label{remark2}
In the symmetric case $\alpha_r=-\alpha_\ell$ (and therefore $\kappa_r(z)=1/\kappa_\ell(z)$), that is considered in \cite{HHL}, the expression 
\eqref{defDelta} of the Lopatinskii determinant reduces to:
$$
\underline{\Delta}(z) = (1-\alpha_\ell) \, (1-\kappa_\ell (z)) \, \big( 1+\alpha_\ell +(1-\alpha_\ell) \, \kappa_\ell (z) \big) \, ,
$$
\emph{independently} of the value of $\alpha_m$. For $z \in \mathscr{O}$, we have $\kappa_\ell(z) \in \U$ so $\kappa_\ell(z)$ can not equal $1$. 
Hence, for $z \in \mathscr{O}$, we have $\underline{\Delta}(z)=0$ if and only if:
$$
\kappa_\ell (z) \, = \, - \, \dfrac{1+\alpha_\ell}{1-\alpha_\ell} \, .
$$
Plugging this value in \eqref{modekappal}, we obtain $z=1 \not \in \mathscr{O}$, meaning that $\underline{\Delta}$ does not vanish on $\mathscr{O}$. 
This spectral stability result is \emph{independent} of the convexity or concavity properties of the flux $f$.
\end{remark}

We explain below how to analyze the general case of a convex (or concave) flux $f$. We summarize the properties of the Lopatinskii determinant 
$\underline{\Delta}$ in the following result. It is important to observe that the first part of Lemma \ref{lem2}, namely the holomorphy properties of 
$\underline{\Delta}$ is independent of $f$.

\begin{lemma}
\label{lem2}
Let the weak solution \eqref{shock} satisfy the Rankine-Hugoniot relation \eqref{RH} and the entropy inequalities \eqref{entropy}. Then under the 
condition \eqref{CFL} on the parameter $\lambda$, the Lopatinskii determinant $\underline{\Delta}$ defined in \eqref{defDelta} is holomorphic on 
the open set:
$$
\mathscr{O} \cup \Big\{ \zeta \in \C \, \Big| \, |\zeta| > \max \Big( \sqrt{1-\alpha_\ell^2},\sqrt{1-\alpha_r^2} \Big) \Big\} \, ,
$$
and $\underline{\Delta}(1)=0$. Furthermore, if the flux $f$ in \eqref{law} is either convex or concave, then $\underline{\Delta}$ satisfies:
\begin{itemize}
 \item $\underline{\Delta}'(1) \neq 0$,
 \item $\underline{\Delta}$ does not vanish on $\mathscr{O}$. 
\end{itemize}
\end{lemma}

\noindent The fact that $\underline{\Delta}$ vanishes at $1$ is associated with the presence of an eigenvalue at $1$ for the operator 
$\mathscr{L}$. This is discussed in more details below (we refer to \cite{Serre-notes} for a general discussion on this fact). In the terminology 
of \cite[Definition 4.1]{Serre-notes}, if the flux $f$ is either convex or concave, the stationary discrete shock profile \eqref{shock} is thus 
\emph{spectrally stable} since the non-vanishing of $\underline{\Delta}$ on $\mathscr{O}$ will imply that the spectrum of $\mathscr{L}$ 
is located within $\D \cup \{ 1 \}$ (see the final argument below in Section \ref{section3-3} after our construction of the spatial Green's 
function).

\begin{proof}[Proof of Lemma \ref{lem2}]
The holomorphy of $\underline{\Delta}$ follows from that of $\kappa_\ell$ and $\kappa_r^{-1}$ as given in Lemma \ref{lem1}. Indeed, $\kappa_\ell$ 
is holomorphic on:
$$
\C \setminus \Big\{ 1-\alpha_\ell^2 \, +\mathbf{i} \, t \, \alpha_\ell \, \sqrt{1-\alpha_\ell^2} \, \Big| \, t \in [-1,1] \Big\} \, ,
$$
and the segment that has to be excluded is contained in the closed ball $\overline{B_{\sqrt{1-\alpha_\ell^2}}(0)}$ and within the spectral curve 
\eqref{courbespectre}. Arguing similarly with $\kappa_r^{-1}$, one sees that $\underline{\Delta}$ is holomorphic on $\mathscr{O} \cup \big\{ 
\zeta \in \C \, | \, |\zeta| > \max (\sqrt{1-\alpha_\ell^2},\sqrt{1-\alpha_r^2}) \big\}$.

For the behavior of $\underline{\Delta}$ at $1$, we have:
$$
\kappa_\ell(1) = - \, \dfrac{1+\alpha_\ell}{1-\alpha_\ell} \in \U \, ,\quad  \kappa_r(1) = - \, \dfrac{1+\alpha_r}{1-\alpha_r} \in \D \, ,
$$
and we thus compute $\underline{\Delta}(1)=0$ (plug the above values at $z=1$ in \eqref{defDelta}). This proves the first part of Lemma \ref{lem2}, 
which holds independently of the convexity or concavity of $f$.

Differentiating $\underline{\Delta}$ with respect to $z$, we also get:
$$
\underline{\Delta}'(1)=(1-\alpha_\ell) \, (1-\alpha_m) \kappa_\ell'(1) 
-(1+\alpha_m) \, \dfrac{(1-\alpha_r)^2}{1+\alpha_r} \, \kappa_r'(1) \, ,
$$
and the derivatives $\kappa_\ell'(1),\kappa_r'(1)$ are obtained by differentiating \eqref{modekappal} and \eqref{modekappar} and evaluating at $z=1$:
$$
\kappa_\ell'(1) = - \, \dfrac{1+\alpha_\ell}{\alpha_\ell \, (1-\alpha_\ell)} \, ,\quad  \kappa_r'(1) = - \, \dfrac{1+\alpha_r}{\alpha_r \, (1-\alpha_r)} \, .
$$
We end up with the expression:
\begin{equation}
\label{deriveeDelta}
\underline{\Delta}'(1)=-\dfrac{(1+\alpha_\ell)(1-\alpha_m)}{\alpha_\ell}+\dfrac{(1-\alpha_r)(1+\alpha_m)}{\alpha_r} \, .
\end{equation}

Let us from now on assume that the flux $f$ is either convex or concave. The crucial consequence is that we have $\alpha_m \in [\alpha_r,\alpha_\ell]$, 
and we recall the entropy inequalities \eqref{entropy} as well as the stability restriction \eqref{CFL}. This implies that $\alpha_m$ belongs to the open 
interval $(-1,1)$ and therefore $\underline{\Delta}'(1)$ is the sum of two negative quantities. It therefore does not vanish.
\bigskip

It remains to study the other possible zeroes of $\underline{\Delta}$ and specifically to show that $\underline{\Delta}$ does not vanish on $\mathscr{O}$. 
We are first going to expand the expression of $\underline{\Delta}$ in \eqref{defDelta}. Because of the form of the dispersion relations \eqref{modekappal} 
and \eqref{modekappar}, we first write:
\begin{align*}
\kappa_\ell(z) &= \dfrac{z-1+\alpha_\ell^2 +\mathcal{W}_\ell(z)}{\alpha_\ell\, (\alpha_\ell-1)} \, ,\quad 
\mathcal{W}_\ell(z)^2 = (z-1+\alpha_\ell^2)^2 + \alpha_\ell^2 \, (1-\alpha_\ell^2) \, ,\\
\kappa_r(z) &= \dfrac{z-1+\alpha_r^2 +\mathcal{W}_r(z)}{\alpha_r\, (\alpha_r-1)} \, ,\quad 
\mathcal{W}_r(z)^2 = (z-1+\alpha_r^2)^2 + \alpha_r^2 \, (1-\alpha_r^2) \, ,
\end{align*}
where we do not mind at this stage which square root should be picked for $\mathcal{W}_\ell(z)$ and $\mathcal{W}_r(z)$. Plugging these expressions for 
$\kappa_\ell(z)$ and $\kappa_r(z)$ in \eqref{defDelta}, we obtain the expression:
\begin{equation}
\label{expressionDelta}
\alpha_\ell \, \alpha_r \, \underline{\Delta}(z) = Z^2 +\alpha_m(\alpha_\ell+\alpha_r) \, Z +\alpha_\ell \, \alpha_r 
+(Z+\alpha_m \alpha_r) \, \mathcal{W}_\ell -(Z+\alpha_m \alpha_\ell) \, \mathcal{W}_r -\mathcal{W}_\ell \, \mathcal{W}_r \, ,
\end{equation}
where here and from now on, we use the notation $Z:=z-1$ and we rather consider $\mathcal{W}_\ell$ and $\mathcal{W}_r$ as functions of $Z$.

Let us now assume that $z \in \mathscr{O}$ is a point where the Lopatinskii determinant $\underline{\Delta}$ vanishes. Then \eqref{expressionDelta} 
provides us with an expression for the product $\mathcal{W}_\ell \, \mathcal{W}_r$ which we can raise to the square (this is the reason why we do not 
care at this stage about which square root should be preferred). Namely, we introduce the polynomials:
$$
\mathcal{Q}(Z) :=Z^2 +\alpha_m(\alpha_\ell+\alpha_r) \, Z +\alpha_\ell \, \alpha_r \, ,\quad 
\mathcal{P}_{\ell,r}(Z) :=Z+\alpha_m \alpha_{\ell,r} \, ,
$$
and we see from \eqref{expressionDelta} that if $\underline{\Delta}$ vanishes at $z \in \mathscr{O}$, then $Z=z-1$ satisfies:
$$
\mathcal{W}_\ell \, \mathcal{W}_r \, = \, \mathcal{Q}(Z)+\mathcal{P}_r(Z) \, \mathcal{W}_\ell -\mathcal{P}_\ell(Z) \, \mathcal{W}_r \, .
$$
Squaring both left and right-hand sides, and then substituting the expression:
$$
\mathcal{P}_r(Z) \, \mathcal{W}_\ell -\mathcal{P}_\ell(Z) \, \mathcal{W}_r=\mathcal{W}_\ell \, \mathcal{W}_r-\mathcal{Q}(Z),
$$ 
we obtain the relation:
$$
2 \, (\mathcal{Q}-\mathcal{P}_\ell \, \mathcal{P}_r) \, \mathcal{W}_\ell \, \mathcal{W}_r \, = \, 
\mathcal{W}_\ell^2 \, \mathcal{W}_r^2 +\mathcal{Q}^2 +\mathcal{P}_r^2 \, \mathcal{W}_\ell^2  +\mathcal{P}_\ell^2 \, \mathcal{W}_r^2 \, .
$$
Expanding each side with respect to $Z$ and simplyfing by the nonzero factor $\alpha_\ell \alpha_r$, we obtain:
$$
(1-\alpha_m^2) \, \mathcal{W}_\ell \, \mathcal{W}_r \, = \, \Big\{ 1+\alpha_m^2+2\, \alpha_\ell \alpha_r -2\, \alpha_m (\alpha_\ell+\alpha_r) \Big\} \, Z^2 
+\alpha_\ell \alpha_r (1-\alpha_m^2) \, (2\, Z+1) \, .
$$
Raising one last time to the square and collecting the various terms, we end up with the following polynomial equation for $Z$:
\begin{equation}
\label{polynomeZ}
Z^2 \, \Big( 4\, \gamma \, Z^2 +2\, \beta \, Z+\beta \Big) =0 \, ,
\end{equation}
with the parameters $\beta$ and $\gamma$ being defined by:
\begin{align}
\beta &:= (1-\alpha_m^2) \, \Big\{ (\alpha_\ell-\alpha_r)^2 - \big( \alpha_m(\alpha_\ell+\alpha_r) -2\, \alpha_\ell \alpha_r \big)^2 \Big\} 
\, ,\label{lem2-defbeta}\\
\gamma &:= (\alpha_m-\alpha_r)(\alpha_\ell-\alpha_m) \, \big( 1+\alpha_\ell \alpha_r-\alpha_m(\alpha_\ell+\alpha_r) \big) 
\, .\label{lem2-defgamma}
\end{align}

Let us summarize where we are at this stage. We have shown that that if $z \in \mathscr{O}$ is a point where the Lopatinskii determinant 
$\underline{\Delta}$ vanishes, then $Z=z-1$ is a root of \eqref{polynomeZ}, with coefficients $\beta,\gamma$ given by 
\eqref{lem2-defbeta}-\eqref{lem2-defgamma}. Since the function:
$$
\alpha_m \in [\alpha_r,\alpha_\ell] \longmapsto (\alpha_\ell-\alpha_r)^2 - \big( \alpha_m(\alpha_\ell+\alpha_r) -2\, \alpha_\ell \alpha_r \big)^2
$$
is concave, its minimum on the segment $[\alpha_r,\alpha_\ell]$ is attained either at $\alpha_\ell$ or $\alpha_r$ (that is, at one of the endpoints) 
and this minimum is therefore necessarily positive. In other words, we have $\beta>0$ for all relevant values of $\alpha_m$ (let us recall that $ 
\alpha_m \in [\alpha_r,\alpha_\ell]$ follows from the convexity or concavity of $f$).

In the cases $\alpha_m=\alpha_r$ or $\alpha_m=\alpha_\ell$, the coefficient $\gamma$ vanishes. In these two extreme cases, it is clear that all 
the roots of \eqref{polynomeZ} are real. If $\alpha_m$ belongs to the open interval $(\alpha_r,\alpha_\ell)$, then $\gamma$ is also positive, and 
we need to compute the discriminant of the second order factor in \eqref{polynomeZ} in order to determine the location of its two roots. Since 
$\beta$ is positive, the discriminant has the same sign as the quantity\footnote{The key final argument is that the quantity $\beta-4 \, \gamma$ 
can be explicitly factorized !}:
$$
\beta-4 \, \gamma =\Big( (\alpha_\ell+\alpha_r)(1+\alpha_m^2)-2\, (1+\alpha_\ell \alpha_r) \alpha_m \Big)^2 \ge 0 \, .
$$
This means that for $\alpha_m \in [\alpha_r,\alpha_\ell]$, all the roots of \eqref{polynomeZ} are real, which means that the only possible roots 
of $\underline{\Delta}$ in $\mathscr{O}$ are real.
\bigskip

To complete the proof of Lemma \ref{lem2}, we thus restrict to real values of the parameter $z \in \mathscr{O}$, that is we consider either $z \in 
(-\infty,1-2\, \alpha^2)$ or $z \in (1,+\infty)$ (see Figure \ref{fig:regionO} for visualizing the set $\mathscr{O} \cap \R$). Let us first consider the case 
$z>1$, that is $Z=z-1>0$ for which we have, with the standard determination of the square root\footnote{Since $Z$ is now restrained to real 
values only, we shall only take square roots of positive real numbers.}:
$$
\kappa_\ell(z)=\dfrac{Z+\alpha_\ell^2 +\sqrt{Z^2 +2 \, \alpha_\ell^2 \, Z+\alpha_\ell^2}}{\alpha_\ell\, (\alpha_\ell-1)} \, ,\quad 
\kappa_r(z)=\dfrac{Z+\alpha_r^2 -\sqrt{Z^2 +2 \, \alpha_r^2 \, Z+\alpha_r^2}}{\alpha_r\, (\alpha_r-1)} \, .
$$
We then compute, for $z>1$, the expression:
\begin{align*}
\alpha_\ell \, \alpha_r \, \underline{\Delta}(z) =& \, Z^2 +\alpha_m(\alpha_\ell+\alpha_r) \, Z +\alpha_\ell \, \alpha_r 
+(Z+\alpha_m \alpha_r) \, \sqrt{Z^2 +2 \, \alpha_\ell^2 \, Z+\alpha_\ell^2} \\
& +(Z+\alpha_m \alpha_\ell) \, \sqrt{Z^2 +2 \, \alpha_r^2 \, Z+\alpha_r^2} 
+\sqrt{Z^2 +2 \, \alpha_\ell^2 \, Z+\alpha_\ell^2} \, \sqrt{Z^2 +2 \, \alpha_r^2 \, Z+\alpha_r^2} \, ,
\end{align*}
where we recall that $\alpha_\ell$, $\alpha_m$ and $\alpha_r$ satisfy:
$$
-1 < \alpha_r \le \alpha_m \le \alpha_\ell<1 \, ,\quad \alpha_r<0<\alpha_\ell \, .
$$
The function $ \underline{\Delta}$ is real valued and smooth (that is, analytic) on $[1,+\infty)$. It vanishes at $1$, and its derivative is given, for 
$z \ge 1$, by:
\begin{align}
\alpha_\ell \, \alpha_r \, \underline{\Delta}'(z) =& \, {\color{blue} 2 \, Z} +\alpha_m(\alpha_\ell+\alpha_r) 
+\sqrt{Z^2 +2 \, \alpha_\ell^2 \, Z+\alpha_\ell^2} +\sqrt{Z^2 +2 \, \alpha_r^2 \, Z+\alpha_r^2} \notag \\
&+(Z+\alpha_m \alpha_r) \, \dfrac{Z+\alpha_\ell^2}{\sqrt{Z^2 +2 \, \alpha_\ell^2 \, Z+\alpha_\ell^2}} 
+(Z+\alpha_m \alpha_\ell) \, \dfrac{Z+\alpha_r^2}{\sqrt{Z^2 +2 \, \alpha_r^2 \, Z+\alpha_r^2}} \label{expressionlem2-1} \\
&+{\color{blue} (Z+\alpha_\ell^2) \, \dfrac{\sqrt{Z^2 +2 \, \alpha_r^2 \, Z+\alpha_r^2}}{\sqrt{Z^2 +2 \, \alpha_\ell^2 \, Z+\alpha_\ell^2}}} 
+{\color{blue} (Z+\alpha_r^2) \, \dfrac{\sqrt{Z^2 +2 \, \alpha_\ell^2 \, Z+\alpha_\ell^2}}{\sqrt{Z^2 +2 \, \alpha_r^2 \, Z+\alpha_r^2}}} \, .\notag
\end{align}
For $z \ge 1$, we have $Z \ge 0$, and this implies:
$$
\dfrac{Z+\alpha_\ell^2}{\sqrt{Z^2 +2 \, \alpha_\ell^2 \, Z+\alpha_\ell^2}} \ge \alpha_\ell \, ,\quad 
\dfrac{Z+\alpha_r^2}{\sqrt{Z^2 +2 \, \alpha_r^2 \, Z+\alpha_r^2}} \ge |\alpha_r| \, ,
$$
which gives (deriving lower bounds for the above three blue terms in \eqref{expressionlem2-1}), for $z \ge 1$:
\begin{align}
\alpha_\ell \, \alpha_r \, \underline{\Delta}'(z) \ge D(\alpha_m,Z) :=& \, \alpha_m(\alpha_\ell+\alpha_r) +2\, \alpha_\ell \, |\alpha_r| 
+\sqrt{Z^2 +2 \, \alpha_\ell^2 \, Z+\alpha_\ell^2} +\sqrt{Z^2 +2 \, \alpha_r^2 \, Z+\alpha_r^2} \notag \\
&+(Z+\alpha_m \alpha_r) \, \dfrac{Z+\alpha_\ell^2}{\sqrt{Z^2 +2 \, \alpha_\ell^2 \, Z+\alpha_\ell^2}} 
+(Z+\alpha_m \alpha_\ell) \, \dfrac{Z+\alpha_r^2}{\sqrt{Z^2 +2 \, \alpha_r^2 \, Z+\alpha_r^2}} \, .\label{defDalphamZ}
\end{align}

We now consider the quantity $D(\alpha_m,Z)$ that is defined in \eqref{defDalphamZ} and try to show that it does not vanish for 
$\alpha_m \in [\alpha_r,\alpha_\ell]$ and $Z>0$. This is shown by deriving a positive lower bound. Indeed, the quantity $D(\alpha_m,Z)$ 
for $\alpha_\ell \, \alpha_r \, \underline{\Delta}'(z)$ is an affine function with respect to $\alpha_m \in [\alpha_r,\alpha_\ell]$. Its value is 
therefore not smaller than its values at the endpoints of the interval $[\alpha_r,\alpha_\ell]$. For $\alpha_m=\alpha_r$, we compute:
\begin{align*}
D(\alpha_r,Z)=& \, \alpha_\ell \, |\alpha_r| +\alpha_r^2 +\sqrt{Z^2 +2 \, \alpha_\ell^2 \, Z+\alpha_\ell^2} +\sqrt{Z^2 +2 \, \alpha_r^2 \, Z+\alpha_r^2} \\
& +(Z+\alpha_r^2) \, \dfrac{Z+\alpha_\ell^2}{\sqrt{Z^2 +2 \, \alpha_\ell^2 \, Z+\alpha_\ell^2}} 
+(Z-\alpha_\ell \, |\alpha_r|) \, \dfrac{Z+\alpha_r^2}{\sqrt{Z^2 +2 \, \alpha_r^2 \, Z+\alpha_r^2}} \\
\ge & \, \alpha_\ell \, |\alpha_r| \underbrace{\left( 1 -\dfrac{Z+\alpha_r^2}{\sqrt{Z^2 +2 \, \alpha_r^2 \, Z+\alpha_r^2}} \right)}_{\ge 0} 
+\sqrt{Z^2 +2 \, \alpha_\ell^2 \, Z+\alpha_\ell^2} +\sqrt{Z^2 +2 \, \alpha_r^2 \, Z+\alpha_r^2} \, .
\end{align*}
We thus obtain the (far from optimal but nevertheless sufficient !) uniform lower bound:
$$
\forall \, Z \ge 0 \, ,\quad D(\alpha_r,Z) \ge \alpha_\ell+|\alpha_r| >0 \, ,
$$
and similar arguments lead to the analogous estimate:
$$
\forall \, Z \ge 0 \, ,\quad D(\alpha_\ell,Z) \ge \alpha_\ell+|\alpha_r| >0 \, .
$$
For $\alpha_m \in [\alpha_r,\alpha_\ell]$, we have thus obtained the lower bound:
$$
\forall \, z \ge 1 \, ,\quad \alpha_\ell \, \alpha_r \, \underline{\Delta}'(z) \ge \alpha_\ell+|\alpha_r| >0 \, ,
$$
which implies that $\underline{\Delta}$ does not vanish on the open interval $(1,+\infty)$. (Let us recall that $\underline{\Delta}$ vanishes at $1$.)
\bigskip

It remains to examine the case $z \in (-\infty,1-2\, \alpha^2)$, for which we now have $Z \le -2 \, \alpha^2$ and we get the expressions:
$$
\kappa_\ell(z)=\dfrac{Z+\alpha_\ell^2 -\sqrt{Z^2 +2 \, \alpha_\ell^2 \, Z+\alpha_\ell^2}}{\alpha_\ell\, (\alpha_\ell-1)} \, ,\quad 
\kappa_r(z)=\dfrac{Z+\alpha_r^2 +\sqrt{Z^2 +2 \, \alpha_r^2 \, Z+\alpha_r^2}}{\alpha_r\, (\alpha_r-1)} \, ,
$$
which yields:
\begin{align*}
\alpha_\ell \, \alpha_r \, \underline{\Delta}(z) =& \, Z^2 +\alpha_m(\alpha_\ell+\alpha_r) \, Z +\alpha_\ell \, \alpha_r 
-(Z+\alpha_m \alpha_r) \, \sqrt{Z^2 +2 \, \alpha_\ell^2 \, Z+\alpha_\ell^2} \\
&-(Z+\alpha_m \alpha_\ell) \, \sqrt{Z^2 +2 \, \alpha_r^2 \, Z+\alpha_r^2} 
+\sqrt{Z^2 +2 \, \alpha_\ell^2 \, Z+\alpha_\ell^2} \, \sqrt{Z^2 +2 \, \alpha_r^2 \, Z+\alpha_r^2} \, .
\end{align*}
Once again, the right-hand side of the latter equality, which we denote $\mathbb{D}(\alpha_m,Z)$, is an affine function with respect to $\alpha_m$ 
and we are going to derive a lower bound for either $\alpha_m=\alpha_r$ or $\alpha_m=\alpha_\ell$, which will give a lower bound for any value 
of $\alpha_m \in [\alpha_r,\alpha_\ell]$. For $\alpha_m=\alpha_r$ and $Z \le -2 \, \alpha^2$, we have:
\begin{align*}
\mathbb{D}(\alpha_r,Z) =& \, Z^2+\alpha_r^2 \, Z +\alpha_\ell \, \alpha_r \, Z \\
&+(|Z|-\alpha_r^2) \, \sqrt{Z^2 +2 \, \alpha_\ell^2 \, Z+\alpha_\ell^2}+(|Z|+\alpha_\ell \, |\alpha_r|) \, \sqrt{Z^2 +2 \, \alpha_r^2 \, Z+\alpha_r^2} \\
&+\sqrt{Z^2 +2 \, \alpha_\ell^2 \, Z+\alpha_\ell^2} \, \sqrt{Z^2 +2 \, \alpha_r^2 \, Z+\alpha_r^2} -\alpha_\ell \, |\alpha_r| \, .
\end{align*}
Recalling that $\alpha$ stands for the maximum of $|\alpha_r|$ and $\alpha_\ell$, we find that both functions:
$$
Z \longmapsto Z^2 +2 \, \alpha_\ell^2 \, Z+\alpha_\ell^2 \, ,\quad Z \longmapsto Z^2 +2 \, \alpha_r^2 \, Z+\alpha_r^2
$$
are decreasing on $(-\infty,-2\, \alpha^2]$, and we thus derive the lower bounds:
$$
\forall \, Z \le -2 \, \alpha^2 \, ,\quad Z^2 +2 \, \alpha_{\ell,r}^2 \, Z+\alpha_{\ell,r}^2 \ge \alpha_{\ell,r}^2 \, .
$$
Using these lower bounds in the expression of $\mathbb{D}(\alpha_r,Z)$ leads to:
\begin{align*}
\forall \, Z \le -2 \, \alpha^2 \, ,\quad 
\mathbb{D}(\alpha_r,Z) \ge & \, Z^2+\alpha_r^2 \, Z +\alpha_\ell \, |\alpha_r| \, |Z| \\
&+ \sqrt{Z^2 +2 \, \alpha_\ell^2 \, Z+\alpha_\ell^2} \, \sqrt{Z^2 +2 \, \alpha_r^2 \, Z+\alpha_r^2} -\alpha_\ell \, |\alpha_r| \\
\ge & \, Z^2+\alpha_r^2 \, Z \ge 2 \, \alpha^2 \, (2 \, \alpha^2 -\alpha_r^2) \ge 2 \, \alpha^4 >0 \, .
\end{align*}
Similar arguments lead to:
$$
\forall \, Z \le -2 \, \alpha^2 \, ,\quad 
\mathbb{D}(\alpha_\ell,Z) \ge 2 \, \alpha^4 >0 \, ,
$$
and we have therefore proved that $\underline{\Delta}(z)$ does not vanish on the interval $z \in (-\infty,1-2\, \alpha^2]$. In other words, the 
Lopatinskii determinant $\underline{\Delta}$ does not vanish on $\mathscr{O}$. The conclusion of Lemma \ref{lem2} follows.
\end{proof}

\noindent Lemma \ref{lem2} has an immediate consequence regarding the solvability of the eigenvalue problem. The proof is a straightforward 
application of all above results on the Lopatinskii determinant $\underline{\Delta}$.

\begin{corollary}
\label{cor1}
Let the flux $f$ in \eqref{law} be either convex or concave. Let the weak solution \eqref{shock} satisfy the Rankine-Hugoniot relation \eqref{RH} and 
the entropy inequalities \eqref{entropy}. Then, under the CFL condition \eqref{CFL}, for any $z \in \mathscr{O}$ and $1 \le q \le +\infty$, the only 
solution $\mathbf{v} \in \ell^q(\Z;\C)$ to the eigenvalue problem $\mathscr{L} \, \mathbf{v}=z\, \mathbf{v}$ is the zero sequence.
\end{corollary}

%%%%%%%%%%%%%%%%%%%
\section{The spatial Green's function}
\label{section3-2}

In this section, we intend to construct the so-called Green's function, which amounts to inverting the operator $z \, \mathrm{Id} -\mathscr{L}$. 
We have already seen that the Lopatinskii determinant $\underline{\Delta}$ plays a crucial role in the location of the eigenvalues of the operator 
$\mathscr{L}$. We are going to show below that the condition $\underline{\Delta}(z) \neq 0$ is actually necessary and sufficient for a complex 
number $z \in \Ubar \setminus \{ 1 \}$ to lie in the resolvent set of $\mathscr{L}$. We do not wish any longer to assume that the flux $f$ is either 
convex or concave. We rather wish to deal more generally with \emph{spectrally stable} configurations. We shall therefore substitute Assumption 
\ref{hyp-stabspectrale} below in place of the convexity (or concavity) assumption on $f$. In particular, there is now no obvious reason for the 
parameter $\alpha_m$ in \eqref{defalphalrm} to belong to the interval $[\alpha_r,\alpha_\ell]$. We make from now on the following assumption.

\begin{assumption}
\label{hyp-stabspectrale}
The Lopatinskii determinant $\underline{\Delta}$ in \eqref{defDelta} associated with the discrete shock \eqref{DSP} satisfies:
\begin{itemize}
 \item $\underline{\Delta}'(1) \neq 0$,
 \item $\underline{\Delta}$ does not vanish on $\Ubar \setminus \{ 1 \}$.
\end{itemize}
\end{assumption}

As we have seen in the proof of Lemma \ref{lem2}, the holomorphy of $\underline{\Delta}$ on an open set that contains $\Ubar$ always holds 
as long as the CFL condition \eqref{CFL} is satisfied (the CFL condition \eqref{CFL} allows us to define the modes $\kappa_{r,\ell}$ by Lemma 
\ref{lem1}, and therefore the function $\underline{\Delta}$). The property $\underline{\Delta}(1)=0$ also holds independently of the convexity 
properties of $f$. Moreover, Lemma \ref{lem2} shows that Assumption \ref{hyp-stabspectrale} is satisfied whenever the flux $f$ is either convex 
or concave. There is no real difficulty to generalize the analysis of the previous Section in order to consider a slightly larger context than the sole 
case of a convex or concave flux $f$. We can indeed extend Corollary \ref{cor1} and obtain the following result.

\begin{corollary}
\label{cor1bis}
Let the weak solution \eqref{shock} satisfy the entropy inequalities \eqref{entropy}. Let the parameter $\lambda$ satisfy the CFL condition 
\eqref{CFL} and let Assumption \ref{hyp-stabspectrale} be satisfied. Then for any $z \in \Ubar \setminus \{ 1 \}$ and $1 \le q \le +\infty$, the 
only solution $\mathbf{v} \in \ell^q(\Z;\C)$ to the eigenvalue problem $\mathscr{L} \, \mathbf{v} =z\, \mathbf{v}$ is the zero sequence.
\end{corollary}

For $z$ in the resolvent set of the operator $\mathscr{L}$, we denote by $\mathcal{G}^{j_0}(z)=\left(\mathcal{G}^{j_0}_j(z)\right)_{j\in\Z} \in 
\ell^q(\Z;\C)$ the solution to the resolvent problem:
\begin{equation}
\label{defGreenspatial}
(z \, \mathrm{Id} -\mathscr{L}) \, \mathcal{G}^{j_0}(z) \, = \, \boldsymbol{\delta}_{j_0} \, ,
\end{equation}
where $\boldsymbol{\delta}_{j_0}=(\delta_{j_0}(j))_{j\in\Z}$ stands for the discrete Dirac mass located at the index $j_0 \in \Z$. The following 
result gives an explicit expression for the spatial Green's function $\mathcal{G}^{j_0}(z)$ for $z \in \Ubar \setminus \{ 1 \}$. This makes use of 
the fact that the Lopatinskii determinant does not vanish (which is the reason for Assumption \ref{hyp-stabspectrale}). Corollary \ref{cor1bis} 
above shows that the solution to \eqref{defGreenspatial} is necessarily unique for $z \in \Ubar \setminus \{ 1 \}$.

\begin{proposition}
\label{prop1}
Let the weak solution \eqref{shock} satisfy the entropy inequalities \eqref{entropy}. Let the parameter $\lambda$ satisfy the CFL condition \eqref{CFL} 
and let Assumption \ref{hyp-stabspectrale} be satisfied. Then for any $z \in \Ubar \setminus \{ 1 \}$ and for any $j_0 \in \Z$, there exists a unique 
solution $\mathcal{G}^{j_0}(z) \in \ell^q(\Z;\C)$ to the equation \eqref{defGreenspatial}.

For $j_0 \ge 1$, this sequence $\mathcal{G}^{j_0}(z)=\left(\mathcal{G}^{j_0}_j(z)\right)_{j\in\Z}$ is explicitly given by:
\begin{equation}
\label{decompositionGz-1}
\mathcal{G}^{j_0}_j(z) \, = \, \begin{cases}
-\dfrac{2 \, (1-\alpha_m)}{\alpha_\ell \, \underline{\Delta}(z)} \, \kappa_r^u(z)^{1-j_0} \, \kappa_\ell(z)^j, & \text{\rm if $j \le 0$,} \\
 & \\
-\dfrac{2 \, (\alpha_\ell-\alpha_m +(1-\alpha_\ell) \, \kappa_\ell (z))}{\alpha_r \, \underline{\Delta}(z)} \, \kappa_r^u(z)^{1-j_0} \, \kappa_r(z)^{j-1} & \\
 & \\
\qquad +\dfrac{2 \left( \kappa_r^u(z)^{1-j_0} \, \kappa_r(z)^{j-1} - \kappa_r^{u}(z)^{j-j_0}\right)}{\alpha_r(1-\alpha_r)(\kappa_r^{u}(z)-\kappa_r(z))}, & 
\text{\rm if $1 \le j \le j_0$,} \\
 & \\
-\dfrac{2 \, (\alpha_\ell-\alpha_m +(1-\alpha_\ell) \, \kappa_\ell (z))}{\alpha_r \, \underline{\Delta}(z)} \, \kappa_r^u(z)^{1-j_0} \, \kappa_r(z)^{j-1} & \\
 & \\
\qquad +\dfrac{2 \left( \kappa_r^u(z)^{1-j_0} \, \kappa_r(z)^{j-1} - \kappa_r(z)^{j-j_0} \right)}{\alpha_r(1-\alpha_r)(\kappa_r^{u}(z)-\kappa_r(z))}, & 
\text{\rm if $j>j_0$,}
\end{cases}
\end{equation}
and for $j_0 \le 0$, $\mathcal{G}^{j_0}(z)$ is explicitly given by:
\begin{equation}
\label{decompositionGz-2}
\mathcal{G}^{j_0}_j(z) \, = \, \begin{cases}
\dfrac{2 \, (1+\alpha_m)}{\alpha_r \, \underline{\Delta}(z)} \, \kappa_\ell^u(z)^{-j_0} \, \kappa_r(z)^{j-1}, & \text{\rm if $j \ge 1$,} \\
 & \\
-\dfrac{2 \, (\alpha_r-\alpha_m -(1+\alpha_r) \, \kappa_r (z)^{-1})}{\alpha_\ell \, \underline{\Delta}(z)} \, \kappa_\ell^u(z)^{-j_0} \, \kappa_\ell(z)^j & \\
 & \\
\qquad +\dfrac{2 \left( \kappa_\ell^u(z)^{-j_0} \, \kappa_\ell(z)^j - \kappa_\ell^{u}(z)^{j-j_0}\right)}{\alpha_\ell(1-\alpha_\ell)(\kappa_\ell(z)-\kappa_\ell^u(z))}, & 
\text{\rm if $j_0 \le j \le 0$,} \\
 & \\
-\dfrac{2 \, (\alpha_r-\alpha_m -(1+\alpha_r) \, \kappa_r (z)^{-1})}{\alpha_\ell \, \underline{\Delta}(z)} \, \kappa_\ell^u(z)^{-j_0} \, \kappa_\ell(z)^j & \\
 & \\
\qquad +\dfrac{2 \left( \kappa_\ell^u(z)^{-j_0} \, \kappa_\ell(z)^j - \kappa_\ell(z)^{j-j_0} \right)}{\alpha_\ell(1-\alpha_\ell)(\kappa_\ell(z)-\kappa_\ell^u(z))}, & 
\text{\rm if $j<j_0 \le 0$.}
\end{cases}
\end{equation}
\end{proposition}

\noindent The result of Proposition \ref{prop1} is independent of the space $\ell^q(\Z,\C)$ that is considered since for any $j_0 \in \Z$ and $z \in 
\Ubar \setminus \{ 1 \}$, the sequence $\mathcal{G}^{j_0}(z)$ belongs to the intersection of all spaces $\ell^p(\Z,\C)$, $p \in [1,+\infty]$ (it has 
exponential decay at infinity with respect to $j$).

\begin{proof}[Proof of Proposition \ref{prop1}]
We shall give the proof of Proposition \ref{prop1} in the case $\alpha_m \neq 1$. This assumption is used below to rewrite the second order scalar 
recurrence relation \eqref{defGreenspatial} as a first order recurrence relation for a vector $V_{j+1}$ in terms of $V_j$. In the case $\alpha_m =1$, 
which is left to the interested reader, one should rather write a first order recurrence relation for a vector $V_{j-1}$ in terms of $V_j$, that is, going 
backwards.

Let us, for a moment, consider the slightly more general problem where we let $\mathbf{h}=(h_j)_{j\in\Z}$ be a given sequence in $\ell^q(\Z,\C)$ 
and take $z \in \Ubar \setminus \{ 1 \}$. We then wish to construct a solution $\mathbf{v}(z)\in\ell^q(\Z;\C)$ to the resolvent equation:
\bqs
(z \, \mathrm{Id}-\mathscr{L}) \, \mathbf{v}(z) \, = \, \mathbf{h} \, . 
\eqs
To do so, we first introduce the augmented vectors
\bqs
\forall \, j \in \Z \, ,\quad V_j(z):=\left( \begin{array}{c}
v_{j-1}(z) \\
v_j(z)
\end{array} \right) \in \C^2, \quad  \text{ and }  \quad H_j:=\left(
\begin{array}{c}
0 \\
h_j
\end{array}
\right) \in \C^2 \, .
\eqs
Using the definition of $\mathscr{L}$, we obtain that
\bqq
\label{resolventsys}
\begin{cases}
V_{j+1}(z)=\M_r(z)V_j(z)+\A_rH_j, & j \geq 2,\\
V_{j+1}(z)=\M_\ell(z)V_j(z)+\A_\ell H_j, & j \leq -1,\\
V_2(z)=\M_{2,1}(z)V_1(z)+\A_r H_1, & \\
V_1(z)=\M_{1,0}(z)V_0(z)+\A_{r,m} H_0, &
\end{cases}
\eqq
where the above matrices in \eqref{resolventsys} are defined as
\bqs
\M_k(z):=\left(\begin{matrix} 0 & 1 \\ \dfrac{1+\alpha_k}{1-\alpha_k} & \dfrac{2(1-\alpha_k^2-z)}{\alpha_k(1-\alpha_k)} \end{matrix}\right) \, ,\quad 
\A_k:=\left(\begin{matrix} 0 & 0 \\ 0 & \dfrac{2}{\alpha_k(1-\alpha_k)}\end{matrix}\right) \, , \quad k \in \{ r,\ell \},
\eqs
and\footnote{The definition of $\M_{1,0}(z)$ and of $\A_{r,m}$ uses the assumption $\alpha_m \neq 1$.}
\begin{align*}
\M_{2,1}(z) &:= \begin{pmatrix} 0 & 1 \\ \dfrac{\alpha_\ell(1+\alpha_m)}{\alpha_r(1-\alpha_r)} & 
\dfrac{2-\alpha_r(\alpha_r+\alpha_m)-2z}{\alpha_r(1-\alpha_r)} \end{pmatrix} \, ,\\ 
\M_{1,0}(z) &:= \begin{pmatrix} 0 & 1 \\ \dfrac{\alpha_\ell(1+\alpha_\ell)}{\alpha_r(1-\alpha_m)} & 
\dfrac{2-\alpha_\ell(\alpha_\ell+\alpha_m)-2z}{\alpha_r(1-\alpha_m)} \end{pmatrix} \, ,\\
\A_{r,m} &:= \begin{pmatrix} 0 & 0 \\ 0 & \dfrac{2}{\alpha_r(1-\alpha_m)} \end{pmatrix} \, .
\end{align*}

Keeping the notation of Lemma \ref{lem1}, we denote, for $z \in \Ubar \setminus \{ 1 \}$, by $\kappa_r(z)\in\D$ and $\kappa_r^u(z)\in\U$ the 
two eigenvalues of $\M_r(z)$ (these are the roots of the dispersion relation \eqref{modekappar}), while we denote by $\kappa_\ell(z)\in\U$ and 
$\kappa_\ell^u(z)\in\D$ the two eigenvalues of $\M_\ell(z)$ (these are the roots of the dispersion relation \eqref{modekappal}). Upon denoting 
$\Pi_{r,\ell}^{s,u}(z)$ the corresponding stable/unstable spectral projections together with $\E_{r,\ell}^{s,u}(z)$ the associated eigenspaces, we 
have that\footnote{It should be kept in mind that the stable eigenvalue for the ``right'' state is $\kappa_r(z) \in \D$ since it corresponds to the 
dynamics for $j \ge 1$, while the stable eigenvalue for the ``left'' state is $\kappa_\ell(z) \in \U$ since it corresponds to the dynamics for $j \le 0$.}:
\bqs
\C^2=\E_k^{s}(z)\oplus\E_k^{u}(z), \quad k\in\left\{r,\ell\right\},
\eqs
with
\bqs
\E_k^{s}(z):=\mathrm{Span}\left(\begin{matrix} 1 \\ \kappa_k(z) \end{matrix}\right), \quad 
\E_k^{u}(z):=\mathrm{Span}\left(\begin{matrix} 1 \\  \kappa_k^u(z) \end{matrix}\right), \quad k\in\left\{r,\ell\right\}.
\eqs
Integrating the stable and unstable parts in \eqref{resolventsys}, we have that
\bqq
\forall \, j \ge 2 \, ,\quad 
\left\{
\begin{split}
\Pi_r^u(z)V_j(z)&=-\sum_{p=0}^{+\infty} \kappa_r^u(z)^{-1-p} \, \Pi_r^{u}(z) \A_r H_{j+p}, \\
\Pi_r^s(z)V_j(z)&=\kappa_r(z)^{j-2} \, \Pi_r^s(z)V_2(z)+\sum_{p=2}^{j-1} \kappa_r(z)^{j-p-1} \, \Pi_r^{s}(z) \A_r H_{p},
\end{split}
\right.
\label{soljpos}
\eqq
from which we already deduce that
\bqs
\Pi_r^u(z)V_2(z)=-\sum_{p=0}^{+\infty} \kappa_r^u(z)^{-1-p} \, \Pi_r^{u}(z)\A_r H_{2+p}.
\eqs
On the other hand, we get
\bqq
\forall \, j \le 0 \, ,\quad 
\left\{
\begin{split}
\Pi_\ell^u(z)V_j(z)&=\sum_{p=0}^{+\infty} \kappa_\ell^u(z)^{p} \, \Pi_\ell^{u}(z)\A_\ell H_{j-p-1}, \\
\Pi_\ell^s(z)V_j(z)&=\kappa_\ell(z)^{j} \, \Pi_\ell^s(z)V_0(z)-\sum_{p=j}^{-1} \kappa_\ell(z)^{p} \, \Pi_\ell^{s}(z)\A_\ell H_{j-p-1}, 
\end{split}
\right.
\label{soljneg}
\eqq
such that
\bqs
\Pi_\ell^u(z)V_0(z)=\sum_{p=0}^{+\infty} \kappa_\ell^u(z)^{p} \, \Pi_\ell^{u}(z)\A_\ell H_{-p-1}.
\eqs
We notice that both vectors $\Pi_r^s(z)V_2(z)$ and  $\Pi_\ell^s(z)V_0(z)$ still need to be determined, and if we are able to do so, then we shall 
have a solution to \eqref{resolventsys} at our disposal. First, we use the remaining two equations in \eqref{resolventsys} to obtain that
\bqs
V_2(z)=\M_{2,1}(z)\M_{1,0}(z)V_0(z)+\M_{2,1}(z)\A_{r,m}H_0+\A_rH_1,
\eqs
which we write instead as:
\begin{multline*}
\Pi_r^s(z)V_2(z)-\M_{2,1}(z)\M_{1,0}(z)\Pi_\ell^s(z)V_0(z) \\
=\M_{2,1}(z)\M_{1,0}(z)\Pi_\ell^u(z)V_0(z)-\Pi_r^u(z)V_2(z)+\M_{2,1}(z)\A_{r,m}H_0+\A_rH_1 \, .
\end{multline*}
Upon writing 
\bqs
\Pi_r^s(z)V_2(z)=\underbrace{\chi_2(z)}_{\in \C} \, E_r^s(z), \quad E_r^s(z):=\left(\begin{matrix} 1 \\ \kappa_r(z) \end{matrix}\right),
\eqs
and
\bqs
\Pi_\ell^s(z)V_0(z)=\underbrace{\chi_0(z)}_{\in \C} \, E_\ell^s(z), \quad E_\ell^s(z):=\left( \begin{matrix} 1 \\ \kappa_\ell(z) \end{matrix} \right),
\eqs
with the vector $(\chi_2(z),\chi_0(z)) \in \C^2$ still to be determined, we have that
\begin{equation}
\label{defBz}
\Pi_r^s(z)V_2(z)-\M_{2,1}(z)\M_{1,0}(z)\Pi_\ell^s(z)V_0(z) = 
\underbrace{\begin{pmatrix} E_r^s(z) & -\M_{2,1}(z)\M_{1,0}(z)E_\ell^s(z) \end{pmatrix}}_{:= \, \B(z) \, \in \, \mathscr{M}_2(\C)} \, 
\begin{pmatrix} \chi_2(z) \\ \chi_0(z) \end{pmatrix} .
\end{equation}
We thus obtain the last two relations:
\begin{align*}
 \chi_2(z)&={\rm e}_1^t \, \B(z)^{-1} \left( \M_{2,1}(z)\M_{1,0}(z)\Pi_\ell^u(z)V_0(z)-\Pi_r^u(z)V_2(z)+\M_{2,1}(z)\A_{r,m}H_0+\A_rH_1\right),\\
  \chi_0(z)&={\rm e}_2^t \, \B(z)^{-1} \left( \M_{2,1}(z)\M_{1,0}(z)\Pi_\ell^u(z)V_0(z)-\Pi_r^u(z)V_2(z)+\M_{2,1}(z)\A_{r,m}H_0+\A_rH_1\right),
\end{align*}
with $({\rm e}_1,{\rm e}_2)$ the canonical basis of $\C^2$, at least as long as the matrix $\B(z) \in\mathscr{M}_2(\C)$ in \eqref{defBz} is invertible. 
The invertibility property of the matrix $\B(z)$ is summarized in the following result which is a mere consequence of the analysis in the previous 
Section and of Assumption \ref{hyp-stabspectrale}.

\begin{lemma}
\label{lem3}
Let the parameter $\lambda$ satisfy the CFL condition \eqref{CFL} and let Assumption \ref{hyp-stabspectrale} be satisfied. Then, for any $z \in 
\Ubar \setminus \{ 1 \}$, the matrix $\B(z)$ defined in \eqref{defBz} is invertible and its determinant $\Delta(z)$ satisfies:
$$
\forall \, z \in \Ubar \setminus \{ 1 \} \, ,\quad 
\Delta (z) \, = \, - \dfrac{\alpha_\ell \, \kappa_\ell(z)}{\alpha_r(1-\alpha_r)(1-\alpha_m)} \, \underline{\Delta}(z) \, ,
$$
with $\underline{\Delta}(z)$ the Lopatinskii determinant given in \eqref{defDelta}. Moreover, the matrix $\B(z)$ is given for $z \in 
\Ubar \setminus \{ 1 \}$ by:
\begin{align*}
\B(z) &= \, \begin{pmatrix}
1 & b_1(z) \\
\kappa_r(z) & b_2(z) \end{pmatrix} \, ,\\
\text{\rm with } \quad 
b_1(z) &=-\dfrac{\alpha_\ell \, \kappa_\ell(z) \, (\alpha_\ell-\alpha_m +(1-\alpha_\ell) \, \kappa_\ell (z))}{\alpha_r \, (1-\alpha_m)} \, ,\\
b_2(z) &= -\dfrac{\alpha_\ell \, \kappa_\ell(z) \, \big( \alpha_r \, (1-\alpha_m^2) 
-(2\, (z-1) +\alpha_r (\alpha_r+\alpha_m)) \, (\alpha_\ell-\alpha_m +(1-\alpha_\ell) \, \kappa_\ell (z)) \big)}{\alpha_r^2 \, (1-\alpha_r) \, (1-\alpha_m)} \, .
\end{align*}
\end{lemma}

We omit the proof of Lemma \ref{lem3}, which is a mere algebra exercise that uses the definition of the matrices $\M_{2,1}(z)$, $\M_{1,0}(z)$ 
and the dispersion relations \eqref{modekappa}. The above methodology shows how to construct a solution $\mathbf{v}(z)$ to the resolvent 
equation for any source term $\mathbf{h}$ and $z \in \Ubar \setminus \{ 1 \}$. From Corollary \ref{cor1bis}, we know that the solution to the 
resolvent equation is necessarily unique in $\ell^q(\Z;\C)$ for any $z \in \Ubar \setminus \{ 1 \}$ since the eigenvalue problem does not have 
any nontrivial solution\footnote{Let us recall that the existence of a nontrivial solution to the eigenvalue problem $\mathscr{L} \, \mathbf{v} =z \, 
\mathbf{v}$ is equivalent to $\underline{\Delta}(z)=0$, which is precluded for $z \in \Ubar \setminus \{ 1 \}$ by Assumption \ref{hyp-stabspectrale}.}. 
We are now going to specify the above calculations to the case $\mathbf{h}=\boldsymbol{\delta}_{j_0}$ (the discrete Dirac mass located at 
$j_0 \in \Z$) and therefore derive the expression of the spatial Green's function $\mathcal{G}^{j_0}(z)$. We split the calculations according 
to the location of $j_0$ with respect to the discontinuity in the discrete shock $\overline{\mathbf{u}}$.

Let us recall that we denote by $\mathcal{G}^{j_0}(z)=\left(\mathcal{G}^{j_0}_j(z)\right)_{j\in\Z}\in\ell^q(\Z,\C)$ the solution to the resolvent equation:
\bqs
(z \mathrm{Id}-\mathscr{L}) \, \mathcal{G}^{j_0}(z) = \boldsymbol{\delta}_{j_0} \, .
\eqs
This solution exists and is unique for any $z \in \Ubar \setminus \{ 1 \}$. We also introduce the corresponding augmented vectors
\bqs
\forall \, j \in \Z \, ,\quad G^{j_0}_j(z):=\left( \begin{array}{c}
\mathcal{G}^{j_0}_{j-1}(z) \\
\mathcal{G}^{j_0}_j(z)
\end{array} \right) \in \C^2, \quad  \text{ and }  \quad H_j^{j_0}:=\left( \begin{array}{c}
0 \\
\delta_{j_0}(j) \end{array} \right) \in \C^2 .
\eqs

\paragraph{Case I: $j_0 \ge 2$.}

We first recall that for $j\geq2$ we have
\bqs
\Pi_r^u(z)G^{j_0}_j(z)=-\sum_{p=0}^{+\infty} \kappa_r^u(z)^{-1-p} \, \Pi_r^{u}(z)\A_r H_{j+p}^{j_0},
\eqs
which yields two cases:
\begin{itemize}
\item if $2\leq j \leq j_0$ then
\bqs
\Pi_r^u(z)G^{j_0}_j(z)=- \kappa_r^u(z)^{-1+j-j_0} \, \Pi_r^{u}(z)\A_r{\rm e}_2;
\eqs
\item if $j>j_0$ then
\bqs
\Pi_r^u(z)G^{j_0}_j(z)=0.
\eqs
\end{itemize}
Next, for $j\leq0$, we readily get that
\bqs
\Pi_\ell^u(z)G^{j_0}_j(z)=\sum_{p=0}^{+\infty} \kappa_\ell^u(z)^{p} \, \Pi_\ell^{u}(z) \A_\ell H_{j-p-1}^{j_0}=0.
\eqs
Using the above results, we deduce that
\begin{align*}
\M_{2,1}(z)\M_{1,0}(z)\Pi_\ell^u(z)G^{j_0}_0(z)-\Pi_r^u(z)G^{j_0}_2(z)+\M_{2,1}(z)\A_{r,m}H_0^{j_0}+\A_rH_1^{j_0}&=-\Pi_r^u(z)G^{j_0}_2(z)\\
&= \kappa_r^u(z)^{1-j_0} \, \Pi_r^{u}(z)\A_r{\rm e}_2,
\end{align*}
which implies that (the matrix $\B(z)$ is invertible for $z \in \Ubar \setminus \{ 1 \}$ by Lemma \ref{lem3}):
\begin{align*}
\Pi_r^s(z)G^{j_0}_2(z)&= \kappa_r^u(z)^{1-j_0} \left[ {\rm e}_1^t \, \B(z)^{-1} \Pi_r^{u}(z) \A_r{\rm e}_2 \right] E_r^s(z) \, ,\\
\Pi_\ell^s(z)G^{j_0}_0(z)&= \kappa_r^u(z)^{1-j_0} \left[ {\rm e}_2^t \, \B(z)^{-1} \Pi_r^{u}(z) \A_r{\rm e}_2 \right] E_\ell^s(z) \, .
\end{align*}

Now, for $j\geq2$, we observe that
\bqs
\Pi_r^s(z)G^{j_0}_j(z)=\kappa_r(z)^{j-2} \, \Pi_r^s(z)G^{j_0}_2(z)+\sum_{p=2}^{j-1} \kappa_r(z)^{j-p-1} \, \Pi_r^{s}(z)\A_r H_{p}^{j_0},
\eqs
which yields two cases:
\begin{itemize}
\item if $2\leq j \le j_0$ then
\bqs
\Pi_r^s(z)G^{j_0}_j(z)=\kappa_r(z)^{j-2} \, \kappa_r^u(z)^{1-j_0} \, \left[ {\rm e}_1^t \, \B(z)^{-1} \Pi_r^{u}(z) \A_r{\rm e}_2 \right] E_r^s(z);
\eqs
\item if $j>j_0$ then
\bqs
\Pi_r^s(z)G^{j_0}_j(z)=\kappa_r(z)^{j-2} \, \kappa_r^u(z)^{1-j_0} \, \left[ {\rm e}_1^t \, \B(z)^{-1} \Pi_r^{u}(z) \A_r{\rm e}_2 \right] E_r^s(z) 
+\kappa_r(z)^{j-j_0-1} \, \Pi_r^{s}(z)\A_r {\rm e}_2.
\eqs
\end{itemize}

Finally, for $j \le 0$, we have 
\begin{align*}
\Pi_\ell^s(z)G^{j_0}_j(z)&=\kappa_\ell(z)^{j} \, \Pi_\ell^s(z)G^{j_0}_0(z)-\sum_{p=j}^{-1} \kappa_\ell(z)^{p} \, \Pi_\ell^{s}(z)\A_\ell H_{j-p-1}^{j_0}\\
&=\kappa_\ell(z)^{j} \, \Pi_\ell^s(z)G^{j_0}_0(z) 
=\kappa_\ell(z)^{j} \, \kappa_r^u(z)^{1-j_0} \, \left[{\rm e}_2^t \, \B(z)^{-1} \Pi_r^{u}(z) \A_r{\rm e}_2 \right] E_\ell^s(z) \, .
\end{align*}
As a consequence, summarizing the above results, we have obtained for $j_0 \geq2$ that
\bqs
G^{j_0}_j(z)=\left\{
\begin{array}{lc}
\kappa_r(z)^{j-2} \, \kappa_r^u(z)^{1-j_0} \, \left[{\rm e}_1^t \, \B(z)^{-1} \Pi_r^{u}(z) \A_r{\rm e}_2 \right] E_r^s(z) 
+\kappa_r(z)^{j-j_0-1} \, \Pi_r^{s}(z) \A_r E_2, & j>j_0, \\
\kappa_r(z)^{j-2} \, \kappa_r^u(z)^{1-j_0} \, \left[{\rm e}_1^t \, \B(z)^{-1} \Pi_r^{u}(z) \A_r{\rm e}_2 \right] E_r^s(z) 
-\kappa_r^u(z)^{j-j_0-1} \, \Pi_r^{u}(z)\A_rE_2, & 2\leq j \leq j_0,\\
\kappa_r^u(z)^{1-j_0} \, \left[{\rm e}_2^t \, \B(z)^{-1} \Pi_r^{u}(z) \A_r{\rm e}_2 \right] \M_{1,0}(z)E_\ell^s(z), & j=1,\\
\kappa_\ell(z)^{j} \, \kappa_r^u(z)^{1-j_0} \left[{\rm e}_2^t \, \B(z)^{-1} \Pi_r^{u}(z) \A_r {\rm e}_2 \right] E_\ell^s(z), & j\leq 0.
\end{array}
\right.
\eqs

To recover the expression for the spatial Green's function $\mathcal{G}^{j_0}_j(z)$, one just notes that
\bqs
\mathcal{G}^{j_0}_j(z)={\rm e}_2^t \, G^{j_0}_j(z) \, ,
\eqs
meaning that $\mathcal{G}^{j_0}_j(z) \in \C$ is the second coordinate in the vector $G^{j_0}_j(z) \in \C^2$. Further computations give
\bqs
{\rm e}_2^t \, \Pi_r^{s}(z) \A_r {\rm e}_2 = \dfrac{2 \, \kappa_r(z)}{\alpha_r(1-\alpha_r)(\kappa_r(z)-\kappa_r^{u}(z))} \, ,\quad 
{\rm e}_2^t \, \Pi_r^{u}(z) \A_r {\rm e}_2 = \dfrac{2 \, \kappa_r^u(z)}{\alpha_r(1-\alpha_r)(\kappa_r^{u}(z)-\kappa_r(z))} \, ,
\eqs
and
\bqs
{\rm e}_1^t \, \B(z)^{-1} \Pi_r^{u}(z) \A_r{\rm e}_2 = 
\dfrac{2 \, (b_2(z)-b_1(z)\kappa_r^{u}(z))}{\alpha_r(1-\alpha_r)\Delta(z)(\kappa_r^{u}(z)-\kappa_r(z))} ,\quad 
{\rm e}_2^t \, \B(z)^{-1} \Pi_r^{u}(z) \A_r{\rm e}_2 = \dfrac{2}{\alpha_r(1-\alpha_r)\Delta(z)}.
\eqs
This yields for $j_0\geq2$ that
\bqs
\mathcal{G}^{j_0}_j(z)=\left\{
\begin{array}{lc}
\kappa_r(z)^{j-1} \, \kappa_r^u(z)^{1-j_0} \, \dfrac{2 \, (b_2(z)-b_1(z)\kappa_r^{u}(z))}{\alpha_r(1-\alpha_r)\Delta(z)(\kappa_r^{u}(z)-\kappa_r(z))} & \\
\qquad +\kappa_r(z)^{j-j_0} \, \dfrac{2}{\alpha_r(1-\alpha_r)(\kappa_r(z)-\kappa_r^{u}(z))}, & j \ge j_0, \\
\kappa_r(z)^{j-1} \, \kappa_r^u(z)^{1-j_0} \, \dfrac{2 \, (b_2(z)-b_1(z)\kappa_r^{u}(z))}{\alpha_r(1-\alpha_r)\Delta(z)(\kappa_r^{u}(z)-\kappa_r(z))} & \\
\qquad -\kappa_r^u(z)^{j-j_0} \, \dfrac{2}{\alpha_r(1-\alpha_r)(\kappa_r^{u}(z)-\kappa_r(z))}, & 2\leq j \leq j_0,\\
\kappa_r^u(z)^{1-j_0} \, 
\dfrac{2 \, \left(\alpha_\ell(1+\alpha_\ell)+\kappa_\ell(z)(2-\alpha_\ell(\alpha_\ell+\alpha_m)-2z) \right)}{\alpha_r^2(1-\alpha_r)(1-\alpha_m)\Delta(z)}, & j=1,\\
\kappa_\ell(z)^{1+j} \, \kappa_r^u(z)^{1-j_0} \, \dfrac{2}{\alpha_r(1-\alpha_r)\Delta(z)}, & j\leq 0.
\end{array}
\right.
\eqs

Next, we remark two points (by using Lemma \ref{lem3}):
\bqs
\dfrac{2 \, \left(\alpha_\ell(1+\alpha_\ell)+\kappa_\ell(z)(2-\alpha_\ell(\alpha_\ell+\alpha_m)-2z) \right)}{\alpha_r^2(1-\alpha_r)(1-\alpha_m)\Delta(z)} 
=-\dfrac{2 \, b_1(z)}{\alpha_r(1-\alpha_r)\Delta(z)},
\eqs
and
\bqs
\frac{2 \, (b_2(z)-b_1(z)\kappa_r^{u}(z))}{\alpha_r(1-\alpha_r)\Delta(z)(\kappa_r^{u}(z)-\kappa_r(z))} 
=-\frac{2 \, b_1(z)}{\alpha_r(1-\alpha_r)\Delta(z)} +\dfrac{2}{\alpha_r(1-\alpha_r)(\kappa_r^{u}(z)-\kappa_r(z))},
\eqs
so that the above expressions for $\mathcal{G}^{j_0}_j(z)$ simplify to
\bqs
\mathcal{G}^{j_0}_j(z)=\left\{
\begin{array}{lc}
-\kappa_r(z)^{j-1} \, \kappa_r^u(z)^{1-j_0} \, \dfrac{2 \, b_1(z)}{\alpha_r(1-\alpha_r)\Delta(z)} 
+\dfrac{2 \, \left( \kappa_r(z)^{j-1} \, \kappa_r^u(z)^{1-j_0} -\kappa_r(z)^{j-j_0}\right)}{\alpha_r(1-\alpha_r)(\kappa_r^{u}(z)-\kappa_r(z))}, & j\ge j_0, \\
-\kappa_r(z)^{j-1} \, \kappa_r^u(z)^{1-j_0} \, \dfrac{2 \, b_1(z)}{\alpha_r(1-\alpha_r)\Delta(z)} 
+\dfrac{2 \, \left( \kappa_r(z)^{j-1} \, \kappa_r^u(z)^{1-j_0} -\kappa_r^{u}(z)^{j-j_0}\right)}{\alpha_r(1-\alpha_r)(\kappa_r^{u}(z)-\kappa_r(z))}, & 
1\leq j \leq j_0,\\
\kappa_\ell(z)^{1+j} \, \kappa_r^u(z)^{1-j_0} \, \dfrac{2}{\alpha_r(1-\alpha_r)\Delta(z)}, & j\leq 0.
\end{array}
\right.
\eqs
To conclude the analysis of this first case ($j_0 \ge 2$), we recall the expression of the coefficient $b_1(z)$ and the link between the determinant 
$\Delta(z)$ of $\B(z)$ and the Lopatinskii determinant $\underline{\Delta}(z)$ (see Lemma \ref{lem3}). This gives, for any $j_0 \ge 2$, the expression 
\eqref{decompositionGz-1} for $\mathcal{G}^{j_0}_j(z)$ as given in Proposition \ref{prop1}.
\bigskip

\paragraph{Case II: $j_0=1$.} We directly notice that for $j\geq2$ we have
\bqs
\Pi_r^u(z)G^{1}_j(z)=-\sum_{p=0}^{+\infty} \kappa_r^u(z)^{-1-p} \, \Pi_r^{u}(z) \A_r H_{j+p}^{1}=0,
\eqs
together with
\bqs
\forall \, j \le 0 \, ,\quad 
\Pi_\ell^u(z)G^{1}_j(z)=\sum_{p=0}^{+\infty} \kappa_\ell^u(z)^{p} \, \Pi_\ell^{u}(z) \A_\ell H_{j-p-1}^{1}=0.
\eqs
Using the above two results, we deduce that
\bqs
\M_{2,1}(z)\M_{1,0}(z)\Pi_\ell^u(z)G^{1}_0(z)-\Pi_r^u(z)G^{1}_2(z)+\M_{2,1}(z)\A_{r,m}H_0^{1}+\A_rH_1^{1}=\A_{r}{\rm e}_2,
\eqs
which implies that
\begin{align*}
\Pi_r^s(z)G^{1}_2(z)&= \left[{\rm e}_1^t \, \B(z)^{-1} \A_{r}{\rm e}_2\right] E_r^s(z),\\
\Pi_\ell^s(z)G^{1}_0(z)&= \left[{\rm e}_2^t \, \B(z)^{-1} \A_{r}{\rm e}_2\right] E_\ell^s(z).
\end{align*}
Now, for $j\geq2$, we get that
\begin{align*}
\Pi_r^s(z)G^{1}_j(z)&=\kappa_r(z)^{j-2} \, \Pi_r^s(z)G^{1}_2(z) 
+\sum_{p=2}^{j-1} \kappa_r(z)^{j-p-1} \, \Pi_r^{s}(z) \A_r H_{p}^{1} =\kappa_r(z)^{j-2} \, \Pi_r^s(z)G^{1}_2(z) \\
&=\kappa_r(z)^{j-2} \, \left[{\rm e}_1^t \, \B(z)^{-1} \A_{r}{\rm e}_2\right] E_r^s(z).
\end{align*}
Finally, for $j\leq 0$, we have 
\begin{align*}
\Pi_\ell^s(z)G^{1}_j(z)&=\kappa_\ell(z)^{j} \, \Pi_\ell^s(z)G^{1}_0(z) -\sum_{p=j}^{-1} \kappa_\ell(z)^{p} \, \Pi_\ell^{s}(z) \A_\ell H_{j-p-1}^{1} 
=\kappa_\ell(z)^{j} \, \Pi_\ell^s(z)G^{1}_0(z)\\
&=\kappa_\ell(z)^{j} \, \left[{\rm e}_2^t \, \B(z)^{-1}\A_{r}{\rm e}_2\right] E_\ell^s(z).
\end{align*}
As a consequence, when $j_0=1$, the spatial Green's function reads (in vector form):
\bqs
G^{1}_j(z)=\left\{
\begin{array}{ll}
\kappa_r(z)^{j-2} \, \left[{\rm e}_1^t \, \B(z)^{-1}\A_{r}{\rm e}_2\right] E_r^s(z), & j \geq 2,\\
\kappa_\ell(z)^{j} \, \left[{\rm e}_2^t \, \B(z)^{-1}\A_{r}{\rm e}_2\right] E_\ell^s(z), & j \leq 0.\\
\end{array}
\right.
\eqs
Retaining only the second coordinate in each vector (or the first for $j=2$, which gives the expression of $\mathcal{G}^{1}_1(z)$), we obtain 
the expressions:
\bqs
\mathcal{G}^{1}_j(z)=\left\{
\begin{array}{ll}
-\kappa_r(z)^{j-1} \, \dfrac{2 \, b_1(z)}{\alpha_r(1-\alpha_r)\Delta(z)}, & j \geq 1,\\
\kappa_\ell(z)^{j+1} \, \dfrac{2}{\alpha_r(1-\alpha_r)\Delta(z)}, & j \leq 0.\\
\end{array}
\right.
\eqs
The expression \eqref{decompositionGz-1} (with $j_0=1$) for $\mathcal{G}^{1}_j(z)$ as given in Proposition \ref{prop1} follows from the 
expression of $b_1(z)$ and of the determinant $\Delta(z)$, see Lemma \ref{lem3}.
\bigskip

\paragraph{Case III: $j_0=0$.} We now feel free to shorten some details of the computations since many steps in Case III are similar to those in 
Case II ($j_0=1$). Since the Dirac mass is located at $j_0=0$, we have $\Pi_r^u(z)G^{0}_j(z)=0$ for $j \ge 2$ and $\Pi_\ell^u(z)G^{0}_j(z)=0$ for 
$j \le 0$. Using these two facts, we deduce that
\bqs
\M_{2,1}(z)\M_{1,0}(z)\Pi_\ell^u(z)G^{0}_0(z)-\Pi_r^u(z)G^{0}_2(z)+\M_{2,1}(z)\A_{r,m}H_0^{0}+\A_rH_1^{0}=\M_{2,1}(z)\A_{r,m}{\rm e}_2,
\eqs
which implies that
\begin{align*}
\Pi_r^s(z)G^{0}_2(z)&= \left[{\rm e}_1^t \, \B(z)^{-1} \M_{2,1}(z)\A_{r,m}{\rm e}_2\right] E_r^s(z),\\
\Pi_\ell^s(z)G^{0}_0(z)&= \left[{\rm e}_2^t \, \B(z)^{-1} \M_{2,1}(z)\A_{r,m}{\rm e}_2\right] E_\ell^s(z).
\end{align*}
Now, for $j \ge 2$, we get that
$$
\Pi_r^s(z)G^{0}_j(z)=\kappa_r(z)^{j-2} \, \left[{\rm e}_1^t \, \B(z)^{-1} \M_{2,1}(z)\A_{r,m}{\rm e}_2\right] E_r^s(z).
$$
and for $j\leq 0$, we have 
$$
\Pi_\ell^s(z)G^{0}_j(z)=\kappa_\ell(z)^{j} \, \left[{\rm e}_2^t \, \B(z)^{-1} \M_{2,1}(z)\A_{r,m}{\rm e}_2\right] E_\ell^s(z).
$$
At this stage, the spatial Green's function (in vector form) for $j_0=0$ reads:
\bqs
G^{0}_j(z)=\left\{
\begin{array}{ll}
\kappa_r(z)^{j-2} \, \left[{\rm e}_1^t \, \B(z)^{-1}\M_{2,1}(z)\A_{r,m}{\rm e}_2\right] E_r^s(z), & j \geq 2,\\
\kappa_\ell(z)^{j} \, \left[{\rm e}_2^t \, \B(z)^{-1} \M_{2,1}(z)\A_{r,m}{\rm e}_2\right] E_\ell^s(z), & j \leq 0.\\
\end{array}
\right.
\eqs

We compute the expressions:
\begin{align*}
{\rm e}_1^t \, \B(z)^{-1} \M_{2,1}(z)\A_{r,m}{\rm e}_2 &= 
\dfrac{2}{\alpha_r(1-\alpha_m)\Delta(z)}\left( b_2(z)-b_1(z)\frac{2(1-z)-\alpha_r(\alpha_r+\alpha_m)}{\alpha_r(1-\alpha_r)}\right),\\
{\rm e}_2^t \, \B(z)^{-1} \M_{2,1}(z)\A_{r,m}{\rm e}_2 &= 
\dfrac{2}{\alpha_r(1-\alpha_m)\Delta(z)}\left( -\kappa_r(z)+\frac{2(1-z)-\alpha_r(\alpha_r+\alpha_m)}{\alpha_r(1-\alpha_r)}\right) \, .
\end{align*}
Looking at either the first or second coordinate of the vector $G^{0}_j(z)$ for $j \ge 2$, we obtain the expression of $\mathcal{G}^{0}_j(z)$ 
for any $j \ge 1$. Looking then at the second coordinate of $G^{0}_j(z)$ for $j \le 0$, we obtain the expression of $\mathcal{G}^{0}_j(z)$ 
for any $j \le 0$. In the end, we obtain that for $j_0=0$, the Green's function reads:
\bqs
\mathcal{G}^{0}_j(z)=\left\{
\begin{array}{ll}
\kappa_r(z)^{j-1} \, \dfrac{2}{\alpha_r(1-\alpha_m)\Delta(z)} \, \left( 
b_2(z)-b_1(z)\dfrac{2(1-z)-\alpha_r(\alpha_r+\alpha_m)}{\alpha_r(1-\alpha_r)}\right), & j \geq 1,\\
\kappa_\ell(z)^{j+1} \, \dfrac{2}{\alpha_r(1-\alpha_m)\Delta(z)} \, \left( 
-\kappa_r(z)+\dfrac{2(1-z)-\alpha_r(\alpha_r+\alpha_m)}{\alpha_r(1-\alpha_r)}\right), & j \leq 0.\\
\end{array}
\right.
\eqs
The expression of $\mathcal{G}^{0}_j(z)$ for $j \ge 1$ is then simplified one last time by using the expressions of $b_1(z)$ and $b_2(z)$ in Lemma 
\ref{lem3} and by using the relation between $\Delta(z)$ and $\underline{\Delta}(z)$, while the expression of $\mathcal{G}^{0}_j(z)$ for $j \le 0$ is 
simplified by using the dispersion relation \eqref{modekappar} as well as the relation between $\Delta(z)$ and $\underline{\Delta}(z)$. We are then 
led to the expression \eqref{decompositionGz-2} (with $j_0=0$) of $\mathcal{G}^{0}_j(z)$ given in Proposition \ref{prop1}.
\bigskip

\paragraph{Case IV: $j_0 \le -1$.} We directly notice that we have $\Pi_r^u(z)G^{j_0}_j(z)=0$ for $j\geq2$, while for $j\leq0$ the expression
\bqs
\Pi_\ell^u(z)G^{j_0}_j(z)=\sum_{p=0}^{+\infty} \kappa_\ell^u(z)^{p} \, \Pi_\ell^{u}(z)\A_\ell H_{j-p-1}^{j_0},
\eqs
gives two cases:
\begin{itemize}
\item if $j_0< j\leq 0$ then
\bqs
\Pi_\ell^u(z)G^{j_0}_j(z)= \kappa_\ell^u(z)^{j-j_0-1} \, \Pi_\ell^{u}(z)\A_\ell {\rm e}_2;
\eqs
\item if $j\leq j_0$ then
\bqs
\Pi_\ell^u(z)G^{j_0}_j(z)=0.
\eqs
\end{itemize}
Now, using the above results, we deduce that
\begin{multline*}
\M_{2,1}(z)\M_{1,0}(z)\Pi_\ell^u(z)G^{j_0}_0(z)-\Pi_r^u(z)G^{j_0}_2(z)+\M_{2,1}(z)\A_{r,m}H_0^{j_0}+\A_rH_1^{j_0} \\
=\M_{2,1}(z)\M_{1,0}(z)\Pi_\ell^u(z)G^{j_0}_0(z) = \kappa_\ell^u(z)^{-j_0-1} \, \M_{2,1}(z)\M_{1,0}(z) \Pi_\ell^{u}(z) \A_\ell {\rm e}_2,
\end{multline*}
which implies that
\begin{align*}
\Pi_r^s(z)G^{j_0}_2(z) &= 
\kappa_\ell^u(z)^{-j_0-1} \, \left[{\rm e}_1^t \, \B(z)^{-1} \M_{2,1}(z)\M_{1,0}(z)\Pi_\ell^{u}(z) \A_\ell {\rm e}_2 \right] E_r^s(z),\\
\Pi_\ell^s(z)G^{j_0}_0(z) &= 
\kappa_\ell^u(z)^{-j_0-1} \, \left[{\rm e}_2^t \, \B(z)^{-1}\M_{2,1}(z)\M_{1,0}(z)\Pi_\ell^{u}(z) \A_\ell {\rm e}_2 \right] E_\ell^s(z).
\end{align*}
Now, for $j \ge 2$, we observe that
\begin{align*}
\Pi_r^s(z)G^{j_0}_j(z)&=\kappa_r(z)^{j-2} \, \Pi_r^s(z)G^{j_0}_2(z) \\
&=\kappa_r(z)^{j-2} \, \kappa_\ell^u(z)^{-j_0-1} \, 
\left[{\rm e}_1^t \, \B(z)^{-1}\M_{2,1}(z)\M_{1,0}(z)\Pi_\ell^{u}(z) \A_\ell {\rm e}_2 \right] E_r^s(z).
\end{align*}
Finally, for $j \le 0 $, we have 
\bqs
\Pi_\ell^s(z)G^{j_0}_j(z)=\kappa_\ell(z)^{j} \, \Pi_\ell^s(z)G^{j_0}_0(z)-\sum_{p=j}^{-1} \kappa_\ell(z)^{p} \, \Pi_\ell^{s}(z) \A_\ell H_{j-p-1}^{j_0},
\eqs
which yields two cases:
\begin{itemize}
\item if $j\leq j_0$ then
\begin{align*}
\Pi_\ell^s(z)G^{j_0}_j(z)=& \, \kappa_\ell(z)^{j} \, \kappa_\ell^u(z)^{-j_0-1} 
\left[{\rm e}_2^t \, \B(z)^{-1}\M_{2,1}(z)\M_{1,0}(z) \Pi_\ell^{u}(z) \A_\ell {\rm e}_2 \right] E_\ell^s(z) \\
&-\kappa_\ell(z)^{j-j_0-1} \, \Pi_\ell^{s}(z)\A_\ell {\rm e}_2;
\end{align*}
\item if $j_0<j\leq-1$ then
\begin{equation*}
\Pi_\ell^s(z)G^{j_0}_j(z)=\kappa_\ell(z)^{j} \, \kappa_\ell^u(z)^{-j_0-1} 
\left[{\rm e}_2^t \, \B(z)^{-1}\M_{2,1}(z)\M_{1,0}(z) \Pi_\ell^{u}(z) \A_\ell {\rm e}_2 \right] E_\ell^s(z).
\end{equation*}
\end{itemize}
As a consequence, summarizing the above results, we have obtained for $j_0\leq-1$ that
\begin{equation}
\label{prop1-derniercas}
G^{j_0}_j(z)=\left\{
\begin{array}{ll}
\kappa_r(z)^{j-2} \, \kappa_\ell^u(z)^{-j_0-1} \left[{\rm e}_1^t \, \B(z)^{-1}\M_{2,1}(z)\M_{1,0}(z)\Pi_\ell^{u}(z) \A_\ell {\rm e}_2 \right] E_r^s(z), & j \geq 2,\\
\kappa_\ell(z)^{j} \, \kappa_\ell^u(z)^{-j_0-1} \, \left[{\rm e}_2^t \, \B(z)^{-1}\M_{2,1}(z)\M_{1,0}(z)\Pi_\ell^{u}(z)\A_\ell {\rm e}_2 \right] E_\ell^s(z) & \\
\qquad +\kappa_\ell^u(z)^{j-j_0-1}\Pi_\ell^{u}(z)\A_\ell {\rm e}_2, & j_0< j \leq 0,\\
\kappa_\ell(z)^{j} \, \kappa_\ell^u(z)^{-j_0-1} \, \left[ {\rm e}_2^t \, \B(z)^{-1}\M_{2,1}(z)\M_{1,0}(z)\Pi_\ell^{u}(z) \A_\ell {\rm e}_2 \right] E_\ell^s(z)& \\
\qquad -\kappa_\ell(z)^{j-j_0-1} \, \Pi_\ell^{s}(z)\A_\ell {\rm e}_2, & j\leq j_0.
\end{array}
\right.
\end{equation}

For later use, we introduce the notation:
\begin{equation*}
\M_{2,1}(z)\M_{1,0}(z) = \begin{pmatrix} m_1(z) & m_2(z) \\ m_3(z) & m_4(z) \end{pmatrix},
\end{equation*}
We then compute:
\bqs
{\rm e}_2^t \, \Pi_\ell^{s}(z)\A_\ell {\rm e}_2 = \dfrac{2 \, \kappa_\ell(z)}{\alpha_\ell(1-\alpha_\ell)(\kappa_\ell(z)-\kappa_\ell^{u}(z))}, \quad 
{\rm e}_2^t \, \Pi_\ell^{u}(z)\A_\ell {\rm e}_2 = \dfrac{2 \, \kappa_\ell^u(z)}{\alpha_\ell(1-\alpha_\ell)(\kappa_\ell^{u}(z)-\kappa_\ell(z))},
\eqs
as well as
\begin{align*}
{\rm e}_1^t \, \B(z)^{-1}\M_{2,1}(z)\M_{1,0}(z)\Pi_\ell^{u}(z) \A_\ell {\rm e}_2 &= 
\dfrac{2 \, \Big( b_2(z)(m_1(z)+\kappa_\ell^{u}(z)m_2(z))-b_1(z)(m_3(z)+\kappa_\ell^{u}(z)m_4(z)) \Big)}{\alpha_\ell(1-\alpha_\ell) \Delta(z) 
(\kappa_\ell^{u}(z)-\kappa_\ell(z))}, \\
{\rm e}_2^t \, \B(z)^{-1}\M_{2,1}(z)\M_{1,0}(z)\Pi_\ell^{u}(z) \A_\ell {\rm e}_2 &= 
\dfrac{2 \, \Big( -\kappa_r(z)(m_1(z)+\kappa_\ell^{u}(z)m_2(z))+m_3(z)+\kappa_\ell^{u}(z)m_4(z) \Big)}{\alpha_\ell(1-\alpha_\ell) \Delta(z) 
(\kappa_\ell^{u}(z)-\kappa_\ell(z))}.
\end{align*}

Looking at either the first or second coordinate of the vector $G^{j_0}_j(z)$ for $j \ge 2$, we obtain the expression of $\mathcal{G}^{j_0}_j(z)$ for any 
index $j \ge 1$:
$$
\mathcal{G}^{j_0}_j(z)=\kappa_r(z)^{j-1} \, \kappa_\ell^u(z)^{-j_0-1} \, 
\dfrac{2 \Big( b_2(z)(m_1(z)+\kappa_\ell^{u}(z)m_2(z))-b_1(z)(m_3(z)+\kappa_\ell^{u}(z)m_4(z)) \Big)}{\alpha_\ell(1-\alpha_\ell) \Delta(z) 
(\kappa_\ell^{u}(z)-\kappa_\ell(z))} \, .
$$
Because of the definition of the matrix $\B(z)$, we have the relations:
\begin{align*}
m_1(z)+\kappa_\ell^{u}(z)m_2(z)&=-b_1(z)+(\kappa_\ell^{u}(z)-\kappa_\ell(z)) \, m_2(z) \, ,\\
m_3(z)+\kappa_\ell^{u}(z)m_4(z)&=-b_2(z)+(\kappa_\ell^{u}(z)-\kappa_\ell(z)) \, m_4(z) \, ,
\end{align*}
from which we deduce:
$$
\forall \, j \ge 1 \, ,\quad \mathcal{G}^{j_0}_j(z)=\kappa_r(z)^{j-1} \, \kappa_\ell^u(z)^{-j_0-1} \, 
\dfrac{2 \, (b_2(z)m_2(z)-b_1(z)m_4(z))}{\alpha_\ell(1-\alpha_\ell) \Delta(z)} \, ,
$$
and we now use the relations:
\begin{equation}
\label{relationsprop1}
b_1(z)=-m_1(z)-\kappa_\ell(z)m_2(z) \, ,\quad b_2(z)=-m_3(z)-\kappa_\ell(z)m_4(z) \, ,
\end{equation}
to further simplify the expression of $\mathcal{G}^{j_0}_j(z)$, $j \ge 1$, into:
\begin{align*}
\mathcal{G}^{j_0}_j(z)=& \, \kappa_r(z)^{j-1} \, \kappa_\ell^u(z)^{-j_0-1} \, 
\dfrac{2 \, (m_1(z)m_4(z)-m_2(z)m_3(z))}{\alpha_\ell(1-\alpha_\ell) \Delta(z)} \\
=& \, \kappa_r(z)^{j-1} \, \kappa_\ell^u(z)^{-j_0-1} \, 
\dfrac{2 \, \det \M_{2,1}(z) \, \det \M_{1,0}(z)}{\alpha_\ell(1-\alpha_\ell) \Delta(z)} \\
=& \, \kappa_r(z)^{j-1} \, \kappa_\ell^u(z)^{-j_0-1} \, 
\dfrac{2 \, \alpha_\ell (1+\alpha_\ell) \, (1+\alpha_m)}{\alpha_r^2 (1-\alpha_r) (1-\alpha_m) (1-\alpha_\ell) \Delta(z)} \, .
\end{align*}
To obtain the expression \eqref{decompositionGz-2} of $\mathcal{G}^{j_0}_j(z)$ for $j_0 \le -1$ and $j \ge 1$, it only remains to use the relation 
$\kappa_\ell(z) \, \kappa_\ell^u(z)=-(1+\alpha_\ell)/(1-\alpha_\ell)$ (see \eqref{modekappal}), and the link between the determinant $\Delta(z)$ 
of $\B(z)$ and $\underline{\Delta}(z)$ (Lemma \ref{lem3}).
\bigskip

The only remaining task is to derive the expression \eqref{decompositionGz-2} of $\mathcal{G}^{j_0}_j(z)$ for $j_0 \le -1$ and $j \le 0$. We go back to 
\eqref{prop1-derniercas} and retain the second coordinate of the vector $G^{j_0}_j(z)$ for $j \le 0$. We obtain:
\bqs
\mathcal{G}^{j_0}_j(z)=\left\{
\begin{array}{ll}
\kappa_\ell(z)^{j+1} \, \kappa_\ell^u(z)^{-j_0-1} \, 
\dfrac{2\left(-\kappa_r(z)(m_1(z)+\kappa_\ell^{u}(z)m_2(z))+m_3(z)+\kappa_\ell^{u}(z)m_4(z) \right)}{\alpha_\ell(1-\alpha_\ell) \Delta(z) 
(\kappa_\ell^{u}(z)-\kappa_\ell(z))} & \\
\qquad +\dfrac{2 \, \kappa_\ell^u(z)^{j-j_0}}{\alpha_\ell(1-\alpha_\ell)(\kappa_\ell^{u}(z)-\kappa_\ell(z))}, & j_0< j \leq 0,\\
\kappa_\ell(z)^{j+1} \, \kappa_\ell^u(z)^{-j_0-1} \, 
\dfrac{2\left(-\kappa_r(z)(m_1(z)+\kappa_\ell^{u}(z)m_2(z))+m_3(z)+\kappa_\ell^{u}(z)m_4(z) \right)}{\alpha_\ell(1-\alpha_\ell) \Delta(z) 
(\kappa_\ell^{u}(z)-\kappa_\ell(z))} & \\
\qquad +\dfrac{2 \, \kappa_\ell(z)^{j-j_0}}{\alpha_\ell(1-\alpha_\ell)(\kappa_\ell^u(z)-\kappa_\ell(z))} , & j\leq j_0.
\end{array}
\right.
\eqs
Recalling the relations \eqref{relationsprop1}, we get:
$$
-\kappa_r(z)(m_1(z)+\kappa_\ell^{u}(z)m_2(z))+m_3(z)+\kappa_\ell^{u}(z)m_4(z) 
=(\kappa_\ell^{u}(z)-\kappa_\ell(z))(m_4(z)-\kappa_r(z) \, m_2(z))-\Delta(z) \, ,
$$
and this simplifies the expression of $\mathcal{G}^{j_0}_j(z)$ into:
\bqs
\mathcal{G}^{j_0}_j(z)=\left\{
\begin{array}{ll}
\kappa_\ell(z)^{j+1} \, \kappa_\ell^u(z)^{-j_0-1} \, \dfrac{2 \, (m_4(z)-\kappa_r(z)m_2(z))}{\alpha_\ell(1-\alpha_\ell)\Delta(z)} & \\
\qquad +\dfrac{2 (\kappa_\ell^u(z)^{j-j_0} -\kappa_\ell(z)^{j+1} \, \kappa_\ell^u(z)^{-j_0-1})}{\alpha_\ell(1-\alpha_\ell)(\kappa_\ell^{u}(z)-\kappa_\ell(z))}, & 
j_0< j \leq 0,\\
\kappa_\ell(z)^{j+1} \, \kappa_\ell^u(z)^{-j_0-1} \, \dfrac{2 \, (m_4(z)-\kappa_r(z)m_2(z))}{\alpha_\ell(1-\alpha_\ell)\Delta(z)} & \\
\qquad +\dfrac{2 (\kappa_\ell(z)^{j-j_0} -\kappa_\ell(z)^{j+1} \, \kappa_\ell^u(z)^{-j_0-1})}{\alpha_\ell(1-\alpha_\ell)(\kappa_\ell^{u}(z)-\kappa_\ell(z))}, & 
j \le j_0.
\end{array}
\right.
\eqs

It turns out that we can slightly modify the above expression as follows. Using the expressions:
\begin{align*}
\Delta(z)=& \, b_2(z)-\kappa_r(z) \, b_1(z)=\kappa_r(z) \, m_1(z)-m_3(z)+\kappa_\ell(z)(\kappa_r(z) \, m_2(z)-m_4(z)) \, , \\
m_1(z)=& \, \dfrac{\alpha_\ell (1+\alpha_\ell)}{\alpha_r (1-\alpha_m)} \, ,\quad 
m_3(z)=-\dfrac{\alpha_\ell (1+\alpha_\ell)}{\alpha_r^2 (1-\alpha_r) (1-\alpha_m)} (2(z-1)+\alpha_r(\alpha_r+\alpha_m)) \, ,
\end{align*}
we notice that
\begin{align*}
\dfrac{1}{\kappa_\ell^{u}(z)} \, \dfrac{m_4(z)-\kappa_r(z)m_2(z)}{\alpha_\ell(1-\alpha_\ell)\Delta(z)} 
&=-\dfrac{1}{\alpha_\ell(1+\alpha_\ell)\Delta(z)}\left( -\Delta(z)+\kappa_r(z)m_1(z)-m_3(z)\right) \\
&=\dfrac{1}{\alpha_\ell(1+\alpha_\ell)} 
+\dfrac{1}{\alpha_r(1-\alpha_m)\Delta(z)}\left( -\kappa_r(z) +\frac{2(1-z)-\alpha_r(\alpha_r+\alpha_m)}{\alpha_r(1-\alpha_r)} \right)\\
&=-\dfrac{1}{\alpha_\ell(1-\alpha_\ell)\kappa_\ell(z)\kappa_\ell^{u}(z)} 
+\dfrac{\alpha_r-\alpha_m-(1+\alpha_r) \, \kappa_r(z)^{-1}}{\alpha_r (1-\alpha_r) (1-\alpha_m) \Delta(z)}.
\end{align*}
We end up with the following expression
\bqs
\mathcal{G}^{j_0}_j(z)=\left\{
\begin{array}{ll}
\dfrac{2 \, \kappa_\ell(z)^{j+1} \kappa_\ell^u(z)^{-j_0}}{\alpha_r (1-\alpha_r)(1-\alpha_m)\Delta(z)} 
\big( \alpha_r-\alpha_m-(1+\alpha_r) \, \kappa_r(z)^{-1} \big) & \\
\qquad +\dfrac{2 \big( \kappa_\ell(z)^{j-j_0} -\kappa_\ell(z)^j \, \kappa_\ell^u(z)^{-j_0} \big)}{\alpha_\ell(1-\alpha_\ell)(\kappa_\ell^{u}(z)-\kappa_\ell(z))}, & 
j \le j_0,\\
\dfrac{2 \, \kappa_\ell(z)^{j+1} \kappa_\ell^u(z)^{-j_0}}{\alpha_r (1-\alpha_r)(1-\alpha_m)\Delta(z)} 
\big( \alpha_r-\alpha_m-(1+\alpha_r) \, \kappa_r(z)^{-1} \big) & \\
\qquad +\dfrac{2 \big( \kappa_\ell^u(z)^{j-j_0} -\kappa_\ell(z)^j \, \kappa_\ell^u(z)^{-j_0} \big)}{\alpha_\ell(1-\alpha_\ell)(\kappa_\ell^{u}(z)-\kappa_\ell(z))}, & 
j_0 \le j \le 0.
\end{array}
\right.
\eqs
We then use the expression of $\Delta(z)$ given in Lemma \ref{lem3} and derive the expression \eqref{decompositionGz-2} of $\mathcal{G}^{0}_j(z)$ 
given in Proposition \ref{prop1} (for $j_0 \le -1$ and $j \le 0$). The proof of Proposition \ref{prop1} is therefore complete.
\end{proof}

The expression of the spatial Green's function given in Proposition \ref{prop1} gives us in a straightforward way the following estimates away 
from the point $1$. From inspection of the expressions \eqref{decompositionGz-1} and \eqref{decompositionGz-2}, it is useful to introduce a 
tiny modification of the spatial Green's function that we define as follows. For any couple of integers $(j_0,j) \in \Z^2$, we define the following 
function:
\begin{equation}
\label{defGtildez'}
\widetilde{\mathcal{G}}^{j_0}_j(z) \, := \, \begin{cases}
\mathcal{G}^{j_0}_j(z) \, + \, \dfrac{2 \, \kappa_r^u(z)^{j-j_0}}{\alpha_r(1-\alpha_r)(\kappa_r^u(z)-\kappa_r(z))} \, ,& 
\text{\rm if $1 \le j \le j_0$,} \\
 & \\
\mathcal{G}^{j_0}_j(z) \, + \, \dfrac{2 \, \kappa_r(z)^{j-j_0}}{\alpha_r(1-\alpha_r)(\kappa_r^u(z)-\kappa_r(z))} \, ,& 
\text{\rm if $1 \le j_0 < j$,} \\
 & \\
\mathcal{G}^{j_0}_j(z) \, + \, \dfrac{2 \, \kappa_\ell^u(z)^{j-j_0}}{\alpha_\ell(1-\alpha_\ell)(\kappa_\ell(z)-\kappa_\ell^u(z))} \, ,& 
\text{\rm if $j_0 \le j \le 0$,} \\
 & \\
\mathcal{G}^{j_0}_j(z) \, + \, \dfrac{2 \, \kappa_\ell(z)^{j-j_0}}{\alpha_\ell(1-\alpha_\ell)(\kappa_\ell(z)-\kappa_\ell^u(z))} \, ,& 
\text{\rm if $j < j_0 \le 0$,} \\
 & \\
\mathcal{G}^{j_0}_j(z) \, ,& \text{\rm otherwise.}
\end{cases}
\end{equation}
Under the assumptions made in Proposition \ref{prop1}, Proposition \ref{prop1} and Lemma \ref{lem1} show that $\widetilde{\mathcal{G}}^{j_0}_j$ 
is well-defined on $\Ubar \setminus \{ 1 \}$. Furthermore, this function is holomorphic on $\U$ and can be holomorphically extended in the 
neighborhood of any point of $\cercle \setminus \{ 1 \}$. The interest for defining this reduced function $\widetilde{\mathcal{G}}^{j_0}_j$ will 
be made clear in Chapter \ref{chapter4}. Our result is the following.

\begin{corollary}
\label{cor2}
Let the weak solution \eqref{shock} satisfy the entropy inequalities \eqref{entropy}. Let the parameter $\lambda$ satisfy the CFL condition \eqref{CFL} 
and let Assumption \ref{hyp-stabspectrale} be satisfied. Then for any $\varepsilon_\star>0$, there exist constants $\eta_\star>0$, $C>0$ and $c>0$ 
such that, if we define the set:
\begin{equation}
\label{defZepsiloneta}
\mathscr{Z}_{\varepsilon_\star,\eta_\star} \, := \, \{ \zeta \in \C \, | \, {\rm e}^{-\eta_\star} \, \le \, |\zeta| \, \le \, 2 \} \, \setminus \, 
\{ \zeta = {\rm e}^\tau \in \C \, | \, \tau \in \mathbf{B}_{\varepsilon_\star} (0) \} \, , 
\end{equation}
then, for any couple $(j,j_0) \in \Z^2$, the function $\widetilde{\mathcal{G}}^{j_0}_j$ defined in \eqref{defGtildez'} depends holomorphically on $z$ on 
$\mathscr{Z}_{\varepsilon_\star,\eta_\star}$ and it satisfies the uniform bound:
$$
\forall \, z \in \mathscr{Z}_{\varepsilon_\star,\eta_\star} \, ,\quad 
\forall \, (j_0,j) \in \Z^2 \, ,\quad \Big| \widetilde{\mathcal{G}}^{j_0}_j(z) \Big| \le C \, \exp (-c \, (|j|+|j_0|)) \, .
$$
\end{corollary}

\noindent The region $\mathscr{Z}_{\varepsilon_\star,\eta_\star}$ is schematically depicted in Figure \ref{fig:regionZ}.

\begin{figure}[h!]
\begin{center}
\begin{tikzpicture}[scale=2,>=latex]
\fill[gray!40] (0.5,-0.5) -- (1,-1) arc (-45:45:1.42) -- (0.5,0.5) arc (45:-45:0.71) ;
\fill[red!50] (0.6,0.6) -- (1,1) arc (45:315:1.42) -- (0.6,-0.6) arc (315:45:0.85) ;
\fill[red!50] (1.41,0) -- (2,0) arc (0:360:2) -- (1.41,0) arc (360:0:1.41) ;
\draw[black,->] (-2.5,0) -- (2.5,0);
%\draw[black,->] (0,-2.5)--(0,2.5);
\draw[black] (0,-2.5)--(0,0.25);
\draw[black,->] (0,0.45)--(0,2.5);
\draw[thick,blue] (0,0) circle (1);
\draw (-0.4,0.6) node[right] {$\cercle$};
\draw (2,2) node[right] {$\C$};
\draw (1,0) node {$\bullet$};
\draw (1,-0.1) node[right] {$1$};
\draw (2,0) node {$\bullet$};
\draw (2,-0.1) node[right] {$2$};
\draw (-0.65,0.15) node[right] {${\rm e}^{-\eta_\star}$};
\draw[thick,black,<->] (-0.85,0)--(0,0);
\draw[thick,black,<->] (0.5,0.5)--(0,0);
\draw (0.38,0.38) node[left] {${\rm e}^{-\varepsilon_\star}$};
\draw[thick,black,<->] (1,-1)--(0,0);
\draw (0.42,-0.42) node[left] {${\rm e}^{\varepsilon_\star}$};
\end{tikzpicture}
\caption{The region $\mathscr{Z}_{\varepsilon_\star,\eta_\star}$ (in red) of Corollary \ref{cor2}. In grey: the set $\{ \zeta = {\rm e}^\tau \in \C \, | \, 
\tau \in \mathbf{B}_{\varepsilon_\star} (0) \}$. In blue: the unit circle $\cercle$.}
\label{fig:regionZ}
\end{center}
\end{figure}

\begin{proof}[Proof of Corollary \ref{cor2}]
The proof of Corollary \ref{cor2} directly follows from the expressions \eqref{decompositionGz-1} and \eqref{decompositionGz-2} and the definition 
\eqref{defGtildez'}. Indeed, let $\varepsilon_\star>0$ be given. Then the set:
$$
\{ \zeta \in \C \, | \, 1 \, \le \, |\zeta| \, \le \, 2 \} \setminus \, \{ \zeta ={\rm e}^\tau \in \C \, | \, \tau \in \mathbf{B}_{\varepsilon_\star} (0) \} \, ,
$$
is a compact subset of $\mathscr{O}$ and also of $\Ubar \setminus \{ 1 \}$. Moreover, thanks to Assumption \ref{hyp-stabspectrale}, we know 
that the Lopatinskii determinant $\underline{\Delta}$ does not vanish on that set. Lemma \ref{lem1} also shows that the dispersion relation 
\eqref{modekappal}, resp. \eqref{modekappar}, has two distinct roots $\kappa_\ell$ and $\kappa_\ell^u$, resp. $\kappa_r$ and $\kappa_r^u$, 
for any $z$ in that set. By using Lemma \ref{lem1} (for the holomorphy properties of the roots of \eqref{modekappal} and \eqref{modekappar}), 
Lemma \ref{lem2} (for the holomorphy properties of $\underline{\Delta}$) and Assumption \ref{hyp-stabspectrale}, we can thus choose some 
$\eta_\star>0$ such that:
\begin{itemize}
 \item the set $\mathscr{Z}_{\varepsilon_\star,\eta_\star}$ defined in \eqref{defZepsiloneta} is a compact subset of $\mathscr{O}$ so that 
 $\underline{\Delta}$ is holomorphic on $\mathscr{Z}_{\varepsilon_\star,\eta_\star}$,
 \item $\kappa_\ell(z),\kappa_r^u(z)$ belong to $\U$, $\kappa_r(z),\kappa_\ell^u(z)$ belong to $\D$ for any $z \in 
 \mathscr{Z}_{\varepsilon_\star,\eta_\star}$ and those four functions depend holomorphically on $z$ on $\mathscr{Z}_{\varepsilon_\star,\eta_\star}$,
 \item $\kappa_\ell(z)$ is different from $\kappa_\ell^u(z)$, resp. $\kappa_r(z)$ is different from $\kappa_r^u(z)$, for any $z \in 
 \mathscr{Z}_{\varepsilon_\star,\eta_\star}$,
 \item $\underline{\Delta}$ is holomorphic and does not vanish on $\mathscr{Z}_{\varepsilon_\star,\eta_\star}$.
\end{itemize}
All the above properties imply that for any couple $(j,j_0) \in \Z^2$, the function $\widetilde{\mathcal{G}}^{j_0}_j$ defined in \eqref{defGtildez'} 
extends to a holomorphic function on $\mathscr{Z}_{\varepsilon_\star,\eta_\star}$.

We now consider $j_0 \ge 1$ and look at the expression \eqref{decompositionGz-1} for the spatial Green's function and the definition 
\eqref{defGtildez'} for $\widetilde{\mathcal{G}}^{j_0}_j(z)$. Since $\mathscr{Z}_{\varepsilon_\star,\eta_\star}$ is compact, we can find 
some constants $C$ such that for any $z \in \mathscr{Z}_{\varepsilon_\star,\eta_\star}$, there holds:
\begin{equation*}
\Big| \widetilde{\mathcal{G}}^{j_0}_j(z) \Big| \, \le \, \begin{cases}
C \, |\kappa_r^u(z)|^{-j_0} \, |\kappa_\ell(z)|^j \, , & \text{\rm if $j \le 0$,} \\
 & \\
C \, |\kappa_r^u(z)|^{-j_0} \, |\kappa_r(z)|^j \, , & \text{\rm if $j \ge 1$.} \\
\end{cases}
\end{equation*}
It remains to use uniform lower or upper bounds:
$$
|\kappa_r^u(z)| \, \ge \, {\rm e}^{-c} \, ,\quad |\kappa_\ell(z)| \, \ge \, {\rm e}^{-c} \, ,\quad |\kappa_r(z)| \, \le \, {\rm e}^{-c} \, ,
$$
with a uniform constant $c>0$, and the conclusion of Corollary \ref{cor2} follows in the case $j_0 \ge 1$. The case $j_0 \le 0$ follows from 
similar arguments by using the expression \eqref{decompositionGz-2} for the spatial Green's function and the definition \eqref{defGtildez'} 
for $\widetilde{\mathcal{G}}^{j_0}_j(z)$.
\end{proof}

%%%%%%%%%%%%%%%%%%%%%%%%%%%%%
\section{Spectral stability. Proof of Theorem \ref{thm1bis}}
\label{section3-3}

This short paragraph is devoted to the proof of Theorem \ref{thm1bis} (which is a more general version of Theorem \ref{thm1}) using all above 
ingredients, that is our analysis of the Lopatinskii determinant and the construction of the spatial Green's function for $z \in \Ubar \setminus \{ 1 \}$.

We shall only give the proof of Theorem \ref{thm1bis} and leave the analogous analysis for either convex or concave fluxes to the interested reader. 
Let us first quickly show that $1$ is an eigenvalue of $\mathscr{L}$, which is reminiscent of the fact that the Lopatinskii determinant vanishes at $1$ 
(Lemma \ref{lem2}), see \cite{HHL,BGS,godillon,Smyrlis,Serre-notes}. The fact that $1$ is an eigenvalue for $\mathscr{L}$ was already proven in 
\cite[Theorem 2.3]{Smyrlis}. We just reproduce a proof here, with our notation, for the sake of completeness.

We consider the expressions \eqref{decompositionGz-1} and \eqref{decompositionGz-2} of the spatial Green's function. Since $\underline{\Delta}$ 
vanishes at $1$, these expressions incorporate a (simple) pole at $z=1$. We thus introduce the sequence 
$(\mathcal{H}_j)_{j \in \Z}$ that is defined by:
$$
\mathcal{H}_j \, := \, \lim_{z \rightarrow 1} \, (z-1) \, \mathcal{G}^{j_0}_j(z) \, ,
$$
and whose precise expression is given by:
\begin{equation}
\label{defH}
\mathcal{H}_j \, = \, 
\begin{cases}
-\dfrac{2 \, (1-\alpha_m)}{\alpha_\ell \, \underline{\Delta}'(1)} \, \kappa_\ell(1)^j, & \text{\rm if $j \le 0$,} \\
& \\
\dfrac{2 \, (1+\alpha_m)}{\alpha_r \, \underline{\Delta}'(1)} \, \kappa_r(1)^{j-1}, & \text{\rm if $j \ge 1$,}
\end{cases}
\end{equation}
the expression being independent of $j_0 \in \Z$. The sequence given in \eqref{defH} is nonzero and it belongs to any $\ell^q(\Z;\C)$ since it has 
exponential decay at infinity (recall that $\kappa_r(1)$ belongs to $\D$ and $\kappa_\ell(1)$ belongs to $\U$). It is also a mere algebra exercise to 
verify that the sequence given in \eqref{defH} belongs to the kernel of the operator $\mathrm{Id}-\mathscr{L}$, as expected from the above formal 
analysis. This means that $1$ is an eigenvalue for $\mathscr{L}$ in any $\ell^q(\Z;\C)$.
\bigskip

Let us now show that the set\footnote{In the case of a convex or concave flux, one can show that the whole set $\mathscr{O}$ lies in the resolvent set 
of $\mathscr{L}$ since the Lopatinskii determinant $\underline{\Delta}$ does not vanish on $\mathscr{O}$ and Proposition \ref{prop1} extends to any 
$z \in \mathscr{O}$ in that case.} $\Ubar \setminus \{ 1 \}$ lies in the resolvent set for $\mathscr{L}$. We consider $\mathbf{h} \in \ell^q(\Z;\C)$ and 
we wish to construct a solution in $\ell^q(\Z;\C)$ to the resolvent equation:
\begin{equation}
\label{resolventh}
(z \, \mathrm{Id}-\mathscr{L}) \, \mathbf{v}(z) \, = \, \mathbf{h} \, .
\end{equation}
The solution will necessarily be unique because of Corollary \ref{cor1bis}. For $z \in \Ubar \setminus \{ 1 \}$ and $j \in \Z$, we define:
$$
v_j(z) \, := \, \sum_{j_0 \in \Z} \, \mathcal{G}^{j_0}_j(z) \, h_{j_0} \, ,
$$
with $\mathcal{G}^{j_0}_j(z)$ given in Proposition \ref{prop1}. Let us first show that the sequence $\mathbf{v}(z)=(v_j(z))_{j \in \Z}$ thus defined 
belongs indeed to $\ell^q(\Z;\C)$. The arguments below are given for a fixed $z \in \Ubar \setminus \{ 1 \}$ and the constants may depend on $z$. 
From the expressions \eqref{decompositionGz-1}, \eqref{decompositionGz-2} and using Lemma \ref{lem1} and Assumption \ref{hyp-stabspectrale}, 
we obtain bounds of the form:
$$
\forall \, (j_0,j) \in \Z^2 \, ,\quad \Big| \mathcal{G}^{j_0}_j(z) \Big| \le C_z \, \exp (-c_z \, (|j|+|j_0|)) +C_z \, \exp (-c_z \, |j-j_0|) \, ,
$$
with positive constants $C_z,c_z$ that may depend on $z$ but that do not depend on $(j_0,j)$. Applying either the H\"older or the Young inequality, 
we obtain that the above defined sequence $\mathbf{v}(z):=(v_j(z))_{j \in \Z}$ belongs to $\ell^q(\Z;\C)$. It is then a mere exercise to verify that 
$\mathbf{v}(z)$ is a solution to the resolvent equation \eqref{resolventh} (this is rather easy in this framework since $\mathscr{L}$ involves a finite 
stencil). We have thus shown that any point $z \in \Ubar \setminus \{ 1 \}$ belongs to the resolvent set of $\mathscr{L}$, which completes the proof 
of Theorem \ref{thm1bis}.

%%%%%%%%%%%%%
\section{Instability cases}
\label{section3-4}

This whole article will be devoted to \emph{stable} discrete shock profiles, but let us just take a little time to discuss two \emph{unstable} cases, 
just to show that spectral instabilities may occur. We go back to the expression \eqref{defDelta} of the Lopatinskii determinant. As explained in 
Remark \ref{remark2}, in the symmetric case $f'(u_\ell)=-f'(u_r)$, the Lopatinskii determinant $\underline{\Delta}$ is independent of the mid-point 
derivative $f'((u_\ell+u_r)/2)$ and spectral stability\footnote{Spectral stability means here that Assumption \ref{hyp-stabspectrale} is satisfied.} 
always holds.

\begin{figure}
\centering
\includegraphics[scale=0.6]{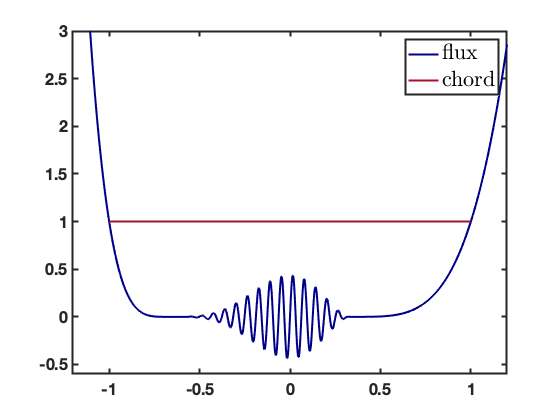}
\caption{An example of flux $f$ that yields spectrally unstable shock profiles. The graph of the flux is depicted in blue and the chord between 
$u_r=-1$ and $u_\ell=+1$ is depicted in red. The Rankine-Hugoniot condition \eqref{RH} and Oleinik's entropy condition are satisfied.}
\label{fig:graphe-instable}
\end{figure}

Let us therefore assume from now on $f'(u_\ell) \neq -f'(u_r)$ so that one has $\alpha_\ell+\alpha_r \neq 0$. We first go back to the expression 
\eqref{deriveeDelta} of the derivative $\underline{\Delta}'(1)$. We observe from this expression that $\underline{\Delta}'(1)$ can be zero if $\alpha_m$ 
is given by:
\begin{equation}
\label{scenario-instab-1}
\alpha_m \, = \, \dfrac{\alpha_r \, - \, \alpha_\ell \, + \, 2 \, \alpha_\ell \, \alpha_r}{\alpha_\ell+\alpha_r} \, .
\end{equation}
This gives for instance the value $\alpha_m=13/3$ in the case $(\alpha_\ell,\alpha_r)=(1/3,-2/3)$. The case $\underline{\Delta}'(1)=0$ is a weak form 
of instability that we do not study here but that can be achieved with a non-convex flux function $f$. An example of such a function $f$ is depicted in 
Figure \ref{fig:graphe-instable} with the choice $u_r=-1$, $u_\ell=+1$ and $f(u_r)=f(u_\ell)=1$ so that the Rankine-Hugoniot condition \eqref{RH} is 
satisfied. The entropy inequalities \eqref{entropy} are also satisfied and we can choose the CFL parameter $\lambda$ sufficiently small in such a way 
that the condition \eqref{CFL} is met. We can then tune the ``small amplitude'' oscillations near the origin in such a way that the derivative $f'(0)$ satisfies 
\eqref{scenario-instab-1} (recall the relation $\alpha_m=\lambda \; f'(0)$, see \eqref{defalphalrm}) and the graph of $f$ between $-1$ and $1$ lies 
below the horizontal chord of height $1$. This means that not only the Lax shock entropy inequalities \eqref{entropy} are satisfied but also Oleinik's 
entropy condition which is stronger, see \cite{Dafermos,Serre1}.
\bigskip

A more severe instability scenario corresponds to finding a root of $\underline{\Delta}$ in the instability region $\U$. For concreteness, we still assume 
that the end points of the shock are $u_r=-1$, $u_\ell=+1$, that the flux $f$ satisfies $f(u_r)=f(u_\ell)=1$, $f'(u_r)=-2\, f'(u_\ell)$ and that the parameter 
$\lambda$ has been chosen in such a way that $(\alpha_\ell,\alpha_r)=(1/3,-2/3)$. We then compute:
$$
\kappa_\ell(2) \, = \, -5 \, - 3 \, \sqrt{3} \, ,\quad \kappa_r(2) \, = \, \dfrac{13 \, - 3 \, \sqrt{21}}{10} \, . 
$$
We then see on the expression \eqref{defDelta} that $\underline{\Delta}$ vanishes at $z=2 \in \U$ provided that we have\footnote{This gives the 
value $\alpha_m \simeq 8,79$.}:
$$
\alpha_m \, \left( \dfrac{3}{2}+2\, \sqrt{3}-\dfrac{\sqrt{21}}{2} \right) \, = \, \dfrac{7}{2} +3 \, \left( \sqrt{3} +\sqrt{7} +\dfrac{\sqrt{21}}{2} \right) \, .
$$
Once again, this can be achieved by tuning small amplitude oscillations near the origin to have the desired value for $f'(0)$, as shown in Figure 
\ref{fig:graphe-instable}. Spectral instabilities may therefore occur for stationary shock profiles of the Lax-Wendroff scheme even though Oleinik's 
entropy condition is satisfied.

%%%%%%%%%%%%%%%%%%%%%%%%%%
\section{Decomposing the spatial Green's function}
\label{section3-5}

The detailed expression of the spatial Green's function $(\mathcal{G}^{j_0}_j(z))_{j \in \Z}$ is given in Proposition \ref{prop1}. For later use, 
we need to decompose the expression of $\mathcal{G}^{j_0}_j(z)$ by isolating several parts in it and specifically its singular behavior near 
$z=1$. A convenient way to do so is to introduce yet two other Green's functions which correspond to that of the Lax-Wendroff scheme for 
a constant coefficient transport operator on the whole real line. The chosen velocity will be either $f'(u_\ell)$ or $f'(u_r)$ depending on the 
sign of the initial position $j_0$. The choice of the velocity $f'(u_\ell)$ corresponds to the expression of the operator $\mathscr{L}$ in the 
region $\{ j \le -1 \}$, see \eqref{linear}, and the choice of the velocity $f'(u_r)$ corresponds to the expression of the operator $\mathscr{L}$ 
in the region $\{ j \ge 2 \}$, see again \eqref{linear}. We thus devote the following paragraph to recalling several facts on the Green's function 
for the Lax-Wendroff on the whole real line. Additional material in this case can be found in \cite{CF1} and \cite{jfcAMBP}.

\subsection{The free Green's function on the whole real line}

We pause for a while and go back to the definition \eqref{linear} of the linearized operator $\mathscr{L}$. In the regions $\{ j \ge 2 \}$ and 
$\{ j \le -1 \}$, the coefficients in the operator $\mathscr{L}$ are independent of the spatial index, meaning that $\mathscr{L}$ reduces to a 
convolution operator that corresponds to the linearization of the Lax-Wendroff scheme at the constant state $u_\ell$ or $u_r$. We thus 
introduce the convolution operators $\mathscr{L}_\ell,\mathscr{L}_r$ that are defined on complex valued sequences $\mathbf{v}=(v_j)_{j \in \Z}$ 
defined on $\Z$ as follows:
\begin{subequations}
\label{cauchyLW}
\begin{align}
\forall \, j \in \Z \, ,\quad 
(\mathscr{L}_\ell \, \mathbf{v})_j \, &:= \, v_j - \dfrac{\alpha_\ell}{2} \left( v_{j+1}-v_{j-1} \right) 
+\dfrac{\alpha_\ell^2}{2} \left( v_{j+1}-2 \, v_j+v_{j-1} \right) \, , \label{cauchy-l} \\
(\mathscr{L}_r \, \mathbf{v})_j \, &:= \, v_j - \dfrac{\alpha_r}{2} \left( v_{j+1}-v_{j-1 }\right) 
+ \dfrac{\alpha_r^2}{2} \left( v_{j+1}-2 \, v_j+v_{j-1} \right) \, .\label{cauchy-r}
\end{align}
\end{subequations}
We recall that $\alpha_\ell$ and $\alpha_r$ are defined in \eqref{defalphalrm}. The operators in \eqref{cauchyLW} are nothing but the operators 
arising from the Lax-Wendroff scheme applied to the transport equation with velocity equal to either $f'(u_\ell)$ or $f'(u_r)$. Thanks to the 
L\'evy-Wiener Theorem \cite{newman}, the spectrum of the operators $\mathscr{L}_\ell$ and $\mathscr{L}_r$ on any space $\ell^q(\Z;\C)$ 
is completely known (see \cite{TE} for more on the spectral analysis of convolution operators). We have:
\begin{align*}
\sigma(\mathscr{L}_\ell) &= \Big\{ 1-2 \, \alpha_\ell^2 \, \sin^2 \dfrac{\xi}{2} +\mathbf{i} \, \alpha_\ell \, \sin \xi \, | \, \xi \in \R \Big\} \, , \\
\sigma(\mathscr{L}_r) &= \Big\{ 1-2 \, \alpha_r^2 \, \sin^2 \dfrac{\xi}{2} +\mathbf{i} \, \alpha_r \, \sin \xi \, | \, \xi \in \R \Big\} \, ,
\end{align*}
where the result is independent of $q \in [1,+\infty]$. For instance, the spectrum of $\mathscr{L}_\ell$ is represented as the black curve 
(an ellipse, actually) in Figure \ref{fig:regionO}. Due to the restriction \eqref{CFL}, the spectrum of both $\mathscr{L}_\ell$ and $\mathscr{L}_r$ 
is included in the closed unit disk $\Dbar$ with a single tangency point at $1$ with the unit circle.
\bigskip

We can then proceed as in Section \ref{section3-2} above and introduce the spatial Green's function for either $\mathscr{L}_\ell$ or 
$\mathscr{L}_r$. Because of the spatial invariance in \eqref{cauchy-l} and \eqref{cauchy-r}, it is sufficient to look at the case where the 
Dirac mass $\boldsymbol{\delta}_{j_0}$ is located at $j_0=0$. Hence, for $z$ in the exterior $\mathscr{O}$ of the curve \eqref{courbespectre}, 
we have that $z$ lies in the resolvent set of both $\mathscr{L}_\ell$ and $\mathscr{L}_r$, and we can therefore introduce the solutions 
$\overline{\mathcal{G}}_\ell(z)=(\overline{\mathcal{G}}_{\ell,j}(z))_{j \in \Z}$ and $\overline{\mathcal{G}}_r(z)=(\overline{\mathcal{G}}_{r,j}(z) 
)_{j \in \Z}$ to the resolvent problems:
\begin{equation}
\label{Green-spatial-Cauchy}
(z \, \mathrm{Id} -\mathscr{L}_\ell) \, \overline{\mathcal{G}}_\ell(z) \, = \, \boldsymbol{\delta}_0 \, ,\quad 
(z \, \mathrm{Id} -\mathscr{L}_r) \, \overline{\mathcal{G}}_r(z) \, = \, \boldsymbol{\delta}_0 \, .
\end{equation}
The computations that lead to the precise expressions for $(\overline{\mathcal{G}}_{\ell,j}(z))_{j \in \Z}$ and $(\overline{\mathcal{G}}_{r,j}(z) 
)_{j \in \Z}$ follow from the same methodology as in Sections \ref{section3-1} and \ref{section3-2} (we refer to \cite{CF1} for an even more 
general analysis in the case of finite difference schemes with arbitrary, possibly infinite, stencils). The analysis is actually much simpler for 
pure convolution operators than what we did in Sections \ref{section3-1} and \ref{section3-2} since there is no Lopatinskii determinant 
involved. We thus feel free to state without proof the following result that is in the same vein as Proposition \ref{prop1}.

\begin{proposition}
\label{prop2}
Under the condition \eqref{CFL} on $\lambda$, for any $z$ in the exterior $\mathscr{O}$ of the spectral curve \eqref{courbespectre}, there 
exist unique solutions $\overline{\mathcal{G}}_\ell(z)=(\overline{\mathcal{G}}_{\ell,j}(z))_{j \in \Z} \in \ell^q(\Z;\C)$ and $\overline{\mathcal{G}}_r(z) 
=(\overline{\mathcal{G}}_{r,j}(z))_{j \in \Z} \in \ell^q(\Z;\C)$ to the equations \eqref{Green-spatial-Cauchy}. These sequences are given by:
\begin{equation}
\label{defGrz}
\overline{\mathcal{G}}_{r,j}(z) \, = 
\begin{cases}
\dfrac{- 2 \, \kappa_r^u(z)^j}{\alpha_r(1-\alpha_r)(\kappa_r^u(z)-\kappa_r(z))}, & \text{\rm if $j \le 0$,} \\
& \\
\dfrac{- 2 \, \kappa_r(z)^j}{\alpha_r(1-\alpha_r)(\kappa_r^u(z)-\kappa_r(z))}, & \text{\rm if $j \ge 0$,}
\end{cases}
\end{equation}
and
\begin{equation}
\label{defGlz}
\overline{\mathcal{G}}_{\ell,j}(z) \, = 
\begin{cases}
\dfrac{- 2 \, \kappa_\ell(z)^j}{\alpha_\ell(1-\alpha_\ell)(\kappa_\ell(z)-\kappa_\ell^u(z))}, & \text{\rm if $j \le 0$,} \\
& \\
\dfrac{- 2 \, \kappa_\ell^u(z)^j}{\alpha_\ell(1-\alpha_\ell)(\kappa_\ell(z)-\kappa_\ell^u(z))}, & \text{\rm if $j \ge 0$.}
\end{cases}
\end{equation}
\end{proposition}

\noindent The expressions in \eqref{defGlz} and \eqref{defGrz} can be recognized in \eqref{decompositionGz-1} and 
\eqref{decompositionGz-2}. This is made more explicit in the following paragraph.

\subsection{Decomposing the spatial Green's function}

The decomposition proceeds as follows. From inspection of the expressions \eqref{decompositionGz-1} and \eqref{decompositionGz-2} and 
the result of Proposition \ref{prop2}, the definition \eqref{defGtildez'} can be recast as follows:
\begin{equation}
\label{defGtildez}
\widetilde{\mathcal{G}}^{j_0}_j(z) \, = \, \begin{cases}
\mathcal{G}^{j_0}_j(z) \, - \, \mathds{1}_{j \ge 1} \, \overline{\mathcal{G}}_{r,j-j_0}(z) \, ,& \text{\rm if $j_0 \ge 1$,} \\
\mathcal{G}^{j_0}_j(z) \, - \, \mathds{1}_{j \le 0} \, \overline{\mathcal{G}}_{\ell,j-j_0}(z) \, ,& \text{\rm if $j_0 \le 0$,}
\end{cases}
\end{equation}
where the notation $\mathds{1}_{j \ge 1}$ is used to denote $1$ if $j \ge 1$ and $0$ otherwise (and similarly for $\mathds{1}_{j \le 0}$). 
It is also useful to introduce the following two fourth degree polynomial functions $\varphi_r$ and $\varphi_\ell$:
\begin{equation}
\label{defvarphirl}
\forall \, \tau \in \C \, ,\quad \varphi_{r,\ell}(\tau) \, := \, - \, \dfrac{1}{\alpha_{r,\ell}} \, \tau \, + \, \dfrac{1-\alpha_{r,\ell}^2}{6 \, \alpha_{r,\ell}^3} \, \tau^3 
 \, - \, \dfrac{1-\alpha_{r,\ell}^2}{8 \, \alpha_{r,\ell}^3} \, \tau^4 \, .
\end{equation}
With such notation, our result is the following.

\begin{proposition}
\label{prop3}
Let the weak solution \eqref{shock} satisfy the entropy inequalities \eqref{entropy}. Let the parameter $\lambda$ satisfy the CFL condition \eqref{CFL} 
and let Assumption \ref{hyp-stabspectrale} be satisfied. Then there exists $\varepsilon_0>0$, there exist constants $C>0$ and $c>0$, there exist two 
complex valued sequences $(\gamma_j^r)_{j \in \Z}$ and $(\gamma_j^\ell)_{j \in \Z}$, there exist two bounded holomorphic functions $\Psi_r$ and 
$\Psi_\ell$ on the square $\mathbf{B}_{\varepsilon_0} (0)$, and there exist sequences $(\Phi_{r,j})_{j \in \Z}$,  $(\Phi_{\ell,j})_{j \in \Z}$, 
$(\Theta_{r,j})_{j \in \Z}$, $(\Theta_{\ell,j})_{j \in \Z}$ and $(\Theta_{1,j})_{j \in \Z}$ of bounded holomorphic functions on the square 
$\mathbf{B}_{\varepsilon_0} (0)$ such that the following hold:
\begin{itemize}
 \item the sequences $(\gamma_j^r)_{j \in \Z}$ and $(\gamma_j^\ell)_{j \in \Z}$ satisfy the estimates:
$$
\forall \, j \in \Z \, ,\quad |\gamma_j^r| \, + \, |\gamma_j^\ell| \, \le \, C \, \exp (-c \, |j|) \, ;
$$
 \item the sequences $(\Phi_{r,j})_{j \in \Z}$, $(\Phi_{\ell,j})_{j \in \Z}$, $(\Theta_{r,j})_{j \in \Z}$ and $(\Theta_{\ell,j})_{j \in \Z}$ satisfy the estimates:
$$
\forall \, j \in \Z \, ,\quad \forall \, \tau \in \mathbf{B}_{\varepsilon_0} (0) \, ,\quad 
|\Phi_{r,j}(\tau)| \, + \, |\Phi_{\ell,j}(\tau)| \,+ \, |\Theta_{r,j}(\tau)| \, + \, |\Theta_{\ell,j}(\tau)| \,+ \, |\Theta_{1,j}(\tau)| \, \le \, C \, \exp (-c \, |j|) \, ;
$$
 \item for any couple of integers $(j_0,j) \in \Z^2$, the function:
$$ 
\tau \in \mathbf{B}_{\varepsilon_0} (0) \cap \Big\{ \zeta \in \C \, | \, \text{\rm Re } \zeta >0 \Big\} \longmapsto 
\widetilde{\mathcal{G}}^{j_0}_j({\rm e}^\tau) \, ,
$$
 whose expression is given in \eqref{defGtildez}, has a \emph{meromorphic} extension to the square $\mathbf{B}_{\varepsilon_0} (0)$ with a first 
 order pole at $0$ only, and there holds:
\begin{equation}
\label{decomposition-pole}
{\rm e}^\tau \, \widetilde{\mathcal{G}}^{j_0}_j({\rm e}^\tau) \, = \, \begin{cases}
\left( \dfrac{\mathcal{H}_j}{\tau} \, + \, \gamma_j^r \, + \, \tau \, \Phi_{r,j}(\tau) \right) \, 
\exp \big( - \, j_0 \, \varphi_r(\tau) \, + \, j_0 \, \tau^5 \, \Psi_r(\tau) \big) \, ,& \text{\rm if $j_0 \ge 1$,} \\
 & \\
\left( \dfrac{\mathcal{H}_j}{\tau} \, + \, \gamma_j^\ell \, + \, \tau \, \Phi_{\ell,j}(\tau) \right) \, 
\exp \big( - \, j_0 \, \varphi_\ell(\tau) \, + \, j_0 \, \tau^5 \, \Psi_\ell(\tau) \big) \, ,& \text{\rm if $j_0 \le 0$,}
\end{cases}
\end{equation}
 for any $\tau \in \mathbf{B}_{\varepsilon_0} (0) \setminus \{ 0 \}$, where we recall that $\mathcal{H}_j$ is defined in \eqref{defH};
 \item for any couple of integers $(j_0,j) \in \Z^2$, the function:
$$ 
\tau \in \mathbf{B}_{\varepsilon_0} (0) \cap \Big\{ \zeta \in \C \, | \, \text{\rm Re } \zeta >0 \Big\} \longmapsto 
\widetilde{\mathcal{G}}^{j_0}_j({\rm e}^\tau)-\widetilde{\mathcal{G}}^{j_0-1}_j({\rm e}^\tau)
$$
 has a \emph{holomorphic} extension to the square $\mathbf{B}_{\varepsilon_0} (0)$, and there holds:
 \begin{equation}
\label{decomposition-derivative}
{\rm e}^\tau \, \left(\widetilde{\mathcal{G}}^{j_0}_j({\rm e}^\tau) -\widetilde{\mathcal{G}}^{j_0-1}_j({\rm e}^\tau) \right)  \, = \, \begin{cases}
\left( \dfrac{\mathcal{H}_j}{\alpha_r} \,  + \, \tau \, \Theta_{r,j}(\tau) \right) \, 
\exp \big( - \, j_0 \, \varphi_r(\tau) \, + \, j_0 \, \tau^5 \, \Psi_r(\tau) \big) \, ,& \text{\rm if $j_0 \ge 2$,} \\
 & \\
 \dfrac{\mathcal{H}_j}{\alpha_r} \,  + \, \gamma_j^r-\gamma_j^\ell + \, \tau \, \Theta_{1,j}(\tau) \, ,& \text{\rm if $j_0 =1$,} \\
& \\
\left( \dfrac{\mathcal{H}_j}{\alpha_\ell} \, + \, \tau \, \Theta_{\ell,j}(\tau) \right) \, 
\exp \big( - \, j_0 \, \varphi_\ell(\tau) \, + \, j_0 \, \tau^5 \, \Psi_\ell(\tau) \big) \, ,& \text{\rm if $j_0 \le 0$,}
\end{cases}
\end{equation}
 for any $\tau \in \mathbf{B}_{\varepsilon_0}(0)$.
\end{itemize}
\end{proposition}

\begin{proof}
We give the proof in the case $j_0 \ge 1$ and construct all quantities associated with the ``right'' state $u_r$. The proof in the case 
$j_0 \le 0$ is entirely similar and is left to the interested reader. We thus always consider from on some $j_0 \ge 1$ and some arbitrary 
integer $j \in \Z$. We also define, for later use, the function:
$$
z \longmapsto \Theta(z) \, := \, \dfrac{\underline{\Delta}(z)}{z-1} \, ,
$$
where the Lopatinskii determinant $\underline{\Delta}$ is defined in \eqref{defDelta}. Thanks to Lemma \ref{lem2} and Assumption 
\ref{hyp-stabspectrale}, we know that $\Theta$ can be holomorphically extended to some set of the form $\{ \zeta \in \C \, | \, 
{\rm e}^{-\delta_0} < |\zeta| \}$ for an appropriate $\delta_0>0$. This is because $\underline{\Delta}$ has a simple zero at $1$. Furthermore, 
we know that, up to restricting $\delta_0$, $\Theta$ does not vanish on the set $\{ \zeta \in \C \, | \, {\rm e}^{-\delta_0} < |\zeta| \}$.

From the definition \eqref{defGtildez'} and the expressions \eqref{decompositionGz-1}, \eqref{defGrz}, we get the factorization:
$$
\forall \, z \in \U \, ,\quad \widetilde{\mathcal{G}}^{j_0}_j(z) \, = \, \chi_j^r(z) \, \overline{\mathcal{G}}_{r,1-j_0}(z) \, ,
$$
where we have set
\begin{equation}
\label{defchijr}
\chi_j^r(z) \, := \, \begin{cases}
\dfrac{\chi^-_r(z)}{z-1} \, \kappa_\ell(z)^j \, , & \text{\rm if $j \le 0$,} \\
 & \\
\left( \dfrac{\chi^+_r(z)}{z-1} \, - \, 1 \right) \, \kappa_r(z)^{j-1} \, ,& \text{\rm if $j \ge 1$,}
\end{cases}
\end{equation}
with
\begin{align*}
\chi^+_r(z) & := \, 
\dfrac{\big( \alpha_\ell-\alpha_m+(1-\alpha_\ell) \, \kappa_\ell(z) \big) \, (1-\alpha_r) \, (\kappa_r^u(z)-\kappa_r(z))}{\Theta(z)} \, ,\\
\chi^-_r(z) & := \, 
\dfrac{\alpha_r \, (1-\alpha_r) \, (1-\alpha_m) \, (\kappa_r^u(z)-\kappa_r(z))}{\alpha_\ell \, \Theta(z)} \, .
\end{align*}
It appears from those expressions and from Lemma \ref{lem1} that, up to restricting $\delta_0$ again, both functions $\chi^\pm_r$ have 
a holomorphic extension to the set $\{ \zeta \in \C \, | \, {\rm e}^{-\delta_0} < |\zeta| \}$ and there is no loss os generality in assuming that 
both $\kappa_\ell$ and $\kappa_r$ are also holomorphic functions and do not vanish on that set (see Lemma \ref{lem1}). Moreover, we 
compute:
$$
\chi^+_r(1) \, = \, - \, \dfrac{2 \, (1+\alpha_m)}{\underline{\Delta}'(1)} \, ,\quad 
\chi^-_r(1) \, = \, \dfrac{2 \, \alpha_r \, (1-\alpha_m)}{\alpha_\ell \, \underline{\Delta}'(1)} \, .
$$

For $z$ in a neighborhood of $1$, we may then write (see \eqref{defchijr}):
$$
\chi_j^r(z) \, = \, \dfrac{\xi_j^r}{z-1} \, + \, \Gamma_j^r \, + \, \Xi_j^r(z) \, ,
$$
where $\xi_j^r$ and $\Gamma_j^r$ are complex numbers defined by:
\begin{align*}
\xi_j^r \, &:= \, \underset{z\rightarrow1}{\lim} \, (z-1) \, \chi_j^r(z) \, = \begin{cases}
\chi^-_r(1) \, \kappa_\ell(1)^j \, ,&\text{\rm if $j \le 0$,} \\
\chi^+_r(1) \, \kappa_r(1)^{j-1} \, ,&\text{\rm if $j \ge 1$,}
\end{cases} \\ 
\Gamma_j^r \, &:= \, \underset{z\rightarrow1}{\lim} \, \left( \chi_j^r(z) \, - \, \dfrac{\xi_j^r}{z-1} \right) \, ,
\end{align*}
and the function $(z \mapsto \Xi_j^r(z))$ is defined by:
$$
\Xi_j^r(z) \, := \, \chi_j^r(z) \, - \, \dfrac{\xi_j^r}{z-1} \, - \, \Gamma_j^r \, ,
$$
so that $\Xi_j^r$ has a holomorphic extension to the set $\{ \zeta \in \C \, | \, {\rm e}^{-\delta_0} < |\zeta| \}$ and vanishes at $1$. In other 
words, we have isolated the first order pole at $1$ in $\chi_j^r$. Moreover, it is not difficult to compare the expression for $\xi_j^r$ and the 
defining equation \eqref{defH} for $\mathcal{H}_j$ and to find the relation $\xi_j^r=-\alpha_r \, \mathcal{H}_j$.

We know from Lemma \ref{lem1} and from Section \ref{section3-3} that $\kappa_\ell(1)$ belongs to $\U$ and $\kappa_r(1)$ belongs to $\D$. 
Hence we can infer from the above definitions the exponentially decaying bounds:
$$
|\xi_j^r| \, + \, |\Gamma_j^r| \, \le \, C \, {\rm e}^{-c \, |j|} \, ,
$$
as well as the local bound in $z$ close to $1$:
$$
|\Xi_j^r(z)| \, \le \, C \, |z-1| \, {\rm e}^{-c \, |j|} \, .
$$

In the same way, we find from the expression \eqref{defGrz} that for $j_0 \ge 1$, the function $\overline{\mathcal{G}}_{r,1-j_0}$ has a 
holomorphic extension to the set $\{ \zeta \in \C \, | \, {\rm e}^{-\delta_0} < |\zeta| \}$ (up to restricting $\delta_0$ one more time). Since 
$\kappa_r^u(1)$ equals $1$, for $\tau$ in a sufficiently small square $\mathbf{B}_{\varepsilon_0} (0)$ centered at the origin, we can 
write:
$$
\kappa_r^u ({\rm e}^\tau) \, = \, \exp (\omega_r(\tau)) \, ,
$$
where $\omega_r$ is holomorphic and bounded on $\mathbf{B}_{\varepsilon_0} (0)$. It is then a mere algebra exercise to infer from 
\eqref{modekappar} the Taylor expansion of $\omega_r$ at $0$ and we get:
$$
\omega_r(\tau) \, = \, \varphi_r(\tau) \, + \, \mathcal{O}(\tau^5) \, ,
$$
with the fourth degree polynomial $\varphi_r$ defined in \eqref{defvarphirl}. We can thus write, for any $\tau \in \mathbf{B}_{\varepsilon_0} (0)$:
$$
{\rm e}^\tau \, \overline{\mathcal{G}}_{r,1-j_0}({\rm e}^\tau) \, = \, 
\left( -\dfrac{1}{\alpha_r} \, + \, \mu_r \, \tau \, + \, \tau^2 \, \widetilde{\Phi}_r(\tau) \right) \, 
\exp \big( - \, j_0 \, \varphi_r(\tau) \, + \, j_0 \, \tau^5 \, \Psi_r(\tau) \big) \, ,
$$
where $\widetilde{\Phi}_r$ and $\Psi_r$ are two holomorphic bounded functions on $\mathbf{B}_{\varepsilon_0} (0)$ and $\mu_r$ is a complex 
number. It then remains to perform the Taylor expansion of $\chi_j^r({\rm e}^\tau)$ at $\tau=0$ and to multiply with the above expansion of 
${\rm e}^\tau \, \overline{\mathcal{G}}_{r,1-j_0}({\rm e}^\tau)$ to get the result of Proposition \ref{prop3}.
\end{proof}

We now turn to the proof of our main estimates for the Green's function of the operator $\mathscr{L}$ (see Theorem \ref{thmGreen} below). 
This will lead to time decay estimates as stated in Theorem \ref{thmLineaire}.

%%%%%%%%%%%%%%%%%%%%%%%%%%%%%%%%%%%%%%%%%%%%%%%%%%%%%%%%%%%%%%%%%%%%%%%%%%%%
\chapter{Linear stability}
\label{chapter4}

The goal of this chapter is to derive sharp bounds for the \emph{temporal} Green's function, that is, for any given $j_0 \in \Z$, the solution 
$(\mathscr{G}^n(j,j_0))_{(n,j) \in \N \times \Z}$ to the recurrence relation:
\begin{equation}
\label{defgreentemporelle}
\begin{cases}
\mathscr{G}^{n+1}(j,j_0) \, = \, (\mathscr{L} \, \mathscr{G}^n(\cdot,j_0))_j, & (n,j) \in \N \times \Z \, ,\\
\mathscr{G}^0(\cdot,j_0) \, = \, \boldsymbol{\delta}_{j_0} \, .
\end{cases}
\end{equation}
The relevance of this sequence is motivated by the fact that for any initial condition $\mathbf{h} \in \ell^q(\Z;\R)$, the solution to the recurrence relation:
\begin{align*}
\forall \, n \in \N \, ,\quad \mathbf{v}^{n+1} \, &= \, \mathscr{L} \, \mathbf{v}^n \, ,\\
\mathbf{v}^0 \, &= \, \mathbf{h} \, ,
\end{align*}
can be decomposed into:
$$
\forall \, (n,j) \in \N \times \Z \, ,\quad v^n_j \, = \, \sum_{j_0 \in \Z} \, \mathscr{G}^n(j,j_0) \, h_{j_0} \, ,
$$
and we expect that sharp bounds on $\mathscr{G}^n(j,j_0)$ will quantify the decay properties of the semigroup of operators $(\mathscr{L}^n)_{n \in \N}$. 
We first decompose the temporal Green's function by following the decomposition given in Proposition \ref{prop3} for the (reduced) spatial Green's 
function. We then analyze each contribution in the decomposition and derive bounds that are meant to be as sharp as possible in order to obtain 
large time decaying bounds for the semigroup $(\mathscr{L}^n)_{n \in \N}$.

%%%%%%%%%%%%%
\section{Preliminary facts}
\label{section4-1}

We always consider from now on that the weak solution \eqref{shock} satisfies the entropy inequalities \eqref{entropy} and that the CFL parameter 
$\lambda$ satisfies the stability condition \eqref{CFL}. We also assume that Assumption \ref{hyp-stabspectrale} is satisfied so that the analysis of 
Chapter \ref{chapter3} can be used. The value of the temporal Green's function $\mathscr{G}^n(j,j_0)$ is given by the so-called functional calculus 
(see \cite{Conway}):
\begin{equation}
\label{formuleGjn}
\mathscr{G}^n(j,j_0) \, = \, \dfrac{1}{2 \pi \mbi} \, \int_{\widetilde{\Gamma}} z^n \, \mathcal{G}^{j_0}_j(z) \, \md z,
\end{equation}
where $\widetilde{\Gamma}$ is any contour that encompasses the spectrum of the operator $\mathscr{L}$. In view of Theorem \ref{thm1bis}, one 
can choose for instance $\widetilde{\Gamma}=(1+\delta) \, \cercle=\{ \zeta \in \C \, | \, |\zeta|=1+\delta \}$ for any $\delta>0$, since the spectrum of 
$\mathscr{L}$ is included in $\D \cup \{ 1 \}$.

Two other key quantities of interest to us are the temporal Green's functions associated with the operators $\mathscr{L}_\ell$ and $\mathscr{L}_r$ 
that are defined in \eqref{cauchyLW}. The associated temporal Green's functions for these operators are defined as the solutions to the recurrences:
\begin{subequations}
\begin{align}
&\begin{cases}
\overline{\mathscr{G}}_\ell^{\, n+1} \, = \, \mathscr{L}_\ell \, \overline{\mathscr{G}}_\ell^{\, n}, & n \in \N \, ,\\
\overline{\mathscr{G}}_\ell^{\, 0} \, = \, \boldsymbol{\delta}_0 \, ,
\end{cases} \label{defgreentemporellel} \\
&\begin{cases}
\overline{\mathscr{G}}_r^{\, n+1} \, = \, \mathscr{L}_r \, \overline{\mathscr{G}}_r^{\, n}, & n \in \N \, ,\\
\overline{\mathscr{G}}_r^{\, 0} \, = \, \boldsymbol{\delta}_0 \, ,
\end{cases} \label{defgreentemporeller}
\end{align}
\end{subequations}
and the same functional calculus rules as before yield the expressions:
$$
\overline{\mathscr{G}}_\ell^{\, n}(j) \, = \, \dfrac{1}{2 \pi \mbi} \, \int_{\widetilde{\Gamma}} z^n \, \overline{\mathcal{G}}_{\ell,j}(z) \, \md z \, ,\quad 
\overline{\mathscr{G}}_r^{\, n}(j) \, = \, \dfrac{1}{2 \pi \mbi} \, \int_{\widetilde{\Gamma}} z^n \, \overline{\mathcal{G}}_{r,j}(z) \, \md z \, ,
$$
where $\widetilde{\Gamma}$ is once again any closed contour that encompasses the spectrum of both $\mathscr{L}_\ell$ and $\mathscr{L}_r$ in 
its interior. We can choose, for instance, $\widetilde{\Gamma}=(1+\delta) \, \cercle$ for any $\delta>0$. A crucial observation for what follows is that 
the sequences $(\overline{\mathscr{G}}_\ell^{\, n}(j))_{(n,j)\in \N \times \Z}$ and $(\overline{\mathscr{G}}_r^{\, n}(j))_{(n,j)\in \N \times \Z}$ have been 
thoroughly studied in \cite{jfcAMBP}, though with a different point of view since the framework allows for the use of Fourier analysis. We shall feel 
free to use repeatedly several results from \cite{jfcAMBP}, which are themselves more accurate estimates for the free Green's functions (in the whole 
space) than previous bounds obtained in \cite{hedstrom1,hedstrom2}. The main results of \cite{jfcAMBP} are gathered in Appendix \ref{appendixA} 
together with some supplementary material that is needed to carry out the analysis below.

We start our analysis with a first elementary observation. Since the Lax-Wendroff scheme as a finite stencil, we readily see from the recurrence 
equation \eqref{defgreentemporelle} defining the Green's function that for all $n \in \N$ and $(j,j_0) \in \Z^2$:
\bqs
|j-j_0|>n\, \Rightarrow \, \mathscr{G}^{n}(j,j_0)=0 \, .
\eqs
this fact is repeatedly used below in order to restrict the possible regimes for $j,j_0,n$. As a second crucial observation, we have the following Lemma.

\begin{lemma}
\label{lem4}
Let $j_0 \in \Z$ and $n \in \N$. If $j_0 \ge 1$ and $n \le j_0-1$, then there holds:
$$
\forall \, j \in \Z \, ,\quad \mathscr{G}^n(j,j_0) \, = \, \overline{\mathscr{G}}_r^n(j-j_0) \, .
$$
If $j_0 \le 0$ and $n \le |j_0|$, then there holds:
$$
\forall \, j \in \Z \, ,\quad \mathscr{G}^n(j,j_0) \, = \, \overline{\mathscr{G}}_\ell^n(j-j_0) \, .
$$
\end{lemma}

\begin{proof}
We give the proof in the case $j_0 \ge 1$ and leave the other situation to the interested reader.

The proof directly follows from the expression \eqref{linear} of the operator $\mathscr{L}$ and from the definition \eqref{cauchy-r} of the convolution 
operator $\mathscr{L}_r$. Indeed, we easily see that if $\mathbf{w}=(w_j)_{j \in \Z}$ is a sequence that is supported in $\{ j \in \Z \, | \, j \ge j_{\rm min} \}$ 
with $j_{\rm min} \ge 2$, then $\mathscr{L} \mathbf{w}$ equals $\mathscr{L}_r \mathbf{w}$ and both sequences are supported in 
$\{ j \in \Z \, | \, j \ge j_{\rm min}-1 \}$.

For $j_0=1$, the proof of Lemma \ref{lem4} is obvious. We thus assume $j_0 \ge 2$. Then we can prove by induction that for any $n=0,\dots,j_0-2$, 
the two sequences $\mathscr{G}^n(\cdot,j_0)$ and $\overline{\mathscr{G}}_r^n(\cdot-j_0)$ are equal and they are supported in 
$\{ j \in \Z \, | \, j \ge j_0-n \}$. Applying one last time the above fact when $n$ equals $j_0-2$, we see that $\mathscr{G}^{j_0-1}(\cdot,j_0)$ and 
$\overline{\mathscr{G}}_r^{j_0-1}(\cdot-j_0)$ are equal and they are supported in $\{ j \in \Z \, | \, j \ge 1 \}$. This is when we cannot use the above 
fact any longer.
\end{proof}

Finally, as a last preparatory step, we shall change variable in the integral representation \eqref{formuleGjn}  to write instead
\bqs
\mathscr{G}^n(j,j_0) \, = \, \dfrac{1}{2 \pi \mbi} \, \int_{\Gamma} \rme^{n \, \tau} \, \mathcal{G}^{j_0}_j(\rme^\tau) \, \rme^\tau\, \md \tau,
\eqs
with $\Gamma=\left\{\tau=\rho+\mbi \, \theta ~|~ \theta \in[-\pi,\pi]\right\}$, for any $\rho>0$. This change of variable justifies the decomposition made 
in Proposition \ref{prop3} (see \eqref{decomposition-pole}). We now introduce some further notations and clarify the bounds that we intend to prove 
on the Green's function of the operator $\mathscr{L}$.

%%%%%%%%%%%%%%%%%%%%%%%%%%%
\section{Notation and bounds for the Green's function}
\label{section4-2}

We first introduce some constants:
\begin{subequations}
\label{defc3c4}
\begin{align}
c_{3,\ell} \, &:= \, \dfrac{\alpha_\ell \, (1 - \alpha_\ell^2)}{6} >0 \, ,\quad & 
c_{4,\ell} \, &:= \, \dfrac{\alpha_\ell^2 \, (1 - \alpha_\ell^2)}{8} \, > \, 0 \, ,\label{defc3c4l} \\
c_{3,r} \, &:= \, \dfrac{\alpha_r \, (1 - \alpha_r^2)}{6} < 0 \, ,\quad & 
c_{4,r} \, &:= \, \dfrac{\alpha_r^2 \, (1 - \alpha_r^2)}{8} \, > \, 0 \, .\label{defc3c4r}
\end{align}
\end{subequations}
We then define the following two functions\footnote{The letter $A$ refers to ``activation'' in analogy with \cite{Coeuret1}.} $\mathbf{A}_\ell$ and 
$\mathbf{A}_r$ on $\R \times \R^{+*}$ as follows. For $\eta>0$, we set:
\begin{subequations}
\label{defactivation}
\begin{align}
\forall \, (x,y) \in \R \times \R^{+*} \, ,\quad \mathbf{A}_\ell(x,y) \, &:= \, \dfrac{1}{2 \, \pi} \, \int_\R {\rm e}^{x \, (\eta+\mbi \, \theta)} 
\, {\rm e}^{-c_{3,\ell} \, y \, (\eta+\mbi \, \theta)^3} \, {\rm e}^{-c_{4,\ell} \, y \, (\eta+\mbi \, \theta)^4} \, \dfrac{{\rm d}\theta}{\eta+\mbi \, \theta} \, ,
\label{defactivationl} \\
\mathbf{A}_r(x,y) \, &:= \, \dfrac{1}{2 \, \pi} \, \int_\R {\rm e}^{x \, (\eta+\mbi \, \theta)} \, {\rm e}^{-c_{3,r} \, y \, (\eta+\mbi \, \theta)^3} \, 
{\rm e}^{-c_{4,r} \, y \, (\eta+\mbi \, \theta)^4} \, \dfrac{{\rm d}\theta}{\eta+\mbi \, \theta} \, ,\label{defactivationr}
\end{align}
\end{subequations}
where both definitions \eqref{defactivationl} and \eqref{defactivationr} make sense since $c_{4,\ell}$ and $c_{4,r}$ are positive. The 
definitions \eqref{defactivation} are shown to be \emph{independent} of $\eta>0$ thanks to the Cauchy formula for holomorphic functions. 
Eventually, we introduce two other functions $\mathbf{M}_\ell$ and $\mathbf{M}_r$ on $\R^{+*} \times \R \times \R^{+*}$ as follows:
\begin{subequations}
\label{defmajorant}
\begin{align}
\forall \, (c,x,y) \in \R^{+*} \times \R \times \R^{+*} \, ,\quad \mathbf{M}_\ell(c,x,y) \, &:= \begin{cases}
\dfrac{1}{y^{1/3}} \, \exp \big( -c \, |x|^{3/2}/y^{1/2}  \big) \, ,& \text{\rm if $x \ge 0$,} \\
 & \\
\dfrac{1}{y^{1/3}} \, ,& \text{\rm if $-y^{1/3} \le x \le 0$,} \\
 & \\
\dfrac{1}{|x|^{1/4} \, y^{1/4}} \, \exp  \big( -c \, x^2/y  \big) \, ,& \text{\rm if $x \le -y^{1/3}$,}
\end{cases}
\label{defamajorantl} \\
\mathbf{M}_r(c,x,y) \, &:= \, \begin{cases}
\dfrac{1}{y^{1/3}} \, \exp \big( -c \, |x|^{3/2}/y^{1/2}  \big) \, ,& \text{\rm if $x \le 0$,} \\
 & \\
\dfrac{1}{y^{1/3}} \, ,& \text{\rm if $0 \le x \le y^{1/3}$,} \\
 & \\
\dfrac{1}{|x|^{1/4} \, y^{1/4}} \, \exp  \big( -c \, x^2/y  \big) \, ,& \text{\rm if $y^{1/3} \le x$.}
\end{cases}
\label{defmajorantr}
\end{align}
\end{subequations}
These functions encode the bounds for the free Green's functions associated with the convolution operators $\mathscr{L}_\ell$ and $\mathscr{L}_r$, 
as recalled in Appendix \ref{appendixA} (see for instance Corollary \ref{coro-A1}).

A crucial property for what follows is that both $\mathbf{M}_\ell$ and $\mathbf{M}_r$ are non-increasing with respect to their first argument. When various 
positive constants $c_1,c_2,\dots$ appear, we can always use the largest function $\mathbf{M}_r(\min_i c_i,\cdot,\cdot)$ as an upper bound for all functions 
$\mathbf{M}_r(c_i,\cdot,\cdot)$. This will be used repeatedly in order to avoid using specific notations for the various small positive constants $c$ that appear.

Our main result for the Green's function of the operator $\mathscr{L}$ reads as follows.

\begin{theorem}[The Green's function of the linearized numerical scheme]
\label{thmGreen}
Let the weak solution \eqref{shock} satisfy the Rankine-Hugoniot condition \eqref{RH} and the entropy inequalities \eqref{entropy}. Let the 
parameter $\lambda$ satisfy the CFL condition \eqref{CFL} 
and let Assumption \ref{hyp-stabspectrale} be satisfied. Then there exist some positive constants $C$ and $c$ such that, for any $j_0 \le 0$, there 
holds:
\begin{align}
\forall \, (n,j) \in \N^* \times \Z \, ,\quad \big| \mathscr{G}^n(j,j_0) \, - \, \mathcal{H}_j \, \mathbf{A}_\ell (j_0+n \, \alpha_\ell,n) \big| \, \le \, 
&C \, \mathbf{M}_\ell \left( c,j_0-j+n \, \alpha_\ell,n \right) \, \mathds{1}_{j \le 0} \notag \\
&+ C \, {\rm e}^{-c \, |j|} \, \mathbf{M}_\ell \left( c,j_0+n \, \alpha_\ell,n \right) \label{borneGreen-1} \\
&+ C \, {\rm e}^{-c \, n} \, {\rm e}^{-c \, |j|} \, {\rm e}^{-c \, |j_0|} \, ,\notag
\end{align}
and for any $j_0 \ge 1$, there holds:
\begin{align}
\forall \, (n,j) \in \N^* \times \Z \, ,\quad \big| \mathscr{G}^n(j,j_0) \, - \, \mathcal{H}_j \, \mathbf{A}_r(-j_0+n \, |\alpha_r|,n) \big| \, \le \, 
&C \, \mathbf{M}_r \left( c,j-j_0+n \, |\alpha_r|,n \right) \, \mathds{1}_{j \ge 1} \notag \\
&+ C \, {\rm e}^{-c \, |j|} \, \mathbf{M}_r \left( c,-j_0+n \, |\alpha_r|,n \right) \label{borneGreen-2} \\
&+ C \, {\rm e}^{-c \, n} \, {\rm e}^{-c \, |j|} \, {\rm e}^{-c \, |j_0|} \, ,\notag
\end{align}
where in both \eqref{borneGreen-1} and \eqref{borneGreen-2}, $\mathcal{H}_j$ is defined in \eqref{defH}.
\end{theorem}

%%%%%%%%%%%%%%%%%%%%%%%%%%%
\section{Decomposing the temporal Green's function}
\label{section4-3}

Throughout this section, we assume that $n \ge1$ is an integer and $(j,j_0)\in\Z^2$ satisfy $|j-j_0|\leq n$. We will mainly focus on the case 
$j_0\geq1$, and present at the end of the section the corresponding results for the case $j_0\leq0$. From Lemma~\ref{lem4}, we shall further 
assume that $1\leq j_0 \leq n$ since for $n\leq j_0-1$ we have already obtained an explicit expression for the temporal Green's function in 
terms of the free Green's function associated to $\mathscr{L}_r$ (so the final derivation of a bound for the Green's function will be simpler).

The starting point of our analysis is to exploit our expression for the spatial Green's function to get
\begin{equation}
\label{decomposition-initial}
\mathscr{G}^n(j,j_0) \, = \, \overline{\mathscr{G}}_r^n(j-j_0) \, \mathds{1}_{j\geq1} \, + \, 
\dfrac{1}{2 \pi \mbi} \, \int_{\Gamma} \rme^{n \, \tau} \, \widetilde{\mathcal{G}}^{j_0}_j(\rme^\tau) \, \rme^\tau \, \md \tau \, ,
\end{equation}
and define 
\bqs
\widetilde{\mathscr{G}}^n(j,j_0) \, := \, 
\dfrac{1}{2 \pi \mbi} \, \int_{\Gamma} \rme^{n \, \tau} \, \widetilde{\mathcal{G}}^{j_0}_j(\rme^\tau) \, \rme^\tau \, \md \tau \, .
\eqs
At this stage, $\Gamma$ denotes the segment $\left\{\tau=\rho+\mbi \, \theta ~|~ \theta \in[-\pi,\pi]\right\}$ for any $\rho>0$, but we shall be 
allowed to deform $\Gamma$ thanks to Cauchy's formula.

Let $\varepsilon_0>0$ be given by Proposition \ref{prop3} and let $C_0>0$ be such for all $\tau \in  \mathbf{B}_{\varepsilon_0}(0)$ one has 
the uniform bound:
\begin{equation}
\label{defreste}
\left| \tau^5 \, \Psi_r(\tau)\right| \, \le \, C_0 \, \left( \, |\Re(\tau)|^5+|\Im(\tau)|^5 \, \right) \, .
\end{equation}
Then, let $\varepsilon_*\in(0,\varepsilon_0)$ be such that for all $\varepsilon\in(0,\varepsilon_*)$ the following conditions are satisfied:
\bqq
\label{condeps1}
0 < \varepsilon < \min\left[\frac{1-\alpha_r^2}{16 \, \alpha_r^2 \, C_0} \, , \left( \frac{1}{3}\right)^{1/5} \right] \, ,
\eqq
\bqq
\label{condeps2}
\dfrac{4 \, |\alpha_r|}{1-\alpha_r^2}\left[\max\left( \frac{|\alpha_r|}{2},1-\alpha_r^2\right)+|\alpha_r|C_0\right]\varepsilon 
+ \varepsilon^3 + \dfrac{3}{2} \, \varepsilon^{8} + \dfrac{1}{3} \, \varepsilon^{11} + \dfrac{4 \, \alpha_r^2}{1-\alpha_r^2} \, C_0 \, \varepsilon^{21} 
< \dfrac{1}{8} \, ,
\eqq
\bqq
\label{condeps3}
\frac{1-\alpha_r^2}{4|\alpha_r|}\left(1+\frac{3}{2}\varepsilon^5\right)\varepsilon^2 + \frac{\alpha_r^2}{2} \, C_0 \, \varepsilon^{20} < \dfrac{|\alpha_r|}{4} \, .
\eqq
Finally, once for all, we fix $\varepsilon\in(0,\varepsilon_*)$, and we let $\eta_\varepsilon>0$ be provided by Corollary \ref{cor2} associated to this 
$\varepsilon>0$ and we also set $0<\eta<\min(\eta_\varepsilon,\varepsilon^5)$.

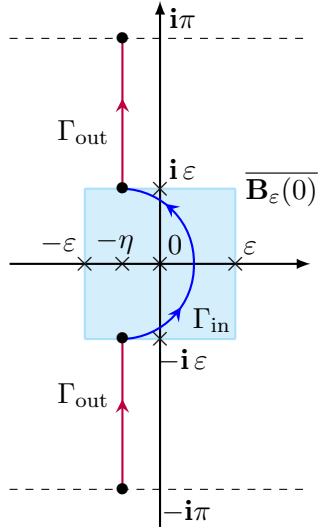
\begin{figure}[t!]
\begin{center}
\begin{tikzpicture}[scale=1,>=latex]

\fill[cyan!15] (-1,-1) -- (1,-1) -- (1,1) -- (-1,1);

\draw[thick,black,->] (-2,0) -- (2,0);
\draw[thick,black,->] (0,-3.5)--(0,3.5);
\draw[dashed,black] (-2,-3)--(2,-3);
\draw[dashed,black] (-2,3)--(2,3);
\draw[thick,purple,directed] (-0.5,-3) -- (-0.5,-1);
\draw[thick,purple,directed] (-0.5,1) -- (-0.5,3);

\draw[thick,cyan!30] (-1,-1) -- (1,-1);
\draw[thick,cyan!30] (1,-1) -- (1,1);
\draw[thick,cyan!30] (-1,1) -- (1,1);
\draw[thick,cyan!30] (-1,-1) -- (-1,1);

\draw[thick,blue,directed] (-0.5,-1) to [bend right=45] (0.45,0) ;
\draw[thick,blue,directed] (0.45,0) to [bend right=45] (-0.5,1) ;

\node (centre) at (-1,0){$\times$};
\node (centre) at (0,0){$\times$};
\node (centre) at (-0.5,0){$\times$};
\node (centre) at (1,0){$\times$};
\node (centre) at (0,-1){$\times$};
\node (centre) at (0,1){$\times$};

\draw (0.2,0) node[above]{$0$};
\draw (1.2,0) node[above]{$\varepsilon$};
\draw (-1.35,0) node[above]{$-\varepsilon$};
\draw (-0.6,0) node[above]{$-\eta$};
\draw (0.3,-1) node[below]{$-\mbi \, \varepsilon$};
\draw (0.3,1) node[above]{$\mbi \, \varepsilon$};
\draw (0.35,-3) node[below]{$-\mbi\pi$};
\draw (0.3,3) node[above]{$\mbi\pi$};

\node (centre) at (-0.5,-3){$\bullet$};
\node (centre) at (-0.5,-1){$\bullet$};
\node (centre) at (-0.5,1){$\bullet$};
\node (centre) at (-0.5,3){$\bullet$};

\draw (-1.5,-1.75) node[right]{$\Gamma_{\mathrm{out}}$};
\draw (-1.5,1.75) node[right]{$\Gamma_{\mathrm{out}}$};
\draw (0.3,-0.75) node[right]{$\Gamma_{\mathrm{in}}$};

\draw (1,1) node[right]{$\overline{\mathbf{B}_{\varepsilon}(0)}$};

\end{tikzpicture}
\caption{Schematic illustration of the contour $\Gamma$ and its decomposition into $\Gamma_{\mathrm{out}}$ (in red) and 
$\Gamma_{\mathrm{in}}$ (in blue). The contour $\Gamma_{\mathrm{in}}$ can be any path joining $-\eta-\mbi \, \varepsilon$ 
to $-\eta+\mbi \, \varepsilon$, which remains within $\overline{\mathbf{B}_{\varepsilon}(0)}$ and passes to the right of the origin. 
The black bullets represent the end points of the contours.}
\label{fig:contourGammainout}
\end{center}
\end{figure}

We can now proceed by choosing an appropriate contour $\Gamma$ in our integral defining $\widetilde{\mathscr{G}}^n(j,j_0)$. We would like 
to choose the segment $\left\{ \tau=-\eta+\mbi \, \theta ~|~ \theta \in[-\pi,\pi]\right\}$ but this is not possible right away because of the pole at the 
origin, so we shall make a detour on the right of the origin. We thus decompose $\Gamma$ into two pieces $\Gamma_{\mathrm{out}}$ and 
$\Gamma_{\mathrm{in}}$ where
\bqs
\Gamma_{\mathrm{out}} \, := \, \left\{ -\eta+\mbi \, \theta ~|~ \varepsilon \leq |\theta| \leq \pi \right\} \, ,
\eqs
and $\Gamma_{\mathrm{in}}$ is any path joining $-\eta-\mbi \, \varepsilon$ to $-\eta+\mbi \, \varepsilon$, which remains within 
$\overline{\mathbf{B}_{\varepsilon}(0)}$ and passes to the right of the origin. These contours are depicted in Figure \ref{fig:contourGammainout}.  From Corollary \ref{cor2} in Chapter \ref{chapter3}, we infer that 
\begin{equation}
\label{borneGammaout}
\left| \dfrac{1}{2 \pi \mbi} \, \int_{\Gamma_{\mathrm{out}}} \rme^{n \, \tau} \, \widetilde{\mathcal{G}}^{j_0}_j(\rme^\tau) \, \rme^\tau\, \md \tau\, \right| 
\, \le \, C \, \rme^{-\eta \, n} \, \rme^{-c \, \left(|j|+|j_0|\right)} \, .
\end{equation}
Since $\Gamma_{\mathrm{in}}\subset \overline{\mathbf{B}_{\varepsilon}(0)}\subset \mathbf{B}_{\varepsilon_0}(0)$, we can use Proposition \ref{prop3} 
and the expression \eqref{decomposition-pole} to decompose the remaining integral as
\bqs
\dfrac{1}{2 \pi \mbi} \, \int_{\Gamma_{\mathrm{in}}} \rme^{n \, \tau} \, \widetilde{\mathcal{G}}^{j_0}_j(\rme^\tau) \, \rme^\tau\, \md \tau\, = \, 
\widetilde{\mathscr{G}}^n_{1,r}(j,j_0) \, + \, \widetilde{\mathscr{G}}^n_{2,r}(j,j_0) \, + \, \mathscr{R}_{1,r}^n(j,j_0) \, ,
\eqs
where we have set
\begin{align}
\widetilde{\mathscr{G}}^n_{1,r}(j,j_0) &:= \mathcal{H}_j \, \times \, \frac{1}{2\pi \mathbf{i}} 
\int_{\Gamma_{\mathrm{in}}} \exp\left(n \, \tau-j_0 \, \varphi_r(\tau)+j_0 \, \tau^5 \, \Psi_r(\tau) \right)\frac{\md \tau}{\tau} \, ,\notag \\
\widetilde{\mathscr{G}}^n_{2,r}(j,j_0) &:= \gamma_j^r \, \times \, 
\frac{1}{2\pi \mathbf{i}} \int_{\Gamma_{\mathrm{in}}} \exp\left(n \, \tau-j_0 \, \varphi_r(\tau)+j_0 \, \tau^5 \, \Psi_r(\tau) \right)\md \tau \, ,\label{defG2rnj0}
\end{align}
and
\begin{equation}
\label{defR1rnj0}
\mathscr{R}_{1,r}^n(j,j_0) \, := \, \dfrac{1}{2\pi \mathbf{i}} \, \int_{\Gamma_{\mathrm{in}}} 
\exp \left( n \, \tau-j_0 \, \varphi_r(\tau) + j_0 \, \tau^5 \, \Psi_r(\tau) \right) \, \tau \, \Phi_{r,j}(\tau) \, \md \tau \, .
\end{equation}
Thank's to Cauchy formula for holomorphic functions, it is important to remark that the above integrals do not depend on $\Gamma_{\mathrm{in}}$, and we shall later on in this section choose specific contours $\Gamma_{\mathrm{in}}$ depending on various regimes between $n$ and $j_0$. In fact, for $\widetilde{\mathscr{G}}^n_{2,r}(j,j_0)$ and $\mathscr{R}_{1,r}^n(j,j_0)$, we can select any contour $\Gamma_{\mathrm{in}}$ that joins $-\eta-\mbi \, \varepsilon$ to $-\eta+\mbi \, \varepsilon$ and which remains within 
$\overline{\mathbf{B}_{\varepsilon}(0)}$ since both integrand are holomorphic functions in $\overline{\mathbf{B}_{\varepsilon}(0)}$, while for $\widetilde{\mathscr{G}}^n_{1,r}(j,j_0)$ we need to keep the constraint that the contour $\Gamma_{\mathrm{in}}$ passes to the right of the origin because of the pole at the 
origin of the integrand.

Our next task is to extract the leading order term of $\widetilde{\mathscr{G}}^n_{1,r}(j,j_0)$. For that, we may simply further decompose 
\bqs
\widetilde{\mathscr{G}}^n_{1,r}(j,j_0)=\mathcal{H}_j \, \mathscr{A}_r^n(j_0)+\widetilde{\mathscr{G}}^n_{3,r}(j,j_0)+\mathscr{R}^n_{2,r}(j,j_0),
\eqs
where we have set
\begin{align}
\mathscr{A}_r^n(j_0) \, &:= \, \dfrac{1}{2\pi \mathbf{i}} \, \int_{\Gamma_{\mathrm{in}}} \, \rme^{n \, \tau - j_0 \, \varphi_r(\tau)} 
\, \dfrac{\md \tau}{\tau} \, , \label{defArnj0} \\
\widetilde{\mathscr{G}}^n_{3,r}(j,j_0)&:= \mathcal{H}_j \, \times \, j_0 \, \times \, \Psi_r(0) \times \dfrac{1}{2\pi \mathbf{i}} 
\int_{\Gamma_{\mathrm{in}}} \rme^{n \, \tau-j_0 \, \varphi_r(\tau)} \, \tau^4 \, \md \tau \, ,\label{defG3rnj0} 
\end{align}
and
\begin{equation}
\label{defR2rnj0}
\mathscr{R}^n_{2,r}(j,j_0) \, := \, \mathcal{H}_j \, \times \, \frac{1}{2\pi \mathbf{i}} \, \int_{\Gamma_{\mathrm{in}}} 
\rme^{n \, \tau-j_0 \, \varphi_r(\tau)} \, \left( \dfrac{\exp\left(j_0 \, \tau^5 \, \Psi_r(\tau)\right) - 1 - j_0 \, \tau^5 \, \Psi_r(0)}{\tau} \right) \, \md \tau \, .
\end{equation}
For convenience, we further define $\mathscr{B}_r^n(j,j_0):=\widetilde{\mathscr{G}}^n_{2,r}(j,j_0)+\widetilde{\mathscr{G}}^n_{3,r}(j,j_0)$ (with the 
definitions \eqref{defG2rnj0} and \eqref{defG3rnj0}) and $\mathscr{R}^n_{r}(j,j_0)=\mathscr{R}^n_{1,r}(j,j_0)+\mathscr{R}^n_{2,r}(j,j_0)$, such that 
we have obtained the intermediate decomposition
\bqq
\label{InterDecomp}
\dfrac{1}{2 \pi \mbi} \, \int_{\Gamma_{\mathrm{in}}} \rme^{n \, \tau} \, \widetilde{\mathcal{G}}^{j_0}_j(\rme^\tau) \, \rme^\tau\, \md \tau\, = \, 
\mathcal{H}_j \, \mathscr{A}_r^n(j_0) \, + \, \mathscr{B}^n_{r}(j,j_0) \, + \, \mathscr{R}_r^n(j,j_0) \, .
\eqq
For future reference, we gather the previous definitions of $\mathscr{B}^n_{r}(j,j_0)$ and $\mathscr{R}^n_{r}(j,j_0)$:
\begin{align}
\mathscr{B}_r^n(j,j_0) \, = & \, \gamma_j^r \, \times \, \dfrac{1}{2 \pi \mathbf{i}} \, 
\int_{\Gamma_{\mathrm{in}}} \exp \left( n \, \tau-j_0 \, \varphi_r(\tau) + j_0 \, \tau^5 \, \Psi_r(\tau) \right) \, \md \tau \notag \\
&+ \Psi_r(0) \times \mathcal{H}_j \times j_0 \times \dfrac{1}{2\pi \mathbf{i}} 
\int_{\Gamma_{\mathrm{in}}} \rme^{n \, \tau-j_0 \, \varphi_r(\tau)} \, \tau^4\md \tau \, , \label{defBrnj0} \\
\mathscr{R}^n_{r}(j,j_0) \, =& \dfrac{1}{2\pi \mathbf{i}} \, \int_{\Gamma_{\mathrm{in}}} 
\exp \left( n \, \tau-j_0 \, \varphi_r(\tau) + j_0 \, \tau^5 \, \Psi_r(\tau) \right) \, \tau \, \Phi_{r,j}(\tau) \, \md \tau \notag \\
&+ \mathcal{H}_j \, \times \, \frac{1}{2\pi \mathbf{i}} \, \int_{\Gamma_{\mathrm{in}}} \rme^{n \, \tau-j_0 \, \varphi_r(\tau)} \, 
\left( \dfrac{\exp\left(j_0 \, \tau^5 \, \Psi_r(\tau)\right) - 1 - j_0 \, \tau^5 \, \Psi_r(0)}{\tau} \right) \, \md \tau \, . \label{defRrnj0}
\end{align}
We shall now study each term appearing in \eqref{InterDecomp} separately.

\subsection{Estimates of the leading order term $\mathscr{A}_r^n(j_0)$}

In this subsection, we obtain a sharp estimate on the term $\mathscr{A}_r^n(j_0)$ in \eqref{defArnj0} and prove that it behaves like an activation function. 
This clarifies the behavior of the first term in the right-hand side of the decomposition \eqref{InterDecomp}. We shall consider here and in all the remainder 
of this Section, two different regimes for $n$ and $j_0$:
\begin{itemize}
\item[(i)] the main regime : $\frac{n \, |\alpha_r|}{2}\leq j_0 \leq n$;
\item[(ii)] the tail : $1 \le j_0 <\frac{n \, |\alpha_r|}{2}$.
\end{itemize}
The main result of this subsection is the following Lemma.

\begin{lemma}
\label{lem:estimateA}
Let $n \ge 1$ and let $1 \le j_0 \le n$. Let $\mathscr{A}_r^n(j_0)$ be defined in \eqref{defArnj0}. Then there exist some positive constants $C$ and $c$ 
that are uniform with respect to $n$ and $j_0$ such that with the functions $\mathbf{A}_r$ defined in \eqref{defactivationr} and $\mathbf{M}_r$ defined 
in \eqref{defmajorantr}, there holds:
\bqq
\label{estimateA}
\big| \mathscr{A}_r^n(j_0) \, - \, \mathbf{A}_r\left(-j_0+n \, |\alpha_r|,n\right) \big| 
\le C \, \mathbf{M}_r (c,-j_0+n \, |\alpha_r|,n) \, + \, C \, \rme^{-c \, n-c \, j_0} \, .
\eqq
\end{lemma}

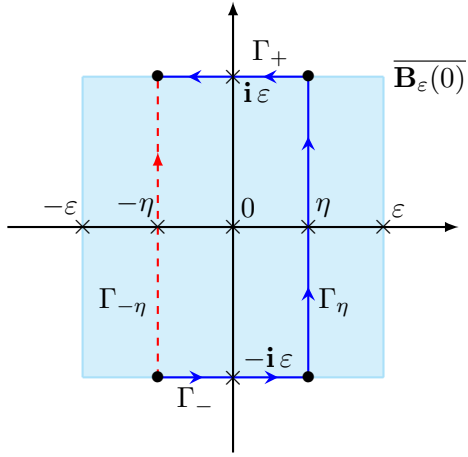
\begin{figure}[t!]
\begin{center}
\begin{tikzpicture}[scale=1,>=latex]
\fill[cyan!15] (-2,-2) -- (2,-2) -- (2,2) -- (-2,2);
\draw[thick,cyan!30] (-2,-2) -- (2,-2);
\draw[thick,cyan!30] (2,-2) -- (2,2);
\draw[thick,cyan!30] (-2,2) -- (2,2);
\draw[thick,cyan!30] (-2,-2) -- (-2,2);
\draw (2,2) node[right]{$\overline{\mathbf{B}_{\varepsilon}(0)}$};
\draw[thick,black,->] (-3,0) -- (3,0);
\draw[thick,black,->] (0,-3)--(0,3);
\draw[thick,blue,directed] (-1,-2) -- (0,-2);
\draw[thick,blue,directed] (0,-2) -- (1,-2);
\draw (-0.5,-2) node[below]{$\Gamma_-$};
\draw[thick,blue,directed] (1,2) -- (0,2);
\draw[thick,blue,directed] (0,2) -- (-1,2);
\draw (0.5,2) node[above]{$\Gamma_+$};
\draw[thick,blue,directed] (1,-2) -- (1,0);
\draw[thick,blue,directed] (1,0) -- (1,2);
\draw (1,-1) node[right]{$\Gamma_\eta$};
\draw[thick,dashed,red,->] (-1,-2) -- (-1,1);
\draw[thick,dashed,red] (-1,1) -- (-1,2);
\draw (-1,-1) node[left]{$\Gamma_{-\eta}$};
\draw (-1.3,0) node[above]{$-\eta$};
\node (centre) at (-1,0){$\times$};
\node (centre) at (0,0){$\times$};
\draw (0.2,0) node[above]{$0$};
\draw (1.2,0) node[above]{$\eta$};
\node (centre) at (1,0){$\times$};
\draw (0,-1.75) node[right]{$-\mbi \, \varepsilon$};
\node (centre) at (0,-2){$\times$};
\draw (0,1.8) node[right]{$\mbi \, \varepsilon$};
\node (centre) at (0,2){$\times$};
\node (centre) at (-1,2){$\bullet$};
\node (centre) at (1,2){$\bullet$};
\node (centre) at (-1,-2){$\bullet$};
\node (centre) at (1,-2){$\bullet$};
\node (centre) at (2,0){$\times$};
\node (centre) at (-2,0){$\times$};
\draw (-2.3,0) node[above]{$-\varepsilon$};
\draw (2.2,0) node[above]{$\varepsilon$};
\end{tikzpicture}
\caption{In blue: the contour $\Gamma_{\rm in}$ used in the regime $\frac{n \, |\alpha_r|}{2}\leq j_0 \leq n$ in Lemma~\ref{lem:estimateA} and its decomposition into three parts ($\Gamma_-$, $\Gamma_\eta$ and $\Gamma_+$). 
In red dash: the contour $\Gamma_{-\eta}$. The black bullets represent the end points of the contours.}
\label{fig:contourGammain}
\end{center}
\end{figure}

\begin{proof}
\textbf{Case (i).}  We assume that $j_0$ and $n$ satisfy $\frac{n \, |\alpha_r|}{2}\leq j_0 \leq n$. The contour $\Gamma_{\rm in}$ is decomposed into 
three parts as depicted in Figure \ref{fig:contourGammain}. From the definition \eqref{defArnj0}, we thus have:
\bqs
\mathscr{A}_r^n(j_0) \, = \, \dfrac{1}{2 \pi \mathbf{i}} \, \int_{\Gamma_-\cup\Gamma_+} \rme^{n \, \tau-j_0 \, \varphi_r(\tau) } \dfrac{\md \tau}{\tau} 
+ \dfrac{1}{2 \pi \mathbf{i}} \, \int_{\Gamma_{\eta}} \rme^{n \, \tau-j_0 \, \varphi_r(\tau) } \dfrac{\md \tau}{\tau},
\eqs
where $\Gamma_\eta=\left\{ \eta + \mbi \, \theta ~|~ |\theta| \leq \varepsilon \right\}$ and $\Gamma_\pm$ are the two horizontal paths joining 
$-\eta \pm \mbi \, \varepsilon$ to $\eta \pm \mbi \, \varepsilon$. Upon denoting 
\begin{equation}
\label{defomegazeta}
\omega \, := \, \dfrac{j_0-n \, |\alpha_r|}{n} \in \, \left[ -\frac{|\alpha_r|}{2},1-|\alpha_r|\right] \quad \text{ and } \quad 
\zeta:= \dfrac{j_0}{n} \, \dfrac{1-\alpha_r^2}{2 \, \alpha_r^2} \in \left[ \frac{1-\alpha_r^2}{4 \, |\alpha_r|},\frac{1-\alpha_r^2}{2 \, \alpha_r^2}\right] \, ,
\end{equation}
we get
\begin{equation}
\label{relationomegazeta}
n \, \tau-j_0 \, \varphi_r(\tau) \, = \, \dfrac{n}{|\alpha_r|} \, \left[ - \omega \, \tau + \dfrac{\zeta}{3} \, \tau^3 - \dfrac{\zeta}{4} \, \tau^4 \right] \, .
\end{equation}

Let now $\tau \in \Gamma_-$ so that $\tau$ reads $\tau=t-\mbi \, \varepsilon$ with $t\in[-\eta,\eta]$. We compute
\bqs
\Re \left( - \omega \, \tau + \dfrac{\zeta}{3} \, \tau^3 - \dfrac{\zeta}{4} \, \tau^4 \right) \, = \, 
-\omega \, t - \zeta \, \varepsilon^2 \, t +\dfrac{\zeta}{3} \, t^3 + \dfrac{3 \, \zeta}{2} \, \varepsilon^2 \, t^2 - \dfrac{\zeta}{4} \, t^4 
- \dfrac{\zeta}{4} \, \varepsilon^4 \, .
\eqs
Recalling that $\eta \le \varepsilon^5$ by assumption, we obtain that
\begin{equation}
\label{estimateA-partiereelle}
\Re \left( - \omega \, \tau + \dfrac{\zeta}{3} \, \tau^3 - \dfrac{\zeta}{4} \, \tau^4 \right) \, \le \, 
- \dfrac{\zeta}{4} \, \varepsilon^4 + |\omega| \, \varepsilon^5 + \zeta \, \varepsilon^7 + \dfrac{3 \, \zeta}{2} \, \varepsilon^{12} 
+ \dfrac{\zeta}{3} \, \varepsilon^{15} \, \le \, - \dfrac{\zeta}{8} \, \varepsilon ^4 \, ,
\end{equation}
thanks to our smallness assumption \eqref{condeps2} on $\varepsilon$. Thus, we get the estimate
\bqs
\left| \frac{1}{2\pi \mathbf{i}} \int_{\Gamma_-} e^{n \, \tau-j_0 \, \varphi_r(\tau) } \frac{\md \tau}{\tau} \right| \leq 
C \, \rme^{- n \, \frac{1-\alpha_r^2}{32\alpha_r^2} \, \varepsilon^4} \, = \, C \, \rme^{-c \, n} \, .
\eqs
There is of course a similar estimate on $\Gamma_+$. Since $n$ dominates $j_0$, one deduces that
\bqs
\left| \mathscr{A}_r^n(j_0) - \dfrac{1}{2\pi \mathbf{i}} \int_{\Gamma_{\eta}} \rme^{n \, \tau-j_0 \, \varphi_r(\tau)} \frac{\md \tau}{\tau} \right| 
\, \le \, C\, \rme^{-c \, n-c \, j_0} \, ,
\eqs
for constants $C,c>0$ that are independent of both $n$ and $j_0$.

Next, we complete the remaining integral along $\Gamma_\eta$ as follows:
\bqs
\dfrac{1}{2\pi \mathbf{i}} \int_{\Gamma_\eta} \rme^{n \, \tau-j_0 \, \varphi_r(\tau) } \dfrac{\md \tau}{\tau} \, = \, 
\dfrac{1}{2\pi \mathbf{i}} \int_{\eta+\mbi\R} \rme^{n \, \tau-j_0 \, \varphi_r(\tau) } \dfrac{\md \tau}{\tau} 
-\dfrac{1}{2\pi \mathbf{i}} \int_{\Gamma_\eta^c} \rme^{n \, \tau-j_0 \, \varphi_r(\tau) } \dfrac{\md \tau}{\tau} \, ,
\eqs
where $\Gamma_\eta^c:=\left\{\eta+\mbi \, \theta ~|~|\theta| > \varepsilon \right\}$. The reader can verify indeed that the integral along the 
line $\eta+\mbi\R$ converges. This is due to the form of the function $\varphi_r$ and the fact that $j_0$ is positive. We now estimate the 
integral along $\Gamma_\eta^c$ and show that it is a remainder term just like several other contributions that we have already estimated.

We keep the notation \eqref{defomegazeta} and use the relation \eqref{relationomegazeta}. For $\tau \in \Gamma_\eta^c$, we write 
$\tau=\eta+\mbi \, \theta$ with $|\theta|>\varepsilon$, so that we have
\bqs
\Re \left( - \omega \, \tau + \dfrac{\zeta}{3} \, \tau^3 - \dfrac{\zeta}{4} \, \tau^4 \right) \, = \, 
- \omega \, \eta + \dfrac{\zeta}{3} \, \eta^3 - \dfrac{\zeta}{4} \, \eta^4 - \dfrac{\zeta}{4} \, \theta^4 
- \zeta \, \eta \, \left( 1 - \dfrac{3}{2} \eta \right) \, \theta^2 \, .
\eqs
Since $0<\eta \le \varepsilon^5$ and $\varepsilon$ is assumed to be small enough to satisfy \eqref{condeps1} and \eqref{condeps2}, we 
have that
\begin{align*}
\Re \left( - \omega \, \tau + \dfrac{\zeta}{3} \, \tau^3 - \dfrac{\zeta}{4} \, \tau^4 \right) \, & \le \, 
- \omega \, \eta + \dfrac{\zeta}{3} \, \eta^3  - \dfrac{\zeta}{4} \, \theta^4 - \dfrac{\zeta \, \eta}{2} \, \theta^2 \\
& \le - \dfrac{\zeta}{4} \, \varepsilon^4 + |\omega| \, \varepsilon^5 + \dfrac{\zeta}{3} \, \varepsilon^{15} - \dfrac{\zeta \, \eta}{2} \, \theta^2 
\le - \dfrac{\zeta}{8}\varepsilon^4 - \frac{\zeta \, \eta}{2} \, \theta^2 \, ,
\end{align*}
 and as a consequence, we obtain
\bqs
\left|\frac{1}{2\pi \mathbf{i}} \int_{\Gamma_\eta^c} \rme^{n \, \tau-j_0 \, \varphi_r(\tau) } \frac{\md \tau}{\tau}\right| \, \le \, 
C \, \rme^{-n \, \frac{1-\alpha_r^2}{32\alpha_r^2} \, \varepsilon^4} \, \int_\varepsilon ^{+\infty} \rme^{-n \, \frac{\zeta \eta}{2} \, \theta^2} \md \theta 
\, \le \, C \, \rme^{-c \, n} \, \le \, C \, \rme^{-c \, n-c \, j_0} \, ,
\eqs
where we have used once again that $n$ dominates $j_0$. As a partial summary, we have obtained that
\bqs
\left| \mathscr{A}_r^n(j_0) - \dfrac{1}{2\pi \mathbf{i}} \int_{\eta +\mbi \R} \rme^{n \, \tau-j_0 \, \varphi_r(\tau)} \frac{\md \tau}{\tau} \right| 
\, \le \, C\, \rme^{-c \, n-c \, j_0} \, ,
\eqs
for positive constants $C$ and $c$ that are independent of $n$ and $j_0$ and for integers that satisfy $ \frac{n \, |\alpha_r|}{2}\leq j_0 \leq n$.

Cauchy's formula shows that the value of the integral:
$$
\dfrac{1}{2\pi \mathbf{i}} \int_{\eta +\mbi \R} \rme^{n \, \tau-j_0 \, \varphi_r(\tau)} \frac{\md \tau}{\tau}
$$
does not depend on $\eta>0$ so choosing the parametrization $\tau=|\alpha_r| \, (\eta+\mbi \, \theta)$ in the integral, we get:
$$
\frac{1}{2\pi \mathbf{i}} \int_{\eta+\mbi\R} \rme^{n \, \tau-j_0 \, \varphi_r(\tau) } \frac{\md \tau}{\tau} \ = \, 
\dfrac{1}{2 \, \pi} \, \int_\R {\rm e}^{-(j_0+n\, \alpha_r) \, (\eta+\mbi \, \theta)} 
\, {\rm e}^{\, j_0 \, \frac{1-\alpha_r^2}{6} \, (\eta+\mbi \, \theta)^3} \, {\rm e}^{j_0 \, \alpha_r \, \frac{1-\alpha_r^2}{8} \, 
(\eta+\mbi \, \theta)^4} \, \dfrac{{\rm d}\theta}{\eta+\mbi \, \theta} \, .
$$
Recalling the definitions \eqref{defc3c4r} and \eqref{defactivationr}, we may write the above integral as
\bqs
\frac{1}{2\pi \mathbf{i}} \int_{\eta+\mbi\R} \rme^{n \, \tau-j_0 \, \varphi_r(\tau) } \frac{\md \tau}{\tau} \ = \, 
\mathbf{A}_r\left(-j_0+n \, |\alpha_r|,\frac{j_0}{|\alpha_r|}\right) \, ,
\eqs
so that we have obtained the estimate:
\bqs
\left| \mathscr{A}_r^n(j_0) - \mathbf{A}_r\left(-j_0+n \, |\alpha_r|,\frac{j_0}{|\alpha_r|}\right) \right| \, \le \, C\, \rme^{-c \, n-c \, j_0} \, .
\eqs

With the definition \eqref{defmajorantr}, Corollary \ref{coro-A7} in Appendix \ref{appendixA} can be rewritten as:
\bqs
\left| \mathbf{A}_r \left( -j_0+n \, |\alpha_r|,\dfrac{j_0}{|\alpha_r|} \right) - \mathbf{A}_r \left( -j_0+n \, |\alpha_r|,n \right) \right| 
\, \le \, C \, \mathbf{M}_r (c,-j_0+n \, |\alpha_r|,n) \, ,
\eqs
for suitable constants $C$ and $c$, since we have assumed $\frac{n \, |\alpha_r|}{2}\leq j_0 \leq n$. We can combine the previous two 
estimates and obtain the estimate \eqref{estimateA} that we were aiming at. This completes the analysis of case (i).
\bigskip

\textbf{Case (ii).} For $1 \le j_0 < \frac{n \, |\alpha_r|}{2}$, we rather use the residue theorem. In other words, as suggested in Figure 
\ref{fig:contourGammain}, we close the contour $\Gamma_{\rm in}$ by using the segment $\Gamma_{-\eta}$ and we thus get:
\bqs
\mathscr{A}_r^n(j_0)= 1+\frac{1}{2\pi \mathbf{i}} \int_{\Gamma_{-\eta}} \rme^{n \, \tau-j_0 \, \varphi_r(\tau) } \frac{\md \tau}{\tau},
\eqs
with $\Gamma_{-\eta}:=\left\{-\eta+\mbi \, \theta ~|~ |\theta| \leq \varepsilon \right\}$. Along $\Gamma_{-\eta}$, for each $\tau=-\eta+\mbi\theta$ 
with $|\theta| \le \varepsilon$, we keep the notation \eqref{defomegazeta} and compute (recall \eqref{relationomegazeta}):
\bqs
\Re \left( - \omega \, \tau + \dfrac{\zeta}{3} \, \tau^3 - \dfrac{\zeta}{4} \, \tau^4 \right) \, = \, 
\omega \, \eta - \dfrac{\zeta}{3} \, \eta^3 - \dfrac{\zeta}{4} \, \eta^4 - \dfrac{\zeta}{4} \, \theta^4 
+ \zeta \, \eta \, \left( 1 + \dfrac{3}{2} \, \eta \right) \, \theta^2 \, ,
\eqs
where $\omega$ and $\zeta$ are defined in \eqref{defomegazeta} and now satisfy
\bqs
\omega \in \left(-|\alpha_r|,-\frac{|\alpha_r|}{2}\right), \quad \text{ and } \quad \zeta \in \left(0,\frac{1-\alpha_r^2}{4|\alpha_r|}\right)\,.
\eqs
As a consequence, we have
\bqs
\Re \left( - \omega \, \tau + \dfrac{\zeta}{3} \, \tau^3 - \dfrac{\zeta}{4} \, \tau^4 \right) \, \le \, 
\eta \, \left( - \dfrac{|\alpha_r|}{2} + \zeta \, \left( 1 + \dfrac{3}{2} \, \varepsilon^5\right) \, \varepsilon^2 \right) \, \le \, - \dfrac{|\alpha_r|}{4} \, \eta \, ,
\eqs
thanks to our smallness assumption \eqref{condeps3} on $\varepsilon$. Thus, we get
\bqs
\left|\frac{1}{2\pi \mathbf{i}} \, \int_{\Gamma_{-\eta}} \rme^{n \, \tau-j_0 \, \varphi_r(\tau)} \, \frac{\md \tau}{\tau} \right| 
\, \le \, C \, \rme^{-c \, n-c \, j_0} \, ,
\eqs
which implies that
\bqs
|\mathscr{A}_r^n(j_0) - 1| \, \le \, C \, \rme^{-c \, n-c \, j_0} \, ,
\eqs
for positive constants $C$ and $c$ that are independent of $n$ and $j_0$, and $1\leq j_0 < \frac{n \, |\alpha_r|}{2}$.

Using now Corollary \ref{coro-A6} in Appendix \ref{appendixA}, we have the estimate:
$$
\left| 1 - \mathbf{A}_r (-j_0+n\, |\alpha_r|,n) \right| \, \le \, C \, \rme^{-c \, n} \le \, C \, \rme^{-c \, n-c \, j_0} \, ,
$$
where the second inequality comes from the fact that $n$ dominates $j_0$. Adding the two previous estimates, we get
\bqs
\left| \mathscr{A}_r^n(j_0) - \mathbf{A}_r (-j_0+n \, |\alpha_r|,n) \right| \, \le \, C \, \rme^{-c \, n-c \, j_0} \, .
\eqs
This completes the analysis of case (ii).
\end{proof}

\subsection{Estimates of the next order term $\mathscr{B}_r^n(j,j_0)$}

We now focus our attention to the next term $\mathscr{B}_r^n(j,j_0)$ appearing in the decomposition~\eqref{InterDecomp}. Our main result is 
the following.

\begin{lemma}
\label{lem:estimateB}
Let $n \ge 1$ and let $1 \le j_0 \le n$. Let $\mathscr{B}_r^n(j,j_0)$ be defined in \eqref{defBrnj0}. Then there exist some positive constants $C$ and $c$ 
that are uniform with respect to $n$, $j \in \Z$ and $j_0$ such that with the function $\mathbf{M}_r$ defined in \eqref{defmajorantr}, there holds:
\bqq
\label{estimateB}
\big| \mathscr{B}_r^n(j,j_0) \big| \le C \, \rme^{-c \, |j|} \, \mathbf{M}_r (c,-j_0+n \, |\alpha_r|,n) \, + \, C \, \rme^{-c \, n-c \, |j|-c \, j_0} \, .
\eqq
\end{lemma}

\begin{proof}
Once again, we shall consider two different regimes identified in the previous section, namely:
\begin{itemize}
\item[(i)] $\frac{n \, |\alpha_r|}{2} \le j_0 \le n$;
\item[(ii)] $1 \le j_0 < \frac{n \, |\alpha_r|}{2}$.
\end{itemize}

From its definition, $\mathscr{B}_r^n(j,j_0)$ splits into two parts:
\bqs
\mathscr{B}_r^n(j,j_0)\, = \, \widetilde{\mathscr{G}}_{2,r}^n(j,j_0) \, + \, \widetilde{\mathscr{G}}_{3,r}^n(j,j_0),
\eqs
with quantities $\widetilde{\mathscr{G}}_{2,r}^n(j,j_0)$ and $\widetilde{\mathscr{G}}_{3,r}^n(j,j_0)$ defined in \eqref{defG2rnj0} and \eqref{defG3rnj0}. 
We shall estimate separately each of theses two terms below for either case (i) or case (ii).

\paragraph{Case (i).} We shall first focus on the term $\widetilde{\mathscr{G}}_{2,r}^n(j,j_0)$ whose expression is given (see \eqref{defG2rnj0}) by:
\bqs
\widetilde{\mathscr{G}}_{2,r}^n(j,j_0) \, = \, \gamma_j^r \, \times \, \dfrac{1}{2\pi \mathbf{i}} 
\int_{\Gamma_{\mathrm{in}}} \exp \left( n \, \tau-j_0 \, \varphi_r(\tau)+j_0 \, \tau^5 \, \Psi_r(\tau) \right) \, \md \tau \, ,
\eqs
where, at first, $\Gamma_{\mathrm{in}}$ is a contour that joins $-\eta -\mbi \, \varepsilon$ to $-\eta +\mbi \, \varepsilon$ and that passes to the right 
of the origin (as depicted in Figure \ref{fig:contourGammainout}). Since we now integrate a holomorphic function (there is no longer a pole at the origin !), 
we shall feel free to deform $\Gamma_{\mathrm{in}}$ and choose any contour that joins $-\eta -\mbi \, \varepsilon$ to $-\eta +\mbi \, \varepsilon$ and 
remains within the closed square $\overline{\mathbf{B}_\varepsilon(0)}$ on which $\Psi_r$ is a holomorphic function. We shall only focus on the integral 
that depends on $j_0$ and $n$ since we already know that the factor $ \gamma_j^r$ satisfies an exponential estimate (see Proposition \ref{prop3}).

For $\frac{n \, |\alpha_r|}{2} \le j_0 \le n$, we may take $\Gamma_{\mathrm{in}}$ as the union of the following paths:
\bqs
\Gamma_{\mathrm{in}} \, = \, \Gamma_- \, \cup \, \Gamma_+ \, \cup \, \Gamma_0 \, ,
\eqs
where $\Gamma_0 := \left\{ \mbi \, \theta ~|~ |\theta| \leq \varepsilon\right\}$ and $\Gamma_\pm$ are horizontal paths joining $-\eta \pm \mbi \, \varepsilon$ 
to $\pm \mbi \, \varepsilon$ (this is rather similar to what is depicted in Figure \ref{fig:contourGammain} except that we have shifted the right segment to 
the abscissa $0$ instead of $+\eta$). Upon writing as usual now (see \eqref{defomegazeta} and \eqref{relationomegazeta}):
\bqs
n \, \tau-j_0 \, \varphi_r(\tau)+j_0 \, \tau^5 \, \Psi_r(\tau) = \dfrac{n}{|\alpha_r|} \, \left( 
-\omega \, \tau + \dfrac{\zeta}{3} \, \tau^3 - \dfrac{\zeta}{4} \, \tau^4 + \dfrac{j_0}{n} \, |\alpha_r| \, \tau^5 \, \Psi_r(\tau) \right) \, ,
\eqs
and noticing that for any $\tau \in \Gamma_\pm \subset \mathbf{B}_{\varepsilon_0}(0)$ there holds (see \eqref{defreste}):
$$
|\tau^5 \, \Psi_r(\tau)| \, \le \, C_0 \, ( \varepsilon^5 \, + \, \eta^5 ) \, ,
$$
we have for each $\tau \in \Gamma_\pm$:
\begin{align*}
\Re \left( -\omega \, \tau + \dfrac{\zeta}{3} \, \tau^3 - \dfrac{\zeta}{4} \, \tau^4 + \dfrac{j_0}{n} \, |\alpha_r| \, \tau^5 \, \Psi_r(\tau) \right) \, 
& \le \, - \dfrac{\zeta}{4} \, \varepsilon^4 + |\omega| \, \eta +\zeta \, \eta \, \varepsilon^2 + \dfrac{3 \, \zeta}{2} \, \eta^2 \, \varepsilon^2 
+ |\alpha_r| \, C_0 \, ( \varepsilon^5 + \eta^5 ) \\
& \le \, - \dfrac{\zeta}{4} \, \varepsilon^4 + (|\omega|+|\alpha_r| \, C_0) \, \varepsilon^5 + \zeta \, \varepsilon^7 + 
\dfrac{3 \, \zeta}{2} \, \varepsilon^{12} + |\alpha_r| \, C_0 \, \varepsilon^{25} \\
& \le \, - \dfrac{\zeta}{8} \, \varepsilon^4 \, ,
\end{align*}
thanks to condition \eqref{condeps2} on $\varepsilon$. As a consequence, we have
\bqs
\left|\frac{1}{2\pi \mathbf{i}} \int_{\Gamma_\pm} \exp\left(n \, \tau -j_0 \, \varphi_r(\tau) +j_0 \, \tau^5 \, \Psi_r(\tau) \right)\md \tau\right| 
\, \le \, C \, \rme^{-c \, n} \, \le \, C \, \rme^{-c \, n-c \, j_0} \, ,
\eqs
since $n$ dominates $j_0$ in this regime.
\bigskip

For the remaining integral along $\Gamma_0$, we use the parametrization $\tau=\mbi \, |\alpha_r| \, \theta$ and we thus get the expression:
\begin{multline*}
\dfrac{1}{2\pi \mathbf{i}} \, \int_{\Gamma_0} \exp \left( n \, \tau-j_0 \, \varphi_r(\tau)+j_0 \, \tau^5 \, \Psi_r(\tau) \right) \, \md \tau \\
= \, \dfrac{|\alpha_r|}{2\pi} \, \int_{-\frac{\varepsilon}{|\alpha_r|}}^{\frac{\varepsilon}{|\alpha_r|}} 
\exp \left ( \mbi \, (-j_0+n \, |\alpha_r|) \, \theta + \mbi \, \frac{j_0}{|\alpha_r|} \, c_{3,r} \, \theta^3 - \frac{j_0}{|\alpha_r|} \, c_{4,r} \, \theta^4 
+ j_0 \, \mbi \, \theta^5 \, |\alpha_r|^5 \, \Psi_r(\mbi \, |\alpha_r| \, \theta) \right) \, \md \theta \, .
\end{multline*}

We can now apply Theorem \ref{thm-A2} from Appendix \ref{appendixA}, and obtain that there exists some small enough $\underline{\varepsilon}>0$ 
with $\underline{\varepsilon} \le \varepsilon/|\alpha_r|$ some constants $C$ and $c$ such that for any $j_0,n \in \N^*$ with $\frac{n \, |\alpha_r|}{2} \le 
j_0 \le n$, there holds\footnote{With the notation of Theorem \ref{thm-A2}, we use $x:=-j_0+n \, |\alpha_r|$ and $y:=j_0/|\alpha_r|$. In case (i), we are 
in a regime where $|x|/y$ is bounded and $y$ is bounded from below so that we can tune the constant $\underline{\mathbf{C}}$.}:
\begin{multline*}
\left|  \int_{-\underline{\varepsilon}}^{\underline{\varepsilon}} 
\exp \left ( \mbi \, (-j_0+n \, |\alpha_r|) \, \theta + \mbi \, \frac{j_0}{|\alpha_r|} \, c_{3,r} \, \theta^3 - \frac{j_0}{|\alpha_r|} \, c_{4,r} \, \theta^4 
+ j_0 \, \mbi \, \theta^5 \, |\alpha_r|^5 \, \Psi_r(\mbi \, |\alpha_r| \, \theta) \right) \, \md \theta \right| \\
\le \, C \, \mathbf{M}_r \left( c,-j_0+n \, |\alpha_r|,\dfrac{j_0}{|\alpha_r|} \right) \, .
\end{multline*}
Since we are in the regime $\frac{n \, |\alpha_r|}{2} \le j_0 \le n$, we deduce that we always have an upper estimate
\bqs
\mathbf{M}_r \left( c,-j_0+n \, |\alpha_r|,\dfrac{j_0}{|\alpha_r|} \right) \, \le \, C \, \mathbf{M}_r (c,-j_0+n \, |\alpha_r|,n) \, ,
\eqs
so that the latter estimate also reads:
\begin{multline*}
\left|  \int_{-\underline{\varepsilon}}^{\underline{\varepsilon}} 
\exp \left ( \mbi \, (-j_0+n \, |\alpha_r|) \, \theta + \mbi \, \frac{j_0}{|\alpha_r|} \, c_{3,r} \, \theta^3 - \frac{j_0}{|\alpha_r|} \, c_{4,r} \, \theta^4 
+ j_0 \, \mbi \, \theta^5 \, |\alpha_r|^5 \, \Psi_r(\mbi \, |\alpha_r| \, \theta) \right) \, \md \theta \right| \\
\le \, C \, \mathbf{M}_r (c,-j_0+n \, |\alpha_r|,n) \, .
\end{multline*}
The only remaining task is to control the integral with respect to $\theta$ on the two intervals $[-\varepsilon/|\alpha_r|,-\underline{\varepsilon}]$ 
and $[\underline{\varepsilon},\varepsilon/|\alpha_r|]$. This is entirely similar to what we have already done on the segments $\Gamma_\pm$ 
so we feel free to skip the details. The final estimate on these segments is of exponential type with respect to both $n$ and $j_0$ (that are of 
comparable sizes in case (i)). Combining all above estimates on the segments $\Gamma_\pm$ and on $\Gamma_0$, we have obtained the 
following estimate for $\widetilde{\mathscr{G}}_{2,r}^n(j,j_0)$:
\bqs
\left| \widetilde{\mathscr{G}}_{2,r}^n(j,j_0) \right| \, 
\le \, C \, \rme^{-c \, |j|} \, \mathbf{M}_r (c,-j_0+n \, |\alpha_r|,n) \, + \, C \, \rme^{-c \, n - c \, |j| - c \, j_0} \, ,
\eqs
for suitable constants $C$ and $c$ and $\frac{n \, |\alpha_r|}{2} \le j_0 \le n$ (the integer $j \in \Z$ is arbitrary).
\bigskip

We finally turn our attention to the next term $\widetilde{\mathscr{G}}_{3,r}^n(j,j_0)$ that enters the definition of $\mathscr{B}_r^n(j,j_0)$. We recall 
that $\widetilde{\mathscr{G}}_{3,r}^n(j,j_0)$ is defined as follows:
\bqs
\widetilde{\mathscr{G}}_{3,r}^n(j,j_0) = \mathcal{H}_j \times \Psi_r(0) \times j_0 \times \dfrac{1}{2\pi \mathbf{i}} 
\, \int_{\Gamma_{\mathrm{in}}} \exp \left( n \, \tau-j_0 \, \varphi_r(\tau) \right) \tau ^4 \md \tau \, ,
\eqs
and we shall mainly focus our efforts on the above integral on the right-hand side since we already know that $\mathcal{H}_j$ decays exponentially 
with respect to $j$, and $\Psi_r(0)$ is just a constant factor. We still focus here on case (i), that is on the regime $\frac{n \, |\alpha_r|}{2} \le j_0 \le n$. 
It is important to observe that we need to absorb the factor $j_0$ that is unbounded.

Since we integrate a holomorphic function, we may use Cauchy's formula and deform the contour $\Gamma_{\mathrm{in}}$. We therefore choose 
$\Gamma_{\mathrm{in}}$ to be the union of the paths $\Gamma_0 := \left\{ \mbi \, \theta ~|~ |\theta| \leq \varepsilon \right\}$ and $\Gamma_\pm$ 
which are horizontal paths joining $- \eta \pm \mbi \, \varepsilon$ to $\pm \mbi \, \varepsilon$, just as we did above for the analysis of the term 
$\widetilde{\mathscr{G}}_{2,r}^n(j,j_0)$. For the integrals along $\Gamma_\pm$, we may use again the estimate \eqref{estimateA-partiereelle} and 
combine with a uniform bound for the factor $\tau^4$. Since $n$ and $j_0$ are comparable in this first regime, we readily obtain the estimate:
\bqs
\left| j_0 \times \dfrac{1}{2\pi \mathbf{i}} \, \int_{\Gamma_{\pm}} \exp \left( n \, \tau-j_0 \, \varphi_r(\tau) \right) \, \tau ^4 \, \md \tau \right| 
\, \le \, C \, j_0 \, \rme^{-c \, n - c \, j_0} \, \le \, C \, \rme^{-c \, n - c \, j_0} \, .
\eqs
For the remaining integral along $\Gamma_0$, we add and subtract to get:
\begin{multline*}
j_0 \times \dfrac{1}{2\pi \mathbf{i}} \, \int_{\Gamma_0} \exp \left( n \, \tau-j_0 \, \varphi_r(\tau) \right) \, \tau ^4 \, \md \tau \\
\, = \, j_0 \times \dfrac{1}{2\pi \mathbf{i}} \, \int_{\mbi \, \R} \exp \left( n \, \tau-j_0 \, \varphi_r(\tau) \right) \, \tau ^4 \, \md \tau 
\, - \, j_0 \times \dfrac{1}{2\pi \mathbf{i}} \, \int_{\Gamma_0^c} \exp \left( n \, \tau-j_0 \, \varphi_r(\tau) \right) \, \tau ^4 \, \md \tau \, ,
\end{multline*}
with rather obvious notation. For $\tau =\mbi \, \theta \in \Gamma_0^c$, we have (see the definition \eqref{defvarphirl}):
$$
\text{\rm Re } \Big( n \, \tau-j_0 \, \varphi_r(\tau) \Big) \, = \, -c \, j_0 \, \theta^4 \, ,
$$
for some constant $c>0$, so we have the straightforward estimate:
$$
\left| \, j_0 \times \dfrac{1}{2\pi \mathbf{i}} \, \int_{\Gamma_0^c} \exp \left( n \, \tau-j_0 \, \varphi_r(\tau) \right) \, \tau ^4 \, \md \tau \, \right| 
\, \le \, C \, j_0 \, \int_\varepsilon^{+\infty} \, \theta^4 \, \rme^{-c \, j_0 \, \theta^4} \, \md \theta 
\, \le \, C \, j_0 \, \rme^{-c \, j_0} \, \le \, C \, \rme^{-c \, j_0} \, .
$$
Since $j_0$ and $n$ are comparable in case (i), we can collect all the above estimates and already obtain the estimate:
$$
\left| \widetilde{\mathscr{G}}_{3,r}^n(j,j_0) \, - \, \mathcal{H}_j \times \Psi_r(0) \times j_0 \times 
\dfrac{1}{2\pi \mathbf{i}} \, \int_{\mbi \, \R} \exp \left( n \, \tau-j_0 \, \varphi_r(\tau) \right) \, \tau ^4 \, \md \tau \right| 
\, \le \, C \, \rme^{-c \, n - c \, |j| - c \, j_0} \, ,
$$
for suitable uniform constants $C$ and $c$, and integers $n$, $j_0$, $j$ that satisfy $\frac{n \, |\alpha_r|}{2} \le j_0 \le n$ and $j \in \Z$.

It thus only remains to estimate the integral over the imaginary axis for which we refer to the definition \eqref{A-defcorrecteur} in Appendix 
\ref{appendixA} of the function $\mathbf{G}_4$ (the constants $c_3$ and $c_4$ should be taken to be $c_{3,r}$ and $c_{4,r}$ since we deal 
here with the right state of the shock). Using the parametrization $\tau=\mbi \, |\alpha_r| \, \theta$ in the integral, we get the relation:
$$
j_0 \times \dfrac{1}{2\pi \mathbf{i}} \, \int_{\mbi \, \R} \exp \left( n \, \tau-j_0 \, \varphi_r(\tau) \right) \, \tau ^4 \, \md \tau \, = \, 
|\alpha_r|^5 \, j_0 \, \mathbf{G}_4 \left( -j_0+n \, |\alpha_r|,\dfrac{j_0}{|\alpha_r|} \right) \, ,
$$
and we can then use the estimates provided by Theorem \ref{thm-A4} to get\footnote{The fact that we consider the function $\mathbf{G}_4$ 
is crucial here since $4$ is the first index where the gain in the estimates of Theorem \ref{thm-A4} is sufficient to absorb the factor $j_0$.}:
\bqs
\left| j_0 \times \dfrac{1}{2\pi \mathbf{i}} \, \int_{\mbi \, \R} \exp \left( n \, \tau-j_0 \, \varphi_r(\tau) \right) \, \tau ^4 \, \md \tau \right| 
\, \le \, C \, \mathbf{M}_r \left( c,-j_0+n \, |\alpha_r|,\dfrac{j_0}{|\alpha_r|} \right) 
\, \le \, C \, \mathbf{M}_r \left( c,-j_0+n \, |\alpha_r|,n \right) \, ,
\eqs
since we have $\frac{n \, |\alpha_r|}{2}\leq j_0 \leq n$. Collecting all the above estimates for the various contributions in the decomposition of 
$\widetilde{\mathscr{G}}_{3,r}^n(j,j_0)$, we end up with the estimate:
\bqs
\left| \widetilde{\mathscr{G}}_{3,r}^n(j,j_0) \right| \, \le \, 
C \, \rme^{-c \, |j|} \, \mathbf{M}_r (c,-j_0+n \, |\alpha_r|,n) \, + \, C \, \rme^{-c \, n - c \, |j| - c \, j_0} \, ,
\eqs
for $\frac{n \, |\alpha_r|}{2} \le j_0 \le n$, $j \in \Z$, and for suitable constants $C$ and $c$ that are uniform with respect to $n$, $j$ and $j_0$. 
This completes the analysis of case (i) by combining with the estimate for $\widetilde{\mathscr{G}}_{2,r}^n(j,j_0)$.

\paragraph{Case (ii).} In the case $1\leq j_0 \leq \frac{n \, |\alpha_r|}{2}$, we simply take $\Gamma_{in}=\Gamma_{-\eta}$ (see Figure 
\ref{fig:contourGammain}) for evaluating the integrals arising in the definition of both $\widetilde{\mathscr{G}}_{2,r}^n(j,j_0)$ and 
$\widetilde{\mathscr{G}}_{3,r}^n(j,j_0)$. This is legitimate because we integrate holomorphic functions on the closed square 
$\overline{\mathbf{B}_\varepsilon(0)}$ so we can apply Cauchy's formula. Reproducing similar computations as in the previous subsection, 
we get that for each $\tau=-\eta+\mbi \, \theta \in \Gamma_{-\eta}$ (with therefore $|\theta| \le \varepsilon$), there holds:
\begin{align*}
\Re \left( - \omega \, \tau + \dfrac{\zeta}{3} \, \tau^3 - \dfrac{\zeta}{4} \, \tau^4 + \dfrac{j_0}{n} \, |\alpha_r| \, \tau^5 \, \Psi_r(\tau) \right) 
& \le \, \eta \, \left( - \dfrac{|\alpha_r|}{2} + \zeta \, \left( 1 + \dfrac{3}{2} \, \varepsilon^5 \right) \, \varepsilon^2 
+ \dfrac{\alpha_r^2}{2} \, C_0 \, \varepsilon^{20} \right) \\
&\quad - \theta^4 \, \left( \dfrac{\zeta}{4} - |\alpha_r| \, C_0 \, |\theta| \right) \\
& \le - \dfrac{|\alpha_r|}{4} \, \eta \, ,
\end{align*}
and similarly:
$$
\Re \left( - \omega \, \tau + \dfrac{\zeta}{3} \, \tau^3 - \dfrac{\zeta}{4} \, \tau^4 \right) \le - \dfrac{|\alpha_r|}{4} \, \eta \, ,
$$
where once again we have used conditions \eqref{condeps1} and \eqref{condeps2} on $\varepsilon$ and our choice for $\eta$. By applying the triangle 
inequality (with a uniform bound for $\tau \in \Gamma_{-\eta}$), we deduce that
\bqs
\left| \dfrac{1}{2\pi \mathbf{i}} \, \int_{\Gamma_{-\eta}} \exp \left( n \, \tau - j_0 \, \varphi_r(\tau) + j_0 \, \tau^5 \, \Psi_r(\tau) \right) \, \md \tau \right| 
\, \le \, C \, \rme^{-c \, n} \, \le \, C \, \rme^{-c \, n - c \, j_0} \, ,
\eqs
and
\bqs
\left| j_0 \, \dfrac{1}{2\pi \mathbf{i}} \, \int_{\Gamma_{-\eta}} \exp \left( n \, \tau - j_0 \, \varphi_r(\tau) \right) \, \tau^4 \, \md \tau \right| 
\, \le \, C \, j_0 \, \rme^{-c \, n - c \, j_0} \, \le \, C \, \rme^{-c \, n - c \, j_0} \, ,
\eqs
where we use once again the fact that $n$ dominates $j_0$. As a consequence, using the exponential estimate of $\gamma_j^r$ from Proposition 
\ref{prop3} and the fact that $\mathcal{H}_j$ is also exponentially decaying, we obtain the uniform exponential estimate:
\bqs
\left| \mathscr{B}_r^n(j,j_0) \right| \, C \, \rme^{-c \, n - c \, |j| - c \, j_0} \, ,
\eqs
for $1 \le j_0 \le \frac{n \, |\alpha_r|}{2}$ and $j \in \Z$. This completes the analysis of case (ii).
\end{proof}

\subsection{Estimates of the remainder term $\mathscr{R}_r^n(j,j_0)$}

In this section, we shall prove some estimates on the remainder term $\mathscr{R}_r^n(j,j_0)$ whose expression is gathered in \eqref{defRrnj0}. 
We recall that $\mathscr{R}_r^n(j,j_0)$ is decomposed into $\mathscr{R}_r^n(j,j_0)=\mathscr{R}_{1,r}^n(j,j_0)+\mathscr{R}_{2,r}^n(j,j_0)$ with 
$\mathscr{R}_{1,r}^n(j,j_0)$ and $\mathscr{R}_{2,r}^n(j,j_0)$ defined in \eqref{defR1rnj0} and \eqref{defR2rnj0}. In both \eqref{defR1rnj0} and 
\eqref{defR2rnj0}, the path $\Gamma_{\mathrm{in}}$ is any path joining $-\eta-\mbi \, \varepsilon$ to $-\eta+\mbi \, \varepsilon$ which remains in 
$\overline{\mathbf{B}_\varepsilon(0)}$ since the integrand is holomorphic with respect to $\tau$ on that set. 
%Note that we can lift our initial 
%constrain that the path passes to the right of the origin since both maps
%\bqs
%\tau \mapsto \exp\left(n \, \tau-j_0 \, \varphi_r(\tau)+j_0 \, \tau^5 \, \Psi_r(\tau) \right) \tau \Phi_{r,j}(\tau)\,,
%\eqs
%and 
%\bqs
%\tau\mapsto \rme^{n \, \tau-j_0 \, \varphi_r(\tau)} \left( \frac{\exp\left(j_0 \, \tau^5 \, \Psi_r(\tau)\right)-1-j_0\tau^5 \Psi_r(0)}{\tau} \right)\,,
%\eqs
%are holomorphic on $\overline{\mathbf{B}_\varepsilon(0)}\subset \mathbf{B}_{\varepsilon_0}(0)$.
We further recall that the sequence $(\mathcal{H}_j)_{j \in \Z}$ in \eqref{defH} is exponentially decreasing and the sequence of holomorphic 
functions $(\Phi_{r,j})_{j \in \Z}$ satisfies the exponential bound stated in Proposition \ref{prop3}, uniformly with respect to $\tau \in 
\mathbf{B}_{\varepsilon_0}(0)$. Restricting to the smaller set $\overline{\mathbf{B}_\varepsilon(0)}\subset \mathbf{B}_{\varepsilon_0}(0)$, 
we thus have
\bqs
\left| \mathcal{H}_j \right| \, + \, \left| \Phi_{r,j}(\tau) \right| \, \le \, C \, \rme^{-c \, |j|} \, ,
\eqs
for all $j\in\Z$ and $\tau \in \overline{\mathbf{B}_\varepsilon(0)}$ with uniform constants $C,c>0$. Our main result in this section is the following.

\begin{proposition}
\label{propestimR}
There exist $C,c>0$ such that for all $n\in\N^*$ and $(j,j_0)\in\Z^2$ such that $1\leq j_0 \leq n$, one has:
\bqs
\left| \mathscr{R}_r^n(j,j_0) \right| \, \le \, C \, 
\begin{cases}
\dfrac{1}{n^{1/4}} \dfrac{e^{-c \, |j|}}{n^{1/3}} \, \exp \left( -c \left(\dfrac{j_0-n \, |\alpha_r|}{n^{1/3}} \right)^{3/2}  \right) \, , & 
\text{\rm if $j_0-n \, |\alpha_r| \ge 0$,} \\
\dfrac{e^{-c \, |j|}}{n^{7/12}} \, , & \text{\rm if $-n^{1/3} \, |\alpha_r|\leq j_0-n \, |\alpha_r| \le 0$,} \\
\dfrac{1}{n^{1/8}} \, \dfrac{e^{-c \, |j|}}{n^{1/2}} \, \exp \left( -c \left(\dfrac{\left|j_0-n \, |\alpha_r|\right|}{n^{1/2}} \right)^{2}  \right) \, , & 
\text{\rm if $j_0-n \, |\alpha_r| \le -n^{1/3} \, |\alpha_r|$.}
\end{cases}
\eqs
In particular, we have:
$$
\left| \mathscr{R}_r^n(j,j_0) \right| \, \le \, C \, \rme^{-c \, |j|} \, \mathbf{M}_r (c,-j_0+|\alpha_r| \, n,n) \, ,
$$
with the function $\mathbf{M}_r$ defined in \eqref{defmajorantr}.
\end{proposition}

\begin{proof}
We will mainly focus on the remainder term $\mathscr{R}_{1,r}^n(j,j_0)$ and briefly explain how to recover similar estimates for the second 
term $\mathscr{R}_{2,r}^n(j,j_0)$ in the last part of this section. We shall decompose the proof into several steps, which corresponds to 
different regimes for $\omega$, which we recall is defined as
\bqs
\omega \, := \, \dfrac{j_0-n \, |\alpha_r|}{n} \in \left(-|\alpha_r|, 1-|\alpha_r|\right) \, , \quad \text{ when } \quad 1\leq j_0 \leq n \, .
\eqs
More precisely, we define the following three regimes:
\bqs
(\mathrm{I})~|\omega| \leq n^{-2/3} |\alpha_r|, \quad (\mathrm{II})~ n^{-2/3} |\alpha_r| \leq \omega \leq 1-|\alpha_r|, \quad 
(\mathrm{III})~ -|\alpha_r|\leq \omega \leq -n^{-2/3} |\alpha_r|\,.
\eqs

\paragraph{Case (I) -- Uniform bound.} For $|\omega| \leq n^{-2/3}|\alpha_r|$, we provide a uniform bound for $\mathscr{R}^n_r(j,j_0)$ using 
classical results from oscillatory integrals (see Proposition \ref{prop-A1} in Appendix \ref{appendixA} for similar arguments).

\begin{lemma}\label{lemestimR1}
There exist constants $C,c>0$ such that for all $n\geq1$, $1\leq j_0 \leq n$ with $|j-j_0|\leq n$ one has
\bqs
\left| \mathscr{R}^n_{1,r}(j,j_0)\right| \leq \frac{C }{n^{7/12}}e^{-c \, |j|},
\eqs
for $|\omega| \leq n^{-2/3}|\alpha_r|$.
\end{lemma}

\begin{proof}
Let us first observe that for $|\omega| \leq n^{-2/3}|\alpha_r|$, there holds:
$$
\left| \dfrac{j_0}{n \, |\alpha_r|} -1 \right| \, \le \, n^{-2/3} \, ,
$$
so for $n \ge 2$, we get a uniform bound from below $j_0/n \ge c>0$. The case $n=1$ should be dealt with separately but it is far easier 
since we have $j_0=1$ and $j \in \{0,1,2 \}$ so we only deal with finitely many integrals then. We shall therefore assume $n \ge 2$ from 
now on and use that the quantity $\zeta$ defined in \eqref{defomegazeta} is uniformly bounded from below by a positive constant.

We take $\Gamma_{\mathrm{in}}$ as the union of the following paths:
\bqs
\Gamma_{\mathrm{in}} = \Gamma_-\cup\Gamma_+\cup \Gamma_0,
\eqs
where $\Gamma_0=\left\{\mbi \, \theta ~|~ |\theta| \leq \varepsilon\right\}$ and $\Gamma_\pm$ are horizontal paths joining $-\eta\pm\mbi \, \varepsilon$ 
to $\pm\mbi \, \varepsilon$. Upon writing as usual now
\bqs
n \, \tau-j_0 \, \varphi_r(\tau)+j_0 \, \tau^5 \, \Psi_r(\tau) 
= \frac{n}{|\alpha_r|} \left(-\omega \tau +\frac{\zeta}{3}\tau^3-\frac{\zeta}{4}\tau^4 +\frac{j_0}{n}|\alpha_r|\tau^5 \Psi_r(\tau) \right),
\eqs
we have already proved that for each $\tau \in \Gamma_\pm$, there holds:
\bqs
\Re \left(-\omega \tau +\frac{\zeta}{3}\tau^3-\frac{\zeta}{4}\tau^4 +\frac{j_0}{n}|\alpha_r|\tau^5 \Psi_r(\tau) \right)\leq - \frac{\zeta}{8}\varepsilon^8,
\eqs
thanks to the smallness condition \eqref{condeps2} on $\varepsilon$. As a consequence, we get
\bqs
\left|\frac{1}{2\pi \mathbf{i}} \int_{\Gamma_{\pm}} \exp\left(n \, \tau-j_0 \, \varphi_r(\tau)+j_0 \, \tau^5 \, \Psi_r(\tau) \right) \tau \, \Phi_{r,j}(\tau) \md \tau\right| 
\leq C \, \rme^{-c \, |j|-c \, j_0}\,.
\eqs
But since $|\omega| \leq n^{-2/3}|\alpha_r|$, we readily get $\rme^{-j_0}\leq C \rme^{-c \, n}$, from which we deduce that
\bqs
\left|\frac{1}{2\pi \mathbf{i}} \int_{\Gamma_{\pm}} \exp\left(n \, \tau-j_0 \, \varphi_r(\tau)+j_0 \, \tau^5 \, \Psi_r(\tau) \right) \tau \, \Phi_{r,j}(\tau) \md \tau\right| 
\leq C \, \frac{\rme^{-c \, |j|}}{n^{7/12}}\,.
\eqs
Along the remaining integral on the vertical segment $\Gamma_0$, we have
\begin{multline*}
\frac{1}{2\pi \mathbf{i}} \int_{\Gamma_0} \exp\left(n \, \tau-j_0 \, \varphi_r(\tau)\right. +\left.j_0 \, \tau^5 \, \Psi_r(\tau) \right) \tau \, \Phi_{r,j}(\tau) \md \tau \\
=\frac{1}{2\pi} \int_{-\varepsilon}^{\varepsilon} \exp\left( \frac{n}{|\alpha_r|} \left[ 
-\omega \mbi \, \theta  -\frac{\zeta}{3} \mbi \, \theta^3-\frac{\zeta}{4}\theta^4 +\frac{j_0}{n}|\alpha_r| \mbi \, \theta^5 \Psi_r(\mbi \, \theta) \right] \right) 
\theta \, \Phi_{r,j}(\mbi \, \theta) \, \md \theta \, .
\end{multline*}
We introduce two functions (that depend on $(j,j_0,n)$):
\bqs
h(\theta):= \exp\left(- n \mbi \left(\frac{\omega}{|\alpha_r|} \theta +\frac{\zeta}{3|\alpha_r|}\theta^3\right)\right), \quad 
g(\theta):= \exp\left(-n \frac{\zeta}{4|\alpha_r|}\theta^4+n\mbi \frac{2\zeta\alpha_r^2}{1-\alpha_r^2} \theta^5\Psi_r(\mbi\theta) \right) 
\, \theta \, \Phi_{r,j}(\mbi \, \theta).
\eqs
Using \cite[Lemma 3.1]{RSF1}, we have the estimate
\bqq
\left| \int_{-\varepsilon}^{\varepsilon} h(\theta) \, g(\theta) \, \md \theta \right| \leq 
\left( \underset{x\in [-\varepsilon,\varepsilon]}{\sup} \left|\int_{-\varepsilon}^x h(\theta) \md \theta \right| \right) 
\left( \| g \| _{L^\infty([-\varepsilon,\varepsilon])}+\| g' \| _{L^1([-\varepsilon,\varepsilon])}\right).
\label{keyestimate}
\eqq
By an application of the van der Corput Lemma, there exists a constant $C>0$, independent of $\omega$ and $n$, such that\footnote{This holds 
because the parameter $\zeta$ is uniformly bounded from below in the considered regime.}
\bqs
\forall x \in [-\varepsilon,\varepsilon], \quad  \left|\int_{-\varepsilon}^x h(\theta) \md \theta \right| \leq \frac{C}{n^{1/3}}.
\eqs
Furthermore, with our choice \eqref{condeps1} of $\varepsilon>0$, we have
\bqs
\forall \theta \in [-\varepsilon,\varepsilon], \quad \left| g(\theta)\right| \leq  C \, |\theta| \, \rme^{-n \frac{\zeta}{8|\alpha_r|}\theta ^4 }\rme^{-c \, |j|}\,.
\eqs
Differentiating the expression for $g(\theta)$, we also get the bound
\bqs
\forall \theta \in [-\varepsilon,\varepsilon], \quad \left| g'(\theta)\right| \leq  C \, \left( 1+ n|\theta|^4\right) \, \rme^{-n \frac{\zeta}{8|\alpha_r|}\theta ^4} 
\, \rme^{- c \, |j|}, 
\eqs
such that we get that
\bqs
\| g' \| _{L^1([-\varepsilon,\varepsilon])} \leq \frac{C}{n^{1/4}} e^{-c \, |j|},
\eqs
since $\zeta$ is uniformly positive in the considered regime. Using estimate \eqref{keyestimate}, we arrive at the final bound
\bqs
\left| \mathscr{R}^n_{1,r}(j,j_0)\right| \leq \frac{C }{n^{1/3+1/4}}\rme^{- c \, |j|},
\eqs
for some constants $C,c>0$ independent of $n$, $j_0$ and $j$. This concludes the proof of the lemma.
\end{proof}

\paragraph{Case $(\mathrm{II})$ -- Fast decaying tail.} We now turn our attention to the second regime $(\mathrm{II})$ where $n^{-2/3} 
|\alpha_r| < \omega \leq 1- |\alpha_r|$ implying that necessarily $j_0-n \, |\alpha_r|>0$ where we expect to observe a fast decaying bound 
for $\mathscr{R}^n_{1,r}(j,j_0)$.

\begin{lemma}\label{lemestimR2}
There exist constants $C,c>0$ such that for all $n\geq1$, $1\leq j_0 \leq n$ with $|j-j_0|\leq n$ one has
\bqs
\left| \mathscr{R}^n_{1,r}(j,j_0)\right| \leq C \, \dfrac{\rme^{-c \, |j|}}{n^{2/3}} \, \left( \frac{j_0-n \, |\alpha_r|}{n^{1/3}} \right)^{-1/2} 
\exp \left( -c \left( \frac{j_0-n \, |\alpha_r|}{n^{1/3}} \right)^{3/2} \right),
\eqs
as long as $ n^{-2/3}|\alpha_r| \leq  \omega \leq 1- |\alpha_r|$.
\end{lemma}

\begin{proof}
We first note that when $\omega>0$, one has $\frac{1-\alpha_r^2}{2|\alpha_r|}<\zeta \leq \frac{1-\alpha_r^2}{2\alpha_r^2}$ so $\zeta$ is 
uniformly positive. With the constant $C_0>0$ in \eqref{defreste} (associated to $\varepsilon_0$ given by Proposition \ref{prop3}), we 
choose $\omega_\varepsilon\in(0,1-|\alpha_r|)$ small enough such that the following inequalities are satisfied
\bqq
\label{condvarepomega}
\omega_\varepsilon \leq \frac{1-\alpha_r^2}{2|\alpha_r|} \epsilon^2, \quad \sqrt{\omega_\varepsilon} \leq \frac{1}{8|\alpha_r|C_0} 
\left(\frac{1-\alpha_r^2}{2|\alpha_r|}\right)^{3/2}, \quad \sqrt{\omega_\varepsilon} \leq \frac{1}{3} \left(\frac{1-\alpha_r^2}{2|\alpha_r|}\right)^{1/2}.
\eqq
Next, we fix $ n^{-2/3}|\alpha_r|\leq \omega \leq \omega_\varepsilon$. We introduce a family of parametrized curves (see\footnote{The reader 
may compare our choice with the one made in \cite{jfcAMBP} that corresponds to the parametrization of the Green's function for the Cauchy 
problem. See also Appendix \ref{appendixA}. The difference here is that we rather parametrize all curves in terms of the time frequency rather 
than with respect to the space frequency.} Figure~\ref{fig:contourCaseII} for an illustration)  indexed by $\omega$ as follows
\bqs
\Gamma_-:=\left\{ t-\mbi \, \varepsilon ~|~ t \in[-\eta,0]\right\}, \quad \Gamma_+:=\left\{ -t+\mbi \, \varepsilon ~|~ t \in[0,\eta]\right\}, \quad 
\Gamma_\omega:=\left\{ \sqrt{\frac{\omega}{\zeta}}+\mbi \, \theta ~|~ \theta\in\left[-\varepsilon,\varepsilon\right]\right\},
\eqs
together with
\bqs
\Gamma_>^\omega:=\left\{ t- \mbi \, \varepsilon  ~|~ t \in\left[0,\sqrt{\frac{\omega}{\zeta}}\right]\right\}, \quad 
\Gamma_<^\omega:=\left\{\sqrt{\frac{\omega}{\zeta}}- t+ \mbi \, \varepsilon ~|~  t \in\left[0,\sqrt{\frac{\omega}{\zeta}}\right]\right\}.
\eqs

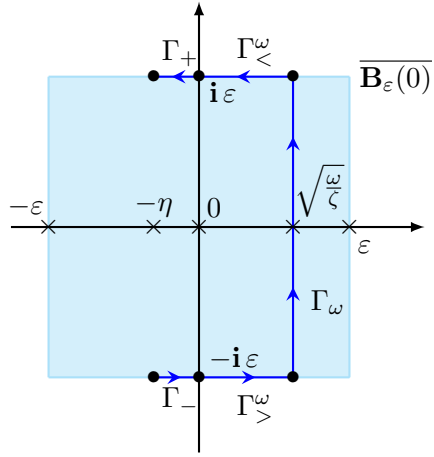
\begin{figure}[h!]
\begin{center}
\begin{tikzpicture}[scale=1,>=latex]
\fill[cyan!15] (-2,-2) -- (2,-2) -- (2,2) -- (-2,2);
\draw[thick,cyan!30] (-2,-2) -- (2,-2);
\draw[thick,cyan!30] (2,-2) -- (2,2);
\draw[thick,cyan!30] (-2,2) -- (2,2);
\draw[thick,cyan!30] (-2,-2) -- (-2,2);
\draw (2,2) node[right]{$\overline{\mathbf{B}_{\varepsilon}(0)}$};
\draw[thick,black,->] (-2.5,0) -- (3,0);
\draw[thick,black,->] (0,-3)--(0,3);
\draw[thick,blue,directed] (-0.6,-2) -- (0,-2);
\draw[thick,blue,directed] (0,-2) -- (1.25,-2);
\draw (-0.25,-2) node[below]{$\Gamma_-$};
\draw[thick,blue,directed] (1.25,2) -- (0,2);
\draw[thick,blue,directed] (0,2) -- (-0.6,2);
\draw (-0.25,2) node[above]{$\Gamma_+$};
\draw (0.75,2) node[above]{$\Gamma_<^\omega$};
\draw (0.75,-2.1) node[below]{$\Gamma_>^\omega$};
\draw[thick,blue,directed] (1.25,-2) -- (1.25,0);
\draw[thick,blue,directed] (1.25,0) -- (1.25,2);
\draw (1.35,-1) node[right]{$\Gamma_\omega$};
\draw (-0.6,0) node[above]{$-\eta$};
\node (centre) at (-0.6,0){$\times$};
\node (centre) at (0,0){$\times$};
\draw (0.2,0) node[above]{$0$};
\draw (1.6,0) node[above]{$\sqrt{\frac{\omega}{\zeta}}$};
\node (centre) at (1.25,0){$\times$};
\draw (0,-1.75) node[right]{$-\mbi \, \varepsilon$};
\node (centre) at (0,-2){$\bullet$};
\draw (0,1.8) node[right]{$\mbi \, \varepsilon$};
\node (centre) at (0,2){$\bullet$};
\node (centre) at (-0.6,2){$\bullet$};
\node (centre) at (-0.6,-2){$\bullet$};
\node (centre) at (1.25,2){$\bullet$};
\node (centre) at (1.25,-2){$\bullet$};
\node (centre) at (-2,0){$\times$};
\node (centre) at (2,0){$\times$};
\draw (-2.3,0) node[above]{$-\varepsilon$};
\draw (2.2,-0.05) node[below]{$\varepsilon$};
\end{tikzpicture}
\caption{In blue: the contour $\Gamma_{\rm in}$ and its decomposition into five parts ($\Gamma_-$, $\Gamma_>^\omega$, $\Gamma_\omega$, 
$\Gamma_<^\omega$ and $\Gamma_+$) in the regime where $ n^{-2/3}|\alpha_r| \leq  \omega \leq 1- |\alpha_r|$. The black bullets represent the 
end points of the contours.}
\label{fig:contourCaseII}
\end{center}
\end{figure}

Along $\Gamma_\pm$, we have already proved that
\bqs
\left|\frac{1}{2\pi \mathbf{i}} \int_{\Gamma_{\pm}} \exp\left(n \, \tau-j_0 \, \varphi_r(\tau)+j_0 \, \tau^5 \, \Psi_r(\tau) \right) \tau \, \Phi_{r,j}(\tau) 
\, \md \tau\right| \leq C \, \rme^{-c \, |j|-c \, j_0}\,.
\eqs
Next, we remark that since $0<\omega \leq 1-|\alpha_r|$ we deduce that
\bqs
-1\leq -\frac{1}{\sqrt{1-|\alpha_r|}}\left(\frac{j_0-n \, |\alpha_r|}{n} \right)^{1/2},
\eqs
and thus
\bqs
-j_0=-n \, |\alpha_r|-\left(j_0-n \, |\alpha_r|\right) \leq -n \, |\alpha_r| -\frac{1}{\sqrt{1-|\alpha_r|}} \left(\frac{j_0-n \, |\alpha_r|}{n^{1/3}}\right)^{3/2}.
\eqs
As a consequence, we have obtained
\bqs
\left|\frac{1}{2\pi \mathbf{i}} \int_{\Gamma_{\pm}} \exp\left(n \, \tau-j_0 \, \varphi_r(\tau)+j_0 \, \tau^5 \, \Psi_r(\tau) \right) \tau \, \Phi_{r,j}(\tau) 
\, \md \tau\right| \leq C \, \rme^{-c \, n-c \, |j|}\exp\left(-c \left(\frac{j_0-n \, |\alpha_r|}{n^{1/3}}\right)^{3/2}\right)\,,
\eqs
which can be subsumed into our desired bound.

Next, we handle the contributions along $\Gamma_>^\omega$ and $\Gamma_<^\omega$. For example, in the former case, upon denoting $\Lambda_r(\tau):= n \, \tau-j_0 \varphi_r(\tau)+j_0 \, \tau^5 \, \Psi_r(\tau)$, we have that for each $\tau =t-\mbi \epsilon \in \Gamma_>^\omega$ with $t \in\left[0,\sqrt{\frac{\omega}{\zeta}}\right]$,
\begin{align*}
\Re\left( \Lambda_r\left(t-\mbi \, \varepsilon \right)\right)&\leq \frac{n}{|\alpha_r|} \left[-\omega t+ \frac{\zeta}{3}\left( t^3-3\varepsilon^2t\right)- \frac{\zeta}{4} \left( t^4+\varepsilon^4-6\varepsilon^2t^2\right)+|\alpha_r|C_0(t^5+\varepsilon^5) \right]\\
&\leq \frac{n}{|\alpha_r|} \left[ -t\left(\frac{2}{3}\omega+\zeta \varepsilon^2\left(1-\frac{3}{2}\sqrt{\frac{\omega}{\zeta}}\right)  \right) -t^4\left( \frac{\zeta}{4}-|\alpha_r|C_0 t\right)-\varepsilon^4\left( \frac{\zeta}{4}-|\alpha_r|C_0 \varepsilon\right) \right]\\
&\leq \frac{n}{|\alpha_r|} \left[ -t\left(\frac{2}{3}\omega+\frac{\zeta}{2} \varepsilon^2  \right) -\frac{\zeta}{8} \varepsilon^4  \right] \,,
\end{align*}
since from \eqref{condeps1} one has $8|\alpha_r|C_0\varepsilon<\zeta$ and from \eqref{condvarepomega} one has $\sqrt{\omega}<\sqrt{\zeta}/3$ and $8|\alpha_r|C_0 \sqrt{\omega} \leq \zeta^{3/2}$.
Hence,
\begin{align*}
\left| \dfrac{1}{ 2\pi \mbi} \, \int_{\Gamma_>^\omega}\exp\left(n \, \tau-j_0 \varphi_r(\tau)+j_0 \, \tau^5 \, \Psi_r(\tau)\right)\tau \Phi_{r,j}(\tau) \md \tau \right| & \leq C e^{-n \frac{\zeta \varepsilon^4}{8|\alpha_r|} -c \, |j|} \int_0^{\sqrt{\frac{\omega}{\zeta}}} \rme^{-nt \frac{\zeta}{2|\alpha_r|} \varepsilon^2}\md t \\
& \leq C \rme^{-nc -c \, |j|} \exp\left(-c \left(\frac{j_0-n \, |\alpha_r|}{n^{1/3}}\right)^{3/2}\right).
\end{align*}
We finally turn our attention to the last integral along $\Gamma_\omega$. We note that in that case, for each $\theta\in[-\varepsilon,\varepsilon]$, one has
\begin{align*}
\Re\left( \Lambda_r\left(\sqrt{\frac{\omega}{\zeta}}+\mbi\theta\right)\right)&\leq \frac{n}{|\alpha_r|}\left(-\frac{2}{3\sqrt{\zeta}}\omega^{3/2}- \frac{\omega^{2}}{\zeta^2}\left(\frac{\zeta}{4}-|\alpha_r|C_0\sqrt{\frac{\omega}{\zeta}}\right)  -\sqrt{\omega}\left( \sqrt{\zeta}-\frac{3}{2}\sqrt{\omega}\right) \theta^2\right)\\
&~~~ - \frac{n}{|\alpha_r|}\left(\frac{\zeta}{4}-|\alpha_r|C_0\varepsilon\right)\theta^4\\
&\leq \frac{n}{|\alpha_r|}\left(-\frac{2}{3\sqrt{\zeta}}\omega^{3/2} -\frac{3\sqrt{\zeta}}{4} \sqrt{\omega} \theta^2 \right),
\end{align*}
thanks to our assumption \eqref{condvarepomega}.  As a consequence, we get
\begin{align*}
\left| \frac{1}{2\pi \mbi} \int_{\Gamma_\omega} \rme^{\Lambda_r(\tau)} \tau\Phi_{r,j}(\tau)\md \tau \right| &\leq C \rme^{-\frac{2n}{3|\alpha_r|\sqrt{\zeta}}\omega^{3/2}-c \, |j|}\int_{-\varepsilon}^\varepsilon \left(\sqrt{\omega}+|\theta|\right) \rme^{-\frac{3\sqrt{\zeta}}{4|\alpha_r|} n \sqrt{\omega}\theta^2}\md\theta\\
&\leq C \rme^{-\frac{2n}{3|\alpha_r|\sqrt{\zeta}}\omega^{3/2}-c \, |j| } \left( \frac{\omega^{1/4}}{n^{1/2}}+\frac{1}{n \omega^{1/2}} \right)\\
&\leq C \frac{e^{-  c \, |j| }}{n^{2/3}}  \left( \frac{|1-j_0|-n \, |\alpha_r|}{n^{1/3}} \right)^{-1/2} \exp\left(-c \left( \frac{j_0-n \, |\alpha_r|}{n^{1/3}} \right)^{3/2} \right).
\end{align*}
\bigskip

We now move to the case where $\omega_\varepsilon \leq \omega \leq 1-|\alpha_r|$. We follow the same strategy and use the contours 
$\Gamma_\pm$, $\Gamma_{\omega_\varepsilon}$ and $\Gamma_{\lessgtr}^{\omega_\varepsilon}$ independently of $\omega$. We readily 
notice that with our careful choice of $\omega_\varepsilon$ all the previous computations remain valid and thus we also have in that case
\begin{align*}
\left| \mathscr{R}^n_{1,r}(j,j_0)\right| &\leq \left| \frac{1}{2\pi \mbi} \int_{\Gamma_\pm} \rme^{\Lambda_r(\tau)} \tau\Phi_{r,j}(\tau)\md \tau \right|+ \left| \frac{1}{2\pi \mbi} \int_{\Gamma_{>}^{\omega_\varepsilon}\, \cup \, \Gamma_{<}^{\omega_\varepsilon}} \rme^{\Lambda_r(\tau)} \tau\Phi_{r,j}(\tau)\md \tau \right| \\
& ~~~ + \left| \frac{1}{2\pi \mbi} \int_{\Gamma_{\omega_\varepsilon}} \rme^{\Lambda_r(\tau)} \tau\Phi_{r,j}(\tau)\md \tau \right|\\
& \leq C\left[ e^{-c \, j_0-c \, |j|} + e^{-c n -c \, |j|} +  \rme^{-\frac{n}{|\alpha_r|}\left(\omega - \frac{\omega_\varepsilon}{\zeta}\right) \sqrt{\frac{\omega_\varepsilon}{\zeta}}-c \, |j|} \left( \frac{\omega_\varepsilon ^{1/4}}{n^{1/2}}+\frac{1}{n \omega_\varepsilon^{1/2}} \right) \right] \\
& \leq C \frac{\rme^{- c \, |j| }}{n^{2/3}}  \left( \frac{j_0-n \, |\alpha_r|}{n^{1/3}} \right)^{-1/2} \exp\left(-c \left( \frac{j_0-n \, |\alpha_r|}{n^{1/3}} \right)^{3/2} \right),
\end{align*}
since $\omega_\varepsilon \leq \omega \leq 1-|\alpha_r|$. This concludes the proof of the lemma.
\end{proof}

\paragraph{Case $(\mathrm{III})$ -- Oscillatory tail.}

We finally move to the last regime $(\mathrm{III})$ where $-|\alpha_r|<\omega \leq -n^{-2/3}|\alpha_r|$. The analysis is further split into two parts.

\begin{lemma}\label{lemestimR3}
There exist constants $\omega_*>0$ and $C,c>0$ such that for each $1\leq j_0\leq n$ and $|j-j_0|\leq n$, on has 
\bqs
\left| \mathscr{R}^n_{1,r}(j,j_0)\right| \leq C \, \frac{ \rme^{-c \, |j|}}{n^{5/8}}\left( \frac{1}{n^{1/8}} \left(\frac{|j_0-n \, |\alpha_r||}{n^{1/2}}\right)^{-1/2} 
+ \left(\frac{|j_0-n \, |\alpha_r||}{n^{1/2}}\right)^{1/4}  \right) \, \rme^{-c\left(\frac{|j_0-n \, |\alpha_r||}{n^{1/2}}\right)^2},
\eqs
as long as $-\omega_*<\omega \leq -n^{-2/3}|\alpha_r|$.
\end{lemma}

\begin{proof}
We let $\omega_*\in(0,|\alpha_r|/3)$ be fixed as follows:
\bqq
\label{condomega2}
\omega_* \leq \frac{1-\alpha_r^2}{3|\alpha_r|} \left(-1+\sqrt{1+4\varepsilon}\right)^2, \quad 
|\alpha_r|C_0 \sqrt{\omega_*} \leq \frac{1}{8}\left(\frac{1-\alpha_r^2}{3|\alpha_r|}\right)^{3/2}, \quad 
\sqrt{\omega_*} \leq \frac{1}{3} \left(\frac{1-\alpha_r^2}{3|\alpha_r|}\right)^{1/2}.
\eqq
As a consequence, for each $|\omega|\in [n^{-2/3}|\alpha_r|,\omega_*]$, one has 
\bqs
\frac{1-\alpha_r^2}{3|\alpha_r|} = \frac{2|\alpha_r|}{3} \frac{1-\alpha_r^2}{2\alpha_r^2} < (|\alpha_r|-\omega_*) \frac{1-\alpha_r^2}{2\alpha_r^2} 
\leq \frac{j_0}{n} \frac{1-\alpha_r^2}{2\alpha_r^2} =\zeta.
\eqs
Defining:
\bqs
\chi_\omega:=\frac{|\omega|}{2\zeta}+\sqrt{\frac{|\omega|}{\zeta}}>0,
\eqs
the above conditions imply in particular that 
\bqs
\chi_\omega \leq \frac{\varepsilon}{4}, \quad |\alpha_r|C_0 \sqrt{\frac{|\omega|}{\zeta}} \leq \frac{\zeta}{8}, \quad \sqrt{\frac{|\omega|}{\zeta}} 
\leq \frac{1}{3} , \quad \text{ with } |\omega|\in [n^{-2/3},\omega_*].
\eqs

We then introduce the following contours which are illustrated in Figure~\ref{fig:contourCaseIII} (compare again with the choice made in 
\cite{jfcAMBP} that is entirely similar):
\begin{itemize}
\item two horizontal contours $\Gamma_\pm$ defined as
\bqs
\Gamma_-=\left\{ t-\mbi \, \varepsilon ~|~ t \in[-\eta,0]\right\}, \quad \Gamma_+=\left\{ -t+\mbi \, \varepsilon ~|~ t \in[0,\eta]\right\};
\eqs
\item two horizontal contours $\Gamma_{\lessgtr}^{\omega}$ defined as
\bqs
\Gamma_>^\omega=\left\{ t- \mbi \, \varepsilon  ~|~ t \in\left[0,\sqrt{\frac{|\omega|}{\zeta}}\right]\right\}, \quad \Gamma_<^\omega 
=\left\{\sqrt{\frac{|\omega|}{\zeta}}- t+ \mbi \, \varepsilon ~|~  t \in\left[0,\sqrt{\frac{|\omega|}{\zeta}}\right]\right\};
\eqs
\item two vertical contours $\Gamma^v_\pm$ defined as
\bqs
\Gamma^v_-=\left\{\sqrt{\frac{|\omega|}{\zeta}}+\mbi \, \theta ~|~ \theta \in\left[-\varepsilon,-\chi_{\omega}-\sqrt{\frac{|\omega|}{\zeta}}\right]\right\}, \quad \Gamma^v_+=\left\{\sqrt{\frac{|\omega|}{\zeta}}+\mbi \, \theta ~|~ \theta \in\left[\chi_{\omega}+\sqrt{\frac{|\omega|}{\zeta}},\varepsilon\right]\right\};
\eqs
\item two oblique contours $\Gamma^o_\pm$
\begin{align*}
\Gamma^o_-&=\left\{-\frac{|\omega|}{2\zeta}-\mbi\sqrt{\frac{|\omega|}{\zeta}}+t e^{3\mbi \pi/4} ~|~ 
t \in\left[-\sqrt{2}\chi_{\omega},\sqrt{2}\sqrt{\frac{|\omega|}{\zeta}}\right]\right\}, \\ 
\Gamma^o_+&=\left\{-\frac{|\omega|}{2\zeta}+\mbi\sqrt{\frac{|\omega|}{\zeta}}+t e^{\mbi \pi/4} ~|~ 
t \in\left[-\sqrt{2}\sqrt{\frac{|\omega|}{\zeta}},\sqrt{2}\chi_{\omega}\right]\right\}.
\end{align*}
\end{itemize}

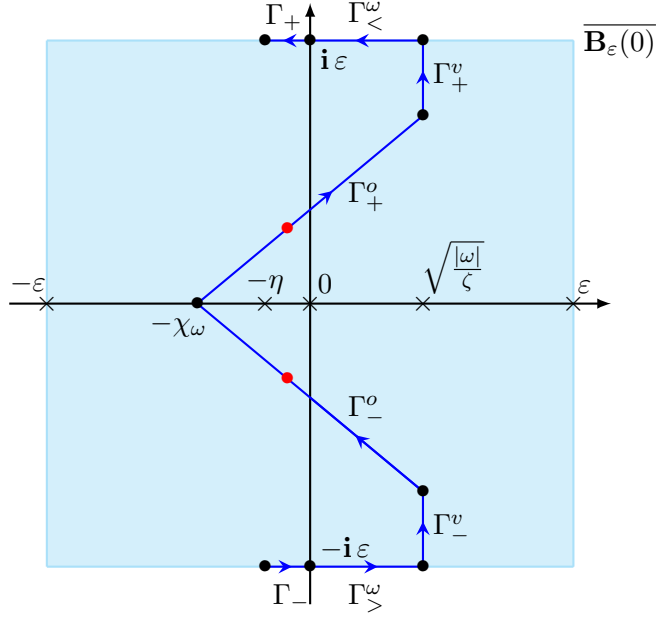
\begin{figure}[h!]
\begin{center}
\begin{tikzpicture}[scale=1,>=latex]
\fill[cyan!15] (-3.5,-3.5) -- (3.5,-3.5) -- (3.5,3.5) -- (-3.5,3.5);
\draw[thick,cyan!30] (-3.5,-3.5) -- (3.5,-3.5);
\draw[thick,cyan!30] (3.5,-3.5) -- (3.5,3.5);
\draw[thick,cyan!30] (-3.5,3.5) -- (3.5,3.5);
\draw[thick,cyan!30] (-3.5,-3.5) -- (-3.5,3.5);
\draw (3.5,3.5) node[right]{$\overline{\mathbf{B}_{\varepsilon}(0)}$};

\draw[thick,black,->] (-4,0) -- (4,0);
\draw[thick,black,->] (0,-4)--(0,4);
\draw[thick,blue,directed] (-0.6,-3.5) -- (0,-3.5);
\draw[thick,blue,directed] (0,-3.5) -- (1.5,-3.5);
\draw[thick,blue,directed] (1.5,3.5) -- (0,3.5);
\draw[thick,blue,directed] (0,3.5) -- (-0.6,3.5);
\draw[thick,blue,directed] (1.5,-3.5) -- (1.5,-2.5);
\draw[thick,blue,directed] (1.5,2.5) -- (1.5,3.5);
\draw[thick,blue] (1.5,-2.5) -- (-1.5,0);
\draw[thick,blue,directed] (1.5,-2.5) -- (0,-1.25);
\draw[thick,blue,directed] (-1.5,0) -- (1.5,2.5);

\node (centre) at (-0.6,-3.5){$\bullet$};
\node (centre) at (0,-3.5){$\bullet$};
\node (centre) at (1.5,-3.5){$\bullet$};
\node (centre) at (1.5,-2.5){$\bullet$};
\node (centre) at (-1.5,0){$\bullet$};
\node (centre) at (-0.6,3.5){$\bullet$};
\node (centre) at (0,3.5){$\bullet$};
\node (centre) at (1.5,3.5){$\bullet$};
\node (centre) at (1.5,2.5){$\bullet$};
\node (centre) at (-0.6,0){$\times$};
\node (centre) at (0,0){$\times$};
\node (centre) at (1.5,0){$\times$};
\node (centre) at (-3.5,0){$\times$};
\node (centre) at (3.5,0){$\times$};

\draw (-0.25,-3.6) node[below]{$\Gamma_-$};
\draw (-0.35,3.5) node[above]{$\Gamma_+$};
\draw (0.75,3.5) node[above]{$\Gamma_<^\omega$};
\draw (0.75,-3.6) node[below]{$\Gamma_>^\omega$};
\draw (1.5,-3) node[right]{$\Gamma_-^v$};
\draw (1.5,3) node[right]{$\Gamma_+^v$};
\draw (0.75,-1.75) node[above]{$\Gamma_-^o$};
\draw (0.75,1.75) node[below]{$\Gamma_+^o$};

\draw (-0.6,0) node[above]{$-\eta$};
\draw (0.2,0) node[above]{$0$};
\draw (1.9,0) node[above]{$\sqrt{\frac{|\omega|}{\zeta}}$};
\draw (-1.75,0) node[below]{$-\chi_\omega$};
\draw (0,-3.25) node[right]{$-\mbi \, \varepsilon$};
\draw (0,3.25) node[right]{$\mbi \, \varepsilon$};

\draw (-3.75,0) node[above]{$-\varepsilon$};
\draw (3.65,0) node[above]{$\varepsilon$};

\node[red] (centre) at (-0.3,-1){$\bullet$};
\node[red] (centre) at (-0.3,1){$\bullet$};

\end{tikzpicture}
\caption{In blue: the contour $\Gamma_{\rm in}$ within $\overline{\mathbf{B}_{\varepsilon}(0)}$ and its decomposition into eight parts 
($\Gamma_-$, $\Gamma_>^\omega$, $\Gamma_-^v$, $\Gamma_-^0$, $\Gamma_+^0$, $\Gamma_v^+$, $\Gamma_<^\omega$ and 
$\Gamma_+$) in the regime where $-\omega_*<\omega \leq -n^{-2/3}|\alpha_r|$. The two red dots correspond to the approximate saddles 
of the phase $|\omega|\tau+\frac{\zeta}{3}\tau^3-\frac{\zeta}{4}\tau^4$ in the complex plane and the black bullets represent the end points 
of the contours.}
\label{fig:contourCaseIII}
\end{center}
\end{figure}

Using Cauchy's formula, we write
 \begin{align*}
\mathscr{R}^n_r(j,j_0)&=\frac{1}{2\pi \mbi} \int_{\Gamma_-\cup\Gamma_+} \rme^{\Lambda_r(\tau)} \tau\Phi_{r,j}(\tau)\md \tau+\frac{1}{2\pi \mbi} \int_{\Gamma_>^\omega\cup\Gamma_<^\omega} \rme^{\Lambda_r(\tau)} \tau\Phi_{r,j}(\tau)\md \tau,\\
&~~~+\frac{1}{2\pi \mbi} \int_{\Gamma_-^v\cup\Gamma_+^v} \rme^{\Lambda_r(\tau)} \tau\Phi_{r,j}(\tau)\md \tau+\frac{1}{2\pi \mbi} \int_{\Gamma_-^o\cup\Gamma_+^o} \rme^{\Lambda_r(\tau)} \tau\Phi_{r,j}(\tau)\md \tau.
\end{align*}
As in the preceding case, we have that the uniform bound
\bqs
\left|\frac{1}{2\pi \mbi} \int_{\Gamma_-\cup\Gamma_+}  \rme^{\Lambda_r(\tau)} \tau\Phi_{r,j}(\tau)\md \tau\right| \leq C \, \rme^{-c \, j_0-c \, |j|},
\eqs
and remark that since $|\omega| \in [n^{-2/3}|\alpha_r|,\omega_*]$, an exponential bound in $j_0$ leads to an exponential bound in both $j_0$ 
and $n$.
\bigskip

\textbf{Bounds along $\Gamma_{\lessgtr}^{\omega}$.}  For each $\tau =t-\mbi \, \varepsilon \in \Gamma_>^\omega$ with 
$t \in\left[0,\sqrt{\frac{|\omega|}{\zeta}}\right]$,
\begin{align*}
\Re\left( \Lambda_r\left(t-\mbi \, \varepsilon \right)\right)&\leq \frac{n}{|\alpha_r|} \left[|\omega| t+ \frac{\zeta}{3}\left( t^3-3\varepsilon^2t\right)- \frac{\zeta}{4} \left( t^4+\varepsilon^4-6\varepsilon^2t^2\right)+|\alpha_r|C_0 (t^5+\varepsilon^5) \right]\\
&\leq \frac{n}{|\alpha_r|} \left[ t\left(\frac{4}{3}|\omega|-\zeta \varepsilon^2\left(1-\frac{3}{2}\sqrt{\frac{|\omega|}{\zeta}}\right)  \right) -t^4\left( \frac{\zeta}{4}-|\alpha_r|C_0 t\right)-\varepsilon^4\left( \frac{\zeta}{4}-|\alpha_r|C_0 \varepsilon\right) \right]\\
&\leq \frac{n}{|\alpha_r|} \left[ t\left(\frac{4}{3}|\omega|-\frac{\zeta}{2} \varepsilon^2  \right) -\frac{\zeta}{8} \varepsilon^4  \right] \leq - \frac{n}{|\alpha_r|} \left[ t\frac{\zeta}{6} \varepsilon^2 +\frac{\zeta}{8} \varepsilon^4  \right].
\end{align*}
Thus, we obtain
\bqs
\left|\frac{1}{2\pi \mbi} \int_{\Gamma_>^\omega} \rme^{\Lambda_r(\tau)} \tau\Phi_{r,j}(\tau)\md \tau\right| \leq C \rme^{ - n \frac{\zeta \varepsilon^4}{8|\alpha_r|}-c \, |j|} \int_0^{\sqrt{\frac{|\omega|}{\zeta}}} e^{- n \frac{\zeta}{6|\alpha_r|}t}\md t \leq C \rme^{-c \, n-c \, |j|}\,.
\eqs
By symmetry, a similar estimate holds along $\Gamma_<^\omega$. As for the bounds along $\Gamma_\pm$, we remark that since $|\omega| \in 
[n^{-2/3}|\alpha_r|,\omega_*]$, an exponential bound in $n$ leads to an exponential bound in both $j_0$ and $n$.
\bigskip

\textbf{Bounds along $\Gamma_{\pm}^v$.} For each $\theta\in \left[-\varepsilon,-\chi_{\omega}-\sqrt{\frac{|\omega|}{\zeta}}\right]$, we have that
\begin{align*}
 \Re\left( \Lambda_r\left(\sqrt{\frac{|\omega|}{\zeta}}+ \mbi\theta\right)\right)&\leq \frac{n}{|\alpha_r|}\left(\frac{4}{3\sqrt{\zeta}}|\omega|^{3/2}-\sqrt{|\omega|}\left(\sqrt{\zeta}-\frac{3}{2}\sqrt{|\omega|}\right) \theta^2-\left(\frac{\zeta}{4}-|\alpha_r|C_0\right)\theta^4\right)\\
 &~~~-\frac{n}{|\alpha_r|}\frac{|\omega|^2}{\zeta^2}\left(\frac{\zeta}{4}-|\alpha_r|C_0\sqrt{\frac{|\omega|}{\zeta}}\right) .
 \end{align*}
Thanks to our careful choice for $\varepsilon$ and $\omega_*$, we get that
\bqs
 \Re\left( \Lambda_r\left(\sqrt{\frac{|\omega|}{\zeta}}+ \mbi\theta\right)\right)\leq \frac{n}{|\alpha_r|}\left(\frac{4}{3\sqrt{\zeta}}|\omega|^{3/2}-\frac{\sqrt{\zeta}}{2}\sqrt{|\omega|}\theta^2\right)- \frac{n |\omega|^2}{8|\alpha_r|\zeta}, \quad \theta\in \left[-\varepsilon,-\chi_{\omega}-\sqrt{\frac{|\omega|}{\zeta}}\,\right].
 \eqs
As a consequence, we obtain
\begin{equation*}
\left|\frac{1}{2\pi \mbi} \int_{\Gamma_-^v} \rme^{\Lambda_r(\tau)} \tau\Phi_{r,j}(\tau)\md \tau\right| 
\leq C \, \rme^{\frac{4}{3\sqrt{\zeta}|\alpha_r|}n|\omega|^{3/2}- \frac{n |\omega|^2}{8|\alpha_r|\zeta}-c \, |j|} 
\int_{-\varepsilon}^{-\chi_{\omega}-\sqrt{\frac{|\omega|}{\zeta}}} 
(\sqrt{|\omega|}+|\theta|)\rme^{-\frac{\sqrt{\zeta}}{2|\alpha_r|}\sqrt{|\omega|}n\theta^2}\md\theta \, .
\end{equation*}
The last integral on the right-hand side should be estimated carefully since there is an exponentially growing factor in front of the integral. 
We use Lemma \ref{lem-A1} in Appendix \ref{appendixA} and obtain the bound:
$$
\int_{-\varepsilon}^{-\chi_{\omega}-\sqrt{\frac{|\omega|}{\zeta}}} \rme^{-\frac{\sqrt{\zeta}}{2|\alpha_r|}\sqrt{|\omega|}n\theta^2} \, \md\theta 
\le \int_{2 \, \sqrt{\frac{|\omega|}{\zeta}}}^{+\infty} \rme^{-\frac{\sqrt{\zeta}}{2|\alpha_r|}\sqrt{|\omega|}n\theta^2} \, \md\theta \le 
\dfrac{C}{n \, |\omega|} \, \rme^{-\frac{2}{\sqrt{\zeta} |\alpha_r|}n |\omega|^{3/2}} \, ,
$$
and the analogous bound:
$$
\int_{-\varepsilon}^{-\chi_{\omega}-\sqrt{\frac{|\omega|}{\zeta}}} |\theta| \, \rme^{-\frac{\sqrt{\zeta}}{2|\alpha_r|}\sqrt{|\omega|}n\theta^2} \, \md\theta 
\le \int_{2 \, \sqrt{\frac{|\omega|}{\zeta}}}^{+\infty} \theta \, \rme^{-\frac{\sqrt{\zeta}}{2|\alpha_r|}\sqrt{|\omega|}n\theta^2} \, \md\theta \le 
\dfrac{C}{n \, \sqrt{|\omega|}} \, \rme^{-\frac{2}{\sqrt{\zeta} |\alpha_r|}n |\omega|^{3/2}} \, ,
$$
This leads to the final estimate:
\begin{equation*}
\left|\frac{1}{2\pi \mbi} \int_{\Gamma_-^v} \rme^{\Lambda_r(\tau)} \tau\Phi_{r,j}(\tau)\md \tau\right| 
\leq C \, \frac{\rme^{-\frac{2}{3\sqrt{\zeta}|\alpha_r|}n|\omega|^{3/2}- \frac{n |\omega|^2}{8|\alpha_r|\zeta}-c \, |j|}}{n\sqrt{|\omega|}}\,,
\end{equation*}
since we have $|\omega| \le \omega_*$. A similar estimate holds along $\Gamma_+^v$.
\bigskip

\textbf{Bounds along $\Gamma_{\pm}^o$.} We finally  handle the contributions along the oblique contours $\Gamma^o_\pm$. We focus first 
on $\Gamma^o_-$. For each $\tau\in\Gamma^o_-$, we have the parametrization
\bqs
\tau(t)=-\frac{|\omega|}{2\zeta}-\mbi\sqrt{\frac{|\omega|}{\zeta}}+t e^{3\mbi \pi/4}, \quad t \in\left[-\sqrt{2}\chi_{\omega},\sqrt{2}\sqrt{\frac{|\omega|}{\zeta}}\right],
\eqs
and we compute that
\bqs
 |\omega|\tau(t)+\frac{\zeta}{3}\tau(t)^3-\frac{\zeta}{4}\tau(t)^4=\sum_{k=0}^4p_k(|\omega|)t^k,
\eqs
where each $p_k$ depends on $|\omega|$ only and is a complex valued function whose expansions as $|\omega|\rightarrow0$ is given by
\begin{align*}
\Re(p_0(|\omega|))&=-\frac{1}{4\zeta}|\omega|^2+\mathcal{O}(|\omega|^3),\quad 
\Im(p_0(|\omega|))=-\frac{2}{3\sqrt{\zeta}}|\omega|^{3/2}+\mathcal{O}(|\omega|^{5/2}),\quad p_1(|\omega|)=\mathcal{O}(|\omega|^2),\\
\Re(p_2(|\omega|))&=-\sqrt{\zeta|\omega|}+\mathcal{O}(|\omega|^{3/2}),\quad \Im(p_2(|\omega|))=\mathcal{O}(|\omega|),\\
p_3(|\omega|)&=\frac{\zeta}{3}e^{\mbi \pi/4}+\mathcal{O}(|\omega|^{1/2}),\quad p_4(|\omega|)=\frac{\zeta}{4}.
\end{align*}
As a consequence, we get that for each $t \in\left[-\sqrt{2}\chi_{\omega},\sqrt{2}\sqrt{\frac{|\omega|}{\zeta}}\right]$
\begin{align*}
\sum_{k=1}^4\Re(p_k(|\omega|))t^k&=-\sqrt{\zeta|\omega|}t^2+\frac{\zeta}{3\sqrt{2}}t^3+\underbrace{\Re(p_1(|\omega|))}_{\mathcal{O}(|\omega|^2)}t+\underbrace{\left(\Re(p_2(|\omega|))+\sqrt{\zeta|\omega|}\right)}_{\mathcal{O}(|\omega|^{3/2})}t^2\\
&~~~+\underbrace{\left(\Re(p_3(|\omega|))-\frac{\zeta}{3\sqrt{2}}\right)}_{\mathcal{O}(|\omega|^{1/2})}t^3+\frac{\zeta}{4}t^4,
\end{align*}
together with
\bqs
\Re(p_0(|\omega|))+\frac{j_0}{n}|\alpha_r|\Re\left( \tau(t)^5\Psi_{r}(\tau(t))\right) \leq -\frac{1}{4\zeta}|\omega|^2+\underbrace{\left(\Re(p_0(|\omega|))+\frac{1}{4\zeta}|\omega|^2\right)}_{\mathcal{O}(|\omega|^3)} +|\alpha_r|C_0 |\tau(t)|^5.
\eqs
Upon decreasing further $\omega_*$ if necessary, we get the existence of a constant $c>0$ independent of $|\omega|$ and $n$, such that
\bqs
\sum_{k=1}^4\Re(p_k(|\omega|))t^k\leq - c |\alpha_r| \sqrt{|\omega|} t^2, \quad \Re(p_0(|\omega|))+\frac{j_0}{n}|\alpha_r|\Re\left( \tau(t)^5\Psi_{r}(\tau(t))\right) \leq -c |\alpha_r| |\omega|^2\,,
\eqs
for all $t \in\left[-\sqrt{2}\chi_{\omega},\sqrt{2}\sqrt{\frac{|\omega|}{\zeta}}\right]$. As a consequence, we have
\bqs
\left|\frac{1}{2\pi \mbi} \int_{\Gamma_-^o}  \rme^{\Lambda_r(\tau)} \tau\Phi_{r,j}(\tau)\md \tau\right| \leq C \sqrt{|\omega|} \rme^{-c \, n|\omega|^2-c \, |j|}\int_{-\sqrt{2}\chi_{\omega}}^{\sqrt{2}\sqrt{\frac{|\omega|}{\zeta}}} \rme^{-c \, n \sqrt{|\omega|} t^2}\md t. 
\eqs
Next, we remark that
\bqs
\int_{-\sqrt{2}\chi_{\omega}}^{\sqrt{2}\sqrt{\frac{|\omega|}{\zeta}}} \rme^{-c \, n \sqrt{|\omega|} t^2}\md t \leq \frac{C}{n^{1/2}|\omega|^{1/4}}, \quad 
n^{-2/3}|\alpha_r| \leq |\omega| \leq \omega_*,
\eqs
such that
\bqs
\left|\frac{1}{2\pi \mbi} \int_{\Gamma_-^o} \rme^{\Lambda_r(\tau)} \tau^5\Phi_{r,j}(\tau) \md \tau\right| \leq 
C \frac{|\omega|^{1/4}}{n^{1/2}} \rme^{-c \, n|\omega|^2-c \, |j| },  \quad n^{-2/3}|\alpha_r| \leq |\omega| \leq \omega_*.
\eqs
And a similar estimate holds along $\Gamma_+^o$.
\bigskip

\textbf{Conclusion.} In summary, combining all the bounds for the eight segments and retaining only the ``worst'' contributions, we have 
obtained the estimate
\bqs
\left| \mathscr{R}^n_{1,r}(j,j_0)\right| \leq C \left(\frac{1}{n\sqrt{|\omega|}} +\frac{|\omega|^{1/4}}{n^{1/2}} \right) \rme^{-cn|\omega|^2- c \, |j| },  \quad 
n^{-2/3}|\alpha_r| \leq |\omega| \leq \omega_*.
\eqs
We may rewrite the above estimate as
\bqs
\left| \mathscr{R}^n_{1,r}(j,j_0)\right| \leq C \frac{\rme^{-c \, |j|}}{n^{5/8}}\left( \frac{1}{n^{1/8}}\left(\frac{|j_0-n \, |\alpha_r||}{n^{1/2}}\right)^{-1/2}+\left(\frac{|j_0-n \, |\alpha_r||}{n^{1/2}}\right)^{1/4}  \right) e^{-c \left(\frac{|j_0-n \, |\alpha_r||}{n^{1/2}}\right)^2}, 
\eqs
valid in the range $n^{-2/3}|\alpha_r| \leq |\omega| \leq \omega_*$.
\end{proof}

We should now deal with the final regime $\omega \le -\omega_*$, with $\omega_*>0$ being given by Lemma \ref{lemestimR3}.

\begin{lemma}\label{lemestimR4}
Let $\omega_*>0$ be given by Lemma \ref{lemestimR3}. Then there exist constants $C,c>0$ such that for each $1\leq j_0\leq n$ and $|j-j_0|\leq n$, 
there holds:
\bqs
\left| \mathscr{R}^n_{1,r}(j,j_0)\right| \leq  C \rme^{-cn-c \, j_0 - c \, |j|},
\eqs
as long as $-|\alpha_r|<\omega \leq -\omega_*$.
\end{lemma}

\begin{proof}
We simply take the vertical contour $\Gamma_{\mathrm{in}}=\Gamma_{-\eta}$, and recall that, thanks to conditions \eqref{condeps1} and 
\eqref{condeps2} on $\varepsilon$, we have
\bqs
\Re\left(\Lambda_r(-\eta+\mbi \, \theta)\right) \leq -\frac{|\alpha_r|}{4}\eta, \quad \theta\in[-\varepsilon,\varepsilon]\,.
\eqs
We readily obtain the desired bound since $-|\alpha_r|<\omega \leq -\omega_*$.
\end{proof}

Combining Lemma~\ref{lemestimR1}, Lemma~\ref{lemestimR2}, Lemma~\ref{lemestimR3} and Lemma~\ref{lemestimR4}, we have proved 
Proposition~\ref{propestimR} for $\mathscr{R}^n_{1,r}(j,j_0)$. Indeed, in the range $n^{-2/3}|\alpha_r| \leq |\omega| \leq \omega_*$, we can 
further bound
\bqs
\frac{1}{n^{1/8}}\left(\frac{|j_0-n \, |\alpha_r||}{n^{1/2}}\right)^{-1/2} \leq n^{1/12-1/8} \leq 1,
\eqs
and
\bqs
\left(\frac{|j_0-n \, |\alpha_r||}{n^{1/2}}\right)^{1/4} \, \rme^{-c \left(\frac{|j_0-n \, |\alpha_r||}{n^{1/2}}\right)^2} \le 
\tilde{C} \, e^{-\tilde{c} \left(\frac{|j_0-n \, |\alpha_r||}{n^{1/2}}\right)^2},
\eqs
with a smaller constant $\tilde{c}>0$ and a suitable constant $\tilde{C}$.

Regarding the second remainder term $\mathscr{R}^n_{2,r}(j,j_0)$, we first observe that for each $\tau\in\overline{\mathbf{B}_{\varepsilon}(0)}$ 
one has
\bqs
\frac{\exp\left(j_0 \, \tau^5 \, \Psi_r(\tau)\right)-1-j_0\tau^5 \Psi_r(0)}{\tau} = 
\frac{\exp\left(j_0 \, \tau^5 \, \Psi_r(\tau)\right)-1-j_0\tau^5 \Psi_r(\tau)}{\tau} + j_0 \tau^4\left( \Psi_r(\tau)-\Psi_r(0)\right),
\eqs
such that
\bqs
\left| \frac{\exp\left(j_0 \, \tau^5 \, \Psi_r(\tau)\right)-1-j_0\tau^5 \Psi_r(0)}{\tau}\right| \leq C \left[ \left(j_0 |\tau|^4\right)+j_0|\tau|^4\right]|\tau| \,.
\eqs
All the previous lemmas can be easily adapted to obtain similar estimates for $\mathscr{R}^n_{2,r}(j,j_0)$ as the ones we derived for 
$\mathscr{R}^n_{1,r}(j,j_0)$. For example, regarding the uniform bound in Lemma \ref{lemestimR1}, the map $g(\theta)$ now reads
\bqs
g(\theta) := \rme^{-n\frac{\zeta}{4}\theta^4} \left(\frac{\exp\left(j_0\mbi\theta^5\Psi_r(\mbi\theta)\right)-1-j_0\mbi\theta^5 \Psi_r(0)}{\theta}\right), \quad 
\theta\in[-\varepsilon,\varepsilon],
\eqs
and we note that
\bqs
\left|g'(\theta)\right| \leq C \left(1+ n \, \theta^4+(n \, \theta^4)^2 \right)\rme^{-cn\theta^4},
\eqs
since $1\leq j_0 \leq n$, and we observe that $\|g'\|_{L^1([-\varepsilon,\varepsilon])}\leq C n^{-1/4}$. As a consequence, we naturally retrieve 
the estimate of Lemma~\ref{lemestimR1} for $\mathscr{R}^n_{2,r}(j,j_0)$. We let the other cases to the interested reader.
\end{proof}

\subsection{Final decomposition of the temporal Green's function. Proof of Theorem \ref{thmGreen}}

Let us recall that, for $j_0 \ge 1$, we have first decomposed (see \eqref{decomposition-initial}):
$$
\mathscr{G}^n(j,j_0) \, = \, \overline{\mathscr{G}}_r^n(j-j_0) \, \mathds{1}_{j \ge 1} \, + \, \widetilde{\mathscr{G}}^n(j,j_0) \, ,
$$
and the reduced Green's function $\widetilde{\mathscr{G}}^n(j,j_0)$ has been decomposed into:
$$
\widetilde{\mathscr{G}}^n(j,j_0) \, = \, 
\dfrac{1}{2 \pi \mbi} \, \int_{\Gamma_{\rm out}} \rme^{n \, \tau} \, \widetilde{\mathcal{G}}^{j_0}_j(\rme^\tau) \, \rme^\tau \, \md \tau  
\, + \, \dfrac{1}{2 \pi \mbi} \, \int_{\Gamma_{\rm in}} \rme^{n \, \tau} \, \widetilde{\mathcal{G}}^{j_0}_j(\rme^\tau) \, \rme^\tau \, \md \tau \, ,
$$
with suitable contours $\Gamma_{\rm out}$ and $\Gamma_{\rm in}$. The part on $\Gamma_{\rm out}$ satisfies the exponential bound 
\eqref{borneGammaout} and the part on $\Gamma_{\rm in}$ has been further decomposed in \eqref{InterDecomp}.

Coming back to our decomposition \eqref{InterDecomp} of part of the (reduced) temporal Green's function $\widetilde{\mathscr{G}}^n(j,j_0)$, and recalling 
the initial decomposition \eqref{decomposition-initial}, we can gather Lemma \ref{lem:estimateA} (together with the exponential decay of $\mathcal{H}_j$), 
Lemma \ref{lem:estimateB} and Proposition \ref{propestimR} to obtain for $j \in \Z$ and $1 \le j_0 \le n$:
\begin{align*}
|\mathscr{G}^n(j,j_0) - \mathcal{H}_j \, \mathbf{A}_r(-j_0+n \, |\alpha_r|,n)| \, & \le \, |\overline{\mathscr{G}}_r^n(j-j_0)| \, \mathds{1}_{j \ge 1} 
\, + \, |\widetilde{\mathscr{G}}^n(j,j_0) - \mathcal{H}_j \, \mathbf{A}_r(-j_0+n \, |\alpha_r|,n)| \\
& \le \, C \, \mathbf{M}_r \left( c,j-j_0+n \, |\alpha_r|,n \right) \, \mathds{1}_{j \ge 1} \\
&\quad + C \, {\rm e}^{-c \, |j|} \, \mathbf{M}_r \left( c,-j_0+n \, |\alpha_r|,n \right) \, + C \, {\rm e}^{-c \, n} \, {\rm e}^{-c \, |j|} \, {\rm e}^{-c \, |j_0|} \, ,
\end{align*}
where we have used Corollary \ref{coro-A1} in Appendix \ref{appendixA} to derive the bound of the free Green's function $\overline{\mathscr{G}}_r$ 
(and the definition \eqref{defmajorantr} of the function $\mathbf{M}_r$). This shows the validity of the bound \eqref{borneGreen-2} in Theorem 
\ref{thmGreen} for the case $1 \le j_0 \le n$.
\bigskip

Let us finally consider the regime $1 \le n \leq j_0-1$ and show that the estimate \eqref{borneGreen-2} in Theorem \ref{thmGreen} is still valid. We apply 
Lemma \ref{lem4} and obtain:
\begin{align*}
\left| \mathscr{G}^n(j,j_0) - \mathcal{H}_j \, \mathbf{A}_r(-j_0+n \, |\alpha_r|,n) \right| \, 
& = \, \left| \overline{\cG}^n_r(j-j_0) - \mathcal{H}_j \, \mathbf{A}_r(-j_0+n \, |\alpha_r|,n) \right| \\
& \le \, \left| \overline{\cG}^n_r(j-j_0) \right| + |\mathcal{H}_j| \, |\mathbf{A}_r(-j_0+n \, |\alpha_r|,n)| \\
& \le \, C \, \mathbf{M}_r ( c,j-j_0+n \, |\alpha_r|,n ) \, + \, C \, \rme^{-c \, |j|} \, \left| \mathbf{A}_r (-j_0+n \, |\alpha_r|,n) \right| \, ,
\end{align*}
where we have again used Corollary \ref{coro-A1} in Appendix \ref{appendixA} to derive the bound of the free Green's function $\overline{\mathscr{G}}_r$. 
It only remains to show that the final term on the right-hand side is exponentially small with respect to both $j_0$ and $n$ and the proof will be complete. 
We apply Corollary \ref{coro-A6} in Appendix \ref{appendixA} to estimate the term with the function $\mathbf{A}_r$ on the right-hand side in the considered 
regime $1 \le n \leq j_0-1$. We get:
$$
\left|\mathbf{A}_r(-j_0+n \, |\alpha_r|,n)\right| \, \le \, C \, \exp \left( -c \, \dfrac{|j_0-n \, |\alpha_r||^{4/3}}{n^{1/3}} \right) 
\, \le \, C \, \rme^{-c \, j_0} \, ,
$$
where the final estimate comes from the fact that we now consider the regime $1 \le n \leq j_0-1$. Since now $j_0$ dominates $n$, we get the 
exponential estimate:
$$
\left| \mathbf{A}_r(-j_0+n \, |\alpha_r|,n)\right| \, \le \, C \, \rme^{-c \, n} \, \rme^{-c \, j_0} \, ,
$$
and we have thus proved the validity of the bound \eqref{borneGreen-2} in Theorem \ref{thmGreen} for the case $1 \le n \leq j_0-1$.

The proof of Theorem \ref{thmGreen} is entirely similar in the case $j_0 \le 0$ except that all involved functions are now associated with the left 
state $u_\ell$ of the shock rather than with $u_r$. We leave the details to the interested reader.

%%%%%%%%%%%%%%%%%%%%%%%%%%
\section{Derivative of the temporal Green's functions}
\label{section4-4}

As it will be made clear in the last chapter of this article, it will also be necessary to obtain large time decaying bounds for the family of operators 
$\left(\cL^n(\mathrm{Id}-\mathbf{S})\right)_{n\in\N}$ where $\mathbf{S}:\ell^q(\Z;\R)\to\ell^q(\Z;\R)$ is the shift operator defined as $(\mathbf{S}\bfh 
)_j:=h_{j+1}$ for all $j\in\Z$ for any sequence $\bfh =(h_j)_{j \in \Z} \in\ell^q(\Z;\R)$. By definition of the operator $\cL^n$, for any $\bfh\in\ell^q(\Z;\R)$, 
we have the decomposition
\bqs
\forall \, (n,j) \in \, \N\times\Z \, , \quad \left(\cL^n(\mathrm{Id}-\mathbf{S}) \bfh\right)_j= \sum_{j_0\in\Z}\left(\cG^n(j,j_0)-\cG^n(j,j_0-1)\right)h_{j_0}\,.
\eqs
This motivates the definition of the following 
quantity
\bqq
\label{GreenDeriv}
\forall \, (n,j,j_0) \in \, \N\times\Z\times\Z\,,\quad  \mathscr{D}^n(j,j_0):=\mathscr{G}^n(j,j_0)-\mathscr{G}^n(j,j_0-1).
\eqq
The above quantity is thus a discrete spatial derivative of $\mathscr{G}^n(j,j_0)$ with respect to its second argument, and we shall refer to it simply 
as the derivative of the temporal Green's function. We will now follow the same strategy as presented in the previous sections of this chapter. That 
is, we shall decompose the derivative of the temporal Green's function into several contributions and derive bounds for each such contributions which 
are meant to be sufficiently sharp in order to obtain  large time decaying bounds for the family of operators $\left(\cL^n(\mathrm{Id}-\mathbf{S}) 
\right)_{n\in\N}$.

Before stating our main result, we recall that the temporal Green's function satisfies $\mathscr{G}^n(j,j_0)=0$ whenever $|j-j_0|>n$, since the 
Lax-Wendroff scheme has a finite stencil. As an elementary consequence, we have that for all $n\in\N$ and $(j,j_0)\in\Z^2$:
\bqs
j-j_0>n \, \text{ or } \, j-j_0<-n-1 \, \Rightarrow \, \mathscr{D}^n(j,j_0)=0.
\eqs
As a consequence, throughout this section, we shall only consider $(n,j,j_0) \in \, \N\times\Z\times\Z$ that satisfies $-n-1\leq j-j_0 \leq n$. Furthermore, 
a direct application of Lemma~\ref{lem4} gives the following result.

\begin{lemma}
\label{lem5obs}
Let $j_0 \in \Z$ and $n \in \N$. If $j_0 \ge 2$ and $n \le j_0-2$, then there holds:
$$
\forall \, j \in \Z \, ,\quad \mathscr{D}^n(j,j_0) \, = \, \overline{\mathscr{G}}_r^n(j-j_0) \,-\,\overline{\mathscr{G}}_r^n(j-j_0+1) \, .
$$
If $j_0 \le 0$ and $n \le |j_0|$, then there holds:
$$
\forall \, j \in \Z \, ,\quad \mathscr{D}^n(j,j_0) \, = \, \overline{\mathscr{G}}_\ell^n(j-j_0) -\, \overline{\mathscr{G}}_\ell^n(j-j_0+1) \, .
$$
\end{lemma}

Finally, we introduce two functions $\mathbf{K}_\ell$ and $\mathbf{K}_r$ on $\R^{+*} \times \R \times \R^{+*}$ as follows:
\begin{subequations}
\label{defmajorantK}
\begin{align}
\mathbf{K}_\ell(c,x,y) \, &:= \begin{cases}
\dfrac{1}{y^{7/12}} \, \exp \big( -c \, |x|^{3/2}/y^{1/2}  \big) \, ,& \text{\rm if $x \ge 0$,} \\
 & \\
\dfrac{1}{y^{7/12}} \, ,& \text{\rm if $-y^{1/3} \le x \le 0$,} \\
 & \\
\dfrac{1}{y^{5/8}} \, \exp  \big( -c \, x^2/y  \big) \, ,& \text{\rm if $x \le -y^{1/3}$,}
\end{cases}
\label{defmajorantKl} \\
\mathbf{K}_r(c,x,y) \, &:= \, \begin{cases}
\dfrac{1}{y^{7/12}} \, \exp \big( -c \, |x|^{3/2}/y^{1/2}  \big) \, ,& \text{\rm if $x \le 0$,} \\
 & \\
\dfrac{1}{y^{7/12}} \, ,& \text{\rm if $0 \le x \le y^{1/3}$,} \\
 & \\
\dfrac{1}{y^{5/8}} \, \exp  \big( -c \, x^2/y  \big) \, ,& \text{\rm if $y^{1/3} \le x$,}
\end{cases}
\label{defmajorantKr}
\end{align}
\end{subequations}
for all $(c,x,y) \in \R^{+*} \times \R \times \R^{+*}$. Once again, we crucially note that both $\mathbf{K}_\ell$ and $\mathbf{K}_r$ are 
non-increasing with respect to their first argument.

We can now state our main result regarding the derivative of the Green's function.

\begin{theorem}[Pointwise bounds on the derivative of the Green's function]
\label{thmGreenDeriv}
Let the weak solution \eqref{shock} satisfy the Rankine-Hugoniot condition \eqref{RH} and the entropy inequalities \eqref{entropy}. Let the 
parameter $\lambda$ satisfy the CFL condition \eqref{CFL} 
and let Assumption \ref{hyp-stabspectrale} be satisfied. Then there exist some positive constants $C$ and $c$ such that for each $n\geq1$ 
and $(j,j_0)\in\Z^2$ with $-n-1\leq j-j_0\leq n$, the derivative $\mathscr{D}^n(j,j_0)$ of the Green's function enjoys the following pointwise 
bounds:
\begin{itemize}
\item for any $j_0\leq0$:
\bqq\label{estimateDj0neg}
\left|\mathscr{D}^n(j,j_0)\right| \leq C \, \mathbf{K}_\ell(c,j_0-j+n \, \alpha_\ell,n) \mathds{1}_{j\leq0} 
+ C \, \rme^{-c \, |j|} \, \mathbf{M}_\ell \left(c,j_0+\alpha_\ell n,n\right) + C \, \rme^{-cn-c|j-j_0|}\,;
\eqq
\item and for any $j_0\geq 1$:
\bqq\label{estimateDj0pos}
\left|\mathscr{D}^n(j,j_0)\right|\leq C \, \mathbf{K}_r(c,j-j_0+n \, |\alpha_r|,n) \mathds{1}_{j\geq1} 
+ C \, \rme^{-c \, |j|} \, \mathbf{M}_r\left(c,-j_0+|\alpha_r|n,n\right) + C \, \rme^{-cn-c|j-j_0|}\,.
\eqq
\end{itemize}
\end{theorem}

\begin{proof} Throughout the proof we assume that $n\geq1$ and $(j,j_0)\in\Z^2$ with $-n-1\leq j-j_0\leq n$ and we only consider $j_0\geq1$ 
since the analysis for $j_0\leq0$ follows similar lines.

We first focus on the regime $2\leq j_0 \leq n+1$. Using the expression for $\cG^n(j,j_0)$, we may instead write
\bqq
\label{expDnj0geg2}
\mathscr{D}^n(j,j_0)=\mathds{1}_{j\geq1}\left[\underbrace{\overline{\cG}_r^n(j-j_0)-\overline{\cG}_r^n(j-j_0+1)}_{:=\mathscr{K}_r^n(j-j_0)}\right] 
+ \underbrace{\frac{1}{2\pi \mathbf{i}} \int_{\Gamma} \rme^{n \, \tau} \left[\widetilde{\G}^{j_0}_j(e^{\tau})-\widetilde{\G}^{j_0-1}_j(\rme^{\tau})\right] 
\rme^{\tau} \md \tau}_{:=\widetilde{\mathscr{D}}^n(j,j_0)}.
\eqq
We first handle the second term $\widetilde{\mathscr{D}}^n(j,j_0)$. We let $\varepsilon\in(0,\varepsilon_*)$ be fixed as in previous section, 
that is satisfying conditions~\eqref{condeps1}-\eqref{condeps2}-\eqref{condeps3}, and $0<\eta<\min(\eta_\varepsilon,\varepsilon^5)$. From 
Proposition~\ref{prop3}, since the map $\tau\mapsto \widetilde{\G}^{j_0}_j(e^{\tau})-\widetilde{\G}^{j_0-1}_j(\rme^{\tau})$ has a holomorphic 
extension to $\mathbf{B}_{\varepsilon_0}(0)$, we can decompose the contour $\Gamma$ into $\Gamma_{\mathrm{out}}=\left\{-\eta+\mbi\theta 
~|~ \varepsilon\leq \theta\leq \pi\right\}$ and $\Gamma_{\mathrm{in}}$, where $\Gamma_{\mathrm{in}}$ can be any path joining $-\eta-\mbi \, 
\varepsilon$ to $-\eta+\mbi \, \varepsilon$ which remains in $\overline{\mathbf{B}_\varepsilon(0)}$ (see Figure \ref{fig:contourGammainout}). 
As a consequence, we have
\bqs
\left|\frac{1}{2\pi \mathbf{i}} \int_{\Gamma_{\mathrm{out}}} \rme^{n \, \tau} \left[\widetilde{\G}^{j_0}_j(e^{\tau})-\widetilde{\G}^{j_0-1}_j(\rme^{\tau})\right] 
\rme^{\tau} \md \tau\right| \leq C \rme^{-\eta n}e^{-c \, |j|-c|j_0|},
\eqs
together with
\begin{align*}
\frac{1}{2\pi \mathbf{i}} \int_{\Gamma_{\mathrm{in}}} \rme^{n \, \tau} 
\left[\widetilde{\G}^{j_0}_j(e^{\tau})-\widetilde{\G}^{j_0-1}_j(\rme^{\tau})\right]\rme^{\tau} \md \tau &=  \dfrac{\mathcal{H}_j}{\alpha_r} \times 
\frac{1}{2\pi \mathbf{i}} \int_{\Gamma_{\mathrm{in}}} \rme^{n \, \tau  - \, j_0 \, \varphi_r(\tau) \, + \, j_0 \, \tau^5 \, \Psi_r(\tau)} \md \tau \\
&~~~+\frac{1}{2\pi \mathbf{i}} \int_{\Gamma_{\mathrm{in}}} \rme^{n \, \tau  - \, j_0 \, \varphi_r(\tau) \, + \, j_0 \, \tau^5 \, \Psi_r(\tau)} \tau \, 
\Theta_{r,j}(\tau) \md \tau \,,
\end{align*}
where the sequence $(\mathcal{H}_j)_{j\in\Z}$ is defined in \eqref{defH} and the sequence $(\Theta_{r,j})_{j\in\Z}$ of bounded of holomorphic 
functions is given by Proposition~\ref{prop3}.  The first integral is similar to $\widetilde{\cG}_{2,r}^n(j,j_0)$ (with definition \eqref{defG2rnj0}) 
and enjoys a similar bound, while the second integral is similar to $\widetilde{\mathscr{R}}_{1,r}^n(j,j_0)$ (with definition \eqref{defR1rnj0}) 
and enjoys a similar bound. As a consequence, for $2\leq j_0 \leq n+1$, we have the estimate
\bqq
\label{intermestimDtilde}
\left|\widetilde{\mathscr{D}}^n(j,j_0)\right|\leq C\, \rme^{-c \, |j|}\mathbf{M}_r\left(c,-j_0+|\alpha_r|n,n\right)+C\, \rme^{-cn-c \, j_0-c \, |j|}\,, 
\eqq
for some $C>0$ and $c>0$ which are independent of $n,j$ and $j_0$.

Let us focus now on the first term $\mathscr{K}_r^n(j-j_0)$ of \eqref{expDnj0geg2} which can also be written as
\bqs
\mathscr{K}_r^n(j-j_0)=\frac{1}{2\pi \mathbf{i}} \int_{\Gamma} \rme^{n \, \tau} \left[\overline{\G}_{r,j-j_0}(e^{\tau})-\overline{\G}_{r,j-j_0+1}(\rme^{\tau})\right]\rme^{\tau} \md \tau.
\eqs
From the expression \eqref{defGrz} of Proposition~\ref{prop2} of the free spatial Green's function there exists $\kappa_\mathscr{O}>0$ depending 
on the set $\mathscr{O}$ defined in Chapter \ref{chapter3}, such that for $j_0\leq j$ one has 
\bqs
\left| \left[\overline{\G}_{r,j-j_0}(e^{\tau})-\overline{\G}_{r,j-j_0+1}(\rme^{\tau})\right]\rme^{\tau} \right| \leq C \, \rme^{-c|j-j_0|},
\eqs
for all $\tau\in\C$ with $\Re(\tau)\geq -\kappa_\mathscr{O}$. As a consequence, upon taking $\Gamma=\left\{-\eta+\mbi\theta~|~ -\pi \le 
\theta \le \pi\right\}$ with $0<\eta<\kappa_\mathscr{O}$, we readily get in that case that
\bqs
\left| \mathscr{K}_r^n(j-j_0) \right| \leq C \, \rme^{-\eta n -c|j-j_0|}.
\eqs
On the other hand, for $1\leq j \leq j_0 \leq n+1$, one has from the expression \eqref{defGrz} that for all $\tau\in\overline{\mathbf{B}_{\varepsilon_0}(0)}$
\bqs
\left[\overline{\G}_{r,j-j_0}(e^{\tau})-\overline{\G}_{r,j-j_0+1}(\rme^{\tau})\right]\rme^{\tau}=\tau \Xi_r(\tau) \exp\left((j-j_0)\varphi_r(\tau)-(j-j_0)\tau^5\Psi_r(\tau)\right),
\eqs
for some bounded holomorphic function $\Xi_r$, with the same $\varepsilon_0$ as the one given by Proposition~\ref{prop3}. We can then 
once again let $\varepsilon\in(0,\varepsilon_*)$ be fixed as in the previous section and set $0<\eta<\min(\eta_\varepsilon,\varepsilon^5)$. 
Using Corollary~\ref{cor2}, we also have
\bqs
\left| \left[\overline{\G}_{r,j-j_0}(e^{\tau})-\overline{\G}_{r,j-j_0+1}(\rme^{\tau})\right]\rme^{\tau} \right| \leq C \, \rme^{-c|j-j_0|},
\eqs
for all $\tau=-\eta+\mbi\theta$ with $\varepsilon\leq \theta \leq \pi$. As a consequence, with our now usual notation:
$$
\Gamma_{\mathrm{out}} =\left\{-\eta+\mbi\theta ~|~ \varepsilon \le \theta \le \pi\right\} \, ,
$$
we get
\bqs
\left| \frac{1}{2\pi \mathbf{i}} \int_{\Gamma_{\mathrm{out}}} \rme^{n \, \tau} \left[\overline{\G}_{r,j-j_0}(e^{\tau})-\overline{\G}_{r,j-j_0+1}(\rme^{\tau}) 
\right] \rme^{\tau} \right| \md\tau \leq C \, \rme^{-\eta n -c|j-j_0|}.
\eqs
On the other hand, if $\Gamma_{\mathrm{in}}$ denotes any path joining $-\eta-\mbi \, \varepsilon$ to $-\eta+\mbi \, \varepsilon$ which remains 
in $\overline{\mathbf{B}_\varepsilon(0)}$, one has
\bqs
\frac{1}{2\pi \mathbf{i}} \int_{\Gamma_{\mathrm{in}}} \rme^{n \, \tau} 
\left[\overline{\G}_{r,j-j_0}(e^{\tau})-\overline{\G}_{r,j-j_0-1}(\rme^{\tau})\right]\rme^{\tau}  \md\tau 
=\frac{1}{2\pi \mathbf{i}} \int_{\Gamma_{\mathrm{in}}}\rme^{n \, \tau+(j-j_0)\varphi_r(\tau)-(j-j_0)\tau^5\Psi_r(\tau)}\tau \Xi_r(\tau)\md \tau.
\eqs
Then applying Lemma~\ref{lemestimR1}, Lemma~\ref{lemestimR2}, Lemma~\ref{lemestimR3} and Lemma~\ref{lemestimR4} with $\Phi_{r,j}$ 
replaced by $\Xi_r(\tau)$ and $-j_0$ by $j-j_0$, we readily obtain that
\bqs
\left| \frac{1}{2\pi \mathbf{i}} \int_{\Gamma_{\mathrm{in}}} \rme^{n \, \tau} 
\left[\overline{\G}_{r,j-j_0}(e^{\tau})-\overline{\G}_{r,j-j_0+1}(\rme^{\tau})\right]\rme^{\tau}  \md\tau \right| \leq C \, \mathbf{K}_r(c,j-j_0+n \, |\alpha_r|,n)\,,
\eqs
where the definition of  $\mathbf{K}_r$ is given in \eqref{defmajorantKr}. Combining the above bound with the estimate~\eqref{intermestimDtilde} 
proves the bound~\eqref{estimateDj0pos} of $\mathscr{D}^n(j,j_0)$ in the range $2\leq j_0 \leq n+1$.

For $j_0\geq2$ and $n\leq j_0-2$, we use Lemma~\ref{lem5obs} which gives that
$$
\forall \, j \in \Z \, ,\quad \mathscr{D}^n(j,j_0) \, = \, \overline{\mathscr{G}}_r^n(j-j_0) \,-\,\overline{\mathscr{G}}_r^n(j-j_0+1) \, .
$$
We can then proceed along similar lines as above, and get that
\bqs
\left|\mathscr{D}^n(j,j_0) \right| \leq C \, \mathbf{K}_r(c,j-j_0+n \, |\alpha_r|,n)+C \, \rme^{-\eta n -c|j-j_0|}\,.
\eqs

Let us now turn to the case $j_0=1$. Using the expressions of $\cG^n(j,1)$ and $\cG^n(j,0)$, we have
\bqs
\mathscr{D}^n(j,1)=\mathds{1}_{j\geq1}\overline{\cG}_r^n(j-1)-\mathds{1}_{j\leq0}\,\overline{\cG}_\ell^n(j) 
+\underbrace{\frac{1}{2\pi \mathbf{i}} \int_{\Gamma} \rme^{n \, \tau} \left[\widetilde{\G}^{1}_j(e^{\tau})-\widetilde{\G}^{0}_j(\rme^{\tau})\right] 
\rme^{\tau} \md \tau}_{:=\widetilde{\mathscr{D}}^n(j,1)}\,.
\eqs
First, inspecting the contribution stemming from $\widetilde{\mathscr{D}}^n(j,1)$ and decomposing $\Gamma$ with $\Gamma_{\mathrm{out}}$ 
and $\Gamma_{\mathrm{in}}$, we observe that from \eqref{decomposition-derivative} of Proposition~\ref{prop3}, we have:
\bqs
\left|\frac{1}{2\pi \mathbf{i}} \int_{\Gamma_{\mathrm{out}}} \rme^{n \, \tau} \left[\widetilde{\G}^{1}_j(e^{\tau})-\widetilde{\G}^{0}_j(\rme^{\tau})\right]\rme^{\tau} \md \tau\right| \leq C \rme^{-\eta n}e^{-c \, |j|},
\eqs
together with
\begin{align*}
\frac{1}{2\pi \mathbf{i}} \int_{\Gamma_{\mathrm{in}}} \rme^{n \, \tau} \left[\widetilde{\G}^{1}_j(e^{\tau})-\widetilde{\G}^{0}_j(\rme^{\tau})\right]\rme^{\tau} \md \tau &=  \left(\dfrac{\mathcal{H}_j}{\alpha_r} +\gamma_j^r-\gamma_j^\ell\right) \times \frac{1}{2\pi \mathbf{i}} \int_{\Gamma_{\mathrm{in}}} \rme^{n \, \tau} \md \tau +\frac{1}{2\pi \mathbf{i}} \int_{\Gamma_{\mathrm{in}}} \rme^{n \, \tau} \tau \, \Theta_{1,j}(\tau) \md \tau \,.
\end{align*}
As a consequence, upon taking $\Gamma_{\mathrm{in}}=\left\{-\eta+\mbi\theta~|~ |\theta|\leq \varepsilon\right\}$, we get
\bqs
\left|\widetilde{\mathscr{D}}^n(j,1)\right|\leq C \rme^{-cn-c \, |j|}.
\eqs
Upon noticing that $\mathbf{M}_r\left(c,-1+|\alpha_r|n,n\right)=\mathcal{O}\left(\rme^{-c \, n}\right)$, estimate \eqref{intermestimDtilde} also holds 
true for $j_0=1$. Coming back to the first two terms in the expression of $\mathscr{D}^n(j,1)$, we simply note, using Corollary~\ref{coro-A1}, that
\bqs
\forall j \geq1\,, \quad \, \left|\overline{\cG}_r^n(j-1)\right| \leq C \rme^{-c \, n}, \text{ and } \quad 
\forall j\leq0\,, \quad \, \left|\overline{\cG}_\ell^n(j)\right| \leq C \rme^{-c \, n}.
\eqs
But since one has $-n\leq j\leq n+1$, the above exponential bound in $n$ leads to an exponential bound in both $j$ and $n$:
\bqs
\left|\mathds{1}_{j\geq1}\overline{\cG}_r^n(j-1)-\mathds{1}_{j\leq0}\,\overline{\cG}_\ell^n(j)\right| \leq C \, \rme^{-cn-c \, |j|}.
\eqs
To conclude, we note that $\mathbf{K}_r(c,j-1+n \, |\alpha_r|,n)=\mathcal{O}\left(\rme^{-cn-c \, |j|}\right)$ which implies that estimate \eqref{estimateDj0pos} 
also holds true for $j_0=1$. 
\end{proof}

%%%%%%%%%%%%%
\section{Linear estimates}
\label{section4-5}

In this final section, we derive large time decaying bounds for the semigroup $(\cL^n)_{n\in\N}$ and the family of operators $(\cL^n(\mathrm{Id}-\mathbf{S}) 
)_{n\in\N}$ acting on algebraically weighted spaces which are crucial for our forthcoming nonlinear stability analysis.

We first recall some notation for the polynomially weighted spaces of sequences. Given a real number $\gamma \geq 0$ and $q\in[1,+\infty]$, if we 
define the weight sequence $\boldsymbol{\omega}_\gamma:=(1+|j|^\gamma)_{j\in\Z}$, we recall our definition of algebraically weighted $\ell^q$ spaces:
\bqs
\ell^q_\gamma(\Z;\R)=\left\{ \mathbf{h} \in \ell^q(\Z;\R) ~|~  \boldsymbol{\omega}_\gamma \mathbf{h} \in\ell^q(\Z;\R)\right\},
\eqs
where $\boldsymbol{\omega}_\gamma \mathbf{h}$ stands for the sequence $((1+|j|^\gamma) \, h_j)_{j\in\Z}$. For any sequence $\mathbf{h} \in 
\ell^q_\gamma(\Z;\R)$, the norm of $\bfh$ is defined as $\| \bfh\|_{\ell^q_\gamma}:=\left\| \boldsymbol{\omega}_\gamma \bfh\right\|_{\ell^q}$.

The main result of this Chapter reads as follows (the reader will observe that this statement is a refinement of Theorem \ref{thmLineaire} that was 
made more simple for the reader's convenience).

\begin{theorem}
\label{thmlinestim}
Let the weak solution \eqref{shock} satisfy the Rankine-Hugoniot relation \eqref{RH} and the entropy inequalities \eqref{entropy}. Let the parameter 
$\lambda$ satisfy the CFL condition \eqref{CFL}  and let Assumption \ref{hyp-stabspectrale} be satisfied.  For any $\gamma_2 \ge \gamma_1 \ge 0$, 
there exists $C_\cL(\gamma_1,\gamma_2)>0$ such that we have the following estimates on the semigroup $(\cL^n)_{n\in\N}$:
\begin{subequations}
\begin{align}
\forall n \in\N \,, \quad \left\|\cL^n \, \bfh \right\|_{\ell^1_{\gamma_1}} \, &\leq \, \frac{C_\cL(\gamma_1,\gamma_2)}{(1+n)^{\gamma_2-\gamma_1-1/8}}\, 
\|\bfh\|_{\ell^1_{\gamma_2}}\,,\quad &\text{ for } \bfh\in\ell^1_{\gamma_2}(\Z;\R)\text{ with } \sum_{j\in\Z}h_j=0\, ,\label{estimLn1} \\
\forall n \in\N \,, \quad \left\|\cL^n \, \bfh \right\|_{\ell^\infty_{\gamma_1}} \,& \leq \, 
\frac{C_\cL(\gamma_1,\gamma_2)}{(1+n)^{\gamma_2-\gamma_1+\min\left(1/3,\gamma_1\right)}} \, \| \bfh\|_{\ell^1_{\gamma_2}}\,,\quad 
&\text{ for } \bfh\in\ell^1_{\gamma_2}(\Z;\R)\text{ with } \sum_{j\in\Z}h_j=0\, ,\label{estimLn2}
\end{align}
\end{subequations}
and the following estimates on the family of operators $(\cL^n(\mathrm{Id}-\mathbf{S}))_{n\in\N}$:
\begin{subequations}
\begin{align}
\forall n \in\N \,, \quad\left\|\cL^n(\mathrm{Id}-\mathbf{S}) \bfh \right\|_{\ell^1_{\gamma_1}} & \leq 
\frac{C_\cL(\gamma_1,\gamma_2)}{(1+n)^{\gamma_2-\gamma_1+1/8}}\,\|\bfh\|_{\ell^1_{\gamma_2}}\, ,\quad 
&\text{ for } \bfh\in\ell^1_{\gamma_2}(\Z;\R)\,,  \label{estimLnshift1} \\
\forall n \in\N \,, \quad\left\|\cL^n(\mathrm{Id}-\mathbf{S}) \bfh \right\|_{\ell^\infty_{\gamma_1}}& \leq 
\frac{C_\cL(\gamma_1,\gamma_2)}{(1+n)^{\gamma_2-\gamma_1+1/3+\min\left(1/4,\gamma_1\right)}}\, \|\bfh\|_{\ell^1_{\gamma_2}}\, ,\quad 
&\text{ for } \bfh\in\ell^1_{\gamma_2}(\Z;\R)\,,\label{estimLnshift2} \\
\forall n \in\N \,, \quad\left\|\cL^n(\mathrm{Id}-\mathbf{S}) \bfh \right\|_{\ell^\infty_{\gamma_1}}& \leq 
\frac{C_\cL(\gamma_1,\gamma_2)}{(1+n)^{\gamma_2-\gamma_1+\min\left(1/8,\gamma_1-1/8\right)}}\,\|\bfh\|_{\ell^\infty_{\gamma_2}}\, ,\quad 
&\text{ for } \bfh\in\ell^\infty_{\gamma_2}(\Z;\R)\, .\label{estimLnshift3}
\end{align}
\end{subequations}
\end{theorem}

\subsection{Proof of the estimates \eqref{estimLn1} and \eqref{estimLn2} on the semigroup $(\cL^n)_{n\in\N}$}

Before proceeding with the proof of the estimates \eqref{estimLn1} and \eqref{estimLn2} on the semigroup $(\cL^n)_{n\in\N}$ of 
Theorem~\ref{thmlinestim}, let us comment on the strategy that we shall follow. For $\gamma\geq0$ and $\bfh\in \ell^q_\gamma(\Z;\R)$ 
we recall that the action of the semigroup $\cL^n$ on $\bfh$ is given for each $n\in\N$ by
\bqs
\forall j \in\Z \, , \quad \left(\cL^n \mathbf{h}\right)_j= \sum_{j_0\in\Z}\cG^n(j,j_0)h_{j_0}\,,
\eqs
where $\cG^n(\cdot,j_0)$ is the temporal Green's function solution of \eqref{defgreentemporelle}. The bounds \eqref{estimLn1}-\eqref{estimLn2} 
are trivial for $n=0$ since we have $\gamma_1 \le \gamma_2$ and therefore $\|\bfh\|_{\ell^1_{\gamma_1}} \le \|\bfh\|_{\ell^1_{\gamma_2}}$ 
so we assume from now on $n \in \N^*$. Motivated by the estimates obtained on the temporal Green's function in Theorem~\ref{thmGreen} 
(for $n \in \N^*$), we introduce a family of operators $(\cL^n_{\mathrm{act}})_{n\in\N^*}$ acting on a given sequence $\bfh\in \ell^q_\gamma 
(\Z;\R)$ as follows:
\begin{align*}
\forall (n,j) \in \N^*\times \Z\,,\quad \left(\cL^n_{\mathrm{act}} \, \bfh \right)_j := \, 
& \, \mathcal{H}_j \, \sum_{j_0\geq1}\mathds{1}_{|j-j_0|\leq n} \, \mathbf{A}_r(-j_0+n \, |\alpha_r|,n) \, h_{j_0} \\
&+\mathcal{H}_j \, \sum_{j_0\leq0} \mathds{1}_{|j-j_0|\leq n} \, \mathbf{A}_\ell(j_0+n \, \alpha_\ell,n) \, h_{j_0} \, .
\end{align*}
It is important to observe that $\cL^n_{\mathrm{act}}$ does not stand for the $n$-th power of $\cL^1_{\mathrm{act}}$, while $\cL^n$ is the 
$n$-th power of $\cL$. However, we hope that this will not create any confusion for the reader and we keep the notation as such to bear in 
mind that all operators are considered at the discrete time $n$.

We recall that the Green's function satisfies $\cG^n(j,j_0)=0$ for $|j-j_0|>n$. With the previous definition for $\cL^n_{\mathrm{act}}$, for all 
$(n,j) \in \N^*\times \Z$, one can therefore decompose
\begin{align}
\left(\cL^n \mathbf{h}\right)_j=\left(\cL^n_{\mathrm{act}} \bfh\right)_j &+\sum_{j_0\geq1} \left(\cG^n(j,j_0)-\mathds{1}_{|j-j_0|\leq n} \, \mathcal{H}_j 
\, \mathbf{A}_r(-j_0+n \, |\alpha_r|,n)\right) \, h_{j_0} \notag \\
&+ \sum_{j_0\leq0} \left(\cG^n(j,j_0)-\mathds{1}_{|j-j_0|\leq n} \, \mathcal{H}_j \, \mathbf{A}_\ell(j_0+n \, \alpha_\ell,n)\right) \, h_{j_0} \, .
\label{decompositionlin-n}
\end{align}
We now introduce some notation. For $(n,j) \in \N^* \times \Z$ and $\bfh \in \ell^1_{\gamma_2}(\Z;\R)$, we define:
\begin{align*}
\mathfrak{S}_{r,1}(n,j,\bfh) \, & \, := \, \mathds{1}_{j\ge1} \, \sum_{j_0\ge1} \, \mathds{1}_{|j-j_0|\leq n} \, 
\mathbf{M}_r \left( c,j-j_0+n \, |\alpha_r|,n \right) \, \left|h_{j_0}\right| \, , \\ 
\mathfrak{S}_{\ell,1}(n,j,\bfh) \, & \, := \, \mathds{1}_{j\le0} \, \sum_{j_0\le0} \, \mathds{1}_{|j-j_0|\leq n} \, 
\mathbf{M}_\ell \left( c,j_0-j+n \, \alpha_\ell,n\right) \, \left|h_{j_0}\right| \, ,
\end{align*}
and
\begin{align*}
\mathfrak{S}_{r,2}(n,j,\bfh) \, & \, := \, \rme^{-c \, |j|} \, \sum_{j_0\ge1} \, \mathbf{M}_r \left( c,-j_0+n \, |\alpha_r|,n \right) \, \left|h_{j_0}\right| \, , \\
\mathfrak{S}_{\ell,2}(n,j,\bfh) \, & \, := \, \rme^{-c \, |j|} \, \sum_{j_0\le0} \, \mathbf{M}_\ell \left( c,j_0+n \, \alpha_\ell,n \right) \, \left|h_{j_0}\right| \, ,
\end{align*}
where we recall that the functions $\mathbf{M}_\ell$ and $\mathbf{M}_r$ are defined in \eqref{defmajorant}. We also define:
\bqs
\mathfrak{S}_{\mathrm{exp}}(n,j,\bfh) \, := \, \rme^{-c \, n-c \, |j|} \, \sum_{j_0\in\Z} \, \rme^{-c \, |j_0|} \, \left| h_{j_0} \right| \, .
\eqs

Using the estimates derived in Theorem~\ref{thmGreen} and applying the triangle inequality, we remark that for some suitable positive constants 
$C$ and $c$ that do not depend on $n$, $j$ nor $\bfh$, we have
\begin{multline*}
\left| \sum_{j_0\ge1} \left( \cG^n(j,j_0)-\mathds{1}_{|j-j_0|\leq n} \, \mathcal{H}_j \, \mathbf{A}_r(-j_0+n \, |\alpha_r|,n) \right) \, h_{j_0} \right| \\
\le \, C \, \mathfrak{S}_{r,1}(n,j,\bfh) \, + \, C \, \mathfrak{S}_{r,2}(n,j,\bfh) \, + C \, \mathfrak{S}_{\mathrm{exp}}(n,j,\bfh) \, ,
\end{multline*}
together with
\begin{multline*}
\left| \sum_{j_0\le0} \left( \cG^n(j,j_0)-\mathds{1}_{|j-j_0|\leq n} \, \mathcal{H}_j \, \mathbf{A}_\ell(j_0+n \, \alpha_\ell,n) \right) \, h_{j_0} \right| \\
\le \, C \, \mathfrak{S}_{\ell,1}(n,j,\bfh) \, + \, C \, \mathfrak{S}_{\ell,2}(n,j,\bfh) \, + C \, \mathfrak{S}_{\mathrm{exp}}(n,j,\bfh) \, .
\end{multline*}
Combining those two inequalities with the triangle inequality in \eqref{decompositionlin-n}, we have for all $(n,j) \in \N^*\times \Z$:
\begin{equation}
\label{estimationlin-n}
\left| \left( \cL^n \bfh \right)_j \right| \le \left| \left( \cL^n_{\mathrm{act}} \bfh \right)_j \right| \, + \, C \, 
\Big( \mathfrak{S}_{r,1}(n,j,\bfh) + \mathfrak{S}_{\ell,1}(n,j,\bfh) + \mathfrak{S}_{r,2}(n,j,\bfh) + \mathfrak{S}_{\ell,2}(n,j,\bfh) 
+\mathfrak{S}_{\mathrm{exp}}(n,j,\bfh) \Big) \, .
\end{equation}
In order to prove the estimates \eqref{estimLn1} and \eqref{estimLn2}, one simply needs to prove estimates for each of the terms appearing in 
the right-hand side of the above inequality. From now on, we shall proceed term by term.

We start by estimating $\left(\mathfrak{S}_{r,1}(n,j,\bfh)\right)_{j\in\Z}$ and $\left(\mathfrak{S}_{\ell,1}(n,j,\bfh)\right)_{j\in\Z}$.

\begin{lemma}
\label{lemSrl1}
For any $\gamma_2 \ge \gamma_1 \ge 0$, there exists $C>0$, such that for all $n\geq1$ and $\bfh \in \ell^1_{\gamma_2}(\Z;\R)$, one has:
\begin{align*}
\left\| \left( \mathfrak{S}_{r,1}(n,j,\bfh) \right)_{j\in\Z} \right\|_{\ell^1_{\gamma_1}} + \left\| \left( \mathfrak{S}_{\ell,1}(n,j,\bfh) \right)_{j\in\Z} 
\right\|_{\ell^1_{\gamma_1}} & \, \leq \, \dfrac{C}{n^{\gamma_2-\gamma_1-1/8}} \, \| \bfh \|_{\ell^1_{\gamma_2}} \, ,\\
\left\| \left( \mathfrak{S}_{r,1}(n,j,\bfh) \right)_{j\in\Z} \right\|_{\ell^\infty_{\gamma_1}} + \left\| \left( \mathfrak{S}_{\ell,1}(n,j,\bfh)\right)_{j\in\Z} 
\right\|_{\ell^\infty_{\gamma_1}} & \, \leq \, \dfrac{C}{n^{\gamma_2-\gamma_1+1/3}} \, \| \bfh \|_{\ell^1_{\gamma_2}} \, .
\end{align*}
\end{lemma}

\begin{proof}
We only prove the estimates for $\left(\mathfrak{S}_{r,1}(n,j,\bfh)\right)_{j\in\Z}$, the other case is handled similarly. For $\bfh \in \ell^1_{\gamma_2}(\Z;\R)$ 
and $n \in \N^*$, we have that
\begin{align*}
\left\| \left( \mathfrak{S}_{r,1}(n,j,\bfh) \right)_{j\in\Z} \right\|_{\ell^1_{\gamma_1}} = \, & \, \sum_{j\in\Z} \, 
(1+|j|^{\gamma_1}) \, \mathds{1}_{j\geq1} \, \left( \sum_{j_0\geq1} \, \mathds{1}_{|j-j_0|\leq n} \, 
\mathbf{M}_r \left( c,j-j_0+n \, |\alpha_r|,n \right) \, \left| h_{j_0} \right| \right) \\
= \, & \, \underbrace{\sum_{j\ge 1} \, (1+|j|^{\gamma_1}) \, \left( \sum_{j_0=1}^j \, \mathds{1}_{|j-j_0|\leq n} \, 
\mathbf{M}_r \left( c,j-j_0+n \, |\alpha_r|,n \right) \left| h_{j_0} \right| \right)}_{=: \mathscr{I}^n_1} \\
&\, +\underbrace{\sum_{j\ge1} \, (1+|j|^{\gamma_1}) \, \left( \sum_{j_0 \geq j+1} \, \mathds{1}_{|j-j_0|\leq n} \, 
\mathbf{M}_r \left( c,j-j_0+n \, |\alpha_r|,n \right) \, \left| h_{j_0} \right| \right)}_{=:\mathscr{I}^n_2} \, .
\end{align*}
From the definition \eqref{defmajorantr} of $\mathbf{M}_r$, we note that for $1\leq j_0 \leq j$ with $|j-j_0|\leq n$, one has
\bqs
\mathbf{M}_r\left(c,j-j_0+n \, |\alpha_r|,n\right)\leq C \rme^{-c \, n} \, .
\eqs
Furthermore, Peetre's inequality implies that for $|j-j_0|\leq n$, there holds:
$$
(1+|j|^{\gamma_1}) \, \le \, C \, (1+|j_0|^{\gamma_1}) \, (1+n^{\gamma_1}) \, .
$$
We thus obtain the following estimate for the first contribution $\mathscr{I}^n_1$:
\begin{align*}
\mathscr{I}^n_1 & \le \, C \, (1+n^{\gamma_1}) \, \rme^{-c \, n} \, 
\sum_{j\ge1} \, \sum_{j_0=1}^j \, \mathds{1}_{|j-j_0|\leq n} \, (1+|j_0|)^{\gamma_1} \, |h_{j_0}| \\
&= \, C \, (1+n^{\gamma_1}) \, \rme^{-c \, n} \, \sum_{j_0\ge1} \, \left( \sum_{j\geq j_0} \, \mathds{1}_{|j-j_0|\leq n} \right) \, (1+|j_0|^{\gamma_1}) \, |h_{j_0}| \\
&\le \, C \, n \, (1+n^{\gamma_1}) \, \rme^{-c \, n} \, \|\bfh\|_{\ell^1_{\gamma_1}} \, \le \, C \, \rme^{-c \, n} \, \|\bfh\|_{\ell^1_{\gamma_2}} \, .
\end{align*}

On the other hand, for the second contribution $\mathscr{I}^n_2$, we decompose the sum into two parts:
\begin{align*}
\mathscr{I}^n_2 = \, & \, \sum_{j\ge1} \, (1+|j|^{\gamma_1}) \, \left( \sum_{j_0\geq j+1} \, \mathds{1}_{j_0< \frac{n \, |\alpha_r|}{2}} \, 
\mathds{1}_{|j-j_0|\leq n} \, \mathbf{M}_r \left( c,j-j_0+n \, |\alpha_r|,n \right) \, \left| h_{j_0} \right| \right) \\
&\, + \sum_{j\ge1} \, (1+|j|^{\gamma_1}) \, \left( \sum_{j_0\geq j+1} \, \mathds{1}_{j_0 \geq \frac{n \, |\alpha_r|}{2}} \, 
\mathds{1}_{|j-j_0|\leq n} \, \mathbf{M}_r \left( c,j-j_0+n \, |\alpha_r|,n \right) \left| h_{j_0} \right| \right) \, .
\end{align*}
In both sums, since we have $1 \le j \le j_0$, we can always use the inequality:
$$
(1+|j|^{\gamma_1}) \, \le \, (1+|j_0|^{\gamma_1}) \, .
$$
Furthermore, we notice once again that for $j\ge1$ and $j_0\geq j+1$ with $1\leq j_0 < \frac{n \, |\alpha_r|}{2}$, one has an exponential bound:
\bqs
\mathbf{M}_r\left(c,j-j_0+n \, |\alpha_r|,n\right) \, \leq \, C \, \rme^{-c \, n} \, .
\eqs
As a consequence, one gets
\begin{align*}
\sum_{j\ge1} \, (1&+|j|^{\gamma_1}) \, \left( \sum_{j_0\geq j+1} \, \mathds{1}_{ j_0< \frac{n \, |\alpha_r|}{2}} \, \mathds{1}_{|j-j_0|\leq n} \, 
\mathbf{M}_r \left( c,j-j_0+n \, |\alpha_r|,n \right) \, \left| h_{j_0} \right| \right) \\
&\le \, C \, \rme^{-c \, n} \, \sum_{j\ge1} \, \sum_{j_0\geq j+1} \, \mathds{1}_{ j_0< \frac{n \, |\alpha_r|}{2}} \, \mathds{1}_{|j-j_0|\leq n} \, 
(1+|j_0|^{\gamma_1}) \, |h_{j_0}| \\
&\le \, C \, \rme^{-c \, n} \, \sum_{j_0\in \Z} \, \left(\sum_{j=1}^{j_0} \mathds{1}_{|j-j_0|\leq n} \right) \, (1+|j_0|^{\gamma_1}) \, |h_{j_0}| 
\, \le \, C \, n \, \rme^{-c \, n} \, \|\bfh\|_{\ell^1_{\gamma_1}} \, \le \, C \, \rme^{-c \, n} \, \|\bfh\|_{\ell^1_{\gamma_2}} \, .
\end{align*}
For the remaining contribution, we have
\begin{align*}
\sum_{j\ge1} (1+|j|^{\gamma_1}) & \, \left( \sum_{j_0\geq j+1} \, \mathds{1}_{ j_0 \geq \frac{n \, |\alpha_r|}{2}} \, \mathds{1}_{|j-j_0|\leq n} 
\, \mathbf{M}_r \left( c,j-j_0+n \, |\alpha_r|,n \right) \, \left| h_{j_0} \right| \right) \\
&\le \, \sum_{j_0\geq1} \, \mathds{1}_{ j_0\geq \frac{n \, |\alpha_r|}{2}} \, \left( \sum_{j=1}^{j_0} \mathbf{M}_r \left( c,j-j_0+n \, |\alpha_r|,n \right) \right) 
\, (1+|j_0|^{\gamma_1}) \, |h_{j_0}| \\
&\le \, C \, \|\mathbf{M}_r\left(c,\cdot+n \, |\alpha_r|,n\right)\|_{\ell^1(\Z)} \, \sum_{j_0 \geq 1} \, \mathds{1}_{ j_0 \geq \frac{n \, |\alpha_r|}{2}} 
\, \dfrac{(1+|j_0|^{\gamma_2})}{(1+|j_0|^{\gamma_2-\gamma_1})}|h_{j_0}| \\
&\le \, \frac{C}{n^{\gamma_2-\gamma_1}} \, \|\mathbf{M}_r\left(c,\cdot+n \, |\alpha_r|,n\right)\|_{\ell^1(\Z)} \, \| \bfh \|_{\ell^1_{\gamma_2}} \, .
\end{align*}
In order to conclude, we just use the following estimate that can be obtained from the definition \eqref{defmajorantr} and a mere comparison 
between a series and an integral:
\begin{equation}
\label{estiml1Mr}
\forall \, n \in \N^* \, ,\quad 
\left\| \left( \mathbf{M}_r \left( c,j+n \, |\alpha_r|,n \right) \right)_{j\in\Z} \right\|_{\ell^1(\Z)} \, \le \, C \, n^{1/8} \, ,
\end{equation}
where, of course, all constants are uniform with respect to $n \in \N^*$. The proof of the $\ell^1_{\gamma_1}$ estimate for $\mathfrak{S}_{r,1}(n,\cdot,\bfh)$ 
is now complete.

We now turn our attention to the second estimate in $\ell^\infty_{\gamma_1}$. By definition of the norm in $\ell^\infty_{\gamma_1}$, we have
\begin{align*}
\left\| \left( \mathfrak{S}_{r,1}(n,j,\bfh) \right)_{j\in\Z} \right\|_{\ell^\infty_{\gamma_1}} = \, & \, \sup_{j\in\Z} \, (1+|j|^{\gamma_1}) \, 
\mathds{1}_{j\geq1} \, \left( \sum_{j_0\geq1} \, \mathds{1}_{|j-j_0|\leq n} \, \mathbf{M}_r \left( c,j-j_0+n \, |\alpha_r|,n \right) \, \left| h_{j_0} \right| \right) \\
\le \, & \, \sup_{j\ge1} \, (1+|j|^{\gamma_1}) \, \left( \sum_{j_0=1}^j \, \mathds{1}_{|j-j_0|\leq n} \, 
\mathbf{M}_r \left( c,j-j_0+n \, |\alpha_r|,n \right) \, \left| h_{j_0} \right| \right) \\
&+ \sup_{j\ge1} \, (1+|j|^{\gamma_1}) \, \left( \sum_{j_0\geq j+1} \, \mathds{1}_{j_0< \frac{n \, |\alpha_r|}{2}} \, \mathds{1}_{|j-j_0|\leq n} \, 
\mathbf{M}_r \left( c,j-j_0+n \, |\alpha_r|,n \right) \, \left| h_{j_0} \right| \right) \\
&+ \sup_{j\ge1} \, (1+|j|^{\gamma_1}) \, \left( \sum_{j_0\geq j+1} \, \mathds{1}_{j_0\geq \frac{n \, |\alpha_r|}{2}} \, \mathds{1}_{|j-j_0|\leq n} \, 
\mathbf{M}_r \left( c,j-j_0+n \, |\alpha_r|,n \right) \, \left| h_{j_0} \right| \right) \\
&=:\mathscr{J}^n_1+\mathscr{J}^n_2+\mathscr{J}^n_3 \, .
\end{align*}
Performing similar computations and estimates as above, one gets first:
\bqs
\mathscr{J}^n_1 \, + \, \mathscr{J}^n_2 \, \le \, C \, \rme^{-c \, n} \, \|\bfh\|_{\ell^1_{\gamma_2}} \, ,
\eqs
which is an even better estimate than the one we are aiming at. On the other hand, we have
\begin{align*}
\mathscr{J}^n_3 \, &\le \, C \, \|\mathbf{M}_r\left(c,\cdot+n \, |\alpha_r|,n\right)\|_{\ell^\infty(\Z)} \, 
\sum_{j_0 \geq \frac{n \, |\alpha_r|}{2}} \, (1+|j_0|^{\gamma_1}) \, \left| h_{j_0} \right| \\
&\le \, \frac{C}{n^{\gamma_2-\gamma_1}} \, \|\mathbf{M}_r\left(c,\cdot+n \, |\alpha_r|,n\right)\|_{\ell^\infty(\Z)} \, \|\bfh\|_{\ell^1_{\gamma_2}} \, ,
\end{align*}
and it simply remains to use the property:
\begin{equation}
\label{estimlinftyMr}
\left\| \left( \mathbf{M}_r \left( c,j+n \, |\alpha_r|,n \right) \right)_{j\in\Z} \right\|_{\ell^\infty(\Z)} \, \le \, \dfrac{1}{n^{1/3}} \, ,
\end{equation}
that can be easily deduced from the definition \eqref{defmajorantr}. This concludes the proof of  Lemma \ref{lemSrl1}.
\end{proof}

Next, we estimate $\left(\mathfrak{S}_{r,2}(n,j,\bfh)\right)_{j\in\Z}$ and $\left(\mathfrak{S}_{\ell,2}(n,j,\bfh)\right)_{j\in\Z}$.

\begin{lemma}
\label{lemSrl2}
For any $\gamma_2 \geq \gamma_1\geq 0$, there exists a constant $C>0$ such that for all $n\geq1$ and $\bfh \in \ell^1_{\gamma_2}(\Z;\R)$, one has:
$$
\left\| \left( \mathfrak{S}_{r,2}(n,j,\bfh) \right)_{j\in\Z} \right\|_{\ell^1_{\gamma_1}} + \left\| \left( 
\mathfrak{S}_{\ell,2}(n,j,\bfh) \right)_{j\in\Z} \right\|_{\ell^1_{\gamma_1}} \, \le \, \dfrac{C}{n^{\gamma_2+1/3}} \, \| \bfh \|_{\ell^1_{\gamma_2}} \, ,
$$
and therefore:
$$
\left\| \left( \mathfrak{S}_{r,2}(n,j,\bfh) \right)_{j\in\Z} \right\|_{\ell^\infty_{\gamma_1}} 
+ \left\| \left( \mathfrak{S}_{\ell,2}(n,j,\bfh) \right)_{j\in\Z} \right\|_{\ell^\infty_{\gamma_1}} \, \le \, 
\dfrac{C}{n^{\gamma_2+1/3}} \, \| \bfh \|_{\ell^1_{\gamma_2}} \, .
$$
\end{lemma}

\begin{proof}
For $\bfh \in \ell^1_{\gamma_2}(\Z;\R)$, we have that
\begin{align*}
\left\| \left( \mathfrak{S}_{r,2}(n,j,\bfh) \right)_{j\in\Z} \right\|_{\ell^1_{\gamma_1}} & \, = \, 
\left( \sum_{j\in\Z} \, (1+|j|^{\gamma_1}) \, \rme^{-c \, |j|} \right) \, \left( \sum_{j_0\geq1} \mathbf{M}_r \left( c,-j_0+n \, |\alpha_r|,n \right) \, |h_{j_0}| \right) \\
& \le \, C \, \sum_{j_0=1}^{\frac{n \, |\alpha_r|}{2}} \, \mathbf{M}_r \left( c,-j_0+n \, |\alpha_r|,n \right) \, |h_{j_0}| \, 
+ \, C \, \sum_{j_0\geq\frac{n \, |\alpha_r|}{2}} \, \mathbf{M}_r \left( c,-j_0+n \, |\alpha_r|,n \right) \, |h_{j_0}| \, .
\end{align*}
We then either use the exponential bound given by the definition \eqref{defmajorantr} or the global bound \eqref{estimlinftyMr}, which yields:
$$
\left\| \mathfrak{S}_{r,2}(n,\cdot,\bfh) \right\|_{\ell^1_{\gamma_1}} \, \le \, C \, \rme^{-c \, n} \, \| \bfh \|_{\ell^1} \, + \, 
\dfrac{C}{n^{1/3}}  \, \sum_{j_0\geq\frac{n \, |\alpha_r|}{2}} \, \dfrac{(1+|j_0|^{\gamma_2})}{n^{\gamma_2}} \, |h_{j_0}| \\
\, \le \, \dfrac{C}{n^{\gamma_2+1/3}} \, \| \bfh \|_{\ell^1_{\gamma_2}} \, .
$$
The $\ell^\infty_{\gamma_1}$ estimate is straightforward by just observing that the $\ell^\infty_{\gamma_1}$ norm is always smaller than the 
$\ell^1_{\gamma_1}$ norm. The proof of Lemma \ref{lemSrl2} is thus complete.
\end{proof}

We now handle the most delicate estimates on the family of operators $\left(\cL^n_{\mathrm{act}}\right)_{n\in\N^*}$.

\begin{lemma}
\label{lem-estim-activation}
For any $\gamma_2 \geq \gamma_1 \geq 0$, there exists a constant $C>0$ such that the family of operators $\left(\cL^n_{\mathrm{act}}\right)_{n\in\N^*}$ 
satisfies the following estimate:
\bqs
\forall n \in\N^*\,,\quad \left\|\cL^n_{\mathrm{act}} \, \bfh \right\|_{\ell^1_{\gamma_1}} \, \le \, \dfrac{C}{n^{\gamma_2}} \, \|\bfh\|_{\ell^1_{\gamma_2}} 
\, ,\quad \text{ for any } \bfh \in \ell^1_{\gamma_2}(\Z;\R) \text{ with } \sum_{j\in\Z}h_j=0 \, ,
\eqs
and therefore:
\bqs
\forall n \in\N^*\, ,\quad \left\|\cL^n_{\mathrm{act}} \, \bfh \right\|_{\ell^\infty_{\gamma_1}} \, \le \, \dfrac{C}{n^{\gamma_2}} \, \|\bfh\|_{\ell^1_{\gamma_2}} 
\, ,\quad \text{ for any } \bfh \in \ell^1_{\gamma_2}(\Z;\R) \text{ with } \sum_{j\in\Z}h_j=0 \, .
\eqs
\end{lemma}

\begin{proof}
We consider first $\bfh \in \ell^1_{\gamma_2}(\Z;\R)$ without making any assumption on its mass. The first step of the proof consists in writing:
\begin{align}
\sum_{j_0\geq1}\mathds{1}_{|j-j_0|\leq n} &\mathbf{A}_r(-j_0+n \, |\alpha_r|,n) \, h_{j_0} 
+\sum_{j_0\leq0}\mathds{1}_{|j-j_0|\leq n}\mathbf{A}_\ell(j_0+n \, \alpha_\ell,n) \, h_{j_0} \notag \\
=&\sum_{j_0\geq1}\mathbf{A}_r(-j_0+n \, |\alpha_r|,n) \, h_{j_0} +\sum_{j_0\leq0}\mathbf{A}_\ell(j_0+n \, \alpha_\ell,n) \, h_{j_0} \label{estim-activation-1} \\
&-\sum_{j_0\geq1}\mathds{1}_{|j-j_0|> n}\mathbf{A}_r(-j_0+n \, |\alpha_r|,n) \, h_{j_0} 
-\sum_{j_0\leq0}\mathds{1}_{|j-j_0|> n}\mathbf{A}_\ell(j_0+n \, \alpha_\ell,n) \, h_{j_0} \, ,\notag
\end{align}
where the two infinite sums on the right-hand side converge since $\mathbf{A}_r(-j_0+n \, |\alpha_r|,n)$ is uniformly bounded (see Corollary \ref{coro-A6} 
in Appendix \ref{appendixA}) and $\bfh$ is integrable. Since $|\alpha_r| \in (0,1)$, we can take $\beta \in (|\alpha_r|,1)$. As a consequence, we can use 
the exponential decay of $(\mathcal{H}_j)_{j \in \Z}$ and further decompose:
\begin{align*}
\left| \mathcal{H}_j \sum_{j_0\geq1}\mathds{1}_{|j-j_0|> n}\mathbf{A}_r(-j_0+n \, |\alpha_r|,n)h_{j_0} \right| 
&\le C \, \sum_{j_0\geq1}\mathds{1}_{|j-j_0|> n}\left|\mathbf{A}_r(-j_0+n \, |\alpha_r|,n)\right| \, \left|h_{j_0}\right| \, \rme^{-c \, |j|} \\
&= C \sum_{j_0=1}^{n\beta}\mathds{1}_{|j-j_0|> n}\left|\mathbf{A}_r(-j_0+n \, |\alpha_r|,n)\right| \, \left|h_{j_0}\right| \, \rme^{-c \, |j|} \\
&~~~+C\sum_{j_0 > n\beta}\mathds{1}_{|j-j_0|> n}\left|\mathbf{A}_r(-j_0+n \, |\alpha_r|,n)\right| \, \left|h_{j_0}\right| \, \rme^{-c \, |j|} \, .
\end{align*}
In the first sum, since $n<|j-j_0|$ and $1\leq j_0\leq n \beta$, we get that $(1-\beta)n < |j|$, and thus
\begin{align*}
\sum_{j_0=1}^{n\beta}\mathds{1}_{|j-j_0|> n} \left|\mathbf{A}_r(-j_0+n \, |\alpha_r|,n)\right| \, \left|h_{j_0}\right| \, \rme^{-c \, |j|} 
&\le \sum_{j_0=1}^{n\beta}\mathds{1}_{|j-j_0|> n}\left|\mathbf{A}_r(-j_0+n \, |\alpha_r|,n)\right| \, \left|h_{j_0}\right| \, 
\rme^{-\frac{c}{2}|j|}\rme^{-\frac{c(1-\beta)}{2}n} \\
& \leq C \, n \, \rme^{-\frac{c}{2}(1-\beta)n} \, \rme^{-\frac{c}{2}|j|} \, \| \bfh\|_{\ell^\infty(\Z)} \, ,
\end{align*}
where we have used the property from Corollary \ref{coro-A6} that the quantity $\left|\mathbf{A}_r(-j_0+n \, |\alpha_r|,n)\right|$ is uniformly bounded 
and we have also used the fact that the sum with respect to $j_0$ gathers at most $n$ terms (since $\beta \ge 1$). We now consider the sum with respect 
to $j_0 > n\beta$. Still using Corollary \ref{coro-A6} in Appendix \ref{appendixA}, but this time for $j_0>n \beta$ with $\beta\in(|\alpha_r|,1)$, we have that
\bqs
\left|\mathbf{A}_r(-j_0+n \, |\alpha_r|,n)\right| \le C \, \exp \left( -c \, \frac{\left|j_0-n \, |\alpha_r|\right|^{4/3}}{n^{1/3}}\right) \, ,
\eqs
and thus
\bqs
\sum_{j_0> n\beta} \left|\mathbf{A}_r(-j_0+n \, |\alpha_r|,n)\right| \leq C \, \rme^{-c \, n} \, .
\eqs
As a consequence, we have
\bqs
\sum_{j_0> n\beta}\mathds{1}_{|j-j_0|> n}\left|\mathbf{A}_r(-j_0+n \, |\alpha_r|,n)\right| \, \left|h_{j_0}\right| \, \rme^{-c \, |j|} 
\le C \, \rme^{-c \, n-c \, |j|} \, \|\bfh\|_{\ell^\infty(\Z)} \, .
\eqs
Summing up, we have proved the following estimate
\bqs
\sum_{j\in\Z}(1+|j|)^{\gamma_1} \, \left| \mathcal{H}_j \, \sum_{j_0\geq1} \mathds{1}_{|j-j_0|> n} \, \mathbf{A}_r(-j_0+n \, |\alpha_r|,n) \, h_{j_0} \right| 
\le C \, \rme^{-c \, n} \, \| \bfh\|_{\ell^\infty(\Z)} \le C \, \rme^{-c \, n} \, \|\bfh\|_{\ell^1_{\gamma_2}} \, ,
\eqs
and similarly, we also have
\bqs
\sum_{j\in\Z}(1+|j|)^{\gamma_1} \, \left| \mathcal{H}_j \, \sum_{j_0\leq0} \mathds{1}_{|j-j_0|> n} \, \mathbf{A}_\ell(j_0+n \, \alpha_\ell,n) \, h_{j_0} \right| 
\le C \, \rme^{-c \, n} \, \|\bfh\|_{\ell^1_{\gamma_2}} \, .
\eqs

We now go back to the decomposition \eqref{estim-activation-1} and consider the two series in the right-hand side. We assume from now on that 
the sequence $\bfh$ has zero mass, that is:
$$
\sum_{j_0 \in \Z} \, h_{j_0} \, = \, 0 \, .
$$
We can therefore write:
\begin{align*}
\sum_{j_0\geq1} \mathbf{A}_r(-j_0+n \, |\alpha_r|,n) \, h_{j_0} 
&+\sum_{j_0\leq0} \mathbf{A}_\ell(j_0+n \, \alpha_\ell,n) \, h_{j_0} \\
=&\sum_{j_0\geq1} \left( \mathbf{A}_r(-j_0+n \, |\alpha_r|,n)-1 \right) \, h_{j_0} +\sum_{j_0\leq0} \left( \mathbf{A}_\ell(j_0+n \, \alpha_\ell,n)-1 \right) \, h_{j_0} \, ,
\end{align*}
and we shall study each sum separately. For the first contribution, we further split the sum into two parts:
\begin{align*}
\sum_{j_0\geq1} \left( \mathbf{A}_r(-j_0+n \, |\alpha_r|,n)-1 \right) \, h_{j_0} 
=& \sum_{j_0=1}^{\frac{n \, |\alpha_r|}{2}} \left( \mathbf{A}_r(-j_0+n \, |\alpha_r|,n)-1 \right) \, h_{j_0} \\
&+\sum_{j_0>\frac{n \, |\alpha_r|}{2}} \left( \mathbf{A}_r(-j_0+n \, |\alpha_r|,n)-1 \right) \, h_{j_0} \, .
\end{align*}
Once again, we will rely on Corollary \ref{coro-A6} in Appendix \ref{appendixA}. On the one hand, we use the fact for $1\leq j_0 \leq \frac{n \, |\alpha_r|}{2}$, 
one has an exponential bound:
\bqs
\left|\mathbf{A}_r(-j_0+n \, |\alpha_r|,n)-1\right| \, \le \, C \, \rme^{-c \, n} \, ,
\eqs
and on the other hand, that the factor $\mathbf{A}_r(-j_0+n \, |\alpha_r|,n)$ is uniformly bounded with respect to $j_0 \in \Z$ and $n \in \N^*$. 
As a consequence, we get
\begin{align*}
\left|\sum_{j_0\geq1} \left( \mathbf{A}_r(-j_0+n \, |\alpha_r|,n)-1 \right) \, h_{j_0}\right| \, 
&\le \, C \, \rme^{-c \, n} \, \| \bfh \|_{\ell^1} \, + \, C \, \sum_{j_0>\frac{n \, |\alpha_r|}{2}} |h_{j_0}| \\
&\le \, C \, \rme^{-c \, n} \, \| \bfh \|_{\ell^1} \, + \, 
C \, \sum_{j_0>\frac{n \, |\alpha_r|}{2}} \dfrac{1+|j_0|^{\gamma_2}}{1+|(n \, |\alpha_r|/2)|^{\gamma_2}} \, |h_{j_0}| 
\leq \dfrac{C}{n^{\gamma_2}} \, \|\bfh\|_{\ell^1_{\gamma_2}} \, .
\end{align*}
And similarly, we also have
\bqs
\left|\sum_{j_0\leq0}\left(\mathbf{A}_\ell(j_0+n \, \alpha_\ell,n)-1\right)h_{j_0}\right| \, \le \, \dfrac{C}{n^{\gamma_2}} \, \|\bfh\|_{\ell^1_{\gamma_2}} \, .
\eqs
Summing up, we have obtained
\bqs
\sum_{j\in\Z}(1+|j|)^{\gamma_1} \, \left| \mathcal{H}_j \, \left( \sum_{j_0\geq1} \mathbf{A}_r(-j_0+n \, |\alpha_r|,n) \, h_{j_0} 
+\sum_{j_0\leq0} \mathbf{A}_\ell(j_0+n \, \alpha_\ell,n) \, h_{j_0} \right) \right| \, \le \, \dfrac{C}{n^{\gamma_2}} \, \|\bfh\|_{\ell^1_{\gamma_2}} \, .
\eqs
Going back to the decomposition \eqref{estim-activation-1}, this concludes the proof of Lemma \ref{lem-estim-activation}.
\end{proof}

The very last contribution to handle is the one coming from the exponential terms $(\mathfrak{S}_{\mathrm{exp}}(n,j,\bfh))_{j\in\Z}$, and we have 
the following result whose proof is trivial.

\begin{lemma}
\label{lem-estim-exp}
For any $\gamma_2 \geq \gamma_1 \geq0$, there exists a constant $C>0$ such that for all $n\geq1$ and $\bfh \in \ell^1_{\gamma_2}(\Z;\R)$, one has
\bqs
\left\|(\mathfrak{S}_{\mathrm{exp}}(n,j,\bfh))_{j\in\Z} \right\|_{\ell^1_{\gamma_1}} \, \le \, C \, \rme^{-c \, n} \, \| \bfh \|_{\ell^1_{\gamma_2}} \, ,
\eqs
and therefore:
\bqs
\left\|(\mathfrak{S}_{\mathrm{exp}}(n,j,\bfh))_{j\in\Z} \right\|_{\ell^\infty_{\gamma_1}} \, \le \, C \, \rme^{-c \, n} \, \| \bfh \|_{\ell^1_{\gamma_2}} \, .
\eqs
\end{lemma}

Going back to the inequality \eqref{estimationlin-n}, the proof of the estimates \eqref{estimLn1} and \eqref{estimLn2} on the semigroup $(\cL^n)_{n\in\N}$ 
of Theorem~\ref{thmlinestim} is then a direct consequence of Lemma \ref{lemSrl1}, Lemma \ref{lemSrl2}, Lemma \ref{lem-estim-activation} and Lemma 
\ref{lem-estim-exp}. Indeed, for proving \eqref{estimLn1}, we consider $\bfh \in \ell^1_{\gamma_2}(\Z;\R)$ with mass zero and combine those four 
preliminary results to obtain:
$$
\left\| \cL^n \, \bfh \right\|_{\ell^1_{\gamma_1}} \, \le \, C \left( \dfrac{1}{n^{\gamma_2}} + \dfrac{1}{n^{\gamma_2-\gamma_1-1/8}} 
+ \dfrac{1}{n^{\gamma_2+1/3}} + \rme^{-c \, n} \right) \, \| \bfh \|_{\ell^1_{\gamma_2}} \, 
\le \, \dfrac{C}{n^{\gamma_2-\gamma_1-1/8}} \, \|\bfh\|_{\ell^1_{\gamma_2}} \, ,
$$
and the proof of \eqref{estimLn2} is entirely similar.

\subsection{Proof of the estimates \eqref{estimLnshift1}, \eqref{estimLnshift2} and \eqref{estimLnshift3} on the family 
$(\cL^n(\mathrm{Id}-\mathbf{S}))_{n\in\N}$}

The starting point of the proof of the estimates \eqref{estimLnshift1}, \eqref{estimLnshift2} and \eqref{estimLnshift3} of Theorem~\ref{thmlinestim} 
follows similar lines as in the previous subsection for proving \eqref{estimLn1} and \eqref{estimLn2} on the semigroup $(\cL^n)_{n\in\N}$. For 
$\gamma \geq 0$ and $\bfh \in \ell^q_\gamma(\Z;\R)$ ($q=1$ or $q=+\infty$) we recall that the action of the family of operators 
$(\cL^n(\mathrm{Id}-\mathbf{S}))_{n\in\N}$ on $\bfh$ is given for each $n \in \N$ by:
\bqs
\forall \, (n,j) \in \, \N\times\Z \, , \quad \left(\cL^n(\mathrm{Id}-\mathbf{S}) \bfh\right)_j 
= \sum_{j_0\in\Z}\left(\cG^n(j,j_0)-\cG^n(j,j_0-1)\right)h_{j_0}=\sum_{j_0\in\Z}\mathscr{D}^n(j,j_0) \, h_{j_0} \,,
\eqs
by definition~\eqref{GreenDeriv} of $\mathscr{D}^n(j,j_0)$. Restricting to $n \ge 1$ (the case $n=0$ is trivial) and using the estimates obtained on 
the derivative of the temporal Green's function in Theorem~\ref{thmGreenDeriv}, we obtain that
\begin{align*}
\left|\left(\cL^n(\mathrm{Id}-\mathbf{S}) \bfh\right)_j\right| &\leq C \sum_{j_0\geq1}\mathds{1}_{|j-j_0|\leq n+1} \, 
\left[\mathds{1}_{j\geq1} \mathbf{K}_r (c,j-j_0+n \, |\alpha_r|,n)+ \rme^{-c \, |j|} \, \mathbf{M}_r (c,-j_0+|\alpha_r|n,n) \right] \, \left|h_{j_0}\right| \\
&~~~+C \sum_{j_0\leq0} \mathds{1}_{|j-j_0|\leq n+1} \left[\mathds{1}_{j\leq0} \, \mathbf{K}_\ell (c,j_0-j+n \, \alpha_\ell,n) 
+\rme^{-c \, |j|} \, \mathbf{M}_\ell (c,j_0+n \, \alpha_\ell,n) \right] \, \left| h_{j_0} \right| \\
&~~~+C \, \rme^{-c \, n} \, \sum_{j_0\in\Z} \,\mathds{1}_{|j-j_0|\leq n+1} \, \rme^{-c \, |j-j_0|} \, \left| h_{j_0} \right| \, .
\end{align*}
This motivates the definition of the following quantities for $(n,j) \in \N^*\times\Z$:
\begin{align*}
\mathfrak{D}_{r,1}(n,j,\bfh) \, & := \, \mathds{1}_{j\geq1} \, \sum_{j_0\geq1} \, \mathds{1}_{|j-j_0|\leq n+1} \, 
\mathbf{K}_r (c,j-j_0+n \, |\alpha_r|,n) \, \left| h_{j_0} \right| \, , \\
\mathfrak{D}_{\ell,1}(n,j,\bfh) \, & := \, \mathds{1}_{j\leq0} \, \sum_{j_0\leq0} \, \mathds{1}_{|j-j_0|\leq n+1} \, 
\mathbf{K}_\ell (c,j_0-j+n \, \alpha_\ell,n) \, \left| h_{j_0} \right| \, ,
\end{align*}
and
\begin{align*}
\mathfrak{D}_{r,2}(n,j,\bfh) \, & := \, \rme^{-c \, |j|} \, \sum_{j_0\geq1} \, \mathbf{M}_r (c,-j_0+n \, |\alpha_r|,n) \, \left| h_{j_0} \right| \, , \\
\mathfrak{D}_{\ell,2}(n,j,\bfh) \, & := \, \rme^{-c \, |j|} \, \sum_{j_0 \leq 0} \, \mathbf{M}_\ell (c,j_0+n \, \alpha_\ell,n) \, \left| h_{j_0} \right|\, ,
\end{align*}
together with:
\bqs
\mathfrak{D}_{\mathrm{exp}}(n,j,\bfh) \, := \, \rme^{-c \, n} \, \sum_{j_0\in\Z} \, \mathds{1}_{|j-j_0|\leq n+1} \, \rme^{-c|j-j_0|} \, \left| h_{j_0} \right| \, .
\eqs
With these notations at hand, we readily obtain that for all $(n,j)\in\N^*\times\Z$, we have:
\bqs
\left| \left( \cL^n(\mathrm{Id}-\mathbf{S}) \, \bfh \right)_j\right| \, \le \, C \, \left( \mathfrak{D}_{r,1}(n,j,\bfh) + \mathfrak{D}_{\ell,1}(n,j,\bfh) 
+ \mathfrak{D}_{r,2}(n,j,\bfh) + \mathfrak{D}_{\ell,2}(n,j,\bfh) + \mathfrak{D}_{\mathrm{exp}}(n,j,\bfh) \right) \, .
\eqs
We will mainly focus our efforts on the terms $\mathfrak{D}_{r,1}(n,j,\bfh)$, $\mathfrak{D}_{\ell,1}(n,j,\bfh)$ and $\mathfrak{D}_{\mathrm{exp}}(n,j,\bfh)$. 
Indeed, we note that $\mathfrak{D}_{r,2}(n,j,\bfh)$ and $\mathfrak{D}_{\ell,2}(n,j,\bfh)$ are identical to $\mathfrak{S}_{r,2}(n,j,\bfh)$ and 
$\mathfrak{S}_{\ell,2}(n,j,\bfh)$ in the previous subsection and thus enjoy the same estimates as the ones derived in Lemma~\ref{lemSrl2}. 
We shall need however an additional estimate to the one proved in Lemma~\ref{lemSrl2} but the proof of it will be very similar to what has 
already been done. Let us start indeed with the estimates for $\mathfrak{D}_{r,2}(n,j,\bfh)$ and $\mathfrak{D}_{\ell,2}(n,j,\bfh)$.

\begin{lemma}
\label{lemDrl2}
For any $\gamma_2 \geq \gamma_1 \geq 0$, there exists a constant $C>0$ such that for all $n\geq1$ and $\bfh \in \ell^1_{\gamma_2}(\Z;\R)$, one has
\begin{align*}
\left\| \left( \mathfrak{D}_{r,2}(n,j,\bfh) \right)_{j\in\Z} \right\|_{\ell^1_{\gamma_1}} 
+ \left\| \left( \mathfrak{D}_{\ell,2}(n,j,\bfh) \right)_{j\in\Z} \right\|_{\ell^1_{\gamma_1}} & \leq \, \dfrac{C}{n^{\gamma_2+1/3}} \, \|\bfh\|_{\ell^1_{\gamma_2}} \, , \\
\left\| \left( \mathfrak{D}_{r,2}(n,j,\bfh) \right)_{j\in\Z} \right\|_{\ell^\infty_{\gamma_1}} 
+ \left\| \left( \mathfrak{D}_{\ell,2}(n,j,\bfh) \right)_{j\in\Z} \right\|_{\ell^\infty_{\gamma_1}} & \leq \, \dfrac{C}{n^{\gamma_2+1/3}} \, \|\bfh\|_{\ell^1_{\gamma_2}} 
\end{align*}
and for $\bfh \in \ell^\infty_{\gamma_2}(\Z;\R)$, one has
$$
\left\| \left( \mathfrak{D}_{r,2}(n,j,\bfh) \right)_{j\in\Z} \right\|_{\ell^\infty_{\gamma_1}} 
+ \left\| \left( \mathfrak{D}_{\ell,2}(n,j,\bfh )\right)_{j\in\Z} \right\|_{\ell^\infty_{\gamma_1}} \, \le \, \dfrac{C}{n^{\gamma_2-1/8}} \, \| \bfh \|_{\ell^\infty_{\gamma_2}} \, .
$$
\end{lemma}

\begin{proof}
The first two estimates in Lemma~\ref{lemDrl2} have already been proved in Lemma~\ref{lemSrl2} so we switch directly to the last 
estimate where the novelty is that now the sequence $\bfh$ is assumed to belong to the larger space $\ell^\infty_{\gamma_2}(\Z;\R)$. 
We thus consider $n \in \N^*$ and $\bfh \in \ell^\infty_{\gamma_2}(\Z;\R)$. From the definition of $\mathfrak{D}_{r,2}(n,j,\bfh)$, we have
\begin{align*}
\left\| \left( \mathfrak{D}_{r,2}(n,j,\bfh) \right)_{j\in\Z} \right\|_{\ell^\infty_{\gamma_1}} =& 
\left( \sup_{j\in\Z} \, (1+|j|^{\gamma_1}) \, \rme^{-c \, |j|} \right) \, 
\left( \sum_{j_0\geq1} \, \mathbf{M}_r (c,-j_0+n \, |\alpha_r|,n) \, |h_{j_0}| \right) \\
\le & \, C \sum_{j_0=1}^{\frac{n \, |\alpha_r|}{2}} \, \mathbf{M}_r (c,-j_0+n \, |\alpha_r|,n) \, |h_{j_0}| 
+C \sum_{j_0\geq\frac{n \, |\alpha_r|}{2}} \, \mathbf{M}_r (c,-j_0+n \, |\alpha_r|,n) \, |h_{j_0}| \\
\le & \, C \, \| \bfh \|_{\ell^\infty} \, \sum_{j_0=1}^{\frac{n \, |\alpha_r|}{2}} \, \mathbf{M}_r (c,-j_0+n \, |\alpha_r|,n) 
+C \sum_{j_0\geq\frac{n \, |\alpha_r|}{2}} \, \mathbf{M}_r (c,-j_0+n \, |\alpha_r|,n) \, |h_{j_0}| \, .
\end{align*}
By summing the definition \eqref{defmajorantr} and comparing the corresponding series with an integral, we get the exponential estimate:
$$
\sum_{j_0=1}^{\frac{n \, |\alpha_r|}{2}} \, \mathbf{M}_r (c,-j_0+n \, |\alpha_r|,n) \, \le \, C \, \rme^{-c \, n} \, .
$$
 For the second sum, we have:
 \begin{align*}
 \sum_{j_0\geq\frac{n \, |\alpha_r|}{2}} \, \mathbf{M}_r (c,-j_0+n \, |\alpha_r|,n) \, |h_{j_0} \, \le & \, 
 C \, \sum_{j_0\geq\frac{n \, |\alpha_r|}{2}} \, \mathbf{M}_r (c,-j_0+n \, |\alpha_r|,n) \, \dfrac{(1+|j_0|^{\gamma_2})}{n^{\gamma_2}} \, |h_{j_0}| \\
 & \, +\dfrac{C}{n^{\gamma_2}} \, \| \bfh \|_{\ell^\infty_{\gamma_2}} \, \sum_{j_0\in\Z} \, \mathbf{M}_r (c,-j_0+n \, |\alpha_r|,n) \, ,
 \end{align*}
and we simply use \eqref{estiml1Mr} to conclude.
\end{proof}

Now, we handle the contributions from $\mathfrak{D}_{r,1}(n,j,\bfh)$ and $\mathfrak{D}_{\ell,1}(n,j,\bfh)$.

\begin{lemma}
\label{lemDrl1}
For any $\gamma_2\geq\gamma_1\geq0$, there exists $C>0$, such that for all $n\geq1$, one has
\begin{align*}
\left\|\left(\mathfrak{D}_{r,1}(n,j,\bfh)\right)_{j\in\Z} \right\|_{\ell^1_{\gamma_1}}+\left\|\left(\mathfrak{D}_{\ell,1}(n,j,\bfh)\right)_{j\in\Z} \right\|_{\ell^1_{\gamma_1}} & \leq \frac{C}{n^{\gamma_2-\gamma_1+1/8}}\|\bfh\|_{\ell^1_{\gamma_2}}\,,\text{ for } \bfh\in\ell^1_{\gamma_2}\,,\\
\left\|\left(\mathfrak{D}_{r,1}(n,j,\bfh)\right)_{j\in\Z} \right\|_{\ell^\infty_{\gamma_1}}+\left\|\left(\mathfrak{D}_{\ell,1}(n,j,\bfh)\right)_{j\in\Z} \right\|_{\ell^\infty_{\gamma_1}} & \leq \frac{C}{n^{\gamma_2-\gamma_1+7/12}}\| \bfh\|_{\ell^1_{\gamma_2}}\,,\text{ for } \bfh\in\ell^1_{\gamma_2}\,,\\
\left\|\left(\mathfrak{D}_{r,1}(n,j,\bfh)\right)_{j\in\Z} \right\|_{\ell^\infty_{\gamma_1}}+\left\|\left(\mathfrak{D}_{\ell,1}(n,j,\bfh)\right)_{j\in\Z} \right\|_{\ell^\infty_{\gamma_1}} & \leq \frac{C}{n^{\gamma_2-\gamma_1+1/8}}\|\bfh\|_{\ell^\infty_{\gamma_2}}\,,\text{ for } \bfh\in\ell^\infty_{\gamma_2}\,.
\end{align*}
\end{lemma}

\begin{proof}
For $h\in\ell^1_{\gamma_2}$, we have that
\begin{align*}
\left\|\left(\mathfrak{D}_{r,1}(n,j,\bfh)\right)_{j\in\Z}  \right\|_{\ell^1_{\gamma_1}}&=  \sum_{j\in\Z}(1+|j|)^{\gamma_1} \mathds{1}_{j\geq1} \sum_{j_0\geq1}\mathds{1}_{|j-j_0|\leq n+1}\mathbf{K}_r(c,j-j_0+n \, |\alpha_r|,n)|h_{j_0}| \\
&= \sum_{j\in\Z}(1+|j|)^{\gamma_1} \mathds{1}_{j\geq1} \sum_{j_0=1}^j\mathds{1}_{|j-j_0|\leq n+1}\mathbf{K}_r(c,j-j_0+n \, |\alpha_r|,n)|h_{j_0}| \\
&~~~+\sum_{j\in\Z}(1+|j|)^{\gamma_1} \mathds{1}_{j\geq1} \sum_{j_0\geq j+1}\mathds{1}_{j_0<\frac{n \, |\alpha_r|}{2}}\mathds{1}_{|j-j_0|\leq n+1}\mathbf{K}_r(c,j-j_0+n \, |\alpha_r|,n)|h_{j_0}|\\
&~~~+\sum_{j\in\Z}(1+|j|)^{\gamma_1} \mathds{1}_{j\geq1} \sum_{j_0\geq j+1}\mathds{1}_{j_0\geq\frac{n \, |\alpha_r|}{2}}\mathds{1}_{|j-j_0|\leq n+1}\mathbf{K}_r(c,j-j_0+n \, |\alpha_r|,n)|h_{j_0}|.
\end{align*}
Now from the definition~\eqref{defmajorantKr} of $\mathbf{K}_r$, we infer the following facts
\begin{itemize}
\item for $1\leq j_0 \leq j$ with $|j-j_0|\leq n+1$, one has $\mathbf{K}_r(c,j-j_0+n \, |\alpha_r|,n)\leq C \, \rme^{-c \, n}$,
\item for $j\geq 1$ and $j_0\geq j+1$ with $1\leq j_0 < \frac{n \, |\alpha_r|}{2}$, one has  $\mathbf{K}_r(c,j-j_0+n \, |\alpha_r|,n) \leq C \, \rme^{-c \, n}$.
\end{itemize}
As a consequence, one readily derives, as in the previous cases, that
\begin{align*}
\sum_{j\in\Z}(1+|j|)^{\gamma_1} \mathds{1}_{j\geq1}& \sum_{j_0=1}^j\mathds{1}_{|j-j_0|\leq n+1}\mathbf{K}_r(c,j-j_0+n \, |\alpha_r|,n)|h_{j_0}|\\
&+\sum_{j\in\Z}(1+|j|)^{\gamma_1} \mathds{1}_{j\geq1} \sum_{j_0\geq j+1}\mathds{1}_{j_0<\frac{n \, |\alpha_r|}{2}}\mathds{1}_{|j-j_0|\leq n+1}\mathbf{K}_r(c,j-j_0+n \, |\alpha_r|,n)|h_{j_0}| \\
&\leq C \rme^{-c \, n}\|\bfh\|_{\ell^1_{\gamma_2}}.
\end{align*}
For the last contribution, inverting the sums gives
\begin{align*}
\sum_{j\in\Z}(1+|j|)^{\gamma_1} \mathds{1}_{j\geq1} & \sum_{j_0\geq j+1}\mathds{1}_{j_0\geq\frac{n \, |\alpha_r|}{2}}\mathds{1}_{|j-j_0|\leq n+1}\mathbf{K}_r(c,j-j_0+n \, |\alpha_r|,n)|h_{j_0}| \\
&\leq C \left\|\left(\mathbf{K}_r(j+n \, |\alpha_r|,n)\right)_{j\in\Z}\right\|_{\ell^1(\Z)} \sum_{j_0\geq 1}\mathds{1}_{j_0\geq\frac{n \, |\alpha_r|}{2}}(1+|j_0|)^{\gamma_1}|h_{j_0}|\leq \frac{C}{n^{\gamma_2-\gamma_1+1/8}}\|\bfh\|_{\ell^1_{\gamma_2}},
\end{align*}
since $\left\|\left(\mathbf{K}_r(j+n \, |\alpha_r|,n)\right)_{j\in\Z}\right\|_{\ell^1(\Z)} \leq C/n^{1/8}$.

Finally, the proof of the second estimate is similar to the one for the estimate of $\left(\mathfrak{S}_{r,1}(n,j,\bfh)\right)_{j\in\Z}$ and $\left(\mathfrak{S}_{\ell,1}(n,j,\bfh)\right)_{j\in\Z}$ in Lemma~\ref{lemSrl1}, this time using that $\left\|\left(\mathbf{K}_r(j+n \, |\alpha_r|,n)\right)_{j\in\Z}\right\|_{\ell^\infty(\Z)}\leq C/n^{7/12}$.

For the last estimate, we first note that, for $\bfh\in\ell^\infty_{\gamma_2}$, we have that
\begin{align*}
\left\|\left(\mathfrak{D}_{r,1}(n,j,\bfh)\right)_{j\in\Z} \right\|_{\ell^\infty_{\gamma_1}}& = \sup_{j\in\Z}(1+|j|)^{\gamma_1} \mathds{1}_{j\geq1} \sum_{j_0\geq1}\mathds{1}_{|j-j_0|\leq n+1}\mathbf{K}_r(c,j-j_0+n \, |\alpha_r|,n)|h_{j_0}| \\
&=\sup_{j\in\Z}(1+|j|)^{\gamma_1} \mathds{1}_{j\geq1} \sum_{j_0=1}^j\mathds{1}_{|j-j_0|\leq n+1}\mathbf{K}_r(c,j-j_0+n \, |\alpha_r|,n)|h_{j_0}| \\
&~~~+\sup_{j\in\Z}(1+|j|)^{\gamma_1} \mathds{1}_{j\geq1} \sum_{j_0\geq j+1}\mathds{1}_{j_0<\frac{n \, |\alpha_r|}{2}}\mathds{1}_{|j-j_0|\leq n+1}\mathbf{K}_r(c,j-j_0+n \, |\alpha_r|,n)|h_{j_0}|\\
&~~~+\sup_{j\in\Z}(1+|j|)^{\gamma_1} \mathds{1}_{j\geq1} \sum_{j_0\geq j+1}\mathds{1}_{j_0\geq\frac{n \, |\alpha_r|}{2}}\mathds{1}_{|j-j_0|\leq n+1}\mathbf{K}_r(c,j-j_0+n \, |\alpha_r|,n)|h_{j_0}|.
\end{align*}
The first two terms contribute to
\begin{align*}
\sup_{j\in\Z}(1+|j|)^{\gamma_1} \mathds{1}_{j\geq1}& \sum_{j_0=1}^j\mathds{1}_{|j-j_0|\leq n+1}\mathbf{K}_r(c,j-j_0+n \, |\alpha_r|,n)|h_{j_0}| \\
&+\sup_{j\in\Z}(1+|j|)^{\gamma_1} \mathds{1}_{j\geq1} \sum_{j_0\geq j+1}\mathds{1}_{j_0<\frac{n \, |\alpha_r|}{2}}\mathds{1}_{|j-j_0|\leq n+1}\mathbf{K}_r(c,j-j_0+n \, |\alpha_r|,n)|h_{j_0}| \\
&\leq C \rme^{-c \, n} \|\bfh\|_{\ell^\infty_{\gamma_2}},
\end{align*}
while for the third term, we note that for each $j\geq1$
\begin{align*}
(1+|j|)^{\gamma_1} \sum_{j_0\geq j+1}\mathds{1}_{j_0\geq\frac{n \, |\alpha_r|}{2}}\mathds{1}_{|j-j_0|\leq n+1}&\mathbf{K}_r(c,j-j_0+n \, |\alpha_r|,n)|h_{j_0}| \\
&\leq C \sum_{j_0\geq1}\mathds{1}_{j_0\geq\frac{n \, |\alpha_r|}{2}}|\mathscr{K}_r^n(c,j-j_0)| \frac{(1+|j_0|)^{\gamma_2}}{(1+|j_0|)^{\gamma_2-\gamma_1}}|h_{j_0}|\\
\leq \frac{C}{n^{\gamma_2-\gamma_1+1/8}}\|\bfh\|_{\ell^\infty_{\gamma_2}}\,,
\end{align*}
since $\|\left(\mathbf{K}_r(j+n \, |\alpha_r|,n)\right)_{j\in\Z}\|_{\ell^1(\Z)}\leq C/n^{1/8}$.
\end{proof}

The very last contribution to handle is the one coming from the exponential terms $(\mathfrak{D}_{\mathrm{exp}}(n,j,\bfh))_{j\in\Z}$. And we have 
the following result whose proof is trivial and let to the interested reader.

\begin{lemma}
\label{lemDexp}
For any $\gamma_2 \geq \gamma_1 \geq 0$, there exists a constant $C>0$ such that for all $n\geq1$, one has:
\begin{align*}
\left\| (\mathfrak{D}_{\mathrm{exp}}(n,j,\bfh))_{j\in\Z} \right\|_{\ell^1_{\gamma_1}} & \leq C \, \rme^{-c \, n} \, \|\bfh\|_{\ell^1_{\gamma_2}} \, ,\quad 
\text{ for any } \bfh \in \ell^1_{\gamma_2}(\Z;\R)\, ,\\
\left\| (\mathfrak{D}_{\mathrm{exp}}(n,j,\bfh))_{j\in\Z} \right\|_{\ell^\infty_{\gamma_1}} & \leq C \, \rme^{-c \, n} \, \|\bfh\|_{\ell^\infty_{\gamma_2}} \, ,\quad 
\text{ for any } \bfh\in\ell^\infty_{\gamma_2}(\Z;\R) \, .
\end{align*}
\end{lemma}

We can now conclude the proof of Theorem~\ref{thmlinestim}. For the first estimate \eqref{estimLnshift1}, we use the estimates provided by 
Lemma \ref{lemDrl1}, Lemma \ref{lemDrl2} and Lemma \ref{lemDexp}. We thus have
\begin{align*}
\left\|\cL^n(\mathrm{Id}-\mathbf{S}) \bfh \right\|_{\ell^1_{\gamma_1}}&\leq C \, \left( 
\left\| \left(\mathfrak{D}_{r,1}(n,j,\bfh)\right)_{j\in\Z} \right\|_{\ell^1_{\gamma_1}} 
+\left\|\left(\mathfrak{D}_{\ell,1}(n,j,\bfh)\right)_{j\in\Z} \right\|_{\ell^1_{\gamma_1}}\right) \\
&~~~+C \left( \left\| \left( \mathfrak{D}_{r,2}(n,j,\bfh) \right)_{j\in\Z} \right\|_{\ell^1_{\gamma_1}} 
+ \left\| \left( \mathfrak{D}_{\ell,2}(n,j,\bfh)\right)_{j\in\Z} \right\|_{\ell^1_{\gamma_1}} 
+ \left\| \left( \mathfrak{D}_{\mathrm{exp}}(n,j,\bfh) \right)_{j\in\Z} \right\|_{\ell^1_{\gamma_1}} \right) \\
&\leq C \, \left( \dfrac{1}{n^{\gamma_2-\gamma_1+1/8}} + \dfrac{1}{n^{\gamma_2+1/3}} + \rme^{-c \, n} \right) \, \|\bfh\|_{\ell^1_{\gamma_2}} 
\le \dfrac{C}{n^{\gamma_2-\gamma_1+1/8}} \, \|\bfh\|_{\ell^1_{\gamma_2}} \, .
\end{align*}
For the second estimate \eqref{estimLnshift2}, we have 
\begin{align*}
\left\|\cL^n(\mathrm{Id}-\mathbf{S}) \bfh \right\|_{\ell^\infty_{\gamma_1}}&\leq C\left(\left\|\left(\mathfrak{D}_{r,1}(n,j,\bfh)\right)_{j\in\Z} \right\|_{\ell^\infty_{\gamma_1}}+\left\|\left(\mathfrak{D}_{\ell,1}(n,j,\bfh)\right)_{j\in\Z} \right\|_{\ell^\infty_{\gamma_1}}\right)\\
&~~~+C\left(\left\|\left(\mathfrak{D}_{r,2}(n,j,\bfh)\right)_{j\in\Z} \right\|_{\ell^\infty_{\gamma_1}}+\left\|\left(\mathfrak{D}_{\ell,2}(n,j,\bfh)\right)_{j\in\Z} \right\|_{\ell^\infty_{\gamma_1}}+\left\|\left(\mathfrak{D}_{\mathrm{exp}}(n,j,\bfh) \right)_{j\in\Z}\right\|_{\ell^\infty_{\gamma_1}}\right)\\
&\leq C\left(\left\|\left(\mathfrak{D}_{r,1}(n,j,\bfh)\right)_{j\in\Z} \right\|_{\ell^\infty_{\gamma_1}}+\left\|\left(\mathfrak{D}_{\ell,1}(n,j,\bfh)\right)_{j\in\Z} \right\|_{\ell^\infty_{\gamma_1}}\right)\\
&~~~+C\left(\left\|\left(\mathfrak{D}_{r,2}(n,j,\bfh)\right)_{j\in\Z} \right\|_{\ell^1_{\gamma_1}}+\left\|\left(\mathfrak{D}_{\ell,2}(n,j,\bfh)\right)_{j\in\Z} \right\|_{\ell^1_{\gamma_1}}+\left\|\left(\mathfrak{D}_{\mathrm{exp}}(n,j,\bfh) \right)_{j\in\Z}\right\|_{\ell^1_{\gamma_1}}\right)\\
&\leq C\left(\frac{1}{n^{\gamma_2-\gamma_1+7/12}}+\frac{1}{n^{\gamma_2+1/3}}+ \rme^{-c \, n}\right)\|\bfh\|_{\ell^1_{\gamma_2}} \leq \frac{C}{n^{\gamma_2-\gamma_1+1/3+\min(1/4,\gamma_1)}}\|\bfh\|_{\ell^1_{\gamma_2}}.
\end{align*}
Finally, for the third estimate \eqref{estimLnshift3}, we use once more the estimates provided by Lemma \ref{lemDrl1}, Lemma \ref{lemDrl2} 
and Lemma \ref{lemDexp} to get:
\begin{align*}
\left\|\cL^n(\mathrm{Id}-\mathbf{S}) \bfh \right\|_{\ell^\infty_{\gamma_1}}&\leq C\left(\left\|\left(\mathfrak{D}_{r,1}(n,j,\bfh)\right)_{j\in\Z} \right\|_{\ell^\infty_{\gamma_1}}+\left\|\left(\mathfrak{D}_{\ell,1}(n,j,\bfh)\right)_{j\in\Z} \right\|_{\ell^\infty_{\gamma_1}}\right)\\
&~~~+C\left(\left\|\left(\mathfrak{D}_{r,2}(n,j,\bfh)\right)_{j\in\Z} \right\|_{\ell^\infty_{\gamma_1}}+\left\|\left(\mathfrak{D}_{\ell,2}(n,j,\bfh)\right)_{j\in\Z} \right\|_{\ell^\infty_{\gamma_1}}+\left\|\left(\mathfrak{D}_{\mathrm{exp}}(n,j,\bfh) \right)_{j\in\Z}\right\|_{\ell^\infty_{\gamma_1}}\right)\\
&\leq \, C \, \left( \dfrac{1}{n^{\gamma_2-\gamma_1+1/8}} + \dfrac{1}{n^{\gamma_2-1/8}} + \rme^{-c \, n} \right) \, \|\bfh\|_{\ell^\infty_{\gamma_2}} 
\leq \, \dfrac{C}{n^{\gamma_2-\gamma_1+\min(1/8,\gamma_1-1/8)}} \, \|\bfh\|_{\ell^\infty_{\gamma_2}} \, .
\end{align*}
This completes the proof of Theorem \ref{thmlinestim}.

%%%%%%%%%%%%%%%%%%%%%%%%%%%%%%%%%%%%%%%%%%%%%%%%%%%%%%%%%%%%%%%%%%%%%%%%%%%%
\chapter{Nonlinear orbital stability}
\label{chapter5}

In this last chapter, we finally prove the nonlinear orbital stability of the family of stationary discrete profiles $\left\{ \mathbf{v}^\theta, \theta \in 
(-\underline{\theta},\underline{\theta}) \right\}$ given by Theorem~\ref{thm:Smyrlis} in some well-chosen algebraically weighted $\ell^p$ spaces. 
More precisely, given an initial sequence $\mathbf{h}=(h_j)_{j\in\Z}$ in $\ell^1(\Z;\R)$ satisfying the mass condition
\bqs
\sum_{j\in\Z}h_j=\theta \, , \quad \theta \in (-\underline{\theta},\underline{\theta}),
\eqs
we let $\mathbf{u}^n=(u_j^n)_{j \in \Z,n\in \N}$ denote the corresponding solution to the Lax-Wendroff scheme \eqref{schemeLW}-\eqref{fluxLW} 
starting from the initial condition $\mathbf{u}^0=\overline{\mathbf{u}}+\mathbf{h}$, where we recall that $\overline{\mathbf{u}}=\mathbf{v}^0$ is 
the element of the family $\left\{ \mathbf{v}^\theta, \theta \in (-\underline{\theta},\underline{\theta}) \right\}$ given by the explicit expression \eqref{DSP}. 
By mass conservation, we automatically have:
$$
\forall \, n \in \N \, ,\quad \sum_{j \in \Z} \, u_j^n- \overline{u}_j \, = \, \sum_{j \in \Z} \, u_j^0-\overline{u}_j \, = \, \sum_{j \in \Z} \, h_j \, = \, \theta \,.
$$
As a consequence, for each $n\in\N$, it is natural to define the \emph{perturbation} $\mathbf{p}^n=(p_j^n)_{j \in \Z}$ as
\bqs
\mathbf{p}^n \, :=\, \mathbf{u}^n-\mathbf{v}^\theta\,,
\eqs
where in particular the mass conservation property of $\mathbf{u}^n$ implies that for any $n \in \N$, the sequence $\mathbf{p}^n$ has zero mass:
 \bqs
 \forall \, n \in \N \, ,\quad \sum_{j \in \Z} \, p_j^n \, = \, 0\,.
 \eqs
From the recurrence formula \eqref{schemeLW} for $\mathbf{u}^n=\mathbf{v}^\theta+\mathbf{p}^n$, we deduce that the sequence $\mathbf{p}^n$ 
satisfies the following equation:
\bqq
\label{pinterm}
 \forall \, n \in \N \, ,\quad \mathbf{p}^{n+1}= \mathcal{L}^\theta \, \mathbf{p}^n+ \left(\mathrm{Id}-\mathbf{S}\right) \mathscr{N}^\theta(\mathbf{p}^n) \, ,
\eqq
where the bounded linear operator $\mathcal{L}^\theta \, : \, \ell^q(\Z;\C)\mapsto \ell^q(\Z;\C)$ is defined as the linearization of the Lax-Wendroff 
scheme around the stationary discrete shock profile $\mathbf{v}^\theta$, that is:
\bqq
\label{defLtheta}
\begin{split}
\forall \, j \in\Z \, ,\quad \left(\mathcal{L}^\theta \, \mathbf{p}\right)_j \,:= \, p_j &-\lambda \left[\partial_u \cF_\lambda(v_j^\theta,v_{j+1}^\theta) \, p_j 
+\partial_v\cF_\lambda(v_j^\theta,v_{j+1}^\theta) \, p_{j+1}\right] \\
&+\lambda\left[\partial_u \cF_\lambda(v_{j-1}^\theta,v_{j}^\theta) \, p_{j-1}+\partial_v\cF_\lambda(v_{j-1}^\theta,v_{j}^\theta) \, p_{j} \right] \, ,
\end{split}
\eqq
and the nonlinear operator $\mathscr{N}^\theta$ is defined by:
\bqq
\label{defNtheta}
\begin{split}
\forall \, j \in\Z \,, \quad \left(\mathscr{N}^\theta(\mathbf{p})\right)_j \, := \, &\lambda \, \cF_\lambda(v_{j-1}^\theta + p_{j-1},v_{j}^\theta+p_j) 
- \lambda \, \cF_\lambda(v_{j-1}^\theta,v_{j}^\theta) \\
& -\lambda \, \left[ \partial_u \cF_\lambda(v_{j-1}^\theta,v_{j}^\theta) \, p_{j-1}+\partial_v\cF_\lambda(v_{j-1}^\theta,v_{j}^\theta) \, p_{j} \right] \,.
\end{split}
\eqq
We also recall that $\mathbf{S}$ is the shift operator defined as $(\mathbf{S} \, \mathbf{p})_j:=p_{j+1}$ for all $j\in\Z$. Actually, we shall rewrite 
the recurrence relation in time \eqref{pinterm} more conveniently as
\bqs
\forall \, n \in \N \, ,\quad \mathbf{p}^{n+1} = \cL \, \mathbf{p}^n + \left(\mathcal{L}^\theta -\cL\right) \, \mathbf{p}^n 
+\left( \mathrm{Id}-\mathbf{S} \right) \mathscr{N}^\theta(\mathbf{p}^n) \, ,
\eqs
where the operator $\cL=\mathcal{L}^0$, given by \eqref{linear}, is the linearization of \eqref{schemeLW}-\eqref{fluxLW} around the discrete shock 
profile $\overline{\mathbf{u}}=\mathbf{v}^0$ given by \eqref{DSP}. Finally, an application of Duhamel's formula gives the final expression for the 
perturbation $\mathbf{p}^{n}$ at any time $n \in \N$:
\bqq
\label{pfinal}
 \forall \, n \in \N \, ,\quad \mathbf{p}^{n}= \cL^n \, \mathbf{p}^0 + \sum_{m=0}^{n-1} \cL^{n-1-m} \, \left(\mathcal{L}^\theta -\cL\right) \, \mathbf{p}^m 
 +\sum_{m=0}^{n-1} \cL^{n-1-m} \, \left(\mathrm{Id}-\mathbf{S}\right) \, \mathscr{N}^\theta(\mathbf{p}^m)\,.
\eqq

As already emphasized in Chapter \ref{chapter2}, our nonlinear orbital stability result (that is, Theorem \ref{thmNLS}) holds true in algebraically 
weighted spaces. Recalling our definition from the previous chapter, for $\gamma \in \R^+$, we define the weight $\boldsymbol{\omega}_\gamma
=(1+|j|^\gamma)_{j\in\Z}$. Then, for $p \in [1,+\infty]$, we introduce the weighted space:
\bqs
\ell^p_\gamma(\Z;\R) \, = \, \left\{ \mathbf{h} \in \ell^p(\Z;\R) ~|~  \boldsymbol{\omega}_\gamma \mathbf{h} \in\ell^p(\Z;\R) \right\},
\eqs
with $\boldsymbol{\omega}_\gamma \mathbf{h}=((1+|j|^\gamma) \, h_j)_{j\in\Z}$ and for $\mathbf{h}\in \ell^p_\gamma$, we denote 
$\| \bfh \|_{\ell^p_\gamma}=\left\| \boldsymbol{\omega}_\gamma \bfh\right\|_{\ell^p}$ the norm of $\bfh$ in the space $\ell^p_\gamma$. Our strategy 
will be to bound each term appearing in the right-hand side of \eqref{pfinal} using the semi-group estimates obtained for $(\cL^n)_{n\in\N}$ and 
$(\cL^n(\mathrm{Id}-\mathbf{S}))_{n\in\N}$ in Theorem~\ref{thmlinestim}. This is exactly the same strategy as in \cite{Coeuret2} except that the 
exponents in Theorem~\ref{thmlinestim} are different so the whole process should be carried out in order to determine the correct constraints 
for $\beta$ and $\sigma$ in Theorem \ref{thmNLS}.

We shall need three technical lemmas which we now state. The first lemma provides semi-group estimates for the family of operators 
$(\cL^n\left(\mathcal{L}^\theta -\cL\right))_{n\in\N}$.

\begin{lemma}
\label{lemLtheta}
For any $(\nu_1,\nu_2)\in [0,+\infty)^2$, there exists a constant $C_\mathcal{L}(\nu_1,\nu_2)>0$ such that for all sequence $\mathbf{p}\in\ell^\infty 
(\Z;\R)$ one has $\left(\mathcal{L}^\theta -\cL\right) \, \mathbf{p} \in \cap_{\gamma \ge 0} \ell^1_\gamma$ and:
\bqs
\forall \, n \in \N \, , \quad \left\| \cL^n\left(\mathcal{L}^\theta -\cL\right) \, \mathbf{p} \right\|_{\ell^1_{\nu_1}} \, \leq \, 
\dfrac{C_\mathcal{L}(\nu_1,\nu_2)}{(n+1)^{\nu_2}} \, |\theta| \, \|\mathbf{p}\|_{\ell^\infty}\,,
\eqs
for all $\theta \in (-\underline{\theta},\underline{\theta})$.
\end{lemma}

The proof of the above lemma can be found in \cite[Proposition 2]{Coeuret2}, and it uses crucially the exponential localisation of the family of 
discrete shock profiles $\mathbf{v}^\theta$ as stated in Theorem~\ref{thm:Smyrlis} and the key remark that for a bounded sequence $\mathbf{p} 
\in \ell^\infty(\Z;\R)$ one has
\bqs
\sum_{j\in\Z} \left( \left(\mathcal{L}^\theta -\cL\right) \mathbf{p} \right)_j=0,
\eqs
such that one can rely on estimate~\eqref{estimLn1} of Theorem~\ref{thmlinestim} with $(\gamma_1,\gamma_2)=\left(\nu_1,\nu_1+\nu_2+\frac{1}{8}\right)$.

The second lemma establishes algebraically weighted bounds for the nonlinear operator $\mathscr{N}^\theta$, and a proof is given in 
\cite[Lemma 3.1]{Coeuret2}.

\begin{lemma}
\label{estimNtheta}
Let $\left(\mathscr{N}^\theta\right)_{\theta \in (-\underline{\theta},\underline{\theta})}$ be the family of nonlinear operators defined by \eqref{defNtheta}. 
Let $\gamma_1 \geq 0$ and $\rho>0$ be given. There exists a constant $C_\mathscr{N}(\gamma_1,\rho)>0$ such that the following holds.
\begin{itemize}
\item[(i)] If $\mathbf{p}\in \ell^1_{\gamma_1}(\Z;\R)$ with $\| \mathbf{p} \|_{\ell^\infty}\leq \rho$, then for all $\theta \in (-\underline{\theta},\underline{\theta})$, 
one has $\mathscr{N}^\theta(\mathbf{p}) \in \ell^1_{2\gamma_1}(\Z;\R)$ and
\bqs
\forall \, \theta \in (-\underline{\theta},\underline{\theta}) \,, \quad \left\| \mathscr{N}^\theta(\mathbf{p})\right\|_{\ell^1_{2\gamma_1}} 
\, \leq \, C_\mathscr{N}(\gamma_1,\rho) \, \|\mathbf{p}\|_{\ell^1_{\gamma_1}} \, \|\mathbf{p}\|_{\ell^\infty_{\gamma_1}} \, .
\eqs
\item[(ii)] If $\mathbf{p}\in \ell^\infty_{\gamma_1}(\Z;\R)$ with $\| \mathbf{p} \|_{\ell^\infty}\leq \rho$, then for all $\theta \in (-\underline{\theta},\underline{\theta})$ 
one has $\mathscr{N}^\theta(\mathbf{p}) \in \ell^\infty_{2\gamma_1}(\Z;\R)$ and
\bqs
\forall \, \theta \in (-\underline{\theta},\underline{\theta}) \,, \quad \left\| \mathscr{N}^\theta(\mathbf{p})\right\|_{\ell^\infty_{2\gamma_1}} 
\, \leq \, C_\mathscr{N}(\gamma_1,\rho) \, \|\mathbf{p}\|_{\ell^\infty_{\gamma_1}}^2 \, .
\eqs
\end{itemize}
\end{lemma}

Finally, we shall also need the following estimates whose proof can be found in \cite[Lemma 3.2]{Coeuret2}. These are discrete versions of 
\cite[Lemma 2.3]{xin}. Here and in all this Chapter, we let $\lfloor x \rfloor$ denote the integer part of a real number $x$.

\begin{lemma}\label{lemabc}
Let $a$, $b$ and $c$ be positive real numbers. Then there exists a constant $\boldsymbol{C}(a,b,c)>0$ such that
\bqs
\forall \, n \in \N \, ,\quad \sum_{m=0}^{\left\lfloor\frac{n+1}{2}\right\rfloor} \dfrac{1}{(1+m-n)^a} \, \dfrac{1}{(1+m)^b} 
\, \le \, \dfrac{\boldsymbol{C}(a,b,c)}{(2+n)^c}\,,
\eqs
whenever $0<a-c$ if $b=1$, or $1-b\leq a-c$ if $b\in[0,1)$ or $0\leq a-c$ if $b>1$, together with
\bqs
\forall \, n \in \N \, ,\quad \sum_{m=\left\lfloor\frac{n+1}{2}\right\rfloor+1}^n \dfrac{1}{(1+m-n)^a} \, \dfrac{1}{(1+m)^b} 
\, \le \, \dfrac{\boldsymbol{C}(a,b,c)}{(2+n)^c} \, ,
\eqs
whenever $0<b-c$ if $a=1$, or $1-a\leq b-c$ if $a\in[0,1)$ or $0\leq b-c$ if $a>1$.
\end{lemma}

%%%%%%%%%%%%%%%
\section{Fixing the constants}
\label{section5-1}

The first step towards the proof of Theorem~\ref{thmNLS} is to fix the two constants $C_0>0$ and $\epsilon>0$ that appear in its statement. 
As a consequence, once for all we fix $\sigma+\beta\geq \frac{5}{12}$ together with $0\leq \sigma< \beta+\frac{1}{8}$, and we set $\gamma:=\sigma
+\beta+\frac{1}{8}$. We also fix a positive constant\footnote{In case the flux $f$ of the conservation law or the numerical flux $\cF_\lambda$ 
would be defined on an open set and not on the whole space ($\R$ or $\R^2$), this parameter $\varrho$ would help controlling the $\ell^\infty$ 
norm of the numerical solution so that the numerical solution would be defined for all times.} $\varrho>0$.

Next, using Theorem~\ref{thm:Smyrlis}(iii), we get the existence of a constant $C_m(\gamma)>0$ such that
\bqs
\forall \, \theta \in (-\underline{\theta},\underline{\theta})\,, \quad 
\|\mathbf{v}^\theta-\overline{\mathbf{u}}\|_{\ell^1_\gamma} \, \leq \, C_m(\gamma) \, |\theta|\,.
\eqs
We shall denote by $C_\cL(\beta,\gamma)>0$ the constant from Theorem~\ref{thmlinestim}  with $(\gamma_1,\gamma_2)=(\beta,\gamma)$.  
We can thus set
\bqq
\label{definitionC0}
C_0 \, := \, (1+C_\cL(\beta,\gamma)) \, (1+C_m(\gamma)) \, .
\eqq
For any real number $x$, we use the standard notation $x_+:=\max (x,0)$. With the above parameter $\sigma$, we introduce the number:
\begin{equation}
\label{defnu2bar}
\underline{\nu}_2 \, := \, \dfrac{7}{12} \, + \, \left( \sigma-\dfrac{7}{12} \right)_+ \, .
\end{equation}
Next, we introduce the following two constants
\begin{align*}
C_1:=& \, 2 \, C_0 \, C_\mathcal{L}(\beta,\underline{\nu}_2) \, \boldsymbol{C} \left( \underline{\nu}_2,\sigma+\dfrac{11}{24},\sigma \right) 
+ 2 \, C_0^2 \, C_\cL(\beta,2\beta) \, C_{\mathscr{N}}(\beta,\varrho) \, \boldsymbol{C}\left(\beta+\frac{1}{8},2\sigma+\frac{11}{24},\sigma\right) \, ,\\
C_2:=& \, 2 \, C_0 \, C_\mathcal{L}\left(\beta,\sigma+\frac{25}{24}\right) \, \boldsymbol{C} \left( \sigma+\dfrac{25}{24},\sigma+\frac{11}{24},\sigma+\frac{11}{24} \right) \\
&+ C_0^2 \, C_\cL(\beta,2\beta) \, C_{\mathscr{N}}(\beta,\varrho) \, \boldsymbol{C}\left(\beta+\frac{7}{12},2\sigma+\frac{11}{24},\sigma+\frac{11}{24}\right) \\
&+ C_0^2 \, C_\cL(\beta,2\beta) \, C_{\mathscr{N}}(\beta,\varrho) \, \boldsymbol{C}\left(\beta+\frac{1}{8},2\sigma+\frac{11}{12},\sigma+\frac{11}{24}\right) \,,
\end{align*}
where the constants $C_\mathcal{L}(\beta,\underline{\nu}_2)$ and $C_\mathcal{L}\left(\beta,\sigma+\frac{25}{24}\right)$ are given by Lemma~\ref{lemLtheta} with $(\nu_1,\nu_2)=(\beta,\underline{\nu}_2)$ and $(\nu_1,\nu_2)=\left(\beta,\sigma+\frac{25}{24}\right)$, $C_\cL(\beta,2\beta)$ 
is given by Theorem~\ref{thmlinestim} with $(\gamma_1,\gamma_2)=(\beta,2\beta)$, $C_{\mathscr{N}}(\beta,\varrho)$  is given by Lemma~\ref{estimNtheta} 
with $(\gamma_1,\rho)=(\beta,\varrho)$, and the four constants $\boldsymbol{C}\left(\beta+\frac{1}{8},2\sigma+\frac{11}{24},\sigma\right)$, $\boldsymbol{C} \left( \sigma+\dfrac{25}{24},\sigma+\frac{11}{24},\sigma+\frac{11}{24} \right)$, 
$\boldsymbol{C}\left(\beta+\frac{7}{12},2\sigma+\frac{11}{24},\sigma+\frac{11}{24}\right)$ and 
$\boldsymbol{C}\left(\beta+\frac{1}{8},2\sigma+\frac{11}{12},\sigma+\frac{11}{24}\right)$ are given by Lemma~\ref{lemabc} with either one of the triplets:
$$
\left(\beta+\frac{1}{8},2\sigma+\frac{11}{24},\sigma\right) \, , \left( \sigma+\dfrac{25}{24},\sigma+\frac{11}{24},\sigma+\frac{11}{24} \right)\, , \left(\beta+\frac{7}{12},2\sigma+\frac{11}{24},\sigma+\frac{11}{24}\right) 
\, , \left(\beta+\frac{1}{8},2\sigma+\frac{11}{12},\sigma+\frac{11}{24}\right).
$$
At last, we choose $\epsilon>0$ small enough such that
\bqq
\label{definitionepsilon}
0<\epsilon < \min\left( \underline{\theta},\frac{\varrho}{C_0},\frac{1+C_m(\gamma)}{C_1}, \frac{1+C_m(\gamma)}{C_2}\right).
\eqq

%%%%%%%%%%%%%%%%%%%%%
\section{Proof of Theorem~\ref{thmNLS}}
\label{section5-2}

We now complete the proof of Theorem~\ref{thmNLS}. We consider an initial perturbation $\bfh\in\ell^1_\gamma(\Z;\R)$, with $\gamma$ previously 
defined, and we assume that $\bfh$ is small enough so that
\bqs
\|\bfh\|_{\ell^1_\gamma} < \epsilon \, .
\eqs
We then define the excess mass:
\bqs
\theta \, := \, \sum_{j\in\Z} h_j \, .
\eqs
We readily remark that necessarily 
\bqs
|\theta|=\left|\sum_{j\in\Z}h_j\right| \leq \|\bfh\|_{\ell^1}\leq \|\bfh\|_{\ell^1_\gamma} < \epsilon<\underline{\theta} \, ,
\eqs
the last inequality coming from our choice \eqref{definitionepsilon} for $\epsilon$. We can therefore use below the discrete shock profile $\mathbf{v}^\theta$ 
associated with the excess mass $\theta$. We define the initial condition $\mathbf{u}^0:=\overline{\mathbf{u}}+\bfh$ and since the numerical flux 
$\cF_\lambda \in \mathscr{C}^\infty(\R^2;\R)$ is globally defined, we directly obtain the existence of $(\mathbf{u}^n 
)_{n\in\N}$ solution of \eqref{schemeLW}-\eqref{fluxLW} starting from this $\mathbf{u}^0$. There is no need here to control the $\ell^\infty$ norm of the 
numerical solution at each time step in order to remain within the set where the numerical flux is defined. However, the $\ell^\infty$ control will be crucial 
in order to obtain uniform bounds for the quadratic remainder term $\mathscr{N}^\theta$. Thus, we can introduce the sequence of perturbations 
$(\mathbf{p}^n)_{n\in\N}$ defined as 
\bqs
\mathbf{p}^n \, :=\, \mathbf{u}^n-\mathbf{v}^\theta\,,
\eqs
which is given by \eqref{pfinal}, that is
\bqs
 \forall \, n \in \N \, ,\quad \mathbf{p}^{n}= \cL^n \, \mathbf{p}^0 +\sum_{m=0}^{n-1} \cL^{n-1-m} \, \left(\mathcal{L}^\theta -\cL\right) \, \mathbf{p}^m 
 +\sum_{m=0}^{n-1} \cL^{n-1-m} \, \left(\mathrm{Id}-\mathbf{S}\right) \, \mathscr{N}^\theta(\mathbf{p}^m) \, .
\eqs
We shall now prove by induction that  for all $n\in\N$ one has $\mathbf{p}^n \in \ell^1_\beta (\Z;\R)$ together with the bounds:
\begin{equation}
\label{hyp-recurrence}
\forall \, m=0,\dots,n \, ,\quad \left\| \mathbf{p}^m \right\|_{\ell^1_\beta} \, \leq \, \dfrac{C_0}{(1+m)^\sigma} \, \|\bfh\|_{\ell^1_\gamma} \, ,\quad 
\left\| \mathbf{p}^m \right\|_{\ell^\infty_\beta} \, \leq \, \dfrac{C_0}{(1+m)^{\sigma+\frac{11}{24}}} \, \| \bfh \|_{\ell^1_\gamma} \, ,\quad 
\text{ and } \quad \| \mathbf{p}^m \|_{\ell^\infty} \leq \varrho \, .
\end{equation}

\subsection{Initialization step}

At $n=0$, we have by definition that  
\bqs
\mathbf{p}^0 \, =\, \mathbf{u}^0-\mathbf{v}^\theta\,=\overline{\mathbf{u}}-\mathbf{v}^\theta+\bfh\, .
\eqs
As a consequence, using $0 \le \beta \le \gamma$ and obvious inequalities between norms, we have the estimates:
\begin{equation}
\label{inegalitesp0}
\|\mathbf{p}^0\|_{\ell^\infty} \leq \|\mathbf{p}^0\|_{\ell^\infty_\beta} \leq \|\mathbf{p}^0\|_{\ell^1_\beta} \leq \|\mathbf{p}^0\|_{\ell^1_\gamma} 
\leq \left\|\overline{\mathbf{u}}-\mathbf{v}^\theta\right\|_{\ell^1_\gamma}+\|\bfh\|_{\ell^1_\gamma}\leq (1+C_m(\gamma))\|\bfh\|_{\ell^1_\gamma} 
\leq C_0 \|\bfh\|_{\ell^1_\gamma} \leq C_0 \epsilon \leq \varrho \, .
\end{equation}
This means that the induction assumption \eqref{hyp-recurrence} is satisfied for $n=0$ (the above chain of inequalities encompasses 
the three estimates of \eqref{hyp-recurrence}). We finally also recall that $\mathbf{p}^0$ has zero mass:
\bqs
\sum_{j\in\Z}p_j^0=0 \, ,
\eqs
and this property will be automatically propagated at later times since we consider a conservative scheme.

\subsection{Induction step}

We assume that \eqref{hyp-recurrence} is satisfied up to some integer $n \in \N$ and we now show that it propagates to the integer $n+1$. 
Using Duhamel's formula at $n+1$, we have the following expression for $\mathbf{p}^{n+1}$:
\bqs
 \mathbf{p}^{n+1}= \cL^{n+1} \, \mathbf{p}^0 + \sum_{m=0}^{n} \cL^{n-m}\left(\mathcal{L}^\theta -\cL\right) \, \mathbf{p}^m 
 +\sum_{m=0}^{n} \cL^{n-m}\left(\mathrm{Id}-\mathbf{S}\right) \mathscr{N}^\theta(\mathbf{p}^m) \, .
\eqs
We shall now bound each term separately.

\paragraph{The $\ell^1_\beta$ estimate.}

Since the initial perturbation $\mathbf{p}^{0} \in \ell^1_\gamma$ is of zero mass, we directly have by Theorem~\ref{thmlinestim}:
\bqq
\label{chapitre5-estim1}
\left\|\cL^{n+1} \, \mathbf{p}^{0} \right\|_{\ell^1_{\beta}} \le \dfrac{C_\cL(\beta,\gamma)}{(2+n)^{\sigma}} \, \|\mathbf{p}^{0}\|_{\ell^1_{\gamma}} 
\le \dfrac{C_\cL(\beta,\gamma) (1+C_m(\gamma))}{(2+n)^{\sigma}}\, \| \bfh \|_{\ell^1_{\gamma}} \, ,
\eqq
where we have used the following consequence of \eqref{inegalitesp0}:
$$
\|\mathbf{p}^{0}\|_{\ell^1_{\gamma}} \, \le \, (1+C_m(\gamma)) \, \| \bfh \|_{\ell^1_{\gamma}} \, .
$$

Regarding the second term, we use Lemma~\ref{lemLtheta} with $\nu_1=\beta$ and $\underline{\nu}_2$ defined in \eqref{defnu2bar}, 
the inequality $|\theta|<\epsilon$ and the induction assumption \eqref{hyp-recurrence} to obtain:
\begin{align*}
\left\| \cL^{n-m}\left(\mathcal{L}^\theta -\cL\right) \, \mathbf{p}^m \right\|_{\ell^1_\beta} \, 
& \, \le \, \dfrac{C_\mathcal{L}(\beta,\underline{\nu}_2)}{(n-m+1)^{\underline{\nu}_2}} \, |\theta| \, \|\mathbf{p}^m\|_{\ell^\infty} \\
& \, \le \, \dfrac{C_\mathcal{L}(\beta,\underline{\nu}_2)}{(n-m+1)^{\underline{\nu}_2}} \, |\theta| \, \|\mathbf{p}^m\|_{\ell^\infty_\beta} \\
& \, \le \, \dfrac{\epsilon \, C_0 \, C_\mathcal{L}(\beta,\underline{\nu}_2)}{(n-m+1)^{\underline{\nu}_2} \, (1+m)^{\sigma+\frac{11}{24}}} 
\, \|\bfh\|_{\ell^1_\gamma} \, .
\end{align*}
Finally, using Lemma~\ref{lemabc} with $a=\underline{\nu}_2$, $b=\sigma+\frac{11}{24}$ and $c=\sigma$ (the reader can easily verify that 
we are always in a position to apply the two inequalities provided by Lemma~\ref{lemabc} with our choice \eqref{defnu2bar} of $\underline{\nu}_2$), 
we directly obtain that
\begin{multline}
\left\| \sum_{m=0}^{n} \cL^{n-m} \left(\mathcal{L}^\theta -\cL\right) \, \mathbf{p}^m\right\|_{\ell^1_\beta} 
\le \, \sum_{m=0}^{n} \dfrac{\epsilon \, C_0 \, C_\mathcal{L}(\beta,\underline{\nu}_2)}{(n-m+1)^{\underline{\nu}_2} \, (1+m)^\sigma} \, \|\bfh\|_{\ell^1_\gamma} \\
\le \, 
\dfrac{2 \, \epsilon \, C_0 \, C_\mathcal{L}(\beta,\underline{\nu}_2) \, \boldsymbol{C} \left( \underline{\nu}_2,\sigma+\dfrac{11}{24},\sigma \right) }{(2+n)^\sigma} 
\, \| \bfh \|_{\ell^1_\gamma} \, .
\end{multline}

For the third term that incorporates the nonlinear contributions, we shall use the estimate~\eqref{estimLnshift1} of Theorem~\ref{thmlinestim} and 
Lemma~\ref{estimNtheta} to derive that
\begin{align*}
\left\| \sum_{m=0}^{n} \cL^{n-m} \left(\mathrm{Id}-\mathbf{S}\right) \mathscr{N}^\theta(\mathbf{p}^m)\right\|_{\ell^1_\beta} 
&\leq \sum_{m=0}^{n} \dfrac{C_\cL(\beta,2\beta)}{(1+n-m)^{\beta+\frac{1}{8}}} \, \left\| \mathscr{N}^\theta(\mathbf{p}^m) \right\|_{\ell^1_{2\beta}} \\
&\leq \sum_{m=0}^{n} \dfrac{C_\cL(\beta,2\beta) \, C_{\mathscr{N}}(\beta,\varrho)}{(1+n-m)^{\beta+\frac{1}{8}}} \, 
\left\| \mathbf{p}^m \right\|_{\ell^1_{\beta}} \, \left\| \mathbf{p}^m \right\|_{\ell^\infty_{\beta}} \\
&\leq \sum_{m=0}^{n} \dfrac{C_0^2 \, C_\cL(\beta,2\beta) \, C_{\mathscr{N}}(\beta,\varrho)}{(1+n-m)^{\beta+\frac{1}{8}}(1+m)^{2\sigma+\frac{11}{24}}} 
\, \left\| \bfh \right\|_{\ell^1_\gamma}^2\\
&\leq \dfrac{2 \, \epsilon \, C_0^2 \, C_\cL(\beta,2\beta) \, C_{\mathscr{N}}(\beta,\varrho) \, \boldsymbol{C}\left(\beta+\frac{1}{8},2\sigma+\frac{11}{24},\sigma\right)}{(2+n)^\sigma} \, \left\| \bfh \right\|_{\ell^1_\gamma} \, .
\end{align*}
For the last inequality, we have applied Lemma~\ref{lemabc} with $a=\beta+\frac{1}{8}$, $b=2\sigma+\frac{11}{24}$ and $c=\sigma$. We can readily verify that
\bqs
0<\beta+\frac{1}{8}-\sigma=a-c \, ,\quad 0<\sigma+\frac{11}{24}=b-c \, ,\quad \text{ and } \quad 
b-c+a-1=\beta+\sigma+\frac{1}{8}+\frac{11}{24}-1=\sigma+\beta-\frac{5}{12} \ge 0 \, ,
\eqs
thanks to our assumptions on $\beta+\sigma \ge \frac{5}{12}$ and $0\leq \sigma<\beta+\frac{1}{8}$.

As a consequence (recalling the definition of the constant $C_1$), we have obtained that
\bqs
\| \mathbf{p}^{n+1} \|_{\ell^1_\beta} \le \dfrac{1}{(2+n)^\sigma} \, \Big( C_\cL(\beta,\gamma) \, (1+C_m(\gamma)) +\epsilon \, C_1 \Big) 
\, \left\| \bfh \right\|_{\ell^1_\gamma} \le \dfrac{C_0}{(2+n)^\sigma} \, \left\| \bfh \right\|_{\ell^1_\gamma} \, ,
\eqs
thanks to our choice of $C_0$ and the restrictions on $\epsilon$.

\paragraph{The $\ell^\infty_\beta$ estimate.}

Using once more that the initial perturbation $\mathbf{p}^{0}\in\ell^1_\gamma$ is of zero mass, we directly find, using the semi-group estimate of 
$(\cL^n)_{n\in\N}$ in $\ell^\infty_\beta$  that
\bqs
\left\| \cL^{n+1} \, \mathbf{p}^{0} \right\|_{\ell^\infty_{\beta}} \, \le \, 
\dfrac{C_\cL(\beta,\gamma)}{(2+n)^{\sigma+\frac{11}{24}}} \, \|\mathbf{p}^{0}\|_{\ell^1_{\gamma}} \, \le \, 
\dfrac{C_\cL(\beta,\gamma) \, (1+C_m(\gamma))}{(2+n)^{\sigma+\frac{11}{24}}} \, \|\bfh\|_{\ell^1_{\gamma}} \, ,
\eqs
where we have used one more time the following consequence of \eqref{inegalitesp0}:
$$
\| \mathbf{p}^{0} \|_{\ell^1_{\gamma}} \, \le \, (1+C_m(\gamma)) \, \| \bfh \|_{\ell^1_{\gamma}} \, .
$$

Regarding the second term, we use Lemma~\ref{lemLtheta} with $\nu_1=\beta$ and $\nu_2=\sigma+\frac{25}{24}$ to obtain
\begin{align*}
\left\| \cL^{n-m} \left(\mathcal{L}^\theta -\cL\right) \, \mathbf{p}^m\right\|_{\ell^\infty_\beta} \leq 
\left\| \cL^{n-m} \left(\mathcal{L}^\theta -\cL\right) \, \mathbf{p}^m\right\|_{\ell^1_\beta} & 
\leq \, \dfrac{C_\mathcal{L}\left(\beta,\sigma+\frac{25}{24}\right)}{(n-m+1)^{\sigma+\frac{25}{24}}} \, |\theta| \, \|\mathbf{p}^m\|_{\ell^\infty} \\
& \leq \, \dfrac{\epsilon \, C_0 \, C_\mathcal{L}\left(\beta,\sigma+\frac{25}{24}\right)}{(n-m+1)^{\sigma+\frac{25}{24}} \, (1+m)^{\sigma+\frac{11}{24}}} 
\, \|\bfh\|_{\ell^1_\gamma} \, .
\end{align*}
Next, using Lemma~\ref{lemabc} with $a=\sigma+\frac{25}{24}>1$ and $b=c=\sigma+\frac{11}{24}$, noticing that $c<a$, we directly obtain that
\begin{align*}
\left\|\sum_{m=0}^{n} \cL^{n-m} \left( \mathcal{L}^\theta -\cL\right) \, \mathbf{p}^m \right\|_{\ell^\infty_\beta} 
& \le \, \sum_{m=0}^{n} \dfrac{\epsilon \, C_0 \, C_\mathcal{L}\left(\beta,\sigma+\frac{25}{24}\right)}{(n-m+1)^{\sigma+\frac{25}{24}}(1+m)^{\sigma+\frac{11}{24}}} 
\, \| \bfh \|_{\ell^1_\gamma} \\
& \le \, \dfrac{2 \, \epsilon \, C_0 \, C_\mathcal{L}\left(\beta,\sigma+\frac{25}{24}\right) \, \boldsymbol{C} \left(\sigma+\frac{25}{24},\sigma+\frac{11}{24},\sigma+\frac{11}{24}\right)}{(2+n)^{\sigma+\frac{11}{24}}} \, \|\bfh\|_{\ell^1_\gamma} \, .
\end{align*}

For the third term, we shall use the estimate~\eqref{estimLnshift2} of Theorem~\ref{thmlinestim} and Lemma~\ref{estimNtheta} to derive that
\begin{align*}
\left\| \sum_{m=0}^{\left\lfloor\frac{n+1}{2}\right\rfloor} \cL^{n-m}\left(\mathrm{Id}-\mathbf{S}\right) \mathscr{N}^\theta(\mathbf{p}^m)\right\|_{\ell^\infty_\beta} &\leq \sum_{m=0}^{\left\lfloor\frac{n+1}{2}\right\rfloor} \frac{C_\cL(\beta,2\beta)}{(1+n-m)^{\beta+\frac{1}{3}+\min\left(\frac{1}{4},\beta\right)}} \left\|\mathscr{N}^\theta(\mathbf{p}^m)\right\|_{\ell^1_{2\beta}} \\
&\leq \sum_{m=0}^{\left\lfloor\frac{n+1}{2}\right\rfloor} \frac{C_\cL(\beta,2\beta) C_{\mathscr{N}}(\beta,\varrho)}{(1+n-m)^{\beta+\frac{7}{12}}} \left\|\mathbf{p}^m\right\|_{\ell^1_{\beta}}\left\|\mathbf{p}^m\right\|_{\ell^\infty_{\beta}} \\
&\leq \sum_{m=0}^{\left\lfloor\frac{n+1}{2}\right\rfloor} 
\frac{C_0^2 \, C_\cL(\beta,2\beta) \, C_{\mathscr{N}}(\beta,\varrho)}{(1+n-m)^{\beta+\frac{7}{12}}(1+m)^{2\sigma+\frac{11}{24}}} 
\, \left\|\bfh\right\|_{\ell^1_\gamma}^2 \\
&\leq \frac{\epsilon \, C_0^2 \, C_\cL(\beta,2\beta) \, C_{\mathscr{N}}(\beta,\varrho) \, 
\boldsymbol{C} \left(\beta+\frac{7}{12},2\sigma+\frac{11}{24},\sigma+\frac{11}{24}\right)}{(2+n)^{\sigma+\frac{11}{24}}} \, \left\| \bfh \right\|_{\ell^1_\gamma} \, .
\end{align*}
For the last inequality, we have used the first inequality of Lemma~\ref{lemabc} with $a=\beta+\frac{7}{12}$, $b=2\sigma+\frac{11}{24}$ and $\sigma+\frac{11}{24}$, noticing that
\bqs
a-c=\beta+\frac{1}{8}-\sigma>0, \quad \text{ and } \quad  a-c+b-1=\sigma+\beta-\frac{5}{12}\geq 0,
\eqs
since $0\leq \sigma < \beta+\frac{1}{8}$ and $\beta+\sigma \ge \frac{5}{12}$.

For the remaining part of the sum, we use estimate~\eqref{estimLnshift3} of Theorem~\ref{thmlinestim} and Lemma~\ref{estimNtheta} to obtain
\begin{align*}
\left\| \sum_{m=\left\lfloor\frac{n+1}{2}\right\rfloor+1}^n \cL^{n-m}\left(\mathrm{Id}-\mathbf{S}\right) \mathscr{N}^\theta(\mathbf{p}^m)\right\|_{\ell^\infty_\beta} &\leq \sum_{m=\left\lfloor\frac{n+1}{2}\right\rfloor+1}^n \frac{C_\cL(\beta,2\beta)}{(1+n-m)^{\beta+\min\left(\frac{1}{8},\beta-\frac{1}{8}\right)}} \left\|\mathscr{N}^\theta(\mathbf{p}^m)\right\|_{\ell^\infty_{2\beta}} \\
&\leq \sum_{m=\left\lfloor\frac{n+1}{2}\right\rfloor+1}^n \frac{C_\cL(\beta,2\beta) C_{\mathscr{N}}(\beta,\varrho)}{(1+n-m)^{\beta+\frac{1}{8}}}\left\|\mathbf{p}^m\right\|_{\ell^\infty_{\beta}}^2\\
&\leq \sum_{m=\left\lfloor\frac{n+1}{2}\right\rfloor+1}^n \frac{C_\cL(\beta,2\beta) C_{\mathscr{N}}(\beta,\varrho)C_0^2}{(1+n-m)^{\beta+\frac{1}{8}}(1+m)^{2\sigma+\frac{11}{12}}} \left\|\bfh\right\|_{\ell^1_\gamma}^2\\
&\leq \frac{\epsilon \, C_0^2 \, C_\cL(\beta,2\beta) \, C_{\mathscr{N}}(\beta,\varrho) \, 
\boldsymbol{C}\left(\beta+\frac{1}{8},2\sigma+\frac{11}{12},\sigma+\frac{11}{24}\right)}{(2+n)^{\sigma+\frac{11}{24}}} \, \left\| \bfh \right\|_{\ell^1_\gamma} \, .
\end{align*}
For the last inequality, we have used the second inequality of Lemma~\ref{lemabc} with $a=\beta+\frac{1}{8}$, $b=2\sigma+\frac{11}{12}$ and 
$\sigma+\frac{11}{24}$, noticing that
\bqs
b-c>0, \quad \text{ and } \quad  a-c+b-1=\sigma+\beta-\frac{5}{12}\geq 0.
\eqs

As a consequence (recalling our definition for the constant $C_2$), we have obtained that
\bqs
\|\mathbf{p}^{n+1}\|_{\ell^\infty_\beta} \leq \frac{1}{(2+n)^{\sigma+\frac{11}{24}}} \, 
\Big( C_\cL(\beta,\gamma) \, (1+C_m(\gamma))+\epsilon \, C_2 \Big) \, \left\|\bfh\right\|_{\ell^1_\gamma} 
\le \dfrac{C_0}{(2+n)^{\sigma+\frac{11}{24}}} \, \left\| \bfh \right\|_{\ell^1_\gamma} \, .
\eqs
From there, we also deduce that
\bqs
\|\mathbf{p}^{n+1}\|_{\ell^\infty} \leq \|\mathbf{p}^{n+1}\|_{\ell^\infty_\beta} \leq C_0 \, \epsilon \leq \varrho \, .
\eqs
This concludes the proof of Theorem~\ref{thmNLS}.

%%%%%%%%%%%%%%%%%%%%%%%%%%%%%%%%%%%%%%%%%%%%%%%%%%%%%%%%%%%%%%%%%%%%%%%%%%%%
\appendix
\chapter{The exact and approximate Green's functions for the Cauchy problem}
\label{appendixA}

In this appendix, we study the Lax-Wendroff scheme when applied to the linear transport equation:
$$
\partial_t v \, + \, a \, \partial_x v \, = \, 0 \, ,
$$
with $a \neq 0$ and the equation is considered on the whole real line $\R$. As we have already seen earlier in this article, in the linear case, 
the Lax-Wendroff scheme can be recast under the form:
\begin{equation}
\label{transportLW}
\forall \, n \in \N \, ,\quad v^{n+1} \, = \, \overline{\mathscr{L}} \, v^n \, ,
\end{equation}
where for any integer $n \in \N$, $v^n$ denotes the sequence $(v^n_j)_{j \in \Z}$, and $\overline{\mathscr{L}}$ is the discrete convolution 
operator defined on any real or complex valued sequence $v=(v_j)_{j \in \Z}$ by:
$$
\forall \, j \in \Z \, ,\quad (\overline{\mathscr{L}} \, v)_j \, := \, 
v_j-\dfrac{\alpha}{2} \left( v_{j+1}-v_{j-1} \right) + \dfrac{\alpha^2}{2} \left( v_{j+1}-2v_j+v_{j-1}\right) \, ,
$$
where, as in the core of this article, $\alpha$ is a short notation for $\lambda \, a$, $\lambda>0$ denoting the fixed ratio $\Delta t/\Delta x$ 
between the time and space steps. The constant $\alpha$ thus has the sign of the transport velocity $a$. In what follows, we shall apply 
the results below to either $\alpha=\alpha_\ell \in (0,1)$ or $\alpha=\alpha_r \in (-1,0)$, so that the above operator $\overline{\mathscr{L}}$ 
corresponds to either one of the operators $\mathscr{L}_\ell$ or $\mathscr{L}_r$ defined in \eqref{cauchyLW}.
\bigskip

In this appendix, we recall or prove several bounds on the \emph{exact} and \emph{approximate} Green's functions for the Lax-Wendroff scheme. 
The approximate Green's function is defined below in \eqref{A-defGapproxjn} and is meant to reproduce the leading qualitative and quantitative 
features of the exact Green's function of \eqref{transportLW}. The study of the Green's function of \eqref{transportLW} was the purpose of the 
article \cite{jfcAMBP} by one of the authors. We shall recall below the main conclusions of \cite{jfcAMBP} since they are useful for our purpose 
here. Unsurprisingly, many arguments below for studying the approximate Green's function follow what has been already done in \cite{jfcAMBP} 
but there are also several new regimes that need to be considered and that did not appear in \cite{jfcAMBP}. The bounds that we prove below 
(see in particular Theorem \ref{thm-A3}) are used in this article to study the so-called activation function $\mathbf{A}$ in our decomposition 
of the Green's function for the linearized operator $\mathscr{L}$ in \eqref{linear} (see Theorem \ref{thmGreen}) and various bounds for some 
remainder terms. This appendix can also be seen as a main building block for a complete justification of the local limit theorem for finite difference 
approximations of the transport equation that exhibit \emph{dispersion} and \emph{dissipation} (we refer to \cite{Petrov} and \cite{RSF1,RSF2} 
for a presentation and some recent advances on the local limit theorem and its connection to probability theory). The local limit theorem in the 
non-dispersive (or rather parabolic) case is justified in \cite{CF-CRAS} and the complete justification of the local limit theorem in the dispersive 
case is a work in progress. We refer to \cite{RSF1} for a justification of the leading order term in the local limit theorem for sequences that exhibit 
dispersion and dissipation as we consider here.

%%%%%%%%%%%%%%%%%%%%%%%%%%%%%%%%%%
\section{The exact and approximate Green's functions. Main results}
\label{sectionA-1}

As we have recalled above, the Lax-Wendroff scheme for the transport equation on the whole real line reads:
\begin{equation}
\label{A-LWlineaire}
v_j^{n+1} \, = \, v_j^n \, - \, \dfrac{\alpha}{2} \, (v_{j+1}^n-v_{j+1}^n) \, + \, \dfrac{\alpha^2}{2} \, (v_{j+1}^n-2\, v_j^n+v_{j+1}^n) \, ,
\end{equation}
with $\alpha :=\lambda \, a$. The \emph{exact} Green's function for \eqref{A-LWlineaire} corresponds to the initial condition defined by:
$$
\forall \, j \in \Z \, ,\quad \overline{\mathscr{G}}_j^0 \, := \begin{cases}
1 \, ,&\text{\rm if $j=0$,} \\
0 \, ,&\text{\rm otherwise,}
\end{cases}
$$
which leads to the solution $(\overline{\mathscr{G}}_j^n)_{(j,n) \in \Z \times \N}$ for \eqref{A-LWlineaire}. This sequence 
$(\overline{\mathscr{G}}_j^n)_{(j,n) \in \Z \times \N}$ is studied in \cite{jfcAMBP}, see also \cite{hedstrom1,hedstrom2}. Its main features 
are recalled below.

We restrict from now on to the case $\alpha \in (-1,1) \setminus \{ 0 \}$ in order to stick to the two particular situations we are interested in, 
that is, either $\alpha =\alpha_\ell \in (0,1)$ or $\alpha=\alpha_r \in (-1,0)$. We compute the so-called amplification factor for \eqref{A-LWlineaire} 
and obtain (see for instance \cite{jfcAMBP}):
$$
\forall \, \theta \in \R \, ,\quad \widehat{F}_{\rm LW}(\theta) \, = \, 1 \, - \, 2 \, \alpha^2 \, \sin^2 \dfrac{\theta}{2} \, + \, \mbi \, \alpha \, \sin \theta \, .
$$
As already reported in \cite{jfcAMBP}, the expansion of the amplification factor $\widehat{F}_{\rm LW}$ near the frequency $0$ reads:
\begin{equation}
\label{A-TaylorLW}
\widehat{F}_{\rm LW}(\theta) \, = \, \exp \left ( \mathbf{i} \, \alpha \, \theta \, - \, \mathbf{i} \, \dfrac{\alpha \, (1 - \alpha^2)}{6} \, \theta^{\, 3} 
\, - \, \dfrac{\alpha^2 \, (1 - \alpha^2)}{8} \, \theta^{\, 4} \, + \, \mathcal{O}(\theta^{\, 5}) \right) \, ,
\end{equation}
and we have furthermore the dissipation property:
$$
\forall \, \theta \in [-\pi,\pi] \setminus \{ 0 \} \, ,\quad |\widehat{F}_{\rm LW}(\theta)| \, < \, 1 \, .
$$
Since $\widehat{F}_{\rm LW}$ is a trigonometric polynomial, it extends as a holomorphic function with respect to $\theta$ on the whole complex 
plane $\C$. For later use, we introduce the coefficients:
\begin{equation}
\label{A-defc3c4}
c_3 \, := \, \dfrac{\alpha \, (1 - \alpha^2)}{6} \neq 0 \, ,\quad c_4 \, := \, \dfrac{\alpha^2 \, (1 - \alpha^2)}{8} \, > \, 0 \, ,
\end{equation}
in such a way that \eqref{A-TaylorLW} equivalently reads:
$$
\widehat{F}_{\rm LW}(\theta) \, = \, \exp \Big( 
\mathbf{i} \, \alpha \, \theta \, - \, \mathbf{i} \, c_3 \, \theta^3 \, - \, c_4 \, \theta^4 \, + \, \mathcal{O}(\theta^5) \Big) \, ,
$$
as $\theta$ tends to $0$. The main result proved in \cite{jfcAMBP} can be formulated as follows. It uses the notation $c_3$ and $c_4$ for the 
coefficients in \eqref{A-defc3c4} that arise in the Taylor's expansion of the amplification factor $\widehat{F}_{\rm LW}$ at $0$.

\begin{theorem}
\label{thm-A1}
Assume that the constant $c_3$ in \eqref{A-defc3c4} is positive, that is $\alpha \in (0,1)$. Then there exist two constants $C>0$ and $c>0$ such 
that the Green's function $(\overline{\mathscr{G}}_j^n)_{(j,n) \in \Z \times \N}$ for \eqref{A-LWlineaire} satisfies the uniform bounds:
\begin{equation}
\label{thm-A1-bound1}
\forall \, n \in \N^* \, ,\quad \forall \, j \in \Z \, ,\quad \big| \, \overline{\mathscr{G}}_j^n \, \big| \, \le \, \dfrac{C}{n^{\, 1/3}} \, 
\exp \left( - \, c \, \left( \dfrac{j \, - \, \alpha \, n}{n^{\, 1/3}} \right)^{3/2} \, \right) \, ,\quad \text{\rm if } j \, - \, \alpha \, n \, \ge \, 0 \, ,
\end{equation}
and:
\begin{multline}
\label{thm-A1-bound2}
\forall \, n \in \N^* \, ,\quad \forall \, j \in \Z \, ,\\
\big| \, \overline{\mathscr{G}}_j^n \, - \, 2 \, \text{\rm Re } \mathfrak{g}_j^n \, \big| \, \le \, 
\dfrac{C}{n^{\, 1/3}} \, \exp \left( - \, c \, \left( \dfrac{|j \, - \, \alpha \, n|}{n^{\, 1/3}} \right)^{3/2} \, \right) 
\, + \, \dfrac{C}{n^{\, 1/2}} \, \exp \left( - \, c \, \dfrac{(j \, - \, \alpha \, n)^2}{n} \, \right) \, ,\\
\text{\rm if } j \, - \, \alpha \, n \, < \, 0 \, ,
\end{multline}
where $\mathfrak{g}_j^n$ is defined for $n \in \N^*$ and $j \in \Z$ as:
\begin{multline}
\label{thm-A1-profil-principal}
\forall \, (n,j) \in \N^* \times \Z \, ,\quad 
\mathfrak{g}_j^n \, := \, \dfrac{1}{2 \, \pi} \, \exp \left( - \, \dfrac{c_4 \, (j-\alpha \, n)^2}{9 \, c_3^2 \, n} \, \right) \, 
\exp \left( \mathbf{i} \, \dfrac{2 \, |j-\alpha \, n|^{3/2}}{3 \, \sqrt{3 \, |c_3| \, n}} \, - \, \mathbf{i} \, \dfrac{\pi}{4} \right) \\
\times \int_{-\sqrt{\frac{2 \, |j-\alpha \, n|}{3\, |c_3| \, n}}}^{\sqrt{\frac{2 \, |j-\alpha \, n|}{3\, |c_3| \, n}}} \, 
{\rm e}^{- \, \sqrt{3 \, |c_3| \, |j-\alpha \, n| \, n} \, u^2} \, {\rm e}^{c_3 \, n \, {\rm e}^{-\mathbf{i} \pi/4} \, u^3} \, {\rm d}u \, .
\end{multline}

If $c_3$ is negative, the bounds depending on the sign of $j \, - \, \alpha \, n$ should be switched.
\end{theorem}

In order to be absolutely complete, it is important to note that Theorem \ref{thm-A1} above is not exactly the statement that is proved 
in \cite{jfcAMBP}. This is because, in the interval between the completion of \cite{jfcAMBP} and the present work, an error was found 
in \cite{jfcAMBP}, which made us modify the statement and the proof of the main result in \cite{jfcAMBP}. The new statement, that 
is correct, and whose corollaries follow exactly as in \cite{jfcAMBP}, is Theorem \ref{thm-A1} above. The error occurred in the proof 
of the estimate \eqref{thm-A1-bound2} and in the definition \eqref{thm-A1-profil-principal} of the approximation $\mathfrak{g}_j^n$ 
that incorporates the damped oscillations of the Green's function. Rather than reproducing the whole proof of Theorem \ref{thm-A1} 
(only a tiny part of the proof needs to be corrected), we shall rather give the proof of Theorem \ref{thm-A3} below for the 
\emph{approximate} Green's function since the analysis of the approximate Green's function needs to incorporate new regimes that 
were not considered in \cite{jfcAMBP}. The proof of the above estimate \eqref{thm-A1-bound2} corresponds to the bound \eqref{A-bound2} 
below for the approximate Green's function. The reader will most certainly experiment no difficulty to adapt the arguments below to derive 
the statement in \eqref{thm-A1-bound2}, \eqref{thm-A1-profil-principal}.

Let us recall an immediate consequence of Theorem \ref{thm-A1}, see \cite{jfcAMBP} for an even more precise statement.

\begin{corollary}
\label{coro-A1}
Let the constant $c_3$ in \eqref{A-defc3c4} be positive, that is, $\alpha \in (0,1)$. Then there exist two constants $C>0$ and $c>0$ such 
that the Green's function $(\overline{\mathscr{G}}_j^n)_{(j,n) \in \Z \times \N}$ for \eqref{A-LWlineaire} satisfies:
$$
\forall \, n \in \N^* \, ,\quad \big| \overline{\mathscr{G}}_j^n \big| \, \le \, 
C \, \begin{cases}
\dfrac{1}{n^{1/3}} \, \exp \big( -c \, |j -\alpha \, n|^{3/2}/n^{1/2}  \big) \, ,& \text{\rm if $j-\alpha \, n \ge 0$,} \\
 & \\
\dfrac{1}{n^{1/3}} \, ,& \text{\rm if $-n^{1/3} \le j-\alpha \, n \le 0$,} \\
 & \\
\dfrac{1}{|j-\alpha \, n|^{1/4} \, n^{1/4}} \, \exp  \big( -c \, |j-\alpha \, n|^2/n  \big) \, ,& \text{\rm if $j- \alpha \, n \le -n^{1/3}$,}
\end{cases}
$$
together with the $\ell^1$ estimate:
$$
\forall \, n \in \N \, ,\quad \sum_{j \in \Z} \, \big| \overline{\mathscr{G}}_j^n \big| \, \le \, C \, (1+n)^{1/8} \, .
$$

If $c_3$ is negative, that is, $\alpha \in (-1,0)$, the pointwise estimates for the Green's function $(\overline{\mathscr{G}}_j^n)_{(j,n) \in \Z \times \N}$ 
read:
$$
\forall \, n \in \N^* \, ,\quad \big| \overline{\mathscr{G}}_j^n \big| \, \le \, 
C \, \begin{cases}
\dfrac{1}{n^{1/3}} \, \exp \big( -c \, |j -\alpha \, n|^{3/2}/n^{1/2}  \big) \, ,& \text{\rm if $j-\alpha \, n \le 0$,} \\
 & \\
\dfrac{1}{n^{1/3}} \, ,& \text{\rm if $0 \le j-\alpha \, n \le n^{1/3}$,} \\
 & \\
\dfrac{1}{|j-\alpha \, n|^{1/4} \, n^{1/4}} \, \exp  \big( -c \, |j-\alpha \, n|^2/n  \big) \, ,& \text{\rm if $j- \alpha \, n \ge n^{1/3}$,}
\end{cases}
$$
and the $\ell^1$ estimate is unchanged.
\end{corollary}

The analysis in Chapter \ref{chapter4} will use a very slight variation on Theorem \ref{thm-A1} which we state here for convenience. Actually, 
the statement in Theorem \ref{thm-A2} below is the core of the proof of Theorem \ref{thm-A1} in \cite{jfcAMBP} even though this result was 
not stated in such generality in \cite{jfcAMBP}. Once again, the proof of Theorem \ref{thm-A2} below follows from the exact same arguments 
as those we develop below in the proofs of Theorem \ref{thm-A3} and Theorem \ref{thm-A4}. The absorption of the remainder term $\theta^5 
\, \Psi(\theta)$ in the integral is made by choosing $\delta$ small enough, just like we did in the proof of Proposition \ref{propestimR} in 
Chapter \ref{chapter4}. We therefore feel free to skip the proof of Theorem \ref{thm-A2}.

\begin{theorem}
\label{thm-A2}
Let $\tilde{c}_3$ and $\tilde{c}_4$ be two positive real numbers. Let $\Psi$ denote a holomorphic function on some neighborhood of 
$0$ in $\C$. Let $\underline{\mathbf{C}}$ be a positive real number. Then there exists some positive real number $\delta_0>0$ such 
that the following property holds: for any $\delta \in (0,\delta_0]$, there exist two constants $C>0$ and $c>0$ such that there holds:
\begin{multline*}
\forall \, (x,y) \in \R \times [\underline{\mathbf{C}}^{-1},+\infty) \, ,\quad \left| \int_{-\delta}^\delta 
\rme^{\mbi \, x \, \theta + \mbi \, \tilde{c}_3 \, y \, \theta^3 - \tilde{c}_4 \, y \, \theta^4 + y \, \theta^5 \, \Psi(\theta)} \, {\rm d}\theta \right| \\
\le \, C \, \begin{cases}
\dfrac{1}{y^{1/3}} \, \exp \big( -c \, |x|^{3/2}/y^{1/2} \big) \, ,& \text{\rm if $0 \le x \le \underline{\mathbf{C}} \, y$,} \\
 & \\
\dfrac{1}{y^{1/3}} \, ,& \text{\rm if $-y^{1/3} \le x \le 0$,} \\
 & \\
\dfrac{1}{|x|^{1/4} \, y^{1/4}} \, \exp  \big( -c \, |x|^2/y \big) \, ,& \text{\rm if $-\underline{\mathbf{C}} \, y \le x \le -y^{1/3}$.}
\end{cases}
\end{multline*}

If now $\tilde{c}_3$ is negative ($\tilde{c}_4$ being kept positive), the result still holds but with estimates that now read:
\begin{multline*}
\forall \, (x,y) \in \R \times [\underline{\mathbf{C}}^{-1},+\infty) \, ,\quad \left| \int_{-\delta}^\delta 
\rme^{\mbi \, x \, \theta + \mbi \, \tilde{c}_3 \, y \, \theta^3 - \tilde{c}_4 \, y \, \theta^4 + y \, \theta^5 \, \Psi(\theta)} \, {\rm d}\theta \right| \\
\le \, C \, \begin{cases}
\dfrac{1}{y^{1/3}} \, \exp \big( -c \, |x|^{3/2}/y^{1/2} \big) \, ,& \text{\rm if $-\underline{\mathbf{C}} \, y \le x \le 0$,} \\
 & \\
\dfrac{1}{y^{1/3}} \, ,& \text{\rm if $0 \le x \le y^{1/3}$,} \\
 & \\
\dfrac{1}{|x|^{1/4} \, y^{1/4}} \, \exp  \big( -c \, |x|^2/y \big) \, ,& \text{\rm if $y^{1/3} \le x \le \underline{\mathbf{C}} \, y$.}
\end{cases}
\end{multline*}
\end{theorem}
\bigskip

We now turn to the approximate Green's function for \eqref{A-LWlineaire}. Recalling that we have the following formula for the exact 
Green's function:
\begin{equation*}
\forall \, (n,j) \in \N \times \Z \, ,\quad \overline{\mathscr{G}}_j^n \, = \, \dfrac{1}{2 \, \pi} \, 
\int_{-\pi}^\pi {\rm e}^{-\mbi \, j \, \theta} \, \widehat{F}_{\rm LW}(\theta)^n \, {\rm d}\theta \, = \, \dfrac{1}{2 \, \pi} \, 
\int_{-\pi}^\pi {\rm e}^{\mbi \, j \, \theta} \, \widehat{F}_{\rm LW}(-\theta)^n \, {\rm d}\theta \, ,
\end{equation*}
the expansion \eqref{A-TaylorLW} and the dissipation property suggests the introduction of the approximate Green's function 
$\mathbb{G}_j^n$ defined by:
\begin{equation}
\label{A-defGapproxjn}
\forall \, (n,j) \in \N^* \times \Z \, ,\quad \mathbb{G}_j^n \, := \, \dfrac{1}{2 \, \pi} \, 
\int_\R {\rm e}^{\mbi \, (j-\alpha \, n) \, \theta} \, {\rm e}^{\mbi \, c_3 \, n \, \theta^3} \, {\rm e}^{-c_4 \, n \, \theta^4} \, {\rm d}\theta \, .
\end{equation}
The relevance of $\mathbb{G}_j^n$ for analyzing the behavior of the exact Green's function for \eqref{A-LWlineaire} is illustrated 
by numerous simulations in \cite{Bouche-Weens}. A rigorous justification (by means of sharp analytical bounds) of these numerical 
observations is, to some extent, the purpose of the so-called local limit theorem and is the content of a future work by the authors. 
We show below that the bounds for $\mathbb{G}_j^n$ are ``consistent'' with those given in Theorem \ref{thm-A1} for the exact 
Green's function $\overline{\mathscr{G}}_j^n$. Let us quickly observe that the above definition \eqref{A-defGapproxjn} only makes 
sense for $n \in \N^*$ so that the integral is absolutely convergent. We have replaced the compact integration interval $[-\pi,\pi]$ by 
the whole real line $\R$ for convenience since high frequencies will not modify much the behavior of $\mathbb{G}_j^n$.

For technical reasons that will be made more clear below, it is useful to ``extend'' the approximate Green's function to a continuous 
setting and we therefore introduce the function of two variables $\mathbf{G}$ that is defined by:
\begin{equation}
\label{A-defGapproxn}
\forall \, (x,y) \in \R \times \R^{+*} \, ,\quad \mathbf{G}(x,y) \, := \, \dfrac{1}{2 \, \pi} \, 
\int_\R {\rm e}^{\mbi \, x \, \theta} \, {\rm e}^{\mbi \, c_3 \, y \, \theta^3} \, {\rm e}^{-c_4 \, y \, \theta^4} \, {\rm d}\theta \, .
\end{equation}
From the definition \eqref{A-defGapproxjn}, we directly get the relation $\mathbb{G}_j^n=\mathbf{G}(j-\alpha \, n,n)$ so that any detailed 
information or bounds on $\mathbf{G}$ will give us information or bounds on $\mathbb{G}_j^n$.
\bigskip

We thus consider from now on the function $\mathbf{G}$ defined in \eqref{A-defGapproxn}. Our main result, that is Theorem \ref{thm-A3} below, 
makes use of an auxiliary function that, to some extent, captures the oscillating behavior of the function $\mathbf{G}$ on the side where $x$ has 
the opposite sign of $c_3$. This auxiliary function $\mathfrak{g}$ is defined as follows:
\begin{multline}
\label{deffrakg}
\forall \, (x,y) \in \R \times \R^{+*} \, ,\quad \mathfrak{g}(x,y) \, := \, \dfrac{1}{2 \, \pi} \, \exp \left( - \, \dfrac{c_4 \, x^2}{9 \, c_3^2 \, y} \, \right) \, 
\exp \left( \mathbf{i} \, \dfrac{2 \, |x|^{3/2}}{3 \, \sqrt{3 \, |c_3| \, y}} \, - \, \mathbf{i} \, \dfrac{\pi}{4} \right) \\
\times \int_{-\sqrt{\frac{2 \, |x|}{3\, |c_3| \, y}}}^{\sqrt{\frac{2 \, |x|}{3\, |c_3| \, y}}} \, 
{\rm e}^{- \, \sqrt{3 \, |c_3| \, |x| \, y} \, t^2} \, {\rm e}^{c_3 \, y \, {\rm e}^{-\mathbf{i} \pi/4} \, t^3} \, {\rm d}t \, .
\end{multline}
Comparing with \eqref{thm-A1-profil-principal}, we see that \eqref{thm-A1-profil-principal} corresponds to $x=j-\alpha \, n$ and $y=n$ in 
\eqref{deffrakg}. This explains why the proof of the above estimate \eqref{thm-A1-bound2} is entirely similar to the proof of \eqref{A-bound2} 
below. Our main result is then the following.

\begin{theorem}
\label{thm-A3}
Let us assume that the coefficient $c_3$ in \eqref{A-defc3c4} is positive, that is $\alpha \in (0,1)$. Let $y_{\rm min}>0$ and let 
$\underline{\mathbf{c}}>0$ be given. Then there exist some constants\footnote{These constants depend only on $y_{\rm min}$, 
$\underline{\mathbf{c}}$, $c_3$ and $c_4$, or equivalently on $y_{\rm min}$, $\underline{\mathbf{c}}$ and $\alpha$.} $C>0$ and 
$c>0$ such that, for any $(x,y) \in \R \times [y_{\rm min},+\infty)$, there holds:
\begin{equation}
\label{A-bound1}
|\mathbf{G}(x,y)| \, \le \begin{cases}
\dfrac{C}{y^{1/4}} \, \exp (-c \, x^{4/3}/y^{1/3}) \, ,& \text{\rm if $x \ge \underline{\mathbf{c}} \, y$,} \\
\dfrac{C}{y^{1/3}} \, \exp (-c \, x^{3/2}/y^{1/2}) \, ,& \text{\rm if $0 \le x \le \underline{\mathbf{c}} \, y$,} \\
\dfrac{C}{y^{1/3}} \, ,& \text{\rm if $-y^{1/3} \le x \le 0$,} \\
\dfrac{C}{y^{1/4}} \, \exp (-c \, |x|^{4/3}/y^{1/3}) \, ,& \text{\rm if $x \le -\underline{\mathbf{c}} \, y$,}
\end{cases}
\end{equation}
and:
\begin{equation}
\label{A-bound2}
|\mathbf{G}(x,y) \, - \, 2 \, \text{\rm Re } \mathfrak{g}(x,y)| \, \le \, \dfrac{C}{y^{1/3}} \, \exp \left( -c \, \dfrac{|x|^{3/2}}{y^{1/2}} \right) \, + \, 
\dfrac{C}{y^{1/2}} \, \exp \left( -c \, \dfrac{x^2}{y} \right) \, ,
\end{equation}
if $-\underline{\mathbf{c}} \, y \le x \le -y^{1/3}$.

In particular, there exist another constant $\widetilde{C} \ge C$ and another constant $\tilde{c} \in (0,c]$ such that for any $(x,y) \in \R 
\times [y_{\rm min},+\infty)$, there holds:
\begin{equation}
\label{A-bound3}
\left| \int_{-\infty}^x \mathbf{G}(\xi,y) \, {\rm d}\xi \right| \, \le \begin{cases}
\widetilde{C} \, \exp (-\tilde{c} \, |x|^{4/3}/y^{1/3}) \, ,& \text{\rm if $x \le -\underline{\mathbf{c}} \, y$,} \\
\widetilde{C} \, ,& \text{\rm if $-\underline{\mathbf{c}} \, y \le x \le \underline{\mathbf{c}} \, y$.}
\end{cases}
\end{equation}
and:
\begin{equation}
\label{A-bound4}
\left| 1 - \int_{-\infty}^x \mathbf{G}(\xi,y) \, {\rm d}\xi \right| \, \le \, \widetilde{C} \, \exp (-\tilde{c} \, x^{4/3}/y^{1/3}) \, ,\quad 
\text{\rm if $x \ge \underline{\mathbf{c}} \, y$.}
\end{equation}
\end{theorem}

Let us note that the case where $c_3$ is negative (that is, where $\alpha$ is negative) simply corresponds to changing the sign of $x$ since 
the definition \eqref{A-defGapproxn} gives:
$$
\mathbf{G}(-x,y) \, := \, \dfrac{1}{2 \, \pi} \, 
\int_\R {\rm e}^{\mbi \, x \, \theta} \, {\rm e}^{\mbi \, (-c_3) \, y \, \theta^3} \, {\rm e}^{-c_4 \, y \, \theta^4} \, {\rm d}\theta \, .
$$
We shall thus feel free to use without proof the analog of Theorem \ref{thm-A3} in the case where $\alpha$ is negative. (This is precisely what 
we shall do in Chapter \ref{chapter4} where one of the relevant values for $\alpha$ is $\alpha_r \in (-1,0)$.)
\bigskip

As follows from \cite{Thomee}, see also \cite{hedstrom1,hedstrom2}, it is already known that the Green's function 
$(\overline{\mathscr{G}}_j^n)_{(j,n) \in \Z \times \N}$ of the Lax-Wendroff scheme \eqref{A-LWlineaire} is not uniformly bounded in 
$\ell^1 (\Z)$ (where ``uniformly'' refers to uniformity with respect to $n \in \N$). This is a characteristic feature of numerical schemes 
that exhibit \emph{dispersion}. Since the function $\mathbf{G}$ is meant to reproduce the leading behavior of the Green's function, 
there is no hope to show that $\mathbf{G}(\cdot,y)$ belongs to $L^1_x$ uniformly with respect to $y>0$ (or, say, $y \ge y_{\rm min}>0$ 
in order to get rid of potential trouble when $y$ becomes arbitrarily small). However, a crucial point in our analysis of the linearized 
operator $\mathscr{L}$ around our reference discrete shock profile is to obtain a uniform bound for the activation function $\mathbf{A}$. 
As we shall see later on, this activation function corresponds to the primitive function of $\mathbf{G}$ with respect to its first variable. 
Deriving a bound for this primitive function \emph{can not} follow from the triangle inequality. We thus need to capture the leading 
oscillating behavior of $\mathbf{G}$ in order to show that its primitive (with respect to its first variable) is uniformly bounded with 
respect to \emph{both} variables. This is the meaning of the estimates \eqref{A-bound3} and \eqref{A-bound4} and this is where 
the estimate \eqref{A-bound2} is crucial.
\bigskip

The proof of Theorem \ref{thm-A3} is split in several paragraphs in order to emphasize the distinctions between the fast decaying side 
of the Green's function and its oscillating side (as for the Airy function). Most of what follows is an adaptation of the results in \cite{jfcAMBP} 
even though several regimes did not appear in \cite{jfcAMBP} because the exact Green's function was considered there and this one has 
compact support for any $n$. Furthermore, it appears that an error occurred in \cite{jfcAMBP} when introducing the analog of the above 
function $\mathfrak{g}$ and in proving and error bound between the Green's function and its (presumably) leading behavior. This appendix 
is therefore the opportunity to correct this error and to generalize some of the results in \cite{jfcAMBP} to a continuous setting. In the final 
paragraph of this appendix, we connect the result of Theorem \ref{thm-A3} with the analysis of the activation function $\mathbf{A}_r$ (or 
 $\mathbf{A}_\ell$), see Corollary \ref{coro-A6}, and we explain why the bounds in Theorem \ref{thm-A3} give exactly all that is needed in 
 the analysis of Sections \ref{section4-3} and \ref{section4-5}.

%%%%%%%%%%%%%%
\section{The uniform bound}
\label{sectionA-2}

It is sometimes useful to consider $y>0$ in \eqref{A-defGapproxn} and later restrict to $y \ge y_{\rm min}$ as in Theorem \ref{thm-A3}. 
This is mostly what we shall do in what follows, where the final restriction $y \ge y_{\rm min}$ will be used to derive bounds as claimed 
in Theorem \ref{thm-A3}. We first derive a uniform bound for $\mathbf{G}(x,y)$ with respect to $x \in \R$.

\begin{proposition}
\label{prop-A1}
Let the function $\mathbf{G}$ be defined in \eqref{A-defGapproxn} with a nonzero constant $c_3$. Then there exists a constant $C>0$ 
such that for any $y>0$, there holds:
$$
\sup_{x \in \R} \, |\mathbf{G}(x,y)| \, \le \, \dfrac{C}{y^{1/3}} \, .
$$
\end{proposition}

\begin{proof}
The proof of Proposition \ref{prop-A1} is a mere adaptation of the proof of \cite[Proposition 2.3]{jfcAMBP}. We recall the details for 
the sake of completeness. From \cite[Lemma 3.1]{RSF1} and the so-called van der Corput Lemma, we have:
\begin{equation}
\label{A1-vandercorupt}
\left| \, \int_a^b \, {\rm e}^{\, \mathbf{i} \, f(\theta)} \, g(\theta) \, {\rm d}\theta \, \right| \, \le \, 
\dfrac{C_0}{\Big( \min_{\theta \in [a,b]} \, |\, f^{(3)}(\theta) \, | \, \Big)^{\, 1/3}} \, \Big( \, \| \, g \, \|_{L^\infty([a,b])} \, + \, \| \, g' \, \|_{L^1([a,b])} \, \Big) \, ,
\end{equation}
as long as $f$ is real valued and the minimum of $|f^{(3)}|$ on the interval $[a,b]$ is positive. The crucial observation here is that 
the constant $C_0$ in \eqref{A1-vandercorupt} \emph{does not} depend on $a$ nor $b$ (nor $f$ and $g$, of course). The function 
$g$ could be complex valued.

We apply the above inequality to $f(\theta):=x \, \theta +c_3 \, y \, \theta^3$ and $g(\theta):=\exp(-c_4 \, y \, \theta^4)$, so that both 
norms $\| \, g \, \|_{L^\infty([a,b])}$ and $ \| \, g' \, \|_{L^1([a,b])}$ can be bounded uniformly with respect to $a,b$ and $y$. We thus 
obtain the bound:
$$
\left| \int_a^b {\rm e}^{\mbi \, x \, \theta} \, {\rm e}^{\mbi \, c_3 \, y \, \theta^3} \, {\rm e}^{-c_4 \, y \, \theta^4} \, {\rm d}\theta \right| 
\, \le \, \dfrac{C}{y^{1/3}} \, ,
$$
and it remains to pass to the limit in both $a$ and $b$ to conclude, the constant $C$ being independent of the interval $[a,b]$.
\end{proof}

%%%%%%%%%%%%%%%%
\section{The fast decaying side}
\label{sectionA-3}

We assume from now on that the coefficient $c_3$ in \eqref{A-defc3c4} is positive, the case where this coefficient is negative being 
obtained by switching the sign of $x$. Now that the sign of $c_3$ is fixed, we first deal with the case $x \ge 0$.

\begin{proposition}
\label{prop-A2}
Let the function $\mathbf{G}$ be defined in \eqref{A-defGapproxn} and let the constant $c_3$ be positive. Then there exist some constants 
$\mathbf{c}_\sharp>0$, $C>0$ and $c>0$ such that for any $(x,y) \in \R \times \R^{+*}$, there holds:
$$
|\mathbf{G}(x,y)| \, \le \, \begin{cases}
\dfrac{C}{y^{1/4} \, \max (1,x^{1/4})} \, \exp \left(  - \, c \, \dfrac{x^{3/2}}{y^{1/2}} \right) \, ,& \text{\rm if $0 \le x \le \mathbf{c}_\sharp \, y$,} \\
\dfrac{C}{y^{1/2}} \, \exp \left( - \, c \, \dfrac{x^{4/3}}{y^{1/3}} \right) \, ,& \text{\rm if $x \ge \mathbf{c}_\sharp \, y$.}
\end{cases}
$$
\end{proposition}

\begin{proof}
Integrating the holomorphic function:
$$
z \in \C \longmapsto {\rm e}^{\mbi \, x \, z} \, {\rm e}^{\mbi \, c_3 \, y \, z^3} \, {\rm e}^{-c_4 \, y \, z^4}
$$
over a large rectangle and passing to the limit, we easily see that the real line $\R$ over which we integrated it to obtain the 
defining equation \eqref{A-defGapproxn} can be switched to any line $\mathbf{i}\mu+\R$ with $\mu \in \R$. In other words, the 
Cauchy formula yields:
\begin{equation}
\label{formule-Cauchy-G}
\forall \, \mu \in \R \, ,\quad \forall \, (x,y) \in \R \times \R^{+*} \, ,\quad 
\mathbf{G}(x,y) \, = \, \dfrac{1}{2 \, \pi} \, \int_\R {\rm e}^{\mbi \, x \, (\mathbf{i}\mu+\theta)} \, 
{\rm e}^{\mbi \, c_3 \, y \, (\mathbf{i}\mu+\theta)^3} \, {\rm e}^{-c_4 \, y \, (\mathbf{i}\mu+\theta)^4} \, {\rm d}\theta \, .
\end{equation}
Expanding all quantities within the integral and applying the triangle inequality, we obtain the bound:
\begin{equation}
\label{borne-Cauchy-G}
\forall \, \mu \in \R \, ,\quad \forall \, (x,y) \in \R \times \R^{+*} \, ,\quad 
|\mathbf{G}(x,y)| \, \le \, \dfrac{{\rm e}^{-x \, \mu +c_3 \, y \, \mu^3 -c_4 \, y \, \mu^4}}{2 \, \pi} \, 
\int_\R {\rm e}^{-3 \, y \, \mu \, (c_3-2\, c_4 \, \mu) \, \theta^2} \, {\rm e}^{-c_4 \, y \, \theta^4} \, {\rm d}\theta \, .
\end{equation}
This bound is rather crude but it allows to deal with many regimes in $(x,y)$ by optimizing with respect to the free parameter $\mu \in \R$.
\bigskip

Let us now consider $x \ge 0$ and let us choose $\mu=\mu_0 :=(x/(3\, c_3 \, y))^{1/2} \ge 0$. We use the bound \eqref{borne-Cauchy-G} 
and use the obvious inequality $\exp(-c_4 \, y \, \mu_0^4) \le 1$ to obtain:
$$
|\mathbf{G}(x,y)| \, \le \, \dfrac{1}{2 \, \pi} \, \exp \left( -\dfrac{2 \, x^{3/2}}{3 \, \sqrt{3 \, c_3} \, y^{1/2}} \right) \, 
\int_\R {\rm e}^{-3 \, y \, \mu_0 \, (c_3-2\, c_4 \, \mu_0) \, \theta^2} \, {\rm e}^{-c_4 \, y \, \theta^4} \, {\rm d}\theta \, .
$$
We define once and for all the positive constant $\mathbf{c}_\sharp:=3\, c_3^3/(16 \, c_4^2)$. If we consider the regime $0 \le x \le 
\mathbf{c}_\sharp \, y$, then we have $2\, c_4 \, \mu_0 \le c_3/2$, so that we get:
$$
|\mathbf{G}(x,y)| \, \le \, \dfrac{1}{2 \, \pi} \, \exp \left( -\dfrac{2 \, x^{3/2}}{3 \, \sqrt{3 \, c_3} \, y^{1/2}} \right) \, 
\int_\R {\rm e}^{-(3/2) \, c_3 \, y \, \mu_0 \, \theta^2} \, {\rm e}^{-c_4 \, y \, \theta^4} \, {\rm d}\theta \, .
$$
Since $x$ is nonnegative (and therefore $\mu_0$ is nonnegative too), we can either use one exponential term or the other within the integral 
to derive a bound for this integral. In the case $x=0$, only the exponential term $\exp(-c_4 \, y \, \theta^4)$ makes the integral converge. 
Recalling the value $\mu_0 =(x/(3\, c_3 \, y))^{1/2}$, we thus get:
$$
\int_\R {\rm e}^{-(3/2) \, c_3 \, y \, \mu_0 \, \theta^2} \, {\rm e}^{-c_4 \, y \, \theta^4} \, {\rm d}\theta \le \dfrac{C}{y^{1/4} \, \max (1,x^{1/4})} \, ,
$$
with a constant $C$ that is independent of $y$ and $x \in [0,\mathbf{c}_\sharp \, y]$. This yields our first bound (see the statement of 
Proposition \ref{prop-A2}):
\begin{equation}
\label{borne-1-propA2}
\forall \, x \in [0,\mathbf{c}_\sharp \, y] \, ,\quad 
|\mathbf{G}(x,y)| \, \le \, \dfrac{C}{y^{1/4} \, \max (1,x^{1/4})} \, \exp \left( -c \, \dfrac{x^{3/2}}{y^{1/2}} \right) \, .
\end{equation}
\bigskip

We assume from now on $x \ge \mathbf{c}_\sharp \, y>0$ with the above (already fixed) positive constant $\mathbf{c}_\sharp=3\, c_3^3 / 
(16 \, c_4^2)$. We go back to \eqref{borne-Cauchy-G} and restrict from now on to positive values of the parameter $\mu$. We use Young's 
inequality:
$$
6 \, c_4 \, y \, \mu^2 \, \theta^2 \, \le \, c_4 \, y \, \theta^4 \, + \, 9 \, c_4 \, y \, \mu^4 \, ,
$$
to derive the bound:
\begin{equation}
\label{borne-2-propA2}
\forall \, \mu>0 \, ,\quad |\mathbf{G}(x,y)| \, \le \, \dfrac{C}{\sqrt{y \, \mu}} \, \exp \big( -x \, \mu +c_3 \, y \, \mu^3 + 8 \, c_4 \, y \, \mu^4 \big) \, ,
\end{equation}
where the constant $C$ is independent of $x,y$ and $\mu$. Our final choice for $\mu$ will depend on $x$ and $y$ so it is crucial to get 
constants that do not depend on $\mu$ all along.

The function :
\begin{equation}
\label{prop-A2-deff}
f \, : \, \mu \in \R^+ \longmapsto -x \, \mu +c_3 \, y \, \mu^3 + 8 \, c_4 \, y \, \mu^4 \, ,
\end{equation}
is smooth and strictly convex. Since $f(0)=0$, $f'(0)<0$ and $f$ tends to $+\infty$ at $+\infty$, it appears that $f$ attains its unique 
global minimum on $\R^+$ at some $\underline{\mu}>0$ that is characterized by $f'(\underline{\mu})=0$. Multiplying the relation 
$f'(\underline{\mu})=0$ by $\underline{\mu}$, we obtain:
$$
f(\underline{\mu}) \, = \, - 2\, c_3 \, y \, \underline{\mu}^3 -24 \, c_4 \, y \, \underline{\mu}^4 \, \le \, -24 \, c_4 \, y \, \underline{\mu}^4 \, .
$$
Since we have $x/y \ge \mathbf{c}_\sharp$, it is not hard to show that we can choose some small positive constant $\mathbf{c}_\natural>0$, 
whose choice only depends on $c_3,c_4$ $\mathbf{c}_\sharp$, and such that:
$$
f' \left( \mathbf{c}_\natural \, \dfrac{x^{1/3}}{y^{1/3}} \right) <0 \, .
$$
From the strict convexity of $f$, this means that we have $\underline{\mu} \ge \mathbf{c}_\natural \, (x/y)^{1/3}$ for some constant 
$\mathbf{c}_\natural$ that only depends on $c_3$ and $c_4$ (we recall that $\mathbf{c}_\sharp$ only depends on $c_3$ and $c_4$).

Going back to \eqref{borne-2-propA2} and using $\underline{\mu} \ge \mathbf{c}_\natural \, (x/y)^{1/3}$, we thus obtain the bound:
$$
|\mathbf{G}(x,y)| \, \le \, \dfrac{C}{\sqrt{y \, \underline{\mu}}} \, \exp \big( f(\underline{\mu}) \big) 
\, \le \, \dfrac{C}{\sqrt{y \, \underline{\mu}}} \, \exp \big( -24 \, c_4 \, y \, \underline{\mu}^4 \big) 
\, \le \, \dfrac{C}{y^{1/3} \, x^{1/6}} \, \exp \left( -c \, \dfrac{x^{4/3}}{y^{1/3}} \right) \, .
$$
We simplify a little bit more this last bound by using again the inequality $x \ge \mathbf{c}_\sharp \, y$ and this yields the estimate that 
we announced in Proposition \ref{prop-A2}, namely:
$$
|\mathbf{G}(x,y)| \, \le \, \dfrac{C}{y^{1/2}} \, \exp \left( -c \, \dfrac{x^{4/3}}{y^{1/3}} \right) \, .
$$
The proof of Proposition \ref{prop-A2} is now complete.
\end{proof}

\noindent Let us observe that in Proposition \ref{prop-A2}, the parameter $y$ can take arbitrarily small positive values. We now 
restrict to $y \ge y_{\rm min}>0$ in order to prove the result of Theorem \ref{thm-A3}. This will be done in two steps.

\begin{corollary}
\label{coro-A2}
Let the function $\mathbf{G}$ be defined in \eqref{A-defGapproxn} and let $c_3$ be positive. Let $y_{\rm min}>0$ be given. 
Then there exist some constants $\mathbf{c}_\sharp>0$, $C>0$ and $c>0$ such that for any $(x,y) \in \R \times [y_{\rm min},+\infty)$, 
there holds:
$$
|\mathbf{G}(x,y)| \, \le \, \begin{cases}
\dfrac{C}{y^{1/3}} \, \exp \left( -c \, \dfrac{x^{3/2}}{y^{1/2}} \right) \, ,& \text{\rm if $0 \le x \le \mathbf{c}_\sharp \, y$,} \\
\dfrac{C}{y^{1/4}} \, \exp \left( -c \, \dfrac{x^{4/3}}{y^{1/3}} \right) \, ,& \text{\rm if $x \ge \mathbf{c}_\sharp \, y$.}
\end{cases}
$$
\end{corollary}

\begin{proof}
We fix the constant $\mathbf{c}_\sharp>0$ as the one given in Proposition \ref{prop-A2}. Let us first assume $x \ge \mathbf{c}_\sharp \, y$ 
so that Proposition \ref{prop-A2} gives:
$$
|\mathbf{G}(x,y)| \, \le \, \dfrac{C}{y^{1/2}} \, \exp \left( -c \, \dfrac{x^{4/3}}{y^{1/3}} \right) \, .
$$
We then use the bound $y^{1/2} \ge y_{\rm min}^{1/4} \, y^{1/4}$ to conclude. We can now assume $0 \le x \le \mathbf{c}_\sharp \, y$ so that 
Proposition \ref{prop-A2} gives:
$$
|\mathbf{G}(x,y)| \, \le \, \dfrac{C}{y^{1/4} \, \max (1,x^{1/4})} \, \exp \left( -c \, \dfrac{x^{3/2}}{y^{1/2}} \right) \, .
$$
In particular, we have $\max (1,x^{1/4}) \ge x^{1/4}$ so if $x \ge y^{1/3}>0$, we get:
\begin{equation}
\label{ineg-1-coro-A1}
|\mathbf{G}(x,y)| \, \le \, \dfrac{C}{y^{1/4} \, x^{1/4}} \, \exp \left( -c \, \dfrac{x^{3/2}}{y^{1/2}} \right) 
\, \le \, \dfrac{C}{y^{1/3}} \, \exp \left( -c \, \dfrac{x^{3/2}}{y^{1/2}} \right) \, .
\end{equation}
It remains to examine the case $0 \le x \le \min (y^{1/3},\mathbf{c}_\sharp \, y)$. Proposition \ref{prop-A1} gives the bound:
$$
|\mathbf{G}(x,y)| \, \le \, \dfrac{C}{y^{1/3}} \, \le \, \dfrac{C \, {\rm e}^c}{y^{1/3}} \, \exp \left( -c \, \dfrac{x^{3/2}}{y^{1/2}} \right) \, ,
$$
with the same constant $c>0$ as in \eqref{ineg-1-coro-A1}. This gives the expected bound for $|\mathbf{G}(x,y)|$ in the regime 
$0 \le x \le \mathbf{c}_\sharp \, y$. The proof of Corollary \ref{coro-A2} is complete.
\end{proof}

\noindent It appears that allowing us to choose the value of the constant $\mathbf{c}_\sharp$ that separates the two regimes in Corollary 
\ref{coro-A2} will be convenient in the analysis (while in Corollary \ref{coro-A2} the constant $\mathbf{c}_\sharp$ is given and is therefore 
not free to be fixed). We thus prove the following result which is a slight adaptation of Corollary \ref{coro-A2}.

\begin{corollary}
\label{coro-A3}
Let the function $\mathbf{G}$ be defined in \eqref{A-defGapproxn} and let $c_3$ be positive. Let $y_{\rm min}>0$ be given. Then for any 
constant $\underline{\mathbf{c}}>0$, there exist some constants $C>0$ and $c>0$ such that for any $(x,y) \in \R \times [y_{\rm min},+\infty)$, 
there holds:
$$
|\mathbf{G}(x,y)| \, \le \, \begin{cases}
\dfrac{C}{y^{1/3}} \, \exp \left( -c \, \dfrac{x^{3/2}}{y^{1/2}} \right) \, ,& \text{\rm if $0 \le x \le \underline{\mathbf{c}} \, y$,} \\
\dfrac{C}{y^{1/4}} \, \exp \left( -c \, \dfrac{x^{4/3}}{y^{1/3}} \right) \, ,& \text{\rm if $x \ge \underline{\mathbf{c}} \, y$.}
\end{cases}
$$
\end{corollary}

\begin{proof}
We let $\underline{\mathbf{c}}>0$ be given and assume that there holds $\underline{\mathbf{c}} \ge \mathbf{c}_\sharp$ where $\mathbf{c}_\sharp$ 
is given in Corollary \ref{coro-A2}. (The case $\underline{\mathbf{c}} \le \mathbf{c}_\sharp$ is dealt with similarly as below.) Assuming first $x \ge 
\underline{\mathbf{c}} \, y$, we have $x \ge \mathbf{c}_\sharp \, y$ and we can therefore use Corollary \ref{coro-A2} and obtain the desired bound 
for $|\mathbf{G}(x,y)|$. Corollary \ref{coro-A2} also gives the desired bound in the regime $0 \le x \le \mathbf{c}_\sharp \, y$ and it therefore only 
remains to consider the regime $\mathbf{c}_\sharp \, y \le x \le \underline{\mathbf{c}} \, y$. From Corollary \ref{coro-A2}, we already have the bound:
$$
|\mathbf{G}(x,y)| \, \le \, \dfrac{C_1}{y^{1/4}} \, \exp \left( -c_1 \, \dfrac{x^{4/3}}{y^{1/3}} \right) \, ,
$$
for some constants $C_1$ and $c_1$ that are independent of $x$ and $y$. Since we have $\mathbf{c}_\sharp \, y \le x$, we can first use a slight part 
of the exponential term to improve the power of the algebraic factor:
$$
|\mathbf{G}(x,y)| \, \le \, \dfrac{C_1}{y^{1/4}} \, \exp (-c_1 \, \mathbf{c}_\sharp^{4/3} \, y) 
 \, \le \, \dfrac{C_2}{y^{1/3}} \, \exp \left( -\dfrac{c_1 \, \mathbf{c}_\sharp^{4/3}}{2} \, y \right) \, ,
$$
and we now wish to show that the right-hand side of the latter inequality is less than:
$$
\dfrac{C}{y^{1/3}} \, \exp \left( -c \, \dfrac{x^{3/2}}{y^{1/2}} \right)
$$
for some appropriate constants $C$ and $c$ (recalling that $x$ belongs to the interval $[\mathbf{c}_\sharp \, y,\underline{\mathbf{c}} \, y]$). 
We fix $c_2:=c_1 \, \mathbf{c}_\sharp^{4/3}/(2 \, \underline{\mathbf{c}}^{3/2})$ so that we have (we use $x \le \underline{\mathbf{c}} \, y$):
$$
\exp \left( -\dfrac{c_1 \, \mathbf{c}_\sharp^{4/3}}{2} \, y \right) \, \le \, \exp \left( -c_2 \, \dfrac{x^{3/2}}{y^{1/2}} \right) \, ,
$$
and we therefore obtain:
$$
|\mathbf{G}(x,y)| \, \le \, \dfrac{C_2}{y^{1/3}} \, \exp \left( -c_2 \, \dfrac{x^{3/2}}{y^{1/2}} \right) \, ,
$$
as expected. The proof of Corollary \ref{coro-A3} is complete.
\end{proof}

%%%%%%%%%%%%%%
\section{The oscillating side}
\label{sectionA-4}

We still assume that the coefficient $c_3$ in \eqref{A-defc3c4} is positive but we now consider the case where $x$ is negative. 
This corresponds to the oscillating side of the Green's function. The analysis is split in four regimes. The case $|x| \le y^{1/3}$ 
is dealt with straightforwardly with the uniform bound of Proposition \ref{prop-A1}. The second (and most difficult) regime 
corresponds to $x \le -y^{1/3}$ and $|x|/y$ is sufficiently small. The third regime corresponds to the case where $|x|/y$ is 
sufficiently large and the fourth and last regime corresponds to the case where $|x|$ and $y$ are of comparable sizes. In 
the end, we collect all bounds and prove the result announced in Theorem \ref{thm-A3}. We start with the most difficult case 
for which $|x|/y$ is small.

\begin{proposition}
\label{prop-A3}
Let the function $\mathbf{G}$ be defined in \eqref{A-defGapproxn} and let $c_3$ be positive. Let the function $\mathfrak{g}$ be 
defined in \eqref{deffrakg}. Then there exist some constants $\mathbf{c}_\flat>0$, $C>0$ and $c>0$ such that for any $(x,y) \in 
\R \times \R^{+*}$, there holds:
$$
|\mathbf{G}(x,y) \, - \, 2 \, \text{\rm Re } \mathfrak{g}(x,y)| \, \le \, 
\dfrac{C}{y^{1/3} \, \min (1,y^{1/3})} \, \exp \left( -c \, \dfrac{|x|^{3/2}}{y^{1/2}} \right) 
\, + \, \dfrac{C}{y^{1/2}} \, \exp \left( -c \, \dfrac{x^2}{y} \right) \, ,
$$
if $-\mathbf{c}_\flat \, y \le x \le -y^{1/3}$.
\end{proposition}

\begin{proof}
We always consider $x<0$ and introduce the notation $\omega:=|x|/y$. We follow the analysis in \cite{jfcAMBP} with slight 
modifications since this is precisely at this stage that \cite{jfcAMBP} has some incorrect argument. Rather than integrating 
the function:
$$
z \in \C \longmapsto {\rm e}^{\mbi \, x \, z} \, {\rm e}^{\mbi \, c_3 \, y \, z^3} \, {\rm e}^{-c_4 \, y \, z^4}
$$
on a horizontal line, we apply once again the Cauchy formula and use the contour depicted in Figure \ref{fig:contour-oscillations}, that is:
\begin{itemize}
 \item A horizontal half-line from $-\infty+\mathbf{i} \, \sqrt{\omega/(3\, c_3)}$ to $-2 \sqrt{\omega/(3\, c_3)}-(2\, c_4 / (9 \, c_3^2)) 
 \omega +\mathbf{i} \sqrt{\omega/(3\, c_3)}$,
 \item A segment (with slope $-\pi/4$) from the point $-2 \sqrt{\omega/(3\, c_3)}-(2\, c_4 / (9 \, c_3^2)) 
 \omega +\mathbf{i} \sqrt{\omega/(3\, c_3)}$ to the point $-\mathbf{i} (\sqrt{\omega/(3 \, c_3)}+(2 \, c_4 / (9 \, c_3^2)) \omega)$,
 \item A segment (with slope $\pi/4$) from the point $-\mathbf{i} (\sqrt{\omega/(3 \, c_3)}+(2 \, c_4 / (9 \, c_3^2)) \omega)$ to the 
 point $2 \sqrt{\omega/(3\, c_3)}+(2\, c_4 / (9 \, c_3^2)) \omega +\mathbf{i} \sqrt{\omega/(3\, c_3)}$,
 \item A horizontal half-line from $2 \sqrt{\omega/(3\, c_3)}+(2\, c_4 / (9 \, c_3^2)) \omega +\mathbf{i} \sqrt{\omega/(3\, c_3)}$ 
 to $+\infty +\mathbf{i} \sqrt{\omega/(3\, c_3)}$.
\end{itemize}
The corresponding contributions $\varepsilon_1(x,y)$, $\mathscr{H}_\flat(x,y)$, $\mathscr{H}_\sharp(x,y)$ and $\varepsilon_2(x,y)$ 
are reported in Figure \ref{fig:contour-oscillations} and their exact expressions are given below.

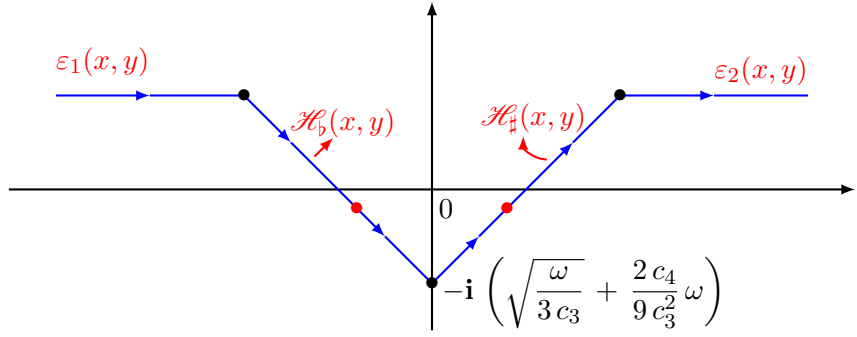
\begin{figure}[h!]
\begin{center}
\begin{tikzpicture}[scale=1.25,>=latex]
\draw[thick,black,->] (-4.5,0) -- (4.5,0);
\draw[thick,black,->] (0,-1.5)--(0,2);
\draw[thick,blue,->] (-4,1) -- (-3,1);
\draw[thick,blue] (-3,1) -- (-2,1);
\draw[thick,blue,->] (-2,1) -- (-1.5,0.5);
\draw[thick,blue] (-1.5,0.5) -- (-1,0);
\draw[thick,blue,->] (-1,0) -- (-0.5,-0.5);
\draw[thick,blue] (-0.5,-0.5) -- (0,-1);
\draw[thick,blue,->] (0,-1) -- (0.5,-0.5);
\draw[thick,blue] (0.5,-0.5) -- (1,0);
\draw[thick,blue,->] (1,0) -- (1.5,0.5);
\draw[thick,blue] (1.5,0.5) -- (2,1);
\draw[thick,blue,->] (2,1) -- (3,1);
\draw[thick,blue] (3,1) -- (4,1);
\draw (-3.5,1.1) node[above]{{\color{red}$\varepsilon_1(x,y)$}};
\draw (3.5,1) node[above]{{\color{red}$\varepsilon_2(x,y)$}};
\draw (0,-1.1) node[right]{$-\mathbf{i} \, \left( \sqrt{\dfrac{\omega}{3 \, c_3}} \, + \, \dfrac{2 \, c_4}{9 \, c_3^2} \, \omega \right)$};
\draw (0.15,0) node[below]{$0$};
\node (centre) at (-2,1){$\bullet$};
\node (centre) at (2,1){$\bullet$};
\node (centre) at (0,-1){$\bullet$};
\node[red] (centre) at (-0.8,-0.2){$\bullet$};
\node[red] (centre) at (0.8,-0.2){$\bullet$};
\draw (-0.95,0.45) node[above]{{\color{red}$\mathscr{H}_\flat(x,y)$}};
\draw (1.08,0.48) node[above]{{\color{red}$\mathscr{H}_\sharp(x,y)$}};
\draw[thick,red,->] (-1.25,0.35) -- (-1.05,0.55);
\draw[thick,red,->] (1.22,0.32) arc (270:180:0.25);
\end{tikzpicture}
\caption{The integration contour in the case $x \le -y^{1/3}$ and $|x|/y$ small. The two red bullets represent the approximate saddle points of 
the phase $-\mbi \, \omega \, z +\mbi \, c_3 \, z^3 -c_4 \, z^4$ and the black bullets represent the end points of the contours along which we 
compute the contributions $\varepsilon_1(x,y)$, $\varepsilon_2(x,y)$, $\mathscr{H}_\flat(x,y)$, $\mathscr{H}_\sharp(x,y)$.}
\label{fig:contour-oscillations}
\end{center}
\end{figure}

In this way, we obtain the decomposition:
\begin{equation}
\label{propA3-decomposition}
\mathbf{G}(x,y) \, = \, \varepsilon_1(x,y) \, + \, \varepsilon_2(x,y) \, + \, \mathscr{H}_\flat(x,y) \, + \, \mathscr{H}_\sharp(x,y) \, ,
\end{equation}
where the four contributions $\varepsilon_1(x,y)$, $\varepsilon_2(x,y)$, $\mathscr{H}_\flat(x,y)$ and $\mathscr{H}_\sharp(x,y)$ have the 
following expressions:
\begin{multline*}
\varepsilon_1(x,y) \, = \, \dfrac{{\rm e}^{-x \, \mu+c_3 \, y \, \mu^3-c_4 \, y \, \mu^4}}{2 \, \pi} \, 
\int_{-\infty}^{-\Xi(\omega)} {\rm e}^{\mbi \, x \, \theta +\mbi \, c_3 \, y \, \theta^3 -\mbi \, 3 \, c_3 \, y \, \mu^2 \, \theta 
+\mbi \, 4 \, c_4 \, y \, \mu^3 \, \theta -\mbi \, 4 \, c_4 \, y \, \mu \, \theta^3} \, \\
\times {\rm e}^{-3 \, c_3 \, y \, \mu \, \theta^2+6 \, c_4 \, y \, \mu^2 \, \theta^2} \, {\rm e}^{-c_4 \, y \, \theta^4} \, {\rm d}\theta \, ,
\end{multline*}
where the upper bound $\Xi(\omega)$ and the parameter $\mu$ in the expression of $\varepsilon_1(x,y)$ are defined as:
\begin{equation}
\label{defXiomega}
\Xi(\omega) \, := \, 2 \, \sqrt{\dfrac{\omega}{3\, c_3}} \, + \, \dfrac{2\, c_4}{9 \, c_3^2} \, \omega \, ,\quad 
\mu \, := \, \sqrt{\dfrac{\omega}{3\, c_3}} \, .
\end{equation}
The contribution $\varepsilon_2(x,y)$ has the following expression (with the same definition for $\Xi(\omega)$ and $\mu$):
\begin{multline}
\label{propA3-epsilon2}
\varepsilon_2(x,y) \, = \, \dfrac{{\rm e}^{-x \, \mu+c_3 \, y \, \mu^3-c_4 \, y \, \mu^4}}{2 \, \pi} \, 
\int_{\Xi(\omega)}^{+\infty} {\rm e}^{\mbi \, x \, \theta +\mbi \, c_3 \, y \, \theta^3 -\mbi \, 3 \, c_3 \, y \, \mu^2 \, \theta 
+\mbi \, 4 \, c_4 \, y \, \mu^3 \, \theta -\mbi \, 4 \, c_4 \, y \, \mu \, \theta^3} \, \\
\times {\rm e}^{-3 \, c_3 \, y \, \mu \, \theta^2+6 \, c_4 \, y \, \mu^2 \, \theta^2} \, {\rm e}^{-c_4 \, y \, \theta^4} \, {\rm d}\theta 
\, = \, \overline{\varepsilon_1(x,y)} \, .
\end{multline}
The contribution $\mathscr{H}_\flat(x,y)$ is given by:
\begin{equation}
\label{propA3-Hbemol}
\mathscr{H}_\flat(x,y) \, = \, \dfrac{{\rm e}^{-\mbi \, \pi/4}}{2 \, \pi} \, 
\int_{-\sqrt{\frac{2 \, \omega}{3 \, c_3}}-\frac{2 \, \sqrt{2} \, c_4}{9 \, c_3^2} \, \omega}^{\sqrt{\frac{2 \, \omega}{3 \, c_3}}} 
{\rm e}^{-\mbi \, \omega \, y \, \Theta(t) +\mbi \, c_3 \, y \, \Theta(t)^3 -c_4 \, y \, \Theta(t)^4} \, {\rm d}t \, ,
\end{equation}
with:
\begin{equation}
\label{defThetat}
\forall \, t \in \left[ -\sqrt{\dfrac{2 \, \omega}{3 \, c_3}}-\dfrac{2 \, \sqrt{2} \, c_4}{9 \, c_3^2} \, \omega,\sqrt{\dfrac{2 \, \omega}{3 \, c_3}} \right] \, ,
\quad \Theta(t) \, := \, -\sqrt{\dfrac{\omega}{3 \, c_3}}-\dfrac{2 \, c_4}{9 \, c_3^2} \, \omega +t \, {\rm e}^{-\mbi \, \pi/4} \, ,
\end{equation}
which corresponds to the parametrization of the first slanted segment in Figure \ref{fig:contour-oscillations} (the one with slope 
$-\pi/4$). The parametrization of the second segment is entirely similar and we obtain the expression:
$$
\mathscr{H}_\sharp(x,y) \, = \, \overline{\mathscr{H}_\flat(x,y)} \, .
$$

Let us start with the estimate of $\varepsilon_2(x,y)$ whose expression is given in \eqref{propA3-epsilon2}. By using the expression for the 
parameter $\mu$ and recalling that we have $x=-\omega \, y$, we get:
$$
\varepsilon_2(x,y) \, = \, \dfrac{1}{2 \, \pi} \, \exp \left( \dfrac{4 \, \omega^{3/2} \, y}{3 \, \sqrt{3\, c_3}} -\dfrac{c_4}{9\, c_3^2} \, \omega^2 \, y \right) 
\, \int_{\Xi(\omega)}^{+\infty} {\rm e}^{\mbi \, \cdots} \times 
{\rm e}^{-\sqrt{3 \, c_3 \, \omega} \, y \, \theta^2 +\frac{2 \, c_4}{c_3} \, \omega \, y \, \theta^2} \, {\rm e}^{-c_4 \, y \, \theta^4} \, {\rm d}\theta \, ,
$$
where the three dots within the integral stand for a \emph{real} quantity whose precise expression is useless since we are going to 
use the triangle inequality so that the modulus of this exponential term will be bounded by $1$. Indeed, applying the triangle inequality 
yields the bound:
\begin{align*}
|\varepsilon_2(x,y)| \, \le& \, \dfrac{1}{2 \, \pi} \, \exp \left( \dfrac{4 \, \omega^{3/2} \, y}{3 \, \sqrt{3\, c_3}} -\dfrac{c_4}{9\, c_3^2} \, \omega^2 \, y \right) 
\, \int_{\Xi(\omega)}^{+\infty} 
{\rm e}^{-\sqrt{3 \, c_3 \, \omega} \, y \, \theta^2 +\frac{2 \, c_4}{c_3} \, \omega \, y \, \theta^2} \, {\rm e}^{-c_4 \, y \, \theta^4} \, {\rm d}\theta \\
\le& \, \dfrac{1}{2 \, \pi} \, \exp \left( \dfrac{4 \, \omega^{3/2} \, y}{3 \, \sqrt{3\, c_3}} \right) \, \int_{\Xi(\omega)}^{+\infty} 
{\rm e}^{-\sqrt{3 \, c_3 \, \omega} \, y \, \theta^2 +\frac{2 \, c_4}{c_3} \, \omega \, y \, \theta^2} \, {\rm e}^{-c_4 \, y \, \theta^4} \, {\rm d}\theta \, .
\end{align*}
We first restrict $\omega=|x|/y$ by imposing:
\begin{equation}
\label{propA3-restriction1}
\omega \, \le \, \dfrac{3\, c_3^3}{16 \, c_4^2} \, ,
\end{equation}
so that we have:
$$
\dfrac{2 \, c_4}{c_3} \, \omega \, \le \, \dfrac{1}{2} \, \sqrt{3 \, c_3 \, \omega} \, .
$$
We thus get:
\begin{align*}
|\varepsilon_2(x,y)| \, &\le \, \dfrac{1}{2 \, \pi} \, \exp \left( \dfrac{4 \, \omega^{3/2} \, y}{3 \, \sqrt{3\, c_3}} \right) 
\, \int_{\Xi(\omega)}^{+\infty} {\rm e}^{-\frac{\sqrt{3 \, c_3 \, \omega}}{2} \, y \, \theta^2}  \, {\rm e}^{-c_4 \, y \, \theta^4} \, {\rm d}\theta \\
&\le \, \dfrac{1}{2 \, \pi} \, \exp \left( \dfrac{4 \, \omega^{3/2} \, y}{3 \, \sqrt{3\, c_3}} \right) 
\, \int_{\Xi(\omega)}^{+\infty} {\rm e}^{-\frac{\sqrt{3 \, c_3 \, \omega}}{2} \, y \, \theta^2}  \, {\rm d}\theta \\
&\le \, \dfrac{1}{2 \, \pi} \, \exp \left( \dfrac{4 \, \omega^{3/2} \, y}{3 \, \sqrt{3\, c_3}} \right) 
\, \int_{2 \, \sqrt{\frac{\omega}{3\, c_3}}}^{+\infty} {\rm e}^{-\frac{\sqrt{3 \, c_3 \, \omega}}{2} \, y \, \theta^2} \, {\rm d}\theta \, ,
\end{align*}
where the final inequality comes from the definition of $\Xi(\omega)$. We then use the inequality:
$$
\forall \, a>0 \, ,\quad \forall \, X>0 \, ,\quad \int_X^{+\infty} {\rm e}^{-a \, \theta^2} \, {\rm d}\theta \, \le \, \dfrac{1}{2\, a \, X} \, {\rm e}^{-a \, X^2} \, ,
$$
to obtain (for a suitable constant $C$):
$$
|\varepsilon_2(x,y)| \, \le \, \dfrac{C}{\omega \, y} \, \exp \left( \dfrac{4 \, \omega^{3/2} \, y}{3 \, \sqrt{3\, c_3}} \right) 
\, \exp \left( -\dfrac{2 \, \omega^{3/2} \, y}{\sqrt{3\, c_3}} \right) 
\, = \, \dfrac{C}{\omega \, y} \, \exp \left( -\dfrac{2 \, \omega^{3/2} \, y}{3 \, \sqrt{3\, c_3}} \right) \, ,
$$
and this gives, going back to $x$ and $y$, the estimate:
$$
|\varepsilon_2(x,y)| \, \le \, \dfrac{C}{|x|} \, \exp \left( -c \, \dfrac{|x|^{3/2}}{y^{1/2}} \right) \, .
$$
Restricting to the regime $x\le -y^{1/3}$, we obtain the final estimate:
\begin{equation}
\label{propA3-estimation1}
|\varepsilon_2(x,y)| \, \le \, \dfrac{C}{y^{1/3}} \, \exp \left( -c \, \dfrac{|x|^{3/2}}{y^{1/2}} \right) \, ,
\end{equation}
as long as we have $x \le -y^{1/3}$ and $|x|/y \le 3\, c_3^3/(16 \, c_4^2)$. Of course a similar estimate holds for $\varepsilon_1(x,y)$ 
since it is the complex conjugate of $\varepsilon_2(x,y)$.
\bigskip

We now turn to the contribution $\mathscr{H}_\flat(x,y)$ in \eqref{propA3-Hbemol}. With the definition \eqref{defThetat} for $\Theta(t)$, 
we expand:
$$
-\mbi \, \omega \, y \, \Theta(t) +\mbi \, c_3 \, y \, \Theta(t)^3 -c_4 \, y \, \Theta(t)^4 \, = \, y \, \sum_{k=0}^4 p_k(\omega) \, t^k \, ,
$$
where the quantities $p_0(\omega),\dots,p_4(\omega)$ have the following behavior\footnote{Here we see that our choice of contour only 
uses approximate saddle points because $p_1(\omega)$ is not zero. This does not matter so much since $p_1(\omega)$ will be small 
enough to produce error terms that can be controlled in our analysis.} as $\omega>0$ tends to zero:
\begin{subequations}
\label{propertiespj}
\begin{align}
\text{\rm Re } p_0(\omega) \, &= \, - \, \dfrac{c_4}{9 \, c_3^2} \, \omega^2 \, + \, O(\omega^3) \, ,\label{partiereellep0} \\
\text{\rm Im } p_0(\omega) \, &= \, \dfrac{2}{3 \, \sqrt{3 \, c_3}} \, \omega^{3/2} \, + \, O(\omega^{5/2}) \, ,\label{partieimagp0} \\
p_1(\omega) \, &= \, O(\omega^2) \, ,\label{p1} \\
\text{\rm Re } p_2(\omega) \, &= \, - \, \sqrt{3 \, c_3 \, \omega} \, + \, O(\omega^{3/2}) \, ,\label{partiereellep2} \\
\text{\rm Im } p_2(\omega) \, &= \, O(\omega) \, ,\label{partieimagp2} \\
p_3(\omega) \, &= \, c_3 \, {\rm e}^{\, - \, \mathbf{i} \, \pi / 4} \, + \, O(\omega^{\, 1/2}) \, ,\label{p3} \\
p_4(\omega) \, &= \, c_4 \, .\label{p4}
\end{align}
\end{subequations}
We thus have:
\begin{equation*}
\mathscr{H}_\flat(x,y) \, = \, \dfrac{{\rm e}^{p_0(\omega) \, y-\mbi \, \pi/4}}{2 \, \pi} \, 
\int_{-\sqrt{\frac{2 \, \omega}{3 \, c_3}}-\frac{2 \, \sqrt{2} \, c_4}{9 \, c_3^2} \, \omega}^{\sqrt{\frac{2 \, \omega}{3 \, c_3}}} 
\exp \Big( \sum_{k=1}^4 p_k(\omega) \, y \, t^k \Big) \, {\rm d}t \, .
\end{equation*}
We now use the inequality:
\begin{equation}
\label{inegexp}
\forall \, z \in \C \, ,\quad |{\rm e}^z -1| \, \le \, |z| \, {\rm e}^{|z|} \, ,
\end{equation}
with the specific value:
$$
z \, := \, y \, \Big\{ p_1(\omega) \, t \, + \, (p_2(\omega) +\sqrt{3 \, c_3 \, \omega}) \, t^2 \, + \, 
(p_3(\omega) -c_3 \, {\rm e}^{\, - \, \mathbf{i} \, \pi / 4}) \, t^3 \, + \, p_4(\omega) \, t^4 \Big\} \, .
$$
On the considered integration interval for $t$, we thus have (see \eqref{p1}, \eqref{partiereellep2}, \eqref{partieimagp2}, \eqref{p3}, \eqref{p4}):
$$
|z| \, \le \, C \, y \, \big( \omega^{5/2} \, + \, \omega \, t^2 \big) \, ,
$$
with a constant $C$ that is independent of $y$, $\omega$ and $t$ (for the relevant values of $t$). This gives the estimate:
\begin{multline*}
\left| \mathscr{H}_\flat(x,y) \, - \, \dfrac{{\rm e}^{p_0(\omega) \, y-\mbi \, \pi/4}}{2 \, \pi} \, 
\int_{-\sqrt{\frac{2 \, \omega}{3 \, c_3}}-\frac{2 \, \sqrt{2} \, c_4}{9 \, c_3^2} \, \omega}^{\sqrt{\frac{2 \, \omega}{3 \, c_3}}} 
\exp \Big( - \sqrt{3 \, c_3 \, \omega} \, y \, t^2 +\mbi \, c_3 \, {\rm e}^{\, - \, \mathbf{i} \, \pi / 4} \, y \, t^3 \Big) \, {\rm d}t \right| \\
\le C \, y \, {\rm e}^{\text{\rm Re } p_0(\omega) \, y +C \, \omega^{5/2} \, y} \, 
\int_{-\sqrt{\frac{2 \, \omega}{3 \, c_3}}-\frac{2 \, \sqrt{2} \, c_4}{9 \, c_3^2} \, \omega}^{\sqrt{\frac{2 \, \omega}{3 \, c_3}}} 
\big( \omega^{5/2} \, + \, \omega \, t^2 \big) \, 
{\rm e}^{- \sqrt{3 \, c_3 \, \omega} \, y \, t^2 +\frac{c_3}{\sqrt{2}} \, y \, t^3 +C \, \omega \, y \, t^2} \, {\rm d}t \, .
\end{multline*}
We first use \eqref{partiereellep0} to estimate the real part of $p_0(\omega)$ and choose $\omega$ small enough so that we have:
$$
\text{\rm Re } p_0(\omega) +C \, \omega^{5/2} \, \le \, - c \, \omega^2 \, ,
$$
for a suitable constant $c>0$ that does not depend on $\omega$. We thus get the estimate:
\begin{multline*}
\left| \mathscr{H}_\flat(x,y) \, - \, \dfrac{{\rm e}^{p_0(\omega) \, y-\mbi \, \pi/4}}{2 \, \pi} \, 
\int_{-\sqrt{\frac{2 \, \omega}{3 \, c_3}}-\frac{2 \, \sqrt{2} \, c_4}{9 \, c_3^2} \, \omega}^{\sqrt{\frac{2 \, \omega}{3 \, c_3}}} 
\exp \Big( - \sqrt{3 \, c_3 \, \omega} \, y \, t^2 +\mbi \, c_3 \, {\rm e}^{\, - \, \mathbf{i} \, \pi / 4} \, y \, t^3 \Big) \, {\rm d}t \right| \\
\le C \, y \, {\rm e}^{-c \, \omega^2 \, y} \, 
\int_{-\sqrt{\frac{2 \, \omega}{3 \, c_3}}-\frac{2 \, \sqrt{2} \, c_4}{9 \, c_3^2} \, \omega}^{\sqrt{\frac{2 \, \omega}{3 \, c_3}}} 
\big( \omega^{5/2} \, + \, \omega \, t^2 \big) \, 
{\rm e}^{-\sqrt{3 \, c_3 \, \omega} \, y \, t^2 +\frac{c_3}{\sqrt{2}} \, y \, t^3 +C \, \omega \, y \, t^2} \, {\rm d}t \, .
\end{multline*}
As in \cite{jfcAMBP}, let us now observe that when $t$ is nonpositive in the integral, we have $c_3 \, t^3 \le 0$ (recall that $c_3$ is positive), 
and when $t$ is positive, on the considered integration interval, we have:
$$
\dfrac{c_3}{\sqrt{2}} \, t^3 \, \le \, \sqrt{\dfrac{c_3 \, \omega}{3}} \, t^2 \, .
$$
If $\omega$ is chosen sufficiently small (and this choice is, of course, independent of $y$), we thus get:
$$
{\rm e}^{-\sqrt{3 \, c_3 \, \omega} \, y \, t^2 +\frac{c_3}{\sqrt{2}} \, y \, t^3 +C \, \omega \, y \, t^2} \, \le \, 
{\rm e}^{-c \, \sqrt{\omega} \, y \, t^2} \, ,
$$
for the relevant values of $t$. This gives the estimate:
\begin{multline*}
\left| \mathscr{H}_\flat(x,y) \, - \, \dfrac{{\rm e}^{p_0(\omega) \, y-\mbi \, \pi/4}}{2 \, \pi} \, 
\int_{-\sqrt{\frac{2 \, \omega}{3 \, c_3}}-\frac{2 \, \sqrt{2} \, c_4}{9 \, c_3^2} \, \omega}^{\sqrt{\frac{2 \, \omega}{3 \, c_3}}} 
\exp \Big( - \sqrt{3 \, c_3 \, \omega} \, y \, t^2 +\mbi \, c_3 \, {\rm e}^{\, - \, \mathbf{i} \, \pi / 4} \, y \, t^3 \Big) \, {\rm d}t \right| \\
\le C \, y \, {\rm e}^{-c \, \omega^2 \, y} \, 
\int_{-\sqrt{\frac{2 \, \omega}{3 \, c_3}}-\frac{2 \, \sqrt{2} \, c_4}{9 \, c_3^2} \, \omega}^{\sqrt{\frac{2 \, \omega}{3 \, c_3}}} 
\big( \omega^{5/2} \, + \, \omega \, t^2 \big) \, {\rm e}^{-c \, \sqrt{\omega} \, y \, t^2} \, {\rm d}t \, .
\end{multline*}
Integrating now with respect to $t$ (over $\R$ rather than over the above compact interval), we thus have:
\begin{multline*}
\left| \mathscr{H}_\flat(x,y) \, - \, \dfrac{{\rm e}^{p_0(\omega) \, y-\mbi \, \pi/4}}{2 \, \pi} \, 
\int_{-\sqrt{\frac{2 \, \omega}{3 \, c_3}}-\frac{2 \, \sqrt{2} \, c_4}{9 \, c_3^2} \, \omega}^{\sqrt{\frac{2 \, \omega}{3 \, c_3}}} 
\exp \Big( - \sqrt{3 \, c_3 \, \omega} \, y \, t^2 +\mbi \, c_3 \, {\rm e}^{\, - \, \mathbf{i} \, \pi / 4} \, y \, t^3 \Big) \, {\rm d}t \right| \\
\le C \, y^{1/2} \, \omega^{9/4} \, {\rm e}^{-c \, \omega^2 \, y} \, + \, \dfrac{C}{\sqrt{y}} \, \omega^{1/4} \, {\rm e}^{-c \, \omega^2 \, y} 
\, \le \, \dfrac{C}{\sqrt{y}} \, (1+\omega^2 \, y) \, {\rm e}^{-c \, \omega^2 \, y} \, .
\end{multline*}
Up to diminishing the constant $c>0$ and increasing the constant $C$, we thus get the final estimate (recall $\omega=|x|/y$):
\begin{multline}
\label{propA3-estimation2}
\left| \mathscr{H}_\flat(x,y) \, - \, \dfrac{{\rm e}^{p_0(\omega) \, y-\mbi \, \pi/4}}{2 \, \pi} \, 
\int_{-\sqrt{\frac{2 \, \omega}{3 \, c_3}}-\frac{2 \, \sqrt{2} \, c_4}{9 \, c_3^2} \, \omega}^{\sqrt{\frac{2 \, \omega}{3 \, c_3}}} 
\exp \Big( - \sqrt{3 \, c_3 \, \omega} \, y \, t^2 +\mbi \, c_3 \, {\rm e}^{\, - \, \mathbf{i} \, \pi / 4} \, y \, t^3 \Big) \, {\rm d}t \right| \\
\le \, \dfrac{C}{y^{1/2}} \, \exp \left( -c \, \dfrac{x^2}{y} \right) \, .
\end{multline}

We now use the inequality \eqref{inegexp} one more time with the value:
$$
z \, := \, y \, 
\Big\{ p_0(\omega) \, + \, \dfrac{c_4}{9 \, c_3^2} \, \omega^2 \, - \, \mbi \, \dfrac{2}{3 \, \sqrt{3 \, c_3}} \, \omega^{3/2} \Big\} \, ,
$$
which gives (use \eqref{partiereellep0} and \eqref{partieimagp0}):
$$
\left| {\rm e}^{p_0(\omega) \, y} \, - \, {\rm e}^{-\frac{c_4}{9 \, c_3^2} \, \omega^2 \, y} \, 
{\rm e}^{\mbi \, \frac{2}{3 \, \sqrt{3 \, c_3}} \, \omega^{3/2} \, y} \right| \, \le \, 
C \, \omega^{5/2} \, y \, {\rm e}^{C \, \omega^{5/2} \, y} \, .
$$
We can then follow the same kind of estimate as just above and derive the estimate (combining with \eqref{propA3-estimation2}):
\begin{multline}
\label{propA3-estimation3}
\left| \mathscr{H}_\flat(x,y) \, - \, \dfrac{{\rm e}^{-\frac{c_4 \, x^2}{9 \, c_3^2 \, y}} \, 
{\rm e}^{\mbi \, \frac{2 \, |x|^{3/2}}{3 \, \sqrt{3 \, c_3} \, y^{1/2}} -\mbi \, \pi/4}}{2 \, \pi} \, 
\int_{-\sqrt{\frac{2 \, \omega}{3 \, c_3}}-\frac{2 \, \sqrt{2} \, c_4}{9 \, c_3^2} \, \omega}^{\sqrt{\frac{2 \, \omega}{3 \, c_3}}} 
\exp \Big( - \sqrt{3 \, c_3 \, \omega} \, y \, t^2 +\mbi \, c_3 \, {\rm e}^{\, - \, \mathbf{i} \, \pi / 4} \, y \, t^3 \Big) \, {\rm d}t \right| \\
\le \, \dfrac{C}{y^{1/2}} \, \exp \left( -c \, \dfrac{x^2}{y} \right) \, ,
\end{multline}
provided that $\omega$ is sufficiently small.

If we compare with the definition \eqref{deffrakg} of $\mathfrak{g}$, we see that we have almost made the quantity $\mathfrak{g}(x,y)$ 
appear, except for the fact that the integration interval is not exactly the good one since it is not symmetric. This final estimate is not more 
difficult than the previous ones. Namely, we easily estimate:
\begin{multline*}
\left| \dfrac{{\rm e}^{-\frac{c_4 \, x^2}{9 \, c_3^2 \, y}} \, 
{\rm e}^{\mbi \, \frac{2 \, |x|^{3/2}}{3 \, \sqrt{3 \, c_3} \, y^{1/2}} -\mbi \, \pi/4}}{2 \, \pi} \, 
\int_{-\sqrt{\frac{2 \, \omega}{3 \, c_3}}-\frac{2 \, \sqrt{2} \, c_4}{9 \, c_3^2} \, \omega}^{-\sqrt{\frac{2 \, \omega}{3 \, c_3}}} 
\exp \Big( - \sqrt{3 \, c_3 \, \omega} \, y \, t^2 +\mbi \, c_3 \, {\rm e}^{\, - \, \mathbf{i} \, \pi / 4} \, y \, t^3 \Big) \, {\rm d}t \right| \\
\le \, C \, {\rm e}^{-c \, \omega^2 \, y} \, 
\int_{-\sqrt{\frac{2 \, \omega}{3 \, c_3}}-\frac{2 \, \sqrt{2} \, c_4}{9 \, c_3^2} \, \omega}^{-\sqrt{\frac{2 \, \omega}{3 \, c_3}}} 
\exp \Big( - \sqrt{3 \, c_3 \, \omega} \, y \, t^2 +\frac{c_3}{\sqrt{2}} \, y \, t^3 \Big) \, {\rm d}t \, .
\end{multline*}
Since $t$ is negative on the considered interval, we have:
\begin{multline*}
\left| \dfrac{{\rm e}^{-\frac{c_4 \, x^2}{9 \, c_3^2 \, y}} \, 
{\rm e}^{\mbi \, \frac{2 \, |x|^{3/2}}{3 \, \sqrt{3 \, c_3} \, y^{1/2}} -\mbi \, \pi/4}}{2 \, \pi} \, 
\int_{-\sqrt{\frac{2 \, \omega}{3 \, c_3}}-\frac{2 \, \sqrt{2} \, c_4}{9 \, c_3^2} \, \omega}^{-\sqrt{\frac{2 \, \omega}{3 \, c_3}}} 
\exp \Big( - \sqrt{3 \, c_3 \, \omega} \, y \, t^2 +\mbi \, c_3 \, {\rm e}^{\, - \, \mathbf{i} \, \pi / 4} \, y \, t^3 \Big) \, {\rm d}t \right| \\
\le \, C \, {\rm e}^{-c \, \omega^2 \, y} \, 
\int_{-\sqrt{\frac{2 \, \omega}{3 \, c_3}}-\frac{2 \, \sqrt{2} \, c_4}{9 \, c_3^2} \, \omega}^{-\sqrt{\frac{2 \, \omega}{3 \, c_3}}} 
\exp \Big( - \sqrt{3 \, c_3 \, \omega} \, y \, t^2 \Big) \, {\rm d}t \, ,
\end{multline*}
and we can then bound the final integral by the maximum of the integrated function times the length of the interval. This gives the 
estimate:
\begin{multline*}
\left| \dfrac{{\rm e}^{-\frac{c_4 \, x^2}{9 \, c_3^2 \, y}} \, 
{\rm e}^{\mbi \, \frac{2 \, |x|^{3/2}}{3 \, \sqrt{3 \, c_3} \, y^{1/2}} -\mbi \, \pi/4}}{2 \, \pi} \, 
\int_{-\sqrt{\frac{2 \, \omega}{3 \, c_3}}-\frac{2 \, \sqrt{2} \, c_4}{9 \, c_3^2} \, \omega}^{-\sqrt{\frac{2 \, \omega}{3 \, c_3}}} 
\exp \Big( - \sqrt{3 \, c_3 \, \omega} \, y \, t^2 +\mbi \, c_3 \, {\rm e}^{\, - \, \mathbf{i} \, \pi / 4} \, y \, t^3 \Big) \, {\rm d}t \right| \\
\le \, C \, \omega \, {\rm e}^{-c \, \omega^2 \, y} \, {\rm e}^{-c \, \omega^{3/2} \, y} \, \le \, 
\dfrac{C}{\omega^{1/2} \, y} \, (\omega^{3/2} \, y) \, {\rm e}^{-c \, \omega^{3/2} \, y} \, .
\end{multline*}
We can then again decrease the constant $c$ and increase $C$ to get:
\begin{multline*}
\left| \dfrac{{\rm e}^{-\frac{c_4 \, x^2}{9 \, c_3^2 \, y}} \, 
{\rm e}^{\mbi \, \frac{2 \, |x|^{3/2}}{3 \, \sqrt{3 \, c_3} \, y^{1/2}} -\mbi \, \pi/4}}{2 \, \pi} \, 
\int_{-\sqrt{\frac{2 \, \omega}{3 \, c_3}}-\frac{2 \, \sqrt{2} \, c_4}{9 \, c_3^2} \, \omega}^{-\sqrt{\frac{2 \, \omega}{3 \, c_3}}} 
\exp \Big( - \sqrt{3 \, c_3 \, \omega} \, y \, t^2 +\mbi \, c_3 \, {\rm e}^{\, - \, \mathbf{i} \, \pi / 4} \, y \, t^3 \Big) \, {\rm d}t \right| \\
\le \, \dfrac{C}{\omega^{1/2} \, y} \, {\rm e}^{-c \, \omega^{3/2} \, y} \, .
\end{multline*}
Since we restrict to the case $|x| \ge y^{1/3}$, we have $\omega^{1/2} \ge y^{-1/3}$ and this gives:
\begin{multline}
\label{propA3-estimation4}
\left| \dfrac{{\rm e}^{-\frac{c_4 \, x^2}{9 \, c_3^2 \, y}} 
{\rm e}^{\mbi \, \frac{2 \, |x|^{3/2}}{3 \, \sqrt{3 \, c_3} \, y^{1/2}} -\mbi \, \pi/4}}{2 \, \pi} \, 
\int_{-\sqrt{\frac{2 \, \omega}{3 \, c_3}}-\frac{2 \, \sqrt{2} \, c_4}{9 \, c_3^2} \, \omega}^{-\sqrt{\frac{2 \, \omega}{3 \, c_3}}} 
\exp \Big( - \sqrt{3 \, c_3 \, \omega} \, y \, t^2 +\mbi \, c_3 \, {\rm e}^{\, - \, \mathbf{i} \, \pi / 4} \, y \, t^3 \Big) \, {\rm d}t \right| \\
\le \, \dfrac{C}{y^{2/3}} \, \exp \left( -c \, \dfrac{|x|^{3/2}}{y^{1/2}} \right) \, .
\end{multline}
We now combine \eqref{propA3-estimation3} and \eqref{propA3-estimation4} and recall the definition \eqref{deffrakg} to get:
$$
\left| \mathscr{H}_\flat(x,y) \, - \, \mathfrak{g}(x,y) \right| \, \le \, \dfrac{C}{y^{2/3}} \, \exp \left( -c \, \dfrac{|x|^{3/2}}{y^{1/2}} \right) 
\, + \, \dfrac{C}{y^{1/2}} \, \exp \left( -c \, \dfrac{x^2}{y} \right) \, .
$$
Taking twice the real part and recalling that the remaining contribution $\mathscr{H}_\sharp(x,y)$ is the complex conjugate of 
$\mathscr{H}_\flat(x,y)$, we get:
$$
\left| \mathscr{H}_\flat(x,y) \, + \,  \mathscr{H}_\sharp(x,y) \, - \, 2 \, \text{\rm Re } \mathfrak{g}(x,y) \right| \, \le \, 
\dfrac{C}{y^{2/3}} \, \exp \left( -c \, \dfrac{|x|^{3/2}}{y^{1/2}} \right) \, + \, \dfrac{C}{y^{1/2}} \, \exp \left( -c \, \dfrac{x^2}{y} \right) \, .
$$
It remains\footnote{The error in \cite{jfcAMBP} occurred here since the author had tried to get rid of the term $t^3$ in the integral 
by applying similar arguments but an error occurred when estimating the corresponding integrals. Hence the final expressions for 
the analog of the function $\mathfrak{g}$ and the error bound are not correct.} to combine this estimate with \eqref{propA3-estimation1} 
and the corresponding one for $\varepsilon_1(x,y)$ to obtain the estimate of Proposition \ref{prop-A3} (recall the decomposition 
\eqref{propA3-decomposition}). We have not kept track all along of the smallness requirement on $\omega=|x|/y$ but the final 
estimate holds as long as $\omega$ is small enough, and this smallness condition is independent of $y$, which corresponds to 
the statement of Proposition \ref{prop-A3} for the condition on $x$: $x \ge -\mathbf{c}_\flat \, y$ for some $\mathbf{c}_\flat>0$.
\end{proof}

\noindent We now deal with the regime where $x$ is negative and $|x|/y$ is large enough. As in Proposition \ref{prop-A2}, this is a 
regime that had not been considered in \cite{jfcAMBP}.

\begin{proposition}
\label{prop-A4}
Let the function $\mathbf{G}$ be defined in \eqref{A-defGapproxn} and let $c_3$ be positive. Let $\mathbf{c}_\flat>0$ be the constant given in 
Proposition \ref{prop-A3}. Then there exist some constants $\mathbf{C}_\flat>\mathbf{c}_\flat$, $C>0$ and $c>0$ such that for any $(x,y) \in \R 
\times \R^{+*}$ with $x \le -\mathbf{C}_\flat \, y$, there holds:
$$
|\mathbf{G}(x,y)| \, \le \, \dfrac{C}{y^{1/4}} \, \exp \left( - c \, \dfrac{|x|^{4/3}}{y^{1/3}} \right) \, .
$$
\end{proposition}

\begin{proof}
We go back to the bound \eqref{borne-Cauchy-G} and only consider from now on \emph{negative} values for both $x$ and $\mu$. We use Young's 
inequality:
$$
6 \, c_4 \, y \, \mu^2 \, \theta^2 \, \le \, \dfrac{c_4}{2} \, y \, \theta^4 \, + \, 18 \, c_4 \, y \, \mu^4 \, ,
$$
as well as the obvious inequality $\exp(c_3 \, y \, \mu^3) \le 1$ to obtain:
$$
\forall \, \mu \le 0 \, ,\quad |\mathbf{G}(x,y)| \, \le \, \dfrac{{\rm e}^{-x \, \mu +17 \, c_4 \, y \, \mu^4}}{2 \, \pi} \, 
\int_\R {\rm e}^{3 \, c_3 \, y \, |\mu| \, \theta^2} \, {\rm e}^{-\frac{c_4}{2} \, y \, \theta^4} \, {\rm d}\theta \, .
$$
We choose the specific value $\mu=\mu_0:=-(|x|/(68 \, c_4 \, y))^{1/3}<0$ so as to minimize the exponential factor before the integral. We thus get a 
bound\footnote{The integral is reduced to the domain $\R^+$ by an obvious change of variable, which only modifies the constant in the inequality by 
a factor $2$.}:
$$
|\mathbf{G}(x,y)| \, \le \, C \, \exp \left( - c \, \dfrac{|x|^{4/3}}{y^{1/3}} \right) \, 
\int_0^{+\infty} {\rm e}^{3 \, c_3 \, y \, |\mu_0| \, \theta^2} \, {\rm e}^{-\frac{c_4}{2} \, y \, \theta^4} \, {\rm d}\theta \, .
$$
The integral is estimated by cutting at $\theta_0>0$ such that $3 \, c_3 \, |\mu_0|=(c_4/4) \, \theta_0^2$, so that we have:
$$
\int_{\theta_0}^{+\infty} {\rm e}^{3 \, c_3 \, y \, |\mu_0| \, \theta^2} \, {\rm e}^{-\frac{c_4}{2} \, y \, \theta^4} \, {\rm d}\theta 
\, = \, \int_{\theta_0}^{+\infty} {\rm e}^{\frac{c_4}{4} \, y \, \theta_0^2 \, \theta^2} \, {\rm e}^{-\frac{c_4}{2} \, y \, \theta^4} \, {\rm d}\theta 
\, \le \, \int_{\theta_0}^{+\infty} {\rm e}^{-\frac{c_4}{4} \, y \, \theta^4} \, {\rm d}\theta \, \le \, \dfrac{C}{y^{1/4}} \, ,
$$
and we also have:
$$
\int_0^{\theta_0} {\rm e}^{3 \, c_3 \, y \, |\mu_0| \, \theta^2} \, {\rm e}^{-\frac{c_4}{2} \, y \, \theta^4} \, {\rm d}\theta \, \le \, 
\int_0^{\theta_0} {\rm e}^{3 \, c_3 \, y \, |\mu_0| \, \theta^2} \, {\rm d}\theta \, \le \, 
\theta_0 \, {\rm e}^{3 \, c_3 \, y \, |\mu_0| \, \theta_0^2} \, \le \, C \, |\mu_0|^{1/2} \, {\rm e}^{C \, y \, \mu_0^2} \, .
$$
Going back to the definition of $\mu_0$ in terms of $x$ and $y$, we thus get the estimate:
$$
|\mathbf{G}(x,y)| \, \le \, \dfrac{C}{y^{1/4}} \, \exp \left( - c \, \dfrac{|x|^{4/3}}{y^{1/3}} \right) 
\, + \, C \, \left( \dfrac{|x|}{y} \right)^{1/6} \, \exp ( C \, y^{1/3} \, |x|^{2/3} ) \, \exp \left( - c \, \dfrac{|x|^{4/3}}{y^{1/3}} \right) \, ,
$$
where all constants $C$ and $c$ are independent of $x<0$ and $y>0$. It is now not difficult to show that if the constant $\mathbf{C}_\flat$ is 
chosen large enough and is we restrict to the regime $x \le -\mathbf{C}_\flat \, y$, then the decaying exponential term $\exp ( -c \, |x|^{4/3}/y^{1/3})$ 
can absorb the large term $ \exp ( C \, y^{1/3} \, |x|^{2/3} )$. There is of course no loss of generality in assuming that $\mathbf{C}_\flat$ is chosen 
larger than the constant $\mathbf{c}_\flat$ given in Proposition \ref{prop-A3}. Now that $\mathbf{C}_\flat$ is fixed and that we restrict to the regime 
$x \le -\mathbf{C}_\flat \, y$, we get the estimate:
\begin{align*}
|\mathbf{G}(x,y)| \, &\le \, \dfrac{C}{y^{1/4}} \, \exp \left( - c \, \dfrac{|x|^{4/3}}{y^{1/3}} \right) 
+ \dfrac{C}{y^{1/8}} \, \left( \dfrac{|x|}{y^{1/4}} \right)^{1/6} \, \exp \left( - c \, \dfrac{|x|^{4/3}}{y^{1/3}} \right) \\
&\le \, \dfrac{C}{y^{1/4}} \, \exp \left( - c \, \dfrac{|x|^{4/3}}{y^{1/3}} \right) 
+ \dfrac{C}{y^{1/8}} \, \exp \left( - c \, \dfrac{|x|^{4/3}}{y^{1/3}} \right) \, ,
\end{align*}
with (new) constants $C$ and $c$ that are still independent of $y>0$ and $x \le -\mathbf{C}_\flat \, y$. We can still spare half of the last decaying 
exponential term on the right-hand side to gain a decaying factor $\exp(-c \, y)$ and then use a crude bound:
$$
\dfrac{1}{y^{1/8}} \, \exp (-c \, y) \, \le \, \dfrac{C}{y^{1/4}} \, ,
$$
so that, eventually, we get the claimed bound:
$$
|\mathbf{G}(x,y)| \, \le \, \dfrac{C}{y^{1/4}} \, \exp \left( - c \, \dfrac{|x|^{4/3}}{y^{1/3}} \right) \, ,
$$
for $x \le -\mathbf{C}_\flat \, y$. The proof of Proposition \ref{prop-A4} is complete.
\end{proof}

\noindent We now turn to the final, intermediate regime where $x$ is negative and $|x|$ and $y$ are comparable.

\begin{proposition}
\label{prop-A5}
Let the function $\mathbf{G}$ be defined in \eqref{A-defGapproxn} and let $c_3$ be positive. Let the constant $\mathbf{c}_\flat$ be the one given 
by Proposition \ref{prop-A3} and let the constant $\mathbf{C}_\flat$ be the one given by Proposition \ref{prop-A4}. Then there exist some constants 
$C>0$ and $c>0$ such that for any $(x,y) \in \R \times \R^{+*}$ with $-\mathbf{C}_\flat \, y \le x \le -\mathbf{c}_\flat \, y$, there holds:
$$
|\mathbf{G}(x,y)| \, \le \, \dfrac{C}{\min (1,y^{1/4})} \, \exp (-c \, y) \, .
$$
\end{proposition}

\begin{proof}
We go back to the definition \eqref{A-defGapproxn} for $\mathbf{G}(x,y)$ but we now deform the integration axis $\R$ as depicted on Figure \ref{fig:contour}. 
Namely, we introduce some parameter $\delta>0$ to be fixed later on and first consider the contribution:
\begin{equation}
\label{prop-A5-def1}
\varepsilon(x,y) \, := \, 
\dfrac{1}{2 \, \pi} \, \int_{-\infty}^{-\delta} {\rm e}^{\mbi \, x \, \theta} \, {\rm e}^{\mbi \, c_3 \, y \, \theta^3} \, {\rm e}^{-c_4 \, y \, \theta^4} \, {\rm d}\theta 
\, + \, \dfrac{1}{2 \, \pi} \, \int_\delta^{+\infty} {\rm e}^{\mbi \, x \, \theta} \, {\rm e}^{\mbi \, c_3 \, y \, \theta^3} \, {\rm e}^{-c_4 \, y \, \theta^4} \, {\rm d}\theta \, .
\end{equation}
The integral $\mathscr{H}_1(x,y)$ corresponds to the integral along the segment from $-\delta$ to $-\mathbf{i} \, \delta$, namely:
\begin{equation}
\label{prop-A5-defH1}
\mathscr{H}_1(x,y) \, := \, \dfrac{\delta \, (1-\mathbf{i})}{2 \, \pi} \, \int_0^1 
{\rm e}^{\mathbf{i} \, x \, (-\delta \, (1-t) \, - \, \mathbf{i} \, \delta \, t)} \, {\rm e}^{\mathbf{i} \, c_3 \, y \, (-\delta \, (1-t) \, - \, \mathbf{i} \, \delta \, t)^3} 
\, {\rm e}^{- \, c_4 \, y \, (-\delta \, (1-t) \, - \, \mathbf{i} \, \delta \, t)^4} \, {\rm d}t \, ,
\end{equation}
while $\mathscr{H}_2(x,y)$ corresponds to the integral along the segment from $-\mathbf{i} \, \delta$ to $\delta$, namely:
\begin{equation}
\label{prop-A5-defH2}
\mathscr{H}_2(x,y) \, := \, \dfrac{\delta \, (1+\mathbf{i})}{2 \, \pi} \, \int_0^1 
{\rm e}^{\mathbf{i} \, x \, (\delta \, t \, - \, \mathbf{i} \, \delta \, (1-t))} \, {\rm e}^{\mathbf{i} \, c_3 \, y \, (\delta \, t \, - \, \mathbf{i} \, \delta \, (1-t))^3} 
\, {\rm e}^{- \, c_4 \, y \, (\delta \, t \, - \, \mathbf{i} \, \delta \, (1-t))^4} \, {\rm d}t \, .
\end{equation}

\begin{figure}[h!]
\begin{center}
\begin{tikzpicture}[scale=1.25,>=latex]
\draw[black,->] (-4,0) -- (4,0);
\draw[black,->] (0,-1.5)--(0,0.5);
\draw[thick,blue,->] (-4,0) -- (-2,0);
\draw[thick,blue] (-2,0) -- (-1,0);
\draw[thick,blue,->] (-1,0) -- (-0.5,-0.5);
\draw[thick,blue] (-0.5,-0.5) -- (0,-1);
\draw[thick,blue,->] (0,-1) -- (0.5,-0.5);
\draw[thick,blue] (0.5,-0.5) -- (1,0);
\draw[thick,blue,->] (1,0) -- (2,0);
\draw[thick,blue] (2,0) -- (4,0);
\draw (-1,0.1) node[above]{$-\delta$};
\draw (1,0.1) node[above]{$\delta$};
\draw (0,-1.1) node[right]{$-\mathbf{i} \, \delta$};
\draw (-1.5,-0.8) node[below]{{\color{red}$\mathscr{H}_1(x,y)$}};
\draw (1.5,-0.8) node[below]{{\color{red}$\mathscr{H}_2(x,y)$}};
\draw[thick,red,->] (-0.65,-0.45) -- (-1.45,-0.8);
\draw[thick,red,->] (0.55,-0.55) -- (1.25,-0.85);
\draw (0.1,0) node[below]{$0$};
\draw (3,-1.5) node {$\C$};
\node (centre) at (0,-1){$\bullet$};
\node (centre) at (-1,0){$\bullet$};
\node (centre) at (1,0){$\bullet$};
\end{tikzpicture}
\caption{The integration contour in the case $-\mathbf{C}_\flat \, y \le x \le -\mathbf{c}_\flat \, y$.}
\label{fig:contour}
\end{center}
\end{figure}

Cauchy's formula allows us to decompose $\mathbf{G}(x,y)$ as:
$$
\mathbf{G}(x,y) \, = \, \mathscr{H}_1(x,y) \, + \, \mathscr{H}_2(x,y) \, + \, \varepsilon(x,y) \, ,
$$
and we now estimate each contribution in $\mathbf{G}(x,y)$ one by one. We first estimate the integral $\mathscr{H}_1(x,y)$ defined in \eqref{prop-A5-defH1} 
and explain how the parameter $\delta>0$ is fixed. We restrict from on to the regime $-\mathbf{C}_\flat \, y \le x \le -\mathbf{c}_\flat \, y$ where the constant 
$\mathbf{c}_\flat$, resp. $\mathbf{C}_\flat$, is given by Proposition \ref{prop-A3}, resp. Proposition \ref{prop-A4}. This implies in particular that $x$ is negative. 
We expand the polynomial quantities within the exponential terms in \eqref{prop-A5-defH1} and use the triangle inequality to obtain:
\begin{multline*}
|\mathscr{H}_1(x,y)| \, \le \, C \, \delta \, \int_0^1 \exp \Big( - \, |x| \, \delta \, t \, - \, c_3 \, y \, \delta^3 \, t^3 \, + \, 3 \, c_3 \, y \, \delta^3 \, (1-t)^2 \, t \, \Big) \\
\times \, \exp \Big( - \, c_4 \, y \, \delta^4 \, (1-t)^4 \, + \, 6 \, c_4 \, y \, \delta^4 \, (1-t)^2 \, t^2 \, - \, c_4 \, y \, \delta^4 \, t^4 \, \Big) \, {\rm d}t \, .
\end{multline*}
Estimating exponentials of negative terms by $1$, we have:
$$
|\mathscr{H}_1(x,y)| \, \le \, C \, \delta \, \int_0^1 \exp \Big( - \, |x| \, \delta \, t \, + \, 3 \, c_3 \, \delta^3 \, y \, t 
 \, + \, 6 \, c_4 \, \delta^4 \, y \, t^2 \, - \, c_4 \, \delta^4 \, y \, (1-t)^4 \Big) \, {\rm d}t \, .
$$

We fix the constant $\delta>0$ by imposing:
\begin{equation}
\label{prop-A5-choixdelta}
3 \, c_3 \, \delta^2 \, \le \, \dfrac{\mathbf{c}_\flat}{3} \, ,\quad \text{\rm and} \quad 6 \, c_4 \, \delta^3 \, \le \, \dfrac{\mathbf{c}_\flat}{3} \, .
\end{equation}
Since $\delta$ is fixed and only depends on already fixed parameters, we allow constants below to depend on $\delta$. We use the restrictions 
\eqref{prop-A5-choixdelta} in the previous estimate for $\mathscr{H}_1(x,y)$ and use furthermore the inequality $|x| \ge \mathbf{c}_\flat \, y$ to obtain:
$$
|\mathscr{H}_1(x,y)| \, \le \, C \, \int_0^1 \exp 
\left( - \, y \, \Big( \dfrac{\mathbf{c}_\flat}{3} \, \delta \, t \, + \, c_4 \, \delta^4 \, (1-t)^4 \Big) \right) \, {\rm d}t \, .
$$
We now observe that \eqref{prop-A5-choixdelta} implies that the function:
$$
t \in [0,1] \longmapsto \dfrac{\mathbf{c}_\flat}{3} \, \delta \, t \, + \, c_4 \, \delta^4 \, (1-t)^4 \, ,
$$
is increasing. We thus have the exponentially decaying bound:
$$
|\mathscr{H}_1(x,y)| \, \le \, C \, \exp ( - \, c_4 \, \delta^4 \, y) \, = \, C \, \exp ( - \, c \, y) \, .
$$
By a simple change of variable $t \rightarrow 1-t$, we easily find that the integral $\mathscr{H}_2(x,y)$ in \eqref{prop-A5-defH2} equals the complex 
conjuugate of $\mathscr{H}_1(x,y)$ so the previous estimate for $\mathscr{H}_1(x,y)$ also applies to $\mathscr{H}_2(x,y)$.

We end the argument with the last remaining term $\varepsilon(x,y)$ defined in \eqref{prop-A5-def1}. We recall that $\delta>0$ has been fixed in the 
analysis of the contributions $\mathscr{H}_1(x,y)$ and $\mathscr{H}_2(x,y)$. By applying the triangle inequality, we have:
$$
|\varepsilon(x,y)| \, \le \, \dfrac{1}{2 \, \pi} \, \int_{-\infty}^{-\delta} {\rm e}^{-c_4 \, y \, \theta^4} \, {\rm d}\theta 
\, + \, \dfrac{1}{2 \, \pi} \, \int_\delta^{+\infty} {\rm e}^{-c_4 \, y \, \theta^4} \, {\rm d}\theta \, = \, 
\dfrac{1}{\pi} \, \int_\delta^{+\infty} {\rm e}^{-c_4 \, y \, \theta^4} \, {\rm d}\theta \, .
$$
We thus get the estimate:
$$
|\varepsilon(x,y)| \, \le \, \dfrac{{\rm e}^{-\frac{c_4}{2} \, \delta^4 \, y}}{\pi} \, 
\int_\delta^{+\infty} {\rm e}^{-\frac{c_4}{2} \, y \, \theta^4} \, {\rm d}\theta \, \le \, \dfrac{C}{y^{1/4}} \, \exp (- \, c \, y) \, .
$$
Going back to the decomposition of $\mathbf{G}(x,y)$, we collect the estimates of $\mathscr{H}_1(x,y)$, $\mathscr{H}_2(x,y)$ and $\varepsilon(x,y)$ 
to get:
$$
|\mathbf{G}(x,y)| \, \le \, C \, \exp (- \, c \, y) \, + \, \dfrac{C}{y^{1/4}} \, \exp (- \, c \, y) \, .
$$
The conclusion of Proposition \ref{prop-A5} follows.
\end{proof}

\noindent We first collect the results of Proposition \ref{prop-A4} and Proposition \ref{prop-A5} to obtain the following unified estimate, the proof 
of which is left to the interested reader. The constant $\mathbf{c}_\flat$ in Corollary \ref{coro-A4} below is, of course, the same as the one given 
by Proposition \ref{prop-A3}.

\begin{corollary}
\label{coro-A4}
Let the function $\mathbf{G}$ be defined in \eqref{A-defGapproxn} and let $c_3$ be positive. Let $y_{\rm min}>0$ be given. 
Then there exist some constants $\mathbf{c}_\flat>0$, $C>0$ and $c>0$ such that for any $(x,y) \in \R \times [y_{\rm min},+\infty)$, 
there holds:
$$
|\mathbf{G}(x,y)| \, \le \, \dfrac{C}{y^{1/4}} \, \exp \left( - c \, \dfrac{|x|^{4/3}}{y^{1/3}} \right) \, ,
$$
as long as $x$ satisfies $x \le -\mathbf{c}_\flat \, y$.
\end{corollary}

\noindent The final Corollary is precisely what we are aiming at, namely at estimates that are analogous to those of Proposition \ref{prop-A3} 
and Corollary \ref{coro-A4} but with now a transition at $-\underline{\mathbf{c}} \, y$ with an arbitrary constant $\underline{\mathbf{c}}$. The 
result is the following.

\begin{corollary}
\label{coro-A5}
Let the function $\mathbf{G}$ be defined in \eqref{A-defGapproxn} and let $c_3$ be positive. Let the function $\mathfrak{g}$ be 
defined in \eqref{deffrakg}. Let also $y_{\rm min}>0$ be given. Then for any constant $\underline{\mathbf{c}}>0$, there exist constants 
$C>0$ and $c>0$ such that for any $(x,y) \in \R \times [y_{\rm min},+\infty)$, there holds:
\begin{equation}
\label{estim1-coroA4}
|\mathbf{G}(x,y) \, - \, 2 \, \text{\rm Re } \mathfrak{g}(x,y)| \, \le \, 
\dfrac{C}{y^{1/3}} \, \exp \left( -c \, \dfrac{|x|^{3/2}}{y^{1/2}} \right) \, + \, \dfrac{C}{y^{1/2}} \, \exp \left( -c \, \dfrac{x^2}{y} \right) \, ,
\end{equation}
if $-\underline{\mathbf{c}} \, y \le x \le -y^{1/3}$, and there holds:
\begin{equation}
\label{estim2-coroA4}
|\mathbf{G}(x,y)| \, \le \, \dfrac{C}{y^{1/4}} \, \exp \left( - c \, \dfrac{|x|^{4/3}}{y^{1/3}} \right) \, ,
\end{equation}
if $x \le -\underline{\mathbf{c}} \, y$.
\end{corollary}

\begin{proof}
The argument is mostly similar to that we used in the proof of Corollary \ref{coro-A3} except that we now need to control also 
$\mathfrak{g}(x,y)$ in the region where $|x|$ and $y$ are comparable. Let us for instance deal with the case $\underline{\mathbf{c}} 
\ge \mathbf{c}_\flat$ where $\mathbf{c}_\flat>0$ is the constant given in Proposition \ref{prop-A3}. If $x \le -\underline{\mathbf{c}} \, y$, 
we have $x \le -\mathbf{c}_\flat \, y$ and the desired estimate \eqref{estim2-coroA4} for $|\mathbf{G}(x,y)|$ is given by Corollary 
\ref{coro-A4}. Moreover, if $x$ satisfies $-\mathbf{c}_\flat \, y \le x \le -y^{1/3}$, the estimate \eqref{estim1-coroA4} is given by 
Proposition \ref{prop-A3} and by using $y \ge y_{\rm min}>0$ in order to estimate from below the quantity $\min(1,y^{1/3})$.

It thus remains to control the left hand side of \eqref{estim1-coroA4} in the case where $x$ is negative and $|x|/y \in 
[\mathbf{c}_\flat,\underline{\mathbf{c}}]$, which we assume from now on. We have:
\begin{align*}
|\mathbf{G}(x,y) \, - \, 2 \, \text{\rm Re } \mathfrak{g}(x,y)| \, \le & \, |\mathbf{G}(x,y)| \, + \, 2 \, |\mathfrak{g}(x,y)| \\
\le & \dfrac{C}{y^{1/4}} \, \exp \left( - c \, \dfrac{|x|^{4/3}}{y^{1/3}} \right) \, + \, 2 \, |\mathfrak{g}(x,y)| \, ,
\end{align*}
where the second inequality follows from Corollary \ref{coro-A4} since we have $x \le -\mathbf{c}_\flat \, y$. The quantity $\mathfrak{g}(x,y)$ 
is estimated directly from the definition \eqref{deffrakg}. Applying the triangle inequality, we have:
$$
|\mathfrak{g}(x,y)| \, \le \, C \, \exp \left( - c \, \dfrac{x^2}{y} \right) \, \int_{-\sqrt{\frac{2 \, |x|}{3\, c_3 \, y}}}^{\sqrt{\frac{2 \, |x|}{3\, c_3 \, y}}} \, 
{\rm e}^{- \, \sqrt{3 \, c_3 \, |x| \, y} \, t^2 +\frac{c_3}{\sqrt{2}} \, y \, t^3} \, {\rm d}t \, ,
$$
and we have already seen that on the considered integration interval, there holds:
$$
{\rm e}^{- \, \sqrt{3 \, c_3 \, |x| \, y} \, t^2 +\frac{c_3}{\sqrt{2}} \, y \, t^3} \, \le \, {\rm e}^{- \, c \, \sqrt{|x| \, y} \, t^2}
$$
for some appropriate numerical constant $c>0$. Since we consider the regime $|x|/y \in [\mathbf{c}_\flat,\underline{\mathbf{c}}]$, 
we thus have:
$$
|\mathfrak{g}(x,y)| \, \le \, C \, \exp \left( - c \, \dfrac{x^2}{y} \right) \, \int_\R \, {\rm e}^{- \, c \, y_{\rm min} \, t^2} \, {\rm d}t 
\, \le \, C \, \exp \left( - c \, y \right) \, ,
$$
which yields:
$$
|\mathbf{G}(x,y) \, - \, 2 \, \text{\rm Re } \mathfrak{g}(x,y)| \, \le \, C \, \exp \left( - c \, y \right) \, .
$$
As we have already seen several times before, there is no difficulty at this stage to derive the estimate \eqref{estim1-coroA4} 
for the regime $|x|/y \in [\mathbf{c}_\flat,\underline{\mathbf{c}}]$ and $y \ge y_{\rm min}$. This completes the proof of Corollary 
\ref{coro-A5}.
\end{proof}

%%%%%%%%%%%%%%%%%%%%
\section{Proof of Theorem \ref{thm-A3}}
\label{sectionA-5}

The combination of Proposition \ref{prop-A1}, Corollary \ref{coro-A3} and Corollary \ref{coro-A5} already gives the estimates \eqref{A-bound1} 
and \eqref{A-bound2} of the function $\mathbf{G}$ as given in Theorem \ref{thm-A3}. It therefore only remains to prove the estimates 
\eqref{A-bound3} and \eqref{A-bound4} for the primitive function of $\mathbf{G}$ with respect to its first variable. We still assume $c_3>0$ 
and start with the first estimate in \eqref{A-bound3}. We consider a fixed positive constant $\underline{\mathbf{c}}$ and let $x \le - \, 
\underline{\mathbf{c}} \, y$. We then use the estimate in the fourth case of \eqref{A-bound1} to get:
$$
\left| \int_{-\infty}^x \mathbf{G}(\xi,y) \, {\rm d}\xi \right| \, \le \, \int_{-\infty}^x |\mathbf{G}(\xi,y)| \, {\rm d}\xi \, \le \, 
\dfrac{C}{y^{1/4}} \, \int_{-\infty}^x \exp \left( - c \, \dfrac{|\xi|^{4/3}}{y^{1/3}} \right) \, {\rm d}\xi \, .
$$
We then perform a change of variable to obtain:
$$
\left| \int_{-\infty}^x \mathbf{G}(\xi,y) \, {\rm d}\xi \right| \, \le \, C \, \int_{-\infty}^{x/y^{1/4}} \exp (- c \,|\eta|^{4/3}) \, {\rm d}\eta 
\, = \, C \, \int_{-\infty}^{x/y^{1/4}} \dfrac{(4/3) \, c \, |\eta|^{1/3}}{(4/3) \, c \, |\eta|^{1/3}} \, \exp (- c \,|\eta|^{4/3}) \, {\rm d}\eta \, .
$$
Since we assume $x \le - \, \underline{\mathbf{c}} \, y$, we have $|\eta|^{1/3} \ge \underline{\mathbf{c}}^{1/3} \, y^{1/4}$ in the last integral on the 
right-hand side, and this gives:
$$
\left| \int_{-\infty}^x \mathbf{G}(\xi,y) \, {\rm d}\xi \right| \, \le \, 
\dfrac{C}{y^{1/4}} \, \int_{-\infty}^{x/y^{1/4}} (4/3) \, c \, |\eta|^{1/3} \, \exp (- c \,|\eta|^{4/3}) \, {\rm d}\eta \, .
$$
We thus end up with the estimate:
$$
\left| \int_{-\infty}^x \mathbf{G}(\xi,y) \, {\rm d}\xi \right| \, \le \, \dfrac{C}{y^{1/4}} \, \exp \left( - c \, \dfrac{|x|^{4/3}}{y^{1/3}} \right) 
\, \le \, C \, \exp \left( - c \, \dfrac{|x|^{4/3}}{y^{1/3}} \right) \, ,
$$
for $x \le - \, \underline{\mathbf{c}} \, y$ and $y \ge y_{\rm min}$, an estimate from which the first half of \eqref{A-bound3} follows directly. 
Note in particular that we have the uniform estimate:
\begin{equation}
\label{estimprimitive1}
\left| \int_{-\infty}^{- \, \underline{\mathbf{c}} \, y} \mathbf{G}(\xi,y) \, {\rm d}\xi \right| \, \le \, C \, ,
\end{equation}
for any $y \ge y_{\rm min}$.
\bigskip

From the definition \eqref{A-defGapproxn}, we know that for any $y>0$, $\mathbf{G}(\cdot,y)$ is the inverse Fourier transform of the Schwartz 
class function:
$$
\theta \in \R \longmapsto {\rm e}^{\mbi \, c_3 \, y \, \theta^3} \, {\rm e}^{-c_4 \, y \, \theta^4} \, .
$$
In particular, this means that the (partial) Fourier transform of $\mathbf{G}$ with respect to its first variable is given by:
$$
\mathscr{F}_x(\mathbf{G})(\theta,y) \, = \, {\rm e}^{\mbi \, c_3 \, y \, \theta^3} \, {\rm e}^{-c_4 \, y \, \theta^4} \, .
$$
Evaluating at $\theta=0$, we get the relation:
$$
\int_\R \mathbf{G}(x,y) \, {\rm d}x \, = \, 1 \, ,
$$
so that proving the bound \eqref{A-bound4} amounts to showing the bound:
$$
\left| \int_x^{+\infty} \mathbf{G}(\xi,y) \, {\rm d}\xi \right| \, \le \, \widetilde{C} \, \exp (-\tilde{c} \, y) \, ,\quad 
\text{\rm if $x \ge \underline{\mathbf{c}} \, y$.}
$$
The proof of this bound follows from the exact same argument as above except that we now use the first case in \eqref{A-bound1}.
\bigskip

The last part of the proof of Theorem \ref{thm-A3} aims at showing the second half of \eqref{A-bound3}, that is, at proving a uniform bound 
for the primitive function of $\mathbf{G}$ in the case $|x| \le \underline{\mathbf{c}} \, y$. Let us first assume $x \ge 0$ and thus $x \in 
[0,\underline{\mathbf{c}} \, y]$. We have already seen the relation:
$$
\int_{-\infty}^x \mathbf{G}(\xi,y) \, {\rm d}\xi \, = \, 1-\int_x^{\underline{\mathbf{c}} \, y} \mathbf{G}(\xi,y) \, {\rm d}\xi 
\, - \, \int_{\underline{\mathbf{c}} \, y}^{+\infty} \mathbf{G}(\xi,y) \, {\rm d}\xi  \, ,
$$
and we thus already have from the arguments above (namely, the case $x \ge \underline{\mathbf{c}} \, y$):
$$
\left| \int_{-\infty}^x \mathbf{G}(\xi,y) \, {\rm d}\xi \right| \, \le \, C \, + \, \int_x^{\underline{\mathbf{c}} \, y} |\mathbf{G}(\xi,y)| \, {\rm d}\xi \, .
$$
The final integral on the right-hand side is estimated by using the second case in \eqref{A-bound1}, which gives the uniform control (after an obvious 
change of variable in $\xi$):
$$
\left| \int_{-\infty}^x \mathbf{G}(\xi,y) \, {\rm d}\xi \right| \, \le \, C \, .
$$
We can actually push a little further this argument and obtain from \eqref{A-bound1} a uniform bound for:
$$
\int_{-\infty}^x \mathbf{G}(\xi,y) \, {\rm d}\xi
$$
as long as $x$ belongs to the interval $[-y^{1/3},\underline{\mathbf{c}} \, y]$ and not only when $x$ belongs to $[0,\underline{\mathbf{c}} \, y]$ 
(use now the third case in \eqref{A-bound1}). It thus remains to examine the case $x \in [-\underline{\mathbf{c}} \, y,-y^{1/3}]$ for which we need 
to take into account the oscillating behavior of the Green's function. This is the only regime where applying the triangle inequality to estimate 
the primitive function does not (and can not !) work. Let therefore $x \in [-\underline{\mathbf{c}} \, y,-y^{1/3}]$. We decompose:
\begin{align*}
\int_{-\infty}^x \mathbf{G}(\xi,y) \, {\rm d}\xi \, =& \, 
\int_{-\infty}^{-\underline{\mathbf{c}} \, y} \mathbf{G}(\xi,y) \, {\rm d}\xi \, + \, 
\int_{-\underline{\mathbf{c}} \, y}^x \big( \mathbf{G}(\xi,y)-2 \, \text{\rm Re } \mathfrak{g}(\xi,y) \big) \, {\rm d}\xi \\
&\, + \, 2 \, \text{\rm Re } \int_{-\underline{\mathbf{c}} \, y}^x \mathfrak{g}(\xi,y) \, {\rm d}\xi \, .
\end{align*}
The first term on the right-hand side has already been estimated, see \eqref{estimprimitive1}, and the second term on the right-hand side is estimated 
by using the bound\footnote{The whole point of \eqref{A-bound2} was precisely to extract from $\mathbf{G}$ its leading oscillating behavior so that 
the difference between $\mathbf{G}$ and this ``leading order term'' would actually become uniformly integrable.} \eqref{A-bound2}. At this stage, we 
already have an estimate that reads:
\begin{equation}
\label{estimprimitive2}
\left| \int_{-\infty}^x \mathbf{G}(\xi,y) \, {\rm d}\xi \right| \, \le \, C \, + \, 2 \, \left| \int_{-\underline{\mathbf{c}} \, y}^x \mathfrak{g}(\xi,y) \, {\rm d}\xi \right| \, ,
\end{equation}
and the remaining point of the proof is to derive a uniform estimate for the primitive function of the \emph{explicit} function $\mathfrak{g}$ 
whose expression is given in \eqref{deffrakg}. The proof relies on integration by parts, as detailed below.

We first introduce the notation:
$$
\beta_0 \, := \, \dfrac{c_4}{9 \, c_3^2} >0 \, ,\quad \beta_1 \, := \, \dfrac{2}{3 \, \sqrt{3 \, c_3}} >0 \, ,
$$
so that, after performing a change of variable in the integral of \eqref{deffrakg}, we obtain the expression:
\begin{equation}
\label{expressiong1}
\mathfrak{g}(x,y) \, := \, \dfrac{3 \, \beta_1 \, {\rm e}^{-\mathbf{i} \, \pi/4}}{2 \, \sqrt{2} \, \pi} \, \dfrac{|x|^{1/2}}{y^{1/2}}
\exp \left( - \, \beta_0 \, \dfrac{x^2}{y} \, \right) \, \exp \left( \mathbf{i} \, \beta_1 \, \dfrac{|x|^{3/2}}{y^{1/2}} \right) 
\int_{-1}^1 \, {\rm e}^{- \, 3 \, \beta_1 \, \frac{|x|^{3/2}}{y^{1/2}} \, u^2} \, {\rm e}^{\beta_1 \, (1-\mbi) \, \frac{|x|^{3/2}}{y^{1/2}} \, u^3} \, {\rm d}u \, .
\end{equation}
Let us define a function $H$ on $\R^+$ as follows (the constant $\beta_1>0$ being fixed as above):
\begin{equation}
\label{deffunctionH}
\forall \, w \ge 0 \, ,\quad H(w) \, := \exp \left( \mathbf{i} \, \beta_1 \, w \right) 
\, \int_{-1}^1 {\rm e}^{- \, 3 \, \beta_1 \, w \, u^2} \, {\rm e}^{\beta_1 \, (1-\mathbf{i}) \, w \, u^3} \, {\rm d}u \, .
\end{equation}
With the help of \eqref{deffunctionH}, we can rewrite \eqref{expressiong1} and obtain the relation:
\begin{equation}
\label{expressiong2}
\forall \, (x,y) \in \R \times \R^{+*} \, ,\quad 
y^{1/2} \, \mathfrak{g}(y^{1/2} \, x,y) \, := \, \dfrac{\beta_1 \, {\rm e}^{-\mathbf{i} \, \pi/4}}{\sqrt{2} \, \pi} \, {\rm e}^{-\beta_0 \, x^2} \, 
\left( \dfrac{3}{2} \, y^{1/4} \, |x|^{1/2} \right) \, H(y^{1/4} \, |x|^{3/2}) \, .
\end{equation}
where the factor $(3/2) \, y^{1/4} \, |x|^{1/2}$ equals, up to a sign, the derivative of the function $(x \mapsto y^{1/4} \, |x|^{3/2})$. 
We are thus in a very favorable position for applying integration by parts. Before going further, we prove the following Lemma.

\begin{lemma}
\label{lemA-1}
Let the function $H$ be defined on $\R^+$ by \eqref{deffunctionH}. Then there exists a constant $C>0$ that only depends on $\beta_1$ and 
such that:
$$
\forall \, w \ge 0 \, ,\quad \left| \int_0^w \, H(w') \, {\rm d}w' \right| \, \le \, C \, .
$$
\end{lemma}

\begin{proof}
We consider $w \ge 0$. By applying the Fubini Theorem, we get the relation:
$$
\int_0^w \, H(w') \, {\rm d}w' \, = \, 
\int_{-1}^1 \int_0^w \, \exp \Big( \beta_1 \, w' \, \big( -3 \, u^2 +u^3 \, + \, \mathbf{i} \, (1-u^3) \big) \Big) \, {\rm d}w' \, {\rm d}u \, ,
$$
and it turns out that the continuous function:
$$
u \in [-1,1] \longmapsto -3 \, u^2 +u^3 \, + \, \mathbf{i} \, (1-u^3) \, ,
$$
does not vanish. Its modulus is thus uniformly bounded from below. We get:
$$
\int_0^w \, H(w') \, {\rm d}w' \, = \, \int_{-1}^1 
\dfrac{\exp \Big( \beta_1 \, w \, \big( -3 \, u^2 +u^3 \, + \, \mathbf{i} \, (1-u^3) \big) \Big)-1}{\beta_1 \, 
\big( -3 \, u^2 +u^3 \, + \, \mathbf{i} \, (1-u^3) \big)} \, {\rm d}u \, ,
$$
and the triangle inequality (as well as a lower bound for the modulus of the denominator) yields:
$$
\left| \int_0^w \, H(w') \, {\rm d}w' \right| \, \le \, C \, \int_{-1}^1 \exp \big( \beta_1 \, w \, (-3 \, u^2 +u^3) \big)+1 \, {\rm d}u \, .
$$
We have $-3 \, u^2 +u^3 \le 0$ for $u \in [-1,1]$ and $w \ge 0$ so the uniform bound of Lemma \ref{lemA-1} follows.
\end{proof}

\noindent We now go back to our main problem, which is to prove a bound for the primitive function of $\mathfrak{g}$, see \eqref{estimprimitive2}. 
We use the expression \eqref{expressiong2} and obtain:
\begin{align*}
\int_{-\underline{\mathbf{c}} \, y}^x \mathfrak{g}(\xi,y) \, {\rm d}\xi \, =& \, 
\int_{-\underline{\mathbf{c}} \, y^{1/2}}^{x/y^{1/2}} y^{1/2} \, \mathfrak{g}(y^{1/2} \, w,y) \, {\rm d}w \\
\, =& \, \dfrac{\beta_1 \, {\rm e}^{-\mathbf{i} \, \pi/4}}{\sqrt{2} \, \pi} \, \int_{|x|/y^{1/2}}^{\underline{\mathbf{c}} \, y^{1/2}} 
{\rm e}^{-\beta_0 \, w^2} \, \left( \dfrac{3}{2} \, y^{1/4} \, w^{1/2} \, H(y^{1/4} \, w^{3/2}) \right) \, {\rm d}w \\
\, =& \, \dfrac{\beta_1 \, {\rm e}^{-\mathbf{i} \, \pi/4}}{\sqrt{2} \, \pi} \, \int_{|x|/y^{1/2}}^{\underline{\mathbf{c}} \, y^{1/2}} 
{\rm e}^{-\beta_0 \, w^2} \, \dfrac{{\rm d}}{{\rm d}w} \left( \int_0^{y^{1/4} \, w^{3/2}} \, H(w') \, {\rm d}w' \right) \, {\rm d}w \, .
\end{align*}
It then only remains to integrate by parts the final integral and to apply Lemma \ref{lemA-1} in order to derive the uniform bound:
$$
\left| \int_{-\underline{\mathbf{c}} \, y}^x \mathfrak{g}(\xi,y) \, {\rm d}\xi \right| \, \le \, C \, .
$$
This final argument is a prototype application of Abel's transform (in the continuous setting). This completes the proof of Theorem \ref{thm-A3}.

%%%%%%%%%%%%
\section{Consequences}
\label{sectionA-6}

This final paragraph is devoted to applying Theorem \ref{thm-A3} in order to derive suitable bounds for the activation function $\mathfrak{A}$ 
and other quantities that arise in our decomposition of the Green's function of the operator $\mathscr{L}$ in \eqref{linear}. We therefore now 
go back to the framework of stationary discrete shock profiles and use the index $r$, resp. $\ell$, to refer to the right, resp. left, state of the 
discrete shock \eqref{shock}. The analysis of Section \ref{section4-3} uses the following quantities defined for any $j_0 \in \N^*$ and $n \in \N^*$:
$$
\mathfrak{A}_r^n(j_0) \ := \, \dfrac{1}{2 \, \mbi \, \pi} \, \int_{\eta+\mbi \, \R} {\rm e}^{n \, \tau -j_0 \, \varphi_r(\tau)} \, 
\dfrac{{\rm d}\tau}{\tau} \, ,
$$
where $\eta$ is any positive number (the Cauchy formula shows that the definition is independent of $\eta$) and:
$$
\mathfrak{B}_r^n(j_0) \ := \, \dfrac{1}{2 \, \mbi \, \pi} \, \int_{\mbi \, \R} {\rm e}^{n \, \tau -j_0 \, \varphi_r(\tau)} \, {\rm d}\tau \, .
$$
In both definitions of $\mathfrak{A}_r^n(j_0)$ and $\mathfrak{B}_r^n(j_0)$, the function $\varphi_r$ is defined by:
$$
\forall \, \tau \in \C \, ,\quad \varphi_r(\tau) \, := \, - \, \dfrac{1}{\alpha_r} \, \tau \, + \, \dfrac{1-\alpha_r^2}{6 \, \alpha_r^3} \, \tau^3 
 \, - \, \dfrac{1-\alpha_r^2}{8 \, \alpha_r^3} \, \tau^4 \, .
$$
At last, we recall that $\alpha_r$ belongs to the interval $(-1,0)$, see \eqref{CFL}. This is a consequence of Lax shock inequalities and the 
choice of the CFL parameter.

By using the parametrization $\tau=\mbi \, |\alpha_r| \, \theta$ in the definition of $\mathfrak{B}_r^n(j_0)$, we obtain the expression:
$$
\mathfrak{B}_r^n(j_0) \ := \, \dfrac{|\alpha_r|}{2 \, \pi} \, \int_\R {\rm e}^{-\mbi \, (j_0+n\, \alpha_r) \, \theta} 
\, {\rm e}^{-\mbi \, j_0 \, \frac{1-\alpha_r^2}{6} \, \theta^3} \, {\rm e}^{j_0 \, \alpha_r \, \frac{1-\alpha_r^2}{8} \, \theta^4} \, {\rm d}\theta \, ,
$$
and the integral is convergent since $j_0$ is positive and $\alpha_r$ belongs to the interval $(-1,0)$. Going back to the definitions 
\eqref{A-defc3c4} (with the choice $\alpha=\alpha_r \in (-1,0)$) and \eqref{A-defGapproxn}, we have thus obtained the relation:
\begin{equation}
\label{A-relationBG}
\forall \, j_0 \in \N^* \, ,\quad \forall \, n \in \N^* \, ,\quad 
\mathfrak{B}_r^n(j_0) \ := \, |\alpha_r| \, \mathbf{G}_r \left( -j_0+n\, |\alpha_r|,\dfrac{j_0}{|\alpha_r|} \right) \, ,
\end{equation}
which is the reason why Theorem \ref{thm-A3} will give us exactly what we need for proving the bounds we need in our analysis. 
The index $r$ in $\mathbf{G}_r$ refers to the fact that we have made the choice $\alpha=\alpha_r$ when defining the constants 
$c_3$ and $c_4$ in \eqref{A-defc3c4}.
\bigskip

Let us now turn to the activation function $\mathfrak{A}_r^n(j_0)$. Choosing the parametrization $\tau=|\alpha_r| \, (\eta+\mbi \, \theta)$ 
for any $\eta>0$, we get:
$$
\mathfrak{A}_r^n(j_0) \ := \, \dfrac{1}{2 \, \pi} \, \int_\R {\rm e}^{-(j_0+n\, \alpha_r) \, (\eta+\mbi \, \theta)} 
\, {\rm e}^{j_0 \, \frac{1-\alpha_r^2}{6} \, (\eta+\mbi \, \theta)^3} \, {\rm e}^{j_0 \, \alpha_r \, \frac{1-\alpha_r^2}{8} \, 
(\eta+\mbi \, \theta)^4} \, \dfrac{{\rm d}\theta}{\eta+\mbi \, \theta} \, .
$$
With the choice $\alpha=\alpha_r$ and the definition \eqref{A-defc3c4} for the coefficients $c_3$ and $c_4$, we can introduce the 
function $\mathbf{A}_r$ defined by:
\begin{equation}
\label{A-deffonctionAr}
\forall \, (x,y) \in \R \times \R^{+*} \, ,\quad \mathbf{A}_r(x,y) \, := \, \dfrac{1}{2 \, \pi} \, \int_\R {\rm e}^{x \, (\eta+\mbi \, \theta)} 
\, {\rm e}^{-c_3 \, y \, (\eta+\mbi \, \theta)^3} \, {\rm e}^{-c_4 \, y \, (\eta+\mbi \, \theta)^4} \, \dfrac{{\rm d}\theta}{\eta+\mbi \, \theta} \, ,
\end{equation}
where the definition is independent of the choice of $\eta>0$ because of Cauchy's formula. With this notation, we have expressed 
the activation function $\mathfrak{A}_r^n(j_0)$ as:
\begin{equation}
\label{A-relationAGtilde}
\mathfrak{A}_r^n(j_0) \, = \, \mathbf{A}_r \left( -j_0+n\, |\alpha_r|,\dfrac{j_0}{|\alpha_r|} \right) \, .
\end{equation}
It remains to connect the function $\mathbf{A}_r$ with the primitive function of $\mathbf{G}_r$ with respect to its first variable. 
From the dominated convergence theorem, we have that $\mathbf{A}_r$ is differentiable with respect to its first variable and 
(passing to the limit $\eta\rightarrow 0$ in the expression of the partial derivative):
$$
\dfrac{\partial \mathbf{A}_r}{\partial x} (x,y) \, = \, \mathbf{G}_r(x,y) \, .
$$
Moreover, the factor $\exp(x \, \eta)$ can be extracted from the integral and we thus have:
$$
\lim_{x \rightarrow -\infty} \, \mathbf{A}_r(x,y) \, = \, 0 \, ,
$$
from which we get the general expression\footnote{The integral is convergent because of the bounds that we proved in Theorem \ref{thm-A3}.}:
$$
\mathbf{A}_r(x,y) \, = \, \int_{-\infty}^x \, \mathbf{G}_r(\xi,y) \, {\rm d}\xi \, .
$$
Going back to \eqref{A-relationAGtilde}, this means that we have expressed the activation function $\mathfrak{A}_r^n(j_0)$ as follows:
\begin{equation}
\label{A-relationAG}
\forall \, j_0 \in \N^* \, ,\quad \forall \, n \in \N^* \, ,\quad 
\mathfrak{A}_r^n(j_0) \ = \, \int_{-\infty}^{-j_0+n \, |\alpha_r|} \mathbf{G}_r \left( \xi,\dfrac{j_0}{|\alpha_r|} \right) \, {\rm d}\xi \, .
\end{equation}
Theorem \ref{thm-A3} can be recast into the following compact form that will be helpful in the analysis of Chapter \ref{chapter4}.

\begin{corollary}
\label{coro-A6}
Let the function $\mathbf{A}_r$ be defined in \eqref{A-deffonctionAr} with constants $c_3$ and $c_4$ as in \eqref{A-defc3c4} with the choice 
$\alpha=\alpha_r \in (-1,0)$. Then for any constant $\underline{\mathbf{c}}>0$, there exist two positive constants $C$ and $c$ such that for 
any $n \in \N^*$ and any $j_0 \in \N^*$, there holds:
\begin{equation*}
\left| \mathbf{A}_r (-j_0+n\, |\alpha_r|,n) \right| \, \le \begin{cases}
C \, \exp \left( -c \, \dfrac{|j_0-n\, |\alpha_r||^{4/3}}{n^{1/3}} \right) \, ,& \text{\rm if $-j_0+n\, |\alpha_r| \le -\underline{\mathbf{c}} \, n$,} \\
C \, ,& \text{\rm if $-\underline{\mathbf{c}} \, n \le -j_0+n\, |\alpha_r| \le \underline{\mathbf{c}} \, n$.}
\end{cases}
\end{equation*}
and:
\begin{equation*}
\left| 1 - \mathbf{A}_r (-j_0+n\, |\alpha_r|,n) \right| \, \le \, C \, \exp \left( -c \, \dfrac{|j_0-n\, |\alpha_r||^{4/3}}{n^{1/3}} \right) \, ,\quad 
\text{\rm if $ -j_0+n\, |\alpha_r| \ge \underline{\mathbf{c}} \, n$.}
\end{equation*}
In particular, there holds :
\begin{equation*}
\sup_{j_0 \in \Z \, , \, n \in \N^*} \quad \left| \mathbf{A}_r (-j_0+n\, |\alpha_r|,n) \right| \, < \, + \infty \, .
\end{equation*}
\end{corollary}

%%%%%%%%%%%%%%%%
\section{Higher order estimates}
\label{sectionA-7}

Another crucial estimate that was needed in the analysis of Section \ref{section4-3} (see the proof of Lemma \ref{lem:estimateA}) aims 
at controlling the difference:
$$
\mathbf{A}_r \left( -j_0+n\, |\alpha_r|,\dfrac{j_0}{|\alpha_r|} \right) \, - \, \mathbf{A}_r \left( -j_0+n\, |\alpha_r|,n \right) \, ,
$$
with the function $\mathbf{A}_r$ defined in \eqref{A-deffonctionAr} and $j_0,n \in \N^*$. Unsurprisingly, the most direct way to estimate 
this difference is to apply the mean value theorem, which gives rise to the partial derivative of $\mathbf{A}_r$ with respect to its second 
variable. This leads us to introduce the family of \emph{correctors}:
\begin{equation}
\label{A-defcorrecteur}
\forall \, p \in \N^* \, ,\quad \forall \, (x,y) \in \R \times \R^{+*} \, ,\quad \mathbf{G}_p(x,y) \, := \, \dfrac{1}{2 \, \pi} \, 
\int_\R (\mbi \, \theta)^p \, {\rm e}^{\mbi \, x \, \theta} \, {\rm e}^{\mbi \, c_3 \, y \, \theta^3} \, {\rm e}^{-c_4 \, y \, \theta^4} \, {\rm d}\theta \, .
\end{equation}
The case $p=0$ corresponds to the definition \eqref{A-defGapproxn} of $\mathbf{G}$ and the relevance of the functions $\mathbf{G}_p$ 
for controlling the above difference between the two values of $\mathbf{A}_r$ will be the purpose of Corollary \ref{coro-A7} below. We aim 
here at generalizing Theorem \ref{thm-A3} and at obtaining sharp bounds on $\mathbf{G}_p$ for any $p \in \N^*$. Our result is the following.

\begin{theorem}
\label{thm-A4}
Let us assume that the coefficient $c_3$ in \eqref{A-defc3c4} is positive, that is $\alpha \in (0,1)$. Let $p \in \N^*$. Let $y_{\rm min}>0$ 
and let $\underline{\mathbf{c}}>0$ be given. Then there exist some constants $C>0$ and $c>0$ such that, for any $(x,y) \in \R \times 
[y_{\rm min},+\infty)$, there holds:
\begin{equation}
\label{A-bound5}
|\mathbf{G}_p(x,y)| \, \le \begin{cases}
\dfrac{C}{y^{(p+1)/4}} \, \exp (-c \, x^{4/3}/y^{1/3}) \, ,& \text{\rm if $x \ge \underline{\mathbf{c}} \, y$,} \\
\dfrac{C}{y^{1/3+p/4}} \, \exp (-c \, x^{3/2}/y^{1/2}) \, ,& \text{\rm if $0 \le x \le \underline{\mathbf{c}} \, y$,} \\
\dfrac{C}{y^{1/3+p/4}} \, ,& \text{\rm if $-y^{1/3} \le x \le 0$,} \\
\dfrac{C}{|x|^{1/4} \, y^{(p+1)/4}} \, \exp (-c \, x^2/y) \, ,& \text{\rm if $-\underline{\mathbf{c}} \, y \le x \le -y^{1/3}$,} \\
\dfrac{C}{y^{(p+1)/4}} \, \exp (-c \, |x|^{4/3}/y^{1/3}) \, ,& \text{\rm if $x \le -\underline{\mathbf{c}} \, y$.}
\end{cases}
\end{equation}
If $c_3$ is negative, the same estimate holds for $\mathbf{G}_p$ with $x$ being switched to $-x$.
\end{theorem}

\noindent Applying Theorem \ref{thm-A4} gives us the desired estimate for the difference between the two evaluations of $\mathbf{A}_r$, 
namely we have the following Corollary.

\begin{corollary}
\label{coro-A7}
Let the function $\mathbf{A}_r$ be defined in \eqref{A-deffonctionAr}. There exist two positive constants $C$ and $c$ such that for any 
$n \in \N^*$ and any $j_0 \in \N^*$ that satisfies $j_0 \in [n \, |\alpha_r|/2,n]$, there holds:
\begin{multline*}
\left| \mathbf{A}_r \left( -j_0+n\, |\alpha_r|,\dfrac{j_0}{|\alpha_r|} \right) \, - \, \mathbf{A}_r \left( -j_0+n\, |\alpha_r|,n \right) \right| \\
\le C \, \begin{cases}
\dfrac{1}{n^{1/3}} \, \exp \big( -c \, |j_0-n\, |\alpha_r||^{3/2}/n^{1/2}  \big) \, ,& \text{\rm if $j_0 \ge n\, |\alpha_r|$,} \\
 & \\
\dfrac{1}{n^{1/3}} \, ,& \text{\rm if $0 \le n\, |\alpha_r|-j_0 \le n^{1/3}$,} \\
 & \\
\dfrac{1}{|j_0-n\, |\alpha_r||^{1/4} \, n^{1/4}} \, \exp  \big( -c \, |j_0-n\, |\alpha_r||^2/n  \big) \, ,& \text{\rm if $n\, |\alpha_r|-j_0 \ge n^{1/3}$.}
\end{cases}
\end{multline*}
\end{corollary}

\begin{proof}[Proof of Theorem \ref{thm-A4}]
A very large part of the proof of Theorem \ref{thm-A4} follows that of Theorem \ref{thm-A3}. We therefore feel free to refer to the various steps 
of the proof of Theorem \ref{thm-A3} (that corresponds to the case $p=0$) and to shorten many details.
\bigskip

$\bullet$ \underline{Step 1: the uniform estimate}. We follow the same argument as in the proof of Proposition \ref{prop-A1} but, keeping 
similar notation, we now use the choice $f(\theta):=x \, \theta +c_3 \, y \, \theta^3$ and $g(\theta):=(\mbi \, \theta)^p \, \exp(-c_4 \, y \, \theta^4)$, 
so that we have the estimates:
$$
\| \, g \, \|_{L^\infty([a,b])} \, \le \, \dfrac{C}{y^{p/4}} \, ,\quad \| \, g' \, \|_{L^1([a,b])} \, \le \, \dfrac{C}{y^{p/4}} \, ,
$$
for any $y>0$, uniformly with respect to the interval $[a,b]$. By applying the same arguments as in the proof of Proposition \ref{prop-A1}, we 
get the uniform estimate:
\begin{equation}
\label{thm-A3-estimation1}
\forall \, y>0 \, ,\quad \sup_{x \in \R} \, |\mathbf{G}_p(x,y)| \, \le \, \dfrac{C}{y^{1/3+p/4}} \, ,
\end{equation}
with a constant $C$ that only depends on $p$.
\bigskip

$\bullet$ \underline{Step 2: the fast decaying side. Part I}. Changing the integration line $\R$ to $\mbi \, \mu +\R$ for any $\mu \in \R$ thanks to 
the Cauchy formula, we then apply the triangle inequality and get a similar bound as the one we had obtained in \eqref{borne-Cauchy-G}, namely:
\begin{multline}
\label{borne-Cauchy-Gp}
\forall \, \mu \in \R \, ,\quad \forall \, (x,y) \in \R \times \R^{+*} \, ,\\
|\mathbf{G}_p(x,y)| \, \le \, C_p \, {\rm e}^{-x \, \mu +c_3 \, y \, \mu^3 -c_4 \, y \, \mu^4} \, 
\int_\R (|\mu|^p+|\theta|^p) \, {\rm e}^{-3 \, y \, \mu \, (c_3-2\, c_4 \, \mu) \, \theta^2} \, {\rm e}^{-c_4 \, y \, \theta^4} \, {\rm d}\theta \, ,
\end{multline}
where the constant $C_p$ only depends on $p \in \N^*$ that is a given fixed integer. The important property is that $C_p$ does not depend on 
$x,y$ nor $\mu$.

Let us assume that $c_3$ is positive. We first consider the regime where $x$ is positive and such that the positive parameter $\mu_0$ defined 
by $\mu_0 :=(x/(3\, c_3 \, y))^{1/2}$ satisfies $2 \, c_4 \, \mu_0 \le c_3/2$, which corresponds to the constraint $0<x \le \mathbf{c}_\sharp \, y$ 
for some well-defined positive constant $\mathbf{c}_\sharp$ (the same one as in the proof of Proposition \ref{prop-A2}). Choosing the parameter 
$\mu_0$ in \eqref{borne-Cauchy-Gp}, we follow the same arguments as in the proof of Proposition \ref{prop-A2} and get the bound (compare 
with \eqref{borne-1-propA2}):
$$
|\mathbf{G}_p(x,y)| \, \le \, C \, \exp \left( -c \, \dfrac{x^{3/2}}{y^{1/2}} \right) \, \int_\R 
\left( \dfrac{x^{p/2}}{y^{p/2}}+|\theta|^p \right) \, {\rm e}^{-c \, x^{1/2} \, y^{1/2} \, \theta^2} \, {\rm d}\theta \, ,
$$
for suitable constants $C$ and $c$. Integrating with respect to $\theta$, we thus get the bound:
\begin{equation}
\label{thm-A3-estimation2}
|\mathbf{G}_p(x,y)| \, \le \, C \, \left\{ \dfrac{1}{y^{1/4+p/3} \, \max (1,x^{1/4})} + \dfrac{1}{y^{(p+1)/4} \, \max (1,x^{(p+1)/4})} \right\} \, 
\exp \left( -c \, \dfrac{x^{3/2}}{y^{1/2}} \right) \, ,
\end{equation}
which holds for any $y>0$ and $x \in (0,\mathbf{c}_\sharp \, y]$. We now argue similarly as in the proof of Corollary \ref{coro-A2}. 
We consider $y \ge y_{\rm min}>0$ and $x \in [0,\mathbf{c}_\sharp \, y]$. For $x \le y^{1/(3(1+p))}$, we use the uniform bound 
\eqref{thm-A3-estimation1} and for $x \in [y^{1/(3(1+p))},\mathbf{c}_\sharp \, y]$, we use \eqref{thm-A3-estimation2}. This 
combination of \eqref{thm-A3-estimation1} and \eqref{thm-A3-estimation2} gives the unified estimate:
\begin{equation}
\label{thm-A3-estimation3}
\forall \, y \ge y_{\rm min} \, ,\quad \forall \, x \in [0,\mathbf{c}_\sharp \, y] \quad 
|\mathbf{G}_p(x,y)| \, \le \, \dfrac{C}{y^{1/3+p/4}} \, \exp \left( -c \, \dfrac{x^{3/2}}{y^{1/2}} \right) \, ,
\end{equation}
where $C$ and $c$ are appropriate constants that do not depend on $y$ and $x$.
\bigskip

$\bullet$ \underline{Step 3: the fast decaying side. Part II}. The constant $\mathbf{c}_\sharp$ has been fixed and we consider the regime 
$y>0$, $x \ge \mathbf{c}_\sharp \, y$. We argue as in the proof of Proposition \ref{prop-A2} and first use Young's inequality to obtain:
\begin{equation}
\label{thm-A3-estimation4}
\forall \, \mu>0 \, ,\quad |\mathbf{G}_p(x,y)| \, \le \, C \, \left\{ \dfrac{\mu^{p-1/2}}{y^{1/2}} \, + \, \dfrac{1}{(\mu \, y)^{(p+1)/2}} \right\} \, 
\exp \, (f(\mu)) \, ,
\end{equation}
where $f$ is the convex function defined in \eqref{prop-A2-deff} whose minimum over $\R^+$ is attained at some $\underline{\mu}>0$. 
We have already shown the lower bound $\underline{\mu} \ge \mathbf{c}_\flat \, (x/y)^{1/3}$ in the proof of Proposition \ref{prop-A2} 
and the equality:
$$
\underbrace{3 \, c_3 \, y \, \underline{\mu}^2}_{\ge 0} \, + \, 32 \, c_4 \, y \, \underline{\mu}^3 \, = \, x \, ,
$$
directly gives the upper bound $\underline{\mu} \le \mathbf{C}_\flat \, (x/y)^{1/3}$ for yet another constant $\mathbf{C}_\flat$. 
Choosing the optimal parameter $\underline{\mu}$ in \eqref{thm-A3-estimation4} and following arguments as in the proof of 
Proposition \ref{prop-A2} for the upper estimate of $f(\underline{\mu})$, we end up with the estimate:
$$
|\mathbf{G}_p(x,y)| \, \le \, C \, \left\{ \dfrac{1}{y^{p/4+9/24}} \, + \, \dfrac{1}{y^{(p+1)/2}} \right\} \, 
\exp \left( -c \, \dfrac{x^{4/3}}{y^{1/3}} \right) \, ,
$$
for $x \ge \mathbf{c}_\sharp \, y$. Using now $y \ge y_{\rm min}$, we end up with the estimate:
\begin{equation}
\label{thm-A3-estimation5}
\forall \, y \ge y_{\rm min} \, ,\forall \, x \ge \mathbf{c}_\sharp \, y \quad 
|\mathbf{G}_p(x,y)| \, \le \, \dfrac{C}{y^{(p+1)/4}} \, \exp \left( -c \, \dfrac{x^{4/3}}{y^{1/3}} \right) \, ,
\end{equation}
for suitable constants $C$ and $c$.

It remains to argue as in Corollaries \ref{coro-A2} and \ref{coro-A3} to pass from a given constant $\mathbf{c}_\sharp$ to an arbitrary 
given constant $\underline{\mathbf{c}}>0$ given a priori. At his stage, we have already shown the validity of the first three estimates in 
\eqref{A-bound5}.
\bigskip

$\bullet$ \underline{Step 4: the oscillating side. Part I}. We still assume that $c_3$ is positive and now assume that $x$ is negative. We  
follow the proof of Proposition \ref{prop-A3} and consider the same contour deformation as the one depicted in Figure \ref{fig:contour-oscillations}. 
This gives rise to a decomposition:
$$
\mathbf{G}_p(x,y) \, = \, \varepsilon_1(x,y) \, + \, \varepsilon_2(x,y) \, + \, \mathscr{H}_\flat(x,y) \, + \, \mathscr{H}_\sharp(x,y) \, ,
$$
that is entirely similar to the one in \eqref{propA3-decomposition} except that the four integrals now incorporate the contribution of the 
polynomial factor $(\mbi \, \theta)^p$. For instance, we have (keeping the notation $\omega:=|x|/y$):
$$
\varepsilon_2(x,y) \, = \, \dfrac{1}{2 \, \pi} \, \exp \left( \dfrac{4 \, \omega^{3/2} \, y}{3 \, \sqrt{3\, c_3}} -\dfrac{c_4}{9\, c_3^2} \, \omega^2 \, y \right) 
\, \int_{\Xi(\omega)}^{+\infty} \left(\mbi \, \theta - \sqrt{\dfrac{\omega}{3\, c_3}} \right)^p \, {\rm e}^{\mbi \, \cdots} \, 
{\rm e}^{-\sqrt{3 \, c_3 \, \omega} \, y \, \theta^2 +\frac{2 \, c_4}{c_3} \, \omega \, y \, \theta^2} \, {\rm e}^{-c_4 \, y \, \theta^4} \, {\rm d}\theta \, ,
$$
where, again, the three dots within the integral stand for a \emph{real} quantity whose precise expression is useless, and $\Xi(\omega)$ 
stands for the quantity defined in \eqref{defXiomega}. Applying the triangle inequality yields the bound:
$$
|\varepsilon_2(x,y)| \, \le \, C \, \exp \left( \dfrac{4 \, \omega^{3/2} \, y}{3 \, \sqrt{3\, c_3}} \right) \, \int_{\Xi(\omega)}^{+\infty} 
\big( \theta^p \, + \, \omega^{p/2} \big) \, 
{\rm e}^{-\sqrt{3 \, c_3 \, \omega} \, y \, \theta^2 +\frac{2 \, c_4}{c_3} \, \omega \, y \, \theta^2} \, {\rm e}^{-c_4 \, y \, \theta^4} \, {\rm d}\theta \, ,
$$
where the constant $C$ does not depend on $\omega$ and $y$. We restrict again $\omega=|x|/y$ by imposing the condition 
\eqref{propA3-restriction1} so that we have:
$$
\dfrac{2 \, c_4}{c_3} \, \omega \, \le \, \dfrac{1}{2} \, \sqrt{3 \, c_3 \, \omega} \, .
$$
This restriction corresponds to an inequality $\omega \le \omega_0$ for some well-chosen constant $\omega_0>0$. This yields the estimate:
\begin{equation}
\label{thm-A3-estimepsilon2}
|\varepsilon_2(x,y)| \, \le \, C \, \exp \left( \dfrac{4 \, \omega^{3/2} \, y}{3 \, \sqrt{3\, c_3}} \right) \, \int_{2 \, \sqrt{\frac{\omega}{3\, c_3}}}^{+\infty} 
\big( \theta^p \, + \, \omega^{p/2} \big) \, {\rm e}^{-\frac{\sqrt{3 \, c_3 \, \omega}}{2} \, y \, \theta^2} \, {\rm d}\theta \, .
\end{equation}
We now use the following Lemma which follows from integration by parts and an induction argument (the proof is left to the reader).

\begin{lemma}
\label{lem-A1}
Let the sequence $(Q_k)_{k \in \N}$ of real polynomials be defined by:
\begin{align*}
Q_0(Y) &:= \, \dfrac{1}{2} \, ,\\
\forall \, k \in \N \, ,\quad Q_{k+1}(Y) &:= \, \dfrac{1}{2} \, Y^{k+1} \, + \, (k+1) \, Q_k(Y) \, .
\end{align*}
Then for any integer $\nu \in \N$ and for any real numbers $a>0$ and $X>0$, there holds:
$$
\int_X^{+\infty} \, \theta^\nu \, {\rm e}^{- \, a \, \theta^2} \, {\rm d}\theta \, \le \begin{cases}
a^{-(\nu+1)/2} \, Q_{(\nu-1)/2}(a \, X^2) \, {\rm e}^{- \, a \, X^2} \, ,&\text{\rm if $\nu$ is odd,}\\
a^{-\nu/2-1} \, X^{-1} \, Q_{\nu/2}(a \, X^2) \, {\rm e}^{- \, a \, X^2} \, ,&\text{\rm if $\nu$ is even.}
\end{cases}
$$
\end{lemma}

Let us assume for a moment that $p$ is odd. Applying Lemma \ref{lem-A1} in \eqref{thm-A3-estimepsilon2}, we obtain the estimate:
$$
|\varepsilon_2(x,y)| \, \le \, C \, \exp \left( -\dfrac{2 \, \omega^{3/2} \, y}{3 \, \sqrt{3\, c_3}} \right) \, 
\left( \dfrac{Q(\omega^{3/2} \, y)}{\omega^{(p+1)/4} \, y^{(p+1)/2}} \, + \, \dfrac{\omega^{p/2}}{\omega \, y} \right) \, ,
$$
where $Q$ is a real polynomial with nonnegative coefficients. Since the exponential term can absorb any polynomial expression of the 
same argument, we end up with the estimate:
\begin{equation}
\label{thm-A3-estimepsilon2-impair}
|\varepsilon_2(x,y)| \, \le \, C \, \exp \left( -c \, \omega^{3/2} \, y \right) \, 
\left( \dfrac{1}{\omega^{(p+1)/4} \, y^{(p+1)/2}} \, + \, \dfrac{\omega^{p/2}}{\omega \, y} \right) \, ,
\end{equation}
if $p$ is odd. If $p$ is even, applying Lemma \ref{lem-A1} in \eqref{thm-A3-estimepsilon2} yields the final estimate:
\begin{equation}
\label{thm-A3-estimepsilon2-pair}
|\varepsilon_2(x,y)| \, \le \, C \, \exp \left( -c \, \omega^{3/2} \, y \right) \, 
\left( \dfrac{1}{\omega^{p/4+1} \, y^{p/2+1}} \, + \, \dfrac{\omega^{p/2}}{\omega \, y} \right) \, .
\end{equation}
Of course, there is a similar estimate for the contribution $\varepsilon_1(x,y)$ that is the complex conjugate of $\varepsilon_2(x,y)$.

Let us now turn to the contribution $\mathscr{H}_\flat(x,y)$ that corresponds to the inclined segment on the left in Figure 
\ref{fig:contour-oscillations}. Keeping the notation of the proof of Proposition \ref{prop-A3}, we have:
\begin{equation*}
\mathscr{H}_\flat(x,y) \, = \, \dfrac{{\rm e}^{p_0(\omega) \, y-\mbi \, \pi/4}}{2 \, \pi} \, 
\int_{-\sqrt{\frac{2 \, \omega}{3 \, c_3}}-\frac{2 \, \sqrt{2} \, c_4}{9 \, c_3^2} \, \omega}^{\sqrt{\frac{2 \, \omega}{3 \, c_3}}} 
\big( \mbi \, \Theta(t) \big)^p \, \exp \Big( y \, \sum_{k=1}^4 p_k(\omega) \, t^k \Big) \, {\rm d}t \, ,
\end{equation*}
where $(t \mapsto \Theta(t))$ parametrizes the segment so that we have a uniform estimate:
$$
|\Theta (t)| \, \le \, C \, \sqrt{\omega} \, ,
$$
for $\omega \le \omega_0$.

Instead of trying to isolate the leading contribution in $\mathscr{H}_\flat(x,y)$ as we did in the proof of Proposition \ref{prop-A3}, we rather 
apply the triangle inequality and use the behavior \eqref{propertiespj} of $p_0,\dots,p_4$ when $\omega$ is small. Starting from:
\begin{equation*}
|\mathscr{H}_\flat(x,y)| \, \le \, C \, \omega^{p/2} \, {\rm e}^{(\text{\rm Re } p_0(\omega)) \, y} \, 
\int_{-\sqrt{\frac{2 \, \omega}{3 \, c_3}}-\frac{2 \, \sqrt{2} \, c_4}{9 \, c_3^2} \, \omega}^{\sqrt{\frac{2 \, \omega}{3 \, c_3}}} 
\exp \Big( y \, \sum_{k=1}^4 (\text{\rm Re }p_k(\omega)) \, t^k \Big) \, {\rm d}t \, ,
\end{equation*}
we use \eqref{p1} to first absorb the term $(\text{\rm Re }p_1(\omega)) \, t=\mathcal{O}(\omega^{5/2})$ (uniformly with respect to $t$) 
by half of $\text{\rm Re } p_0(\omega)$. We can then absorb $(\text{\rm Re }p_3(\omega)) \, t^3$ on the considered interval by part of 
$(\text{\rm Re }p_2(\omega)) \, t^2$ (this argument was also used in the proof of Proposition \ref{prop-A3}). The final term $(\text{\rm Re } 
p_4(\omega)) \, t^4$ is $\mathcal{O}(\omega \, t^2)$ so it can also be absorbed by half of what is left from $(\text{\rm Re }p_2(\omega)) \, t^2$ 
(up to choosing $\omega$ small enough). In the end we get an estimate (with uniform constants $C$ and $c$):
\begin{equation*}
|\mathscr{H}_\flat(x,y)| \, \le \, C \, \omega^{p/2} \, {\rm e}^{- c \, \omega^2 \, y} \, 
\int_{-\sqrt{\frac{2 \, \omega}{3 \, c_3}}-\frac{2 \, \sqrt{2} \, c_4}{9 \, c_3^2} \, \omega}^{\sqrt{\frac{2 \, \omega}{3 \, c_3}}} 
\exp \Big( -c \, \omega^{1/2} \, y \, t^2 \Big) \, {\rm d}t \, .
\end{equation*}
Estimating crudely the integral over the segment by the integral of the same function over $\R$, we end up with:
\begin{equation*}
|\mathscr{H}_\flat(x,y)| \, \le \, C \, \dfrac{\omega^{p/2-1/4}}{y^{1/2}} \, {\rm e}^{- c \, \omega^2 \, y} \, ,
\end{equation*}
and the same estimate holds for $\mathscr{H}_\sharp(x,y)$ that is the complex conjugate of $\mathscr{H}_\flat(x,y)$.

Let us assume that $p$ is odd. Combining the latter estimate with \eqref{thm-A3-estimepsilon2-impair}, we get:
$$
|\mathbf{G}_p(x,y)| \, \le \, C \, \dfrac{\omega^{p/2-1/4}}{y^{1/2}} \, {\rm e}^{- c \, \omega^2 \, y} 
\, + C \, {\rm e}^{-c \, \omega^{3/2} \, y} \, \left( \dfrac{1}{\omega^{(p+1)/4} \, y^{(p+1)/2}} \, + \, \dfrac{\omega^{p/2}}{\omega \, y} \right) \, .
$$
Since $\omega$ has been chosen smaller than some constant $\omega_0$, there is no loss of generality in assuming $\omega_0 \le 1$ 
and we therefore have:
$$
{\rm e}^{-c \, \omega^{3/2} \, y} \, \le \, {\rm e}^{-c \, \omega^2 \, y} \, .
$$
We recall the definition $\omega=|x|/y$ and rewrite the latter estimate in terms of $x$ and $y$ to obtain:
$$
|\mathbf{G}_p(x,y)| \, \le \, C \, \exp \left(-c \, \dfrac{x^2}{y} \right) \, \left\{ 
\left( \dfrac{|x|}{\sqrt{y}} \right)^{p/2} \, \dfrac{1}{|x|^{1/4} \, y^{(p+1)/4}} \, + \, \dfrac{1}{|x|^{(p+1)/4} \, y^{(p+1)/4}} \, + \, 
\left( \dfrac{|x|}{\sqrt{y}} \right)^{p/2} \, \dfrac{1}{|x| \, y^{p/4}} \right\} \, .
$$
We can absorb all polynomial expressions of $|x|/\sqrt{y}$ by the Gaussian function and we also use the inequalities:
$$
|x| \, \ge \, |x|^{1/4} \, y^{1/4} \, ,\quad |x|^{p/4} \, \ge \, y_{\rm min}^{p/12} >0 \, ,
$$
that hold for $|x| \ge y^{1/3}$ and $y \ge y_{\rm min}$. Eventually, we have shown that there exists some small constant $\mathbf{c}_\flat>0$ 
and some constants $C$ and $c$ such that there holds:
\begin{equation}
\label{thm-A3-estimation6}
\forall \, y \ge y_{\rm min} \, ,\quad \forall \, x \in [-\mathbf{c}_\flat \, y,-y^{1/3}] \, ,\quad 
|\mathbf{G}_p(x,y)| \, \le \, \dfrac{C}{|x|^{1/4} \, y^{(p+1)/4}} \, \exp \left( -c \, \dfrac{x^2}{y} \right) \, .
\end{equation}
The same kind of arguments lead to the estimate \eqref{thm-A3-estimation6} in the case where $p$ is even (starting now from 
\eqref{thm-A3-estimepsilon2-pair}).
\bigskip

$\bullet$ \underline{Step 5: the oscillating side. Part II}. There is no real difficulty in adapting the proof of Proposition \ref{prop-A4} to this 
slightly more general framework that incorporates the factor $(\mbi \, \theta)^p$ in the definition \eqref{A-defcorrecteur} of $\mathbf{G}_p$. 
By following the same arguments as in the proof of Proposition \ref{prop-A4} and absorbing polynomial terms by exponentially decaying ones, 
we can show that there exists a constant $\mathbf{C}_\flat>\mathbf{c}_\flat$ and some  
\begin{equation}
\label{thm-A3-estimation7}
\forall \, y \ge y_{\rm min} \, ,\quad \forall \, x \le -\mathbf{C}_\flat \, y \, ,\quad 
|\mathbf{G}_p(x,y)| \, \le \, \dfrac{C}{y^{(p+1)/4}} \, \exp \left( -c \, \dfrac{|x|^{4/3}}{y^{1/3}} \right) \, .
\end{equation}
\bigskip

$\bullet$ \underline{Step 6: the oscillating side. Part III}. It remains to deals with the case $x \in [-\mathbf{C}_\flat \, y,-\mathbf{c}_\flat \, y]$ 
and this is done by merely adapting Proposition \ref{prop-A5}. We choose again the contour depicted in Figure \ref{fig:contour} so that along 
the two inclined segments, the complex number $\theta$ is $\mathcal{O}(\delta)$ and the parameter $\delta$ is chosen such that it satisfies 
\eqref{prop-A5-choixdelta}. This choice, that is uniform with respect to $x$ and $y$ in the considered regime, allows us to absorb the term 
$(\mbi \, \theta)^p$ into a constant along the two inclined segments. For the integrals along the two horizontal half-lines, the polynomial factor 
$(\mbi \, \theta)^p$ simply gives an algebraic factor that is harmless when compared with the exponentially decaying term $\exp (-c \, y)$. 
Overall, we leave as an exercise to the interested reader to prove that for suitable constants $C$ and $c$, there holds:
\begin{equation}
\label{thm-A3-estimation8}
\forall \, y \ge y_{\rm min} \, ,\quad \forall \, x \in [-\mathbf{C}_\flat \, y,-\mathbf{c}_\flat \, y] \, ,\quad 
|\mathbf{G}_p(x,y)| \, \le \, \dfrac{C}{y^{(p+1)/4}} \, \exp (-c \, y) \, .
\end{equation}
We can then use the above estimates \eqref{thm-A3-estimation6}, \eqref{thm-A3-estimation7}, \eqref{thm-A3-estimation8} and adapt the arguments 
of Corollaries \ref{coro-A2} and \ref{coro-A3} to show that for any given constant $\underline{\mathbf{c}}>0$, there exist constants $C$ and $c$ 
such that the estimates corresponding to the last two cases of \eqref{A-bound5} are valid. This completes the proof of Theorem \ref{thm-A3}.
\end{proof}

\begin{proof}[Proof of Corollary \ref{coro-A7}]
From the definition \eqref{A-deffonctionAr} of $\mathbf{A}_r$ and the definition \eqref{A-defcorrecteur} of the correctors $\mathbf{G}_p$, 
we have\footnote{We first differentiate the definition \eqref{A-deffonctionAr} with respect to $y$ and then apply once again the Cauchy formula 
to let the abscissa $\eta$ tend to zero.}:
$$
\dfrac{\partial \mathbf{A}_r}{\partial y} (x,y) \, = \, - \, c_3 \, \mathbf{G}_2(x,y) \, - \, c_4 \, \mathbf{G}_3(x,y) \, .
$$
We now consider $n \in \N^*$ and $j_0 \in \N^*$ such that $j_0$ belongs to the segment $[n \, |\alpha_r|/2,n]$. We observe that the segment 
$[j_0/|\alpha_r|,n]$ is included in $[n/2,n/|\alpha_r|]$ so the mean value theorem gives the bound:
\begin{multline}
\label{estim-coroA6-1}
\left| \mathbf{A}_r \left( -j_0+n\, |\alpha_r|,\dfrac{j_0}{|\alpha_r|} \right) \, - \, \mathbf{A}_r \left( -j_0+n\, |\alpha_r|,n \right) \right| \\
\le \, C \, |n\, |\alpha_r|-j_0| \, \sup_{y \in [n/2,n/|\alpha_r|]} \, |\mathbf{G}_2(-j_0+n\, |\alpha_r|,y)| \, + \, 
C \, |n\, |\alpha_r|-j_0| \, \sup_{y \in [n/2,n/|\alpha_r|]} \, |\mathbf{G}_3(-j_0+n\, |\alpha_r|,y)| \, .
\end{multline}
Corollary \ref{coro-A7} then follows by applying Theorem \ref{thm-A4}, keeping in mind that $\alpha_r$ is negative (so the relevant constant 
$c_3$ is negative) and that the relevant values of $x$ and $y$ satisfy here $|x|/y \le \underline{\mathbf{c}}$ for some constant 
$\underline{\mathbf{c}}$ that only depends on $\alpha_r$ (this is because of the bounds on $j_0$ in terms of $n$). Actually, the most critical 
case arises in the right-hand side of \eqref{estim-coroA6-1} with the term:
$$
|n\, |\alpha_r|-j_0| \, \sup_{y \in [n/2,n/|\alpha_r|]} \, |\mathbf{G}_2(-j_0+n\, |\alpha_r|,y)| \, ,
$$
in the regime $n\, |\alpha_r|-j_0 \ge n^{1/3}$. We then use Theorem \ref{thm-A4} for $p=2$ and obtain a bound of the form:
$$
C \, |n\, |\alpha_r|-j_0| \, \dfrac{1}{|n\, |\alpha_r|-j_0|^{1/4} \, n^{3/4}} \, \exp \left( -c \, \dfrac{|n\, |\alpha_r|-j_0|^2}{n} \right) \, ,
$$
so the Gaussian term can absorb the polynomial expression $|n\, |\alpha_r|-j_0|/\sqrt{n}$ and this gives a bound:
$$
\dfrac{C}{|n\, |\alpha_r|-j_0|^{1/4} \, n^{1/4}} \, \exp \left( -c \, \dfrac{|n\, |\alpha_r|-j_0|^2}{n} \right) \, ,
$$
just like we had in Corollary \ref{coro-A1} for the free Green's function. In all other regimes, the situation is more favorable and there are extra 
positive powers of $n$ that can even be omitted in the end.
\end{proof}

\bibliographystyle{plain}
\bibliography{CF}
\end{document}